\documentclass[arma,numbook,envcountsame,final]{svjour}
\usepackage{amssymb}
\usepackage{times}

\title{Lipschitz Bounds and\\ Nonautonomous Integrals}

\author{Cristiana De Filippis \& Giuseppe Mingione}

\usepackage{verbatim}
\usepackage[T1]{fontenc}
\usepackage[notref,notcite]{showkeys}
\usepackage{graphicx}
\usepackage{enumerate}
\usepackage{amsmath,amsfonts,amssymb}
\usepackage{color}
\usepackage{cite}
\definecolor{citation}{rgb}{0,0,0} 
\definecolor{formula}{rgb}{0,0,0}
\definecolor{url}{rgb}{0,0,0}
\usepackage{pgf,tikz}
\usepackage{mathrsfs}
\usepackage{fancyhdr}
\usepackage{dutchcal}
\usepackage{dsfont}
\usepackage{hyperref}
\newcommand{\reqnomode}{\tagsleft@false}

\def\dx{\, dx}

\def \d{\, d }
\def \diver{\,{\rm div}}
\def\dist{\,{\rm dist}}

\def\supp{\,{\rm supp}}
\newcommand\R{\mathbb{R}}

\newcommand\loc{{\rm loc}}
\newcommand\setm{\texttt{set}}

\newcommand\setuni{\texttt{set}_{\textnormal{u}}}
\newcommand\data{\texttt{data}}
\newcommand\datai{\data}
\newcommand{\laz}{\langle z \rangle}
\newcommand\ia{i_{\textnormal{a}}}
\newcommand\sa{s_{\textnormal{a}}}
\newcommand\tcu{\tilde{c}_{\textnormal{u}}}
\newcommand\ccc{\mathfrak{c}}
\newcommand\cu{c_{\textnormal{u}}}
\newcommand\ca{c_{\textnormal{a}}}
\newcommand\cur{c_{\textnormal{r}}}
\newcommand\cb{c_{\textnormal{b}}}
\newcommand\B{\mathcal B}
\newcommand\datauni{\texttt{data}_{\textnormal{u}}}
\newcommand\rad{\texttt{r}}
\newcommand\oo{\texttt{o}}

\def\ccp{(\cdot)}

\allowdisplaybreaks
\makeatletter
\DeclareRobustCommand*{\bfseries}{%
  \not@math@alphabet\bfseries\mathbf
  \fontseries\bfdefault\selectfont
  \boldmath
}

\makeatother

\newlength{\defbaselineskip}
\newcommand{\mint}{\mathop{\int\hskip -1,05em -\, \!\!\!}\nolimits}

\numberwithin{equation}{section}
\allowdisplaybreaks[4]

\newcommand{\kk}{\kappa}

\def\en{\mathbb N}
\def\er{\mathbb R}

\newcommand\eps\varepsilon

\def\eqn#1$$#2$${\begin{equation}\label#1#2\end{equation}}

\newcommand{\go}{\mathfrak{g}}

\newcommand{\be}{\begin{equation}}
\newcommand{\ee}{\end{equation}}

\newcommand{\rr}{\varrho}

\newcommand{\snr}[1]{\lvert #1\rvert}

\newcommand{\nr}[1]{\lVert #1 \rVert}

\def\er{\mathbb R}

\newcommand{\RN}{\mathbb{R}^{N}}

\newcommand{\N}{\mathbb{N}}

\def\name[#1, #2]{#1 #2}

\newcommand{\rif}[1]{(\ref{#1})}
\newcommand{\trif}[1] {\textnormal{\rif{#1}}}

\newcommand{\stackleq}[1]{\stackrel{\rif{#1}}{ \leq}}

\allowdisplaybreaks

\begin{document}

%\subjclass[2010]{49N60, 35J60\vspace{1mm}} %%ALERT CHECK 35J60 23J70 35B65 35D40

%\keywords{Regularity, nonautonomous functionals, nonuniform ellipticity\vspace{1mm}}

\maketitle

\begin{abstract}
We provide a general approach to Lipschitz regularity of solutions for a large class of vector-valued, nonautonomous variational problems exhibiting nonuniform ellipticity. The functionals considered here range from those with unbalanced polynomial growth conditions to those with fast, exponential type growth. The results obtained are sharp with respect to all the data considered and yield new, optimal regularity criteria also in the classical uniformly elliptic case. We give a classification of different types of nonuniform ellipticity, accordingly identifying suitable conditions to get regularity theorems.
\end{abstract}
\setcounter{tocdepth}{1}
%\vspace{3mm}
\tableofcontents

\section{Introduction}
The aim of this paper is to provide a comprehensive treatment of Lipschitz regularity of solutions for a very large class of vector-valued nonautonomous variational problems, involving integral functionals of the type
\eqn{genF}
$$
W^{1,1}_{\loc}(\Omega;\er^N) \ni w \mapsto \mathcal F(w;\Omega):=  \int_{\Omega}  \left[ F(x,Dw)-f\cdot w \right]\, dx\,.
$$ 
Here $\Omega \subset \er^n$ is an open subset with $n\geq 2$ and $N\geq 1$. In the following we shall assume the structure condition $F(x,Dw)\equiv \tilde F(x,|Dw|)$, which is natural in the vectorial case, where $\tilde F: \Omega \times [0, \infty) \to [0, \infty)$ is a suitably regular function (see Section \ref{asssec1} below for the precise assumptions). 
The vector field $f\colon \Omega \mapsto \er^N$ will be at least $L^n$-integrable
\begin{flalign}\label{fff}
f\in L^{n}(\Omega;\RN)\,.
\end{flalign}
The notion of local minimizer used in this paper is quite standard in the literature.
\begin{definition}\label{defi-min} A map $u \in W^{1,1}_{\loc}(\Omega;\er^N)$ is a local minimizer of the functional $\mathcal F$ in~\eqref{genF} with $f \in L^n(\Omega;\R^N)$ if, for every open subset $\tilde \Omega\Subset \Omega$, we have $\mathcal F(u;\tilde \Omega) <\infty$ and $\mathcal F(u;\tilde \Omega)\leq \mathcal F(w;\tilde \Omega)$ holds for every competitor $w \in u + W^{1,1}_0(\tilde \Omega; \er^N)$. 
\end{definition}
In the rest of the paper we shall abbreviate local minimizer simply by minimizer. We just remark that, thanks to \rif{fff} and Sobolev embedding, requiring $\mathcal F(u;\tilde \Omega) <\infty$ for every $\tilde \Omega$ in Definition \ref{defi-min} is the same to requiring $F(\cdot, Du)\in L^{1}_{\loc}( \Omega)$. In this paper we deal with the following, classical
\vspace{1mm}

\noindent {\bf Problem.} {\em Find minimal regularity assumptions on $f$ and $x \mapsto F(x, \cdot)$, guaranteeing local Lipschitz continuity of minima of the functional $\mathcal F$ in \trif{genF}, provided this type of regularity holds when $f\equiv 0$ and no $x$-dependence occurs, i.e., $F(x,z)\equiv F(z)$.}
\vspace{1mm}

Eventually, Lipschitz continuity opens the way to higher order regularity. In the problem above, a situation where no dependence on $x$ typically occurs when considering frozen integrands of the type
$z \mapsto F(x_0, z)$, for some $x_0\in \Omega$. We recall that, under suitable growth conditions, the analysis of \rif{genF} usually involves the related Euler-Lagrange system 
\eqn{general1}
$$
-\diver(\tilde a(x,|Du|)Du)=f\,, \quad \ \tilde a(x,|Du|):= \frac{\tilde F'(x,|Du|)}{|Du|}\,.
$$
Here $\tilde F'$ denotes the partial derivative of $\tilde F$ with respect to the second variable.  

Recently, Cianchi \& Maz'ya \cite{CM0, CM1, CM2, CM3, CM4} (global estimates) and Kuusi and the second author \cite{KMstein, KMbull, KMjems} (local estimates), investigated the above problem in the uniformly elliptic, autonomous case, i.e. $\tilde a(x,|Du|)\equiv \tilde a(|Du|)$ and 
\eqn{uni-ell}
$$
\left\{
\begin{array}{c}
\displaystyle -1< \ia \leq \frac{\tilde a'(t)t}{\tilde a(t)}\leq \sa< \infty \quad \mbox{for every $t>0$}\\ [10 pt]
\mbox{$\tilde a\colon (0, \infty) \to [ 0, \infty)$ is of class $C^1_{\loc}(0, \infty)$}
\,.
\end{array}
\right.
$$ 
A special, yet important model, is given by the $p$-Laplacean system with coefficients 
\eqn{plap}
$$
-\diver(\ccc(x)|Du|^{p-2}Du)=f  \,, \quad p>1\,, \quad 0 < \nu \leq \ccc(\cdot) \leq L\,. 
$$
For this we have 
\vspace{1mm}

\noindent {\bf Nonlinear Stein Theorem} \cite{KMstein}. {\em Let $u\in W^{1, p}_{\loc}(\Omega; \er^N)$ be a weak solution to \eqref{plap}. If $f \in L(n,1)(\Omega;\er^N)$, and $\ccc(\cdot)$ is Dini continuous, then $Du$ is continuous. }

\vspace{0.5mm}
In particular, $Du$ is locally bounded. We recall that $f \in L(n,1)(\Omega;\er^N)$ means that \eqn{cond-lo}
$$
\| f\|_{L(n,1)(\Omega)}:=  \int_0^\infty |\{x \in \Omega\,  \colon  |f(x)|> \lambda\}|^{1/n} \, d\lambda <\infty\,,
$$ and also that $L^{q}\subset L(n,1) \subset L^n$ for every $q>n$. Moreover, denoting by $\omega(\cdot)$ the modulus of continuity of $\ccc(\cdot)$, the Dini continuity of $\ccc(\cdot)$ amounts to require that 
\eqn{dini}
$$\int_0 \omega(\varrho)\, \frac{d \varrho}{\varrho} <\infty\,.$$
The above theorem extends to general equations \cite{KMbull} and to systems depending on forms \cite{sil}; it also extends classical results of Uhlenbeck \cite{Uh} and Uraltseva \cite{Ur}; we again refer to Cianchi \& Maz'ya \cite{CM0, CM1, CM3} for global statements. The terminology is motivated by the fact that, for $\ccc(\cdot)\equiv 1$ and $p=2$, this is another classical result of Stein \cite{St}. It is optimal both with respect to condition \rif{cond-lo}, as shown by Cianchi \cite{CiGA}, and with respect to \rif{dini}, as shown by Jin, Maz'ya \& Van Schaftingen \cite{mazyaex}. The relevant fact here is that the conditions on $f$ and $\ccc(\cdot)$ implying local Lipschitz continuity are independent of $p$. In fact, when considering more general equations, they are independent of the vector field the divergence operator is applied to; for this, see \cite{KMbull}, and \cite{Baroni} for conditions \rif{general1}-\rif{uni-ell}. 

In the case of nonuniformly elliptic operators, the problem of deriving sharp conditions with respect to data for Lipschitz regularity is considerably more difficult. This has been attacked only recently by Beck and the second author \cite{BM}, but only for the case of autonomous functionals in the principal part, i.e., $F(x,z)\equiv F(z)$ (with some abuse of terminology, in this paper we shall refer to $F(x,z)\equiv F(z)$ as to the autonomous case, no matter $f(\cdot)$ can still be present in \rif{genF}, to emphasize that coefficients $x$ do not appear in the part of the integrand containing gradient terms). The outcome of \cite{BM} is that, when $n>2$, condition \rif{cond-lo} is still sufficient to guarantee the local Lipschitz regularity of minima, thereby revealing itself as a sort of universal property. In the case $n=2$, the alternative, slightly stonger borderline condition $L^{2}(\log L)^{\mathfrak {a}}(\Omega; \er^N)$ with $\mathfrak {a} >2$, implies Lipschitz continuity:
\eqn{llog}
$$
f \in  L^{2}(\log L)^{\mathfrak {a}}(\Omega; \er^N) \Longleftrightarrow\int_{\Omega} |f|^2\log^{\mathfrak {a}}\left(\textnormal{e}+|f|\right)\, dx < \infty\,.
$$
Therefore, in the nonuniformly, autonomous case, the condition on $f$ can be summarized as 
\eqn{lox}
$$|f|\in \mathfrak X(\Omega)=
\left\{
\begin{array}{cc} 
L(n,1)(\Omega)  & \mbox{ if $n>2$} \\ [4 pt]
L^{2}(\log L)^{\mathfrak {a}}(\Omega) \,,   \ \mathfrak {a} >2&  \mbox{ if $n= 2$}\,.
\end{array}
\right.
$$
In this paper we deal with the general, fully nonautonomous case \rif{genF}. This is by no means a technical extension as, in fact, when passing to the nonuniformly elliptic case, the role of coefficients drastically changes and they can no longer be treated via perturbation as in \cite{KMstein}. To give a glimpse of the situation, let us consider the so called double phase functional 
\eqn{pqfunctional}
$$
\begin{cases}
\displaystyle w \mapsto \int_{\Omega} \left[H(x,Dw)-f\cdot w\right]\dx\\
\displaystyle  H(x,z):= 
\tilde H(x,|z|):=|z|^p/p+a(x)|z|^{q}/q
\end{cases}
$$
with  $1 < p < q$, $0 \leq a(\cdot)\in L^{\infty}(\Omega)$. This functional has been originally introduced by Zhikov \cite{Z1, Z2} in the setting of Homogeneization of strongly anisotropic materials, and the corresponding regularity theory has been studied at length starting by \cite{BCM, CoM0, CoM}. The functional in \rif{pqfunctional} changes its rate of ellipticity/coercivity - from $p$ to $q$ - around the zero set $\{a(x)=0\}$. As shown in \cite{sharp, FMM}, already when $f\equiv 0$, local minima fail to be continuous if the ratio $q/p$ is too far from $1$, in dependence on the rate of H\"older continuity $\alpha$. Specifically, the condition
\eqn{pqbound}
$$
\frac qp \leq 1 + \frac \alpha n\,, \qquad  a(\cdot)\in C^{0, \alpha}(\Omega)\,, \qquad \alpha \in (0,1]
$$
is necessary \cite{sharp, FMM} and sufficient \cite{BCM} to get gradient local continuity, thereby linking growth conditions of the integrand with respect to the gradient variable, to the smoothness of coefficients. In particular, classical Schauder's theory generally fails. This is the main theme of this paper. 
Condition in \rif{pqbound} reveals a typical phenomenon occurring when nonuniform ellipticity is directly generated by the presence of the $x$-variable as in \rif{pqfunctional}. In this case, it is indeed the very presence of $x$ making functionals as in \rif{pqfunctional} fail to meet the standard, two-sided polynomial conditions with the same exponent, i.e., $H(x,z)\not \approx |z|^p$. We shall also deal with more drastic examples of such an interplay, as for instance 
\eqn{exp-mod1}
$$ w \mapsto \int_{\Omega} \left[\ccc_1(x)\exp(\ccc_{2}(x)|Dw|^p)-f\cdot w\right]\, dx\,, \qquad p>1\,,$$
where $0 < \nu\leq \ccc_1(\cdot), \ccc_{2}(\cdot)\leq L$. Here the dependence on $x$ becomes even more delicate as it makes the ellipticity rate vary more drastically.    
Such integrands fail to satisfy the so-called $\triangle_2$-condition, i.e., $\tilde F(x, 2 t)\not \lesssim \tilde F(x, t)$. This reflects in a loss of related integrability conditions on minimizers as one tries to use perturbation methods, that is, considering a specific point $x_0 \in \Omega$ and making small variations of $x$ around $x_0$. In other words, $\exp(\ccc_{2}(\cdot)|Dw|^p)\in L^1$ does not necessarily imply $\exp(\ccc_{2}(x_0)|Dw|^p)\in L^1$, and vice versa, and plain perturbation methods are again banned. Exponential type functionals are classical in the Calculus of Variations starting by the work of Duc \& Eells \cite{DE} and Marcellini \cite{M2}. In the nonautonomous version, they are treated for instance in the setting of weak KAM-theory by Evans \cite{evans}, but under special assumptions and boundary conditions. More recent progress is in \cite{DM}, for $f\equiv 0$. 

Nonuniform ellipticity is a very classical topic in partial differential equations, and it is often motivated by geometric and physical problems. Seminal papers on this subject are for instance \cite{DE, LU, Simon, Z1, Z2}; a classical monograph is Ivanov's \cite{ivanov}. In the setting of the Calculus of Variations there is a wide literature available, starting from the basic papers of Uraltseva \& Urdaletova \cite{UU} and Marcellini \cite{M1, M2, M3, M4}. More recently, the study of the nonautonomous case has intensified; many papers have been devoted to study specific structures as well as genereal non-uniformly elliptic problems \cite{BS, Bi, BM, BF1, BB2, BOR, BO, BY, Cellinastaicu, DeMi1, DeMi, DM, HO, HS}. Connections to related function spaces have been studied too \cite{diening, HH, RR}.  

The results obtained in this paper are very general and cover large classes of different models cases simultaneously. For this, a number of technical assumptions is necessary; see Section \ref{general results} below. Anyway, when applied to single models, such assumptions reveal to be minimal and produce sharp results. In the autonomous case $F(x, z)\equiv F(z)$, they coincide with the sharp ones introduced in \cite{BM}.  For this reason, and also to ease the reading, in this introductory part we shall present a few main corollaries of the general theorems, in connection to some relevant instances of nonuniformly elliptic functionals often considered in the literature. These models fall in three different general classes, detailed in Sections \ref{pqsec}-\ref{unisec0} below. We refer the reader to Section \ref{notsec} for a full account of the notation used in this paper, while more remarks on nonuniform ellipticity are in Section \ref{diffsec} below. 

\subsection{Nonuniform ellipticity at polynomial rates}\label{pqsec} We start considering functionals with unbalanced polynomial growth conditions, of so-called $(p,q)$-type after Marcellini \cite{M1, M2}. The idea is to provide general conditions on the partial map $x \mapsto F(x, \cdot)$ implying regularity of minima and matching those suggested by counterexamples \cite{sharp, FMM}. In this respect, we consider, for exponents $1 < p \leq q$, the conditions 
\eqn{asp1}
$$
\left\{
\begin{array}{c}
\ F(x,z)=\tilde{F}(x,\snr{z})  \quad \mbox{for all} \ (x,z)\in \Omega\times \mathbb{R}^{N\times n}\\ [3 pt]
\nu (|z|^2+\mu^2)^{p/2} \leq F(x,z) \leq \Lambda(|z|^2+\mu^2)^{q/2}+\Lambda(|z|^2+\mu^2)^{p/2}\\[3 pt]
(|z|^2+\mu^2) |\partial_{zz} F(x,z)| \leq \Lambda(|z|^2+\mu^2)^{q/2}+\Lambda(|z|^2+\mu^2)^{p/2}\\ [3 pt]
\nu (|z|^2+\mu^2)^{(p-2)/2}|\xi|^2\leq   \partial_{zz}F(x,z)\xi \cdot \xi\,, 
 \end{array}\right.
$$
for every choice of $z, \xi \in \er^n$ such that $|z|\not=0$. Here $0<\nu\leq 1 \leq \Lambda$ and $\mu \in [0,1]$ are fixed constants. $\tilde{F}$ satisfies \rif{0ass}$_{2,3}$ below. We moreover assume that
\eqn{nondecresce}
 $$t \mapsto \frac{\tilde F'(x,t)}{(t^2+\mu^2)^{(p-2)/2}t}\quad \mbox{is non-decreasing}$$
for every $x\in \Omega$; as in the rest of the paper, we denote $\tilde F'\equiv \partial_t \tilde F$.   
As for the crucial dependence on $x$, we assume that \eqn{asp2}
$$
\begin{cases}
\ |\partial_{x} \tilde F'(x,t)| \leq h(x)
\left[(t^2+\mu^2)^{(q-2)/2}t+(t^2+\mu^2)^{(p-2)/2}t\right]\\
\   h(\cdot) \in L^{d}(\Omega), \  d >n
\end{cases}
$$
holds for a.e.\,$x \in \Omega$ and every $t>0$. Of course, $F, \tilde F',$ are also assumed to be continuous on $\Omega \times (0,\infty)$, while $\tilde F''$ and $\partial_x \tilde F$ are Carath\'eodory regular. Using Sobolev regularity on coefficients, is a natural approach, also considered elsewhere. See for instance the paper \cite{KM} in the case of uniformly elliptic integrals. As for the nonuniformly elliptic setting, this approach has been used in \cite{EMM}; see also Marcellini's survey \cite{M4}  for a general overview. We also mention that, over the last years, Sobolev coefficients have been systematically considered as a replacement of usual Lipschitz ones to find optimal conditions in several other fields of analysis and PDE; see for instance \cite{bcrippa, FGS, crippadelellis}. 
 \begin{theorem}\label{sample1} Let $u \in W^{1,1}_{\loc}(\Omega;\er^N)$ be a minimizer of $\mathcal F$ in \eqref{genF}, under assumptions \trif{asp1}-\trif{asp2}. Assume that \eqref{lox} holds and 
\eqn{lox2}
$$
\frac{q}{p}<1+\min\left\{\frac{1}{n}-\frac{1}{d},\mathfrak{m}_{p}\right\}\quad \mbox{with} \ \ \mathfrak{m}_{p}:=\left\{ 
\begin{array}{ccc}
\frac{4(p-1)}{\vartheta p(n-2)} &\mbox{if}  & n\ge 3 \\ [6 pt]
\ \frac{2(p-1)}{\vartheta p}&\mbox{if}  & n=2\,,
\end{array}
\right.
$$
where $\vartheta=1$ if $p\ge 2$ and $\vartheta=2$ otherwise. Then $Du\in L^{\infty}_{\loc}(\Omega;\er^{N\times n})$. When either $p\geq 2$ or $f \equiv 0$, \trif{lox2} can be replaced by
\eqn{lox2bis}
$$
\frac qp<  1+\frac 1n-\frac1d\,.
$$ 
 %Moreover, whenever $\B\Subset \Omega$ is a ball with radius $\rad(\B)\le 1$ the following $L^{\infty}$-estimate holds:
%\begin{flalign}\label{s2l}
%\nr{Du}_{L^{\infty}(s\B)}^{p}&\le \frac{c}{(1-s)^{\beta}[\rad(\mathcal{B})]^{\beta}}\left[\nr{F(\cdot,Du)}_{L^{1}(\B)}+\nr{f}_{\mathfrak X(\B)}+1\right]^{\theta},
%\end{flalign}
%for all $t\in (0,1)$, $c\equiv c(\data,\nr{h}_{L^{d}(\B)})$ (resp. $c\equiv c(\texttt{data}_{\textnormal{two}},\nr{h}_{L^{d}(\B)})$ if $n=2$) and $\beta,\theta\equiv \beta,\theta(n,d,p,q)$.
\end{theorem} 
Theorem \ref{sample1} actually follows from Theorem \ref{t1} in Section \ref{general results} below and, as all the other results presented in this Introduction, comes along with explicit local a priori estimates. In particular, for splitting structures as 
\eqn{split}
$$w \mapsto  \int_{\Omega}  \left[ \ccc(x)F(Dw)-f\cdot w \right]\, dx \,, \qquad 0 < \nu \leq \ccc(\cdot)\leq L\,,$$
Theorem \ref{sample1} becomes
\begin{theorem}\label{sample2} Let $u \in W^{1,1}_{\loc}(\Omega;\er^N)$ be a minimizer of the functional in \eqref{split}, under assumptions \trif{asp1}-\trif{nondecresce} with $F(\cdot)\equiv F(z)\equiv \tilde F(|z|)$, and $f$ as in \eqref{lox}. Assume that  $\ccc(\cdot) \in W^{1,d}(\Omega)$ with $d>n$ and that \trif{lox2} is satisfied. Then $Du\in L^{\infty}_{\loc}(\Omega;\er^{N\times n})$. When either $p\geq 2$ or $f \equiv 0$, \trif{lox2} can be replaced by \trif{lox2bis}. 
 \end{theorem}

For double phase functionals in \rif{pqfunctional}, condition \rif{asp2} again amounts to assume that $a(\cdot)\in W^{1,d}(\Omega)$, and indeed we have 
\begin{theorem}\label{sample3} Let $u \in W^{1,1}_{\loc}(\Omega;\er^N)$ be a minimizer of the functional in 
\trif{pqfunctional}, such that $0 \leq a(\cdot) \in W^{1,d}(\Omega)$, $d>n$, and that 
 \trif{lox} holds together with 
\eqn{lox22}
$$
q/p\le 1+1/n-1/d\,,  \ \mbox{if $n\ge2$} \ \   \mbox{and, when $n=2$, also $q/p < p$}\,.
$$
%$$
%\begin{cases}
%\frac{q}{p}\le 1+\frac{d-n}{nd}\quad \mbox{if} \ \ n\ge 3\\
%\frac{q}{p}\le 1+\frac{d-n}{nd} \quad \mbox{and}\quad q<p^{2}\quad \mbox{if} \ \ n=2.
%\end{cases}
%$$
Then $Du\in L^{\infty}_{\loc}(\Omega;\er^{N\times n})$. 
\end{theorem}
Theorem \ref{sample3} allows to clarify in which sense assumptions \rif{asp2} and \rif{lox2bis} \& \rif{lox22} are sharp. Indeed, Sobolev-Morrey embedding yields $a(\cdot)\in C^{0, \alpha}$ with $\alpha =1-n/d$. This last identity makes conditions \rif{pqbound} and \rif{lox2bis} coincide. In turn, \rif{pqbound} is sharp by the counterexamples in \cite{sharp, FMM}. Therefore, \rif{asp2} is the sharp differentiable version of \rif{pqbound}, which is stronger than \rif{pqbound}, but weaker than assuming that $a(\cdot)$ is Lipschitz. Lipschitz continuity of coefficients is typically  assumed in the literature in the nonautonomous case (see for instance \cite{M2, dark} and related references).

\subsection{Nonuniform ellipticity at fast rates and a different phenomenon}\label{fastsec} A prototype we have in mind here is given by \rif{exp-mod1}. Looking at the case of polynomial growth in Section \ref{pqsec}, from \rif{asp2} and \rif{lox2} we see that the required integrability rate of coefficients $d$ increases with the ratio $q/p$. A naive, but seemingly natural bet, would then assert that the exponential case needs more stringent conditions on the integrability exponent $d$. On the contrary, the situation reverses, and any $d>n$ implies local Lipschitz continuity:
\begin{theorem}\label{sample4} Let $u \in W^{1,1}_{\loc}(\Omega; \er^N)$ be a minimizer of the functional in 
\trif{exp-mod1}, such that $\ccc_1(\cdot), \ccc_{2}(\cdot) \in W^{1,d}(\Omega)$ with $d>n$ and $f$ satisfies \trif{lox}. Then $Du\in L^{\infty}_{\loc}(\Omega;\er^{N\times n})$. 
\end{theorem}
The same applies to more general functionals with faster growth, involving arbitrary compositions of exponentials, and therefore even faster growth conditions. 
Specifically, we fix exponent functions $\{p_k(\cdot)\}$ and coefficients  $\{\ccc_k(\cdot)\}$, all defined on the open subset $\Omega \subset \er^n$, such that
\begin{flalign}\label{assexp}
\begin{cases}
1 < p_{\mathfrak{m}} \leq p_0(\cdot) \leq p_M \,, \quad  0 < p_m \leq p_k(\cdot)\leq p_M\,, \ \mbox{for $k \geq 1$} \\
 0 < \nu_{\mathfrak{m}} \leq \ccc_k(\cdot)\leq L \,, \quad p_k(\cdot), \ccc_k(\cdot) \in W^{1,d}(\Omega)\,, \ \  d>n\,, \ \mbox{for $k \geq 0$}\,.
\end{cases}
\end{flalign}
We then inductively define, for every $k\in \en$, the functions ${\bf e}_{k}\colon \Omega \times  [ 0, \infty) \to \er$ as
\eqn{esponenciali}
 $$\left\{
 \begin{array}{ccc}
 {\bf e}_{k+1}(x,t) & := & \exp \left(\ccc_{k+1}(x)\left[{\bf e}_{k}(x,t)\right]^{p_{k+1}(x)}\right) \\ [6 pt]{\bf e}_{0}(x,t) &:= & \exp\left(\ccc_0(x)t^{p_0(x)}\right)\,,
 \end{array}
 \right.
 $$
and consider the variational integrals
\eqn{espik}
$$
w \mapsto   \int_{\Omega}  \left[ {\bf e}_{k}(x,|Dw|)-f\cdot w \right]\, dx\,.
$$
Functionals as in \rif{espik} have been studied at length in the literature also because they provide the best case study to test how far one can go in relaxing the standard uniform ellipticity assumptions; see \cite{BM, M2} and related references. The nonautonomous case is of special interest as the sensitivity to the $x$-dependence is magnified by taking multiple compositions of exponentials; see comments after display \rif{exp-mod1}. We have the following result, which, as also Theorem \ref{sample4}, is completely new already in the case $f\equiv 0$:
\begin{theorem}\label{sample6} Let $u \in W^{1,1}_{\loc}(\Omega; \er^N)$ be a minimizer of the functional in 
\trif{espik} for some $k \in \en$, under assumptions \trif{assexp} and such that $f$ satisfies \trif{lox}. Then $Du\in L^{\infty}_{\loc}(\Omega;\er^{N\times n})$. 
\end{theorem}
In other words, this fact brings functionals as in \rif{exp-mod1} closer to the realm of uniformly elliptic ones. The next step comes in fact in the subsequent section. 
\subsection{New results in the uniformly elliptic setting}\label{unisec0}
New results follow in the classical uniformly elliptic setting too. This time the model is 
\eqn{funz-uni}
$$
w \mapsto \int_{\Omega}  \left[A(x,|Dw|)-f\cdot w \right]\, dx\,, \quad A(x,t)
:=  \ccc(x)\int_0^t \tilde a(s)s\, ds\,,
$$
for $t \geq 0$, with \rif{uni-ell} being in force and such that $ 0 < \nu \leq \ccc(\cdot) \leq L$. Under such conditions, every solution to the system in \rif{general1} is a minimizer of the functional in \rif{funz-uni}  and the second identity in \rif{general1} is automatically satisfied.
\begin{theorem}\label{sample7} Let $u \in W^{1,1}_{\loc}(\Omega; \er^N)$ be a minimizer of the functional in 
\trif{funz-uni}, under assumptions \trif{uni-ell}. If $|f|, |D\ccc|\in \mathfrak X(\Omega)$ as defined in \trif{lox}, then $Du\in L^{\infty}_{\loc}(\Omega;\er^{N\times n})$. Moreover, there exists a positive radius $R_{*}\equiv R_{*}(n,N,\ia, \sa,\linebreak\ccc(\cdot)) \leq 1$ such that if $\B \Subset \Omega$ is a ball with $\rad(\B)\leq R_*$, then the following estimate holds for 
$c \equiv c(n,N,\nu, L, \ia, \sa)$:
\eqn{stimaAA}
$$
\|A(\cdot, \snr{Du}) \|_{L^\infty(\B/2)} \leq 
c \mint_{\B} A(\cdot, \snr{Du})\, dx+c \|f\|_{\mathfrak X(\B)}^{\frac{\ia+2}{\ia+1}}+c\,.
$$
\end{theorem}
In other words, $f$ and $D\ccc$ are this time required to have the same degree of regularity. Theorem \ref{sample7} applies to \rif{plap} by taking $A(x,t)\equiv \ccc(x)t^{p}/p$ and it is sufficient to require that $D\ccc \in L(n,1)(\Omega)$ for $n>2$. This is, already when $f\equiv 0$, a new regularity criterion on coefficients, which goes beyond the known and classical one in \rif{dini}. Indeed, $D\ccc \in L(n,1)$ implies that $\ccc(\cdot)$ is continuous \cite{St}, but not necessarily with a modulus of continuity $\omega(\cdot)$ satisfying \rif{dini}. Moreover, this criterion works for the general cases as in \rif{funz-uni}, to which methods from \cite{KMstein} do not apply under the only considered structure assumption \rif{uni-ell}. When considered in the special case \rif{plap}, it is $\ia=p-2$ and estimate \rif{stimaAA} gives back the classical one valid for the $p$-Laplacean system in \rif{plap}; see for instance \cite{KMstein, KMjems}. 
\subsection{Calder\'on-Zygmund theory} 
In Theorems \ref{sample1}-\ref{sample7}, we can replace \rif{lox} by the weaker $f \in L^{n}(\Omega;\er^N)$, getting, as a corresponding outcome, that $Du \in L^p_{\loc}(\Omega; \er^{N\times n})$ for every $p\geq 1$; see Theorem \ref{t3} below. This result is new in the nonuniformly elliptic case and is in perfect accordance with the Nonlinear Calder\'on-Zygmund theory known for the uniformly elliptic one \cite{KMbull}. For instance, considering the system in \rif{plap}, $D\ccc\in L^n$ implies that $\ccc(\cdot)\in$ VMO, the space of functions with vanishing mean oscillations \cite{sarason}. At this point, $Du \in L^p_{\loc}(\Omega)$, for every $p\geq 1$, follows from the standard theory (see \cite{BOR, DoK} and references). In fact, we provide the first Calder\'on-Zygmund type estimates in problems with non-polynomial growth conditions. An example is the following result, which is completely new already in the autonomous case:
\begin{theorem}\label{sample8} Let $u \in W^{1,1}_{\loc}(\Omega; \er^N)$ be a minimizer of the functional in 
\trif{espik} for some $k \in \en$, under assumptions \trif{assexp} with $n>2$, and such that $f\in L^n(\Omega; \er^N)$. Then ${\bf e}_{k}(\cdot,|Dw|)\in L^{p}_{\loc}(\Omega)$ for every $p \geq 1$. 
\end{theorem}

\subsection{Obstacles} Applications follow to obstacle problems, leading to completely new and sharp results, already in the classical, uniformly elliptic case. For instance, we give the first results for fast growth functionals as in \rif{exp-mod1}, and these are new already in the case of smooth obstacles. For this we consider the functional 
\eqn{genF0}
$$\mathcal F_0(w;\Omega):=  \int_{\Omega}   F(x,Dw) \, dx\;
$$
defined on $W^{1,1}(\Omega)$, where $F$ is for instance one of the integrands considered in Theorems \ref{sample1}-\ref{sample7}; here we of course consider the scalar case $N=1$. Next we consider a measurable function $\psi \colon \Omega \to \er$  
and the convex set $\mathcal{K}_{\psi}(\Omega):=\{w\in W^{1,1}_{\loc}(\Omega)\colon  w(x)\ge \psi(x) \  \mbox{for a.e.} \ x\in \Omega \}$. We then say that a function $u \in W^{1,1}_{\loc}(\Omega)\cap \mathcal{K}_{\psi}(\Omega)$ is a constrained local minimizer of $\mathcal F_0$ if, for every open subset $\tilde \Omega\Subset \Omega$, we have $\mathcal F_0(u;\tilde \Omega) <\infty$ and if $\mathcal F_0(u;\tilde \Omega)\leq \mathcal F_0(w;\tilde \Omega)$ holds for every competitor $w \in u + W^{1,1}_0(\tilde \Omega)$ such that $w \in \mathcal{K}_{\psi}(\tilde \Omega)$. We then have the following, far reaching extension of classical theorems from \cite{choe, CL, fuobs, KS, liebe2}:
\begin{theorem}\label{sample9}
Let  $u\in W^{1,1}_{\loc}(\Omega)$ be a constrained local minimizer of $\mathcal F_0$ in \trif{genF0}, where $F\colon \Omega\times \mathbb{R}^{n}\to \mathbb{R}$ is one of the integrands from Theorems \ref{sample1}-\ref{sample7} with $p\geq 2, \ia\geq 0$ (whenever such parameters are involved). If $\psi\in W^{2,1}_{\loc}(\Omega)$ and $\snr{D^{2}\psi}\in \mathfrak{X}(\Omega)$, then $Du\in L^{\infty}_{\loc}(\Omega;\er^n)$. \end{theorem}
This last result is new already in the classical $p$-Laplacean case $F(x,z)\equiv |z|^p/p$, for which it offers a criterion which is alternative to those given in \cite{choe, CL} -- see also \cite{KL} for double phase type functionals. In such papers Lipschitz estimates are obtained assuming that $D\psi$ is locally H\"older continuous. Here we trade this last condition with 
$\snr{D^{2}\psi}\in \mathfrak{X}(\Omega)$, that in turn implies the mere continuity of $D\psi$. This is essentially the same phenomenon 
seen in Theorem \ref{sample7}, where the condition $\snr{D\ccc}\in \mathfrak{X}(\Omega)$ replaces the Dini-continuity of  $\ccc(\cdot)$. We note that in the constrained versions of Theorems \ref{sample1} and \ref{sample2} we can allow for $p\geq 1$ provided $\mu>0$; for this see Remark \ref{obre} below. 

\subsection{Remarks and organization of the paper} 
Some of the methods here also extend to general scalar functionals, i.e., when minima and competitors are real valued functions. In this case there is no need to assume the radial structure $F(x,Dw)\equiv \tilde F(x,|Dw|)$. On the other hand, additional conditions ensuring the absence of the so-called Lavrentiev phenomenon are needed to build suitable approximation arguments, see for instance \cite{DeMi, sharp}. The radial structure is usually assumed in the vectorial case, otherwise singular minimizers might occur, even when data are smooth \cite{dark, SY1}. Again in the scalar case, we mention the recent, very interesting paper \cite{HO}, where gradient regularity results are obtained for minimizers of functionals as in \rif{genF0}. These results cover functionals with polynomial growth and special structure - the double phase functional is an instance - under H\"older continuity assumptions on coefficients. Anyway, they miss to cover all the classes of integrands described in Sections \ref{pqsec}-\ref{fastsec}. 

The rest of the paper goes as follows. After Sections \ref{notsec} and \ref{notsec2}, containing notations and preliminaries, respectively, in Section \ref{resultsec} we describe in detail the assumptions and the main results of the paper, that is Theorems \ref{t1}-\ref{t4} in Section \ref{general results}. These will imply, directly or with a few additional arguments, Theorems \ref{sample1}-\ref{sample9} above. We then proceed to Section \ref{apro}, that contains the necessary approximation tools for the proofs. One word here: this is a delicate point, as the approximations considered must carefully match the shape of the a priori estimates found later, on one side, and reflect the original structure assumptions on the other. The core of the paper is Section \ref{apri}, where we derive all the main a priori estimates. The proofs here involve a series of ingredients. First, we employ a delicate version of Moser's iteration scheme developed in Proposition \ref{lp}. This is  based on a peculiar choice of test maps suited to the structural assumptions considered. It goes via a finite step procedure taking advantage of suitable smallness conditions; by the way, this is sufficient to get the basic $L^p$-estimates of Theorem \ref{t3}. In turn, this is a preliminary ingredient used to make a nonlinear potential theoretic approach work; this last one is encoded in the abstract iteration Lemma \ref{noniter} below. This approach works in the case $n>2$, but breaks down in two dimensions $n=2$, where more difficulties appear, as for instance already noted in \cite{BM, CM0, CM1, CM2, CM3}. In this case we take another path, as devised in Section \ref{casetwo} below. We use a different interpolation type approach, eventually culminating in the use of Lemma \ref{noniter} again. Sections \ref{pt1} features the proofs of Theorems \ref{t1} and \ref{t3}, combining the approximation scheme of Section \ref{apro} and the estimates in Section \ref{apri}. Section \ref{unisec} contains the results and the proofs for the uniformly elliptic case, i.e., the proof of Theorem \ref{t2}; additional a priori estimates are included here. In Section \ref{applications} we demonstrate the derivation of Theorems \ref{sample1}-\ref{sample8}. These are all direct corollaries of the main results, but Theorem \ref{sample3}. This in fact requires some additional arguments to reach the equality borderline case in the first bound from \rif{lox22},  thereby reconnecting Theorem \ref{sample3} to the known literature \cite{BCM, CoM0, CoM} in the case $f\equiv 0$; the details can be found in Section \ref{sample3sec}. Section \ref{obsec} contains applications to obstacle problems and the proof of Theorem \ref{sample9}. Finally, Section \ref{appendix} features some auxiliary technical facts aimed at making certain computations in Sections \ref{apri} and \ref{unisec} legal.

\section{Basic notation}\label{notsec}
In the following, $\Omega\subset \er^n$ denotes an open domain, and $n \geq 2$ and there is no loss of generality, in assuming that~$\Omega$ is also bounded, as all our results are local in nature.
%%% control: do we need to assume that $\Omega$ is bounded?
We denote by $c$ a general constant larger than~$1$. Different occurrences from line to line will be still denoted by $c$. Special occurrences will be denoted by $c_*,  \tilde c$ or likewise. Relevant dependencies on parameters will be as usual emphasized by putting the corresponding parameters in parentheses. We denote by $ B_r(x_0):= \{x \in \er^n  :   |x-x_0|< r\}$ the open ball with center $x_0$ and radius $r>0$; we omit denoting the center when it is not necessary, i.e., $B \equiv B_r \equiv B_r(x_0)$; this especially happens when various balls in the same context will share the same center. Given a ball $B$, we denote by $\rad(B)$ its radius; with $\gamma$ being a positive number, we denote by $\gamma B$ the ball concentric to $B$, with radius $\gamma \rad(B)$, and set $B/\gamma \equiv (1/\gamma)B$. In denoting several function spaces like $L^p(\Omega), W^{1,p}(\Omega)$, we shall denote the vector valued version by $L^p(\Omega;\er^k), W^{1,p}(\Omega;\er^k)$ in the case the maps considered take values in $\er^k$, $k\in \en$. We shall often abbreviate $L^p(\Omega;\er^k)\equiv L^p(\Omega), W^{1,p}(\Omega;\er^k)\equiv W^{1,p}(\Omega)$. We denote $\{e^\alpha\}_{\alpha=1}^N$ and $\{e_i\}_{i=1}^n$ standard bases for $\er^N$ and $\er^n$, respectively; we shall always assume $n\geq 2$ and $N \geq 1$. The general second-order tensor of size $(N,n)$ as $\zeta = \zeta_i^\alpha e^\alpha \otimes e_i$ is identified with an element of $\er^{N\times n}$. The Frobenius product of second-order tensors $z$ and $\xi$ is defined as $ z \cdot \xi = z_i^\alpha \xi_i^\alpha$ so that $  \xi \cdot \xi = |\xi|^2$, and this is the norm we use here for tensors, vectors, matrixes; needless to say, we are using Einstein's convention on repeated indexes. For instance, if $v, u \in \er^k$, then $v\cdot u= v_iu_i$. The gradient of a map $u = u^\alpha e^\alpha$ is thus defined as 
$
Du =  \partial_{x_i} u^\alpha e^\alpha \otimes e_i\equiv D_i u^\alpha e^\alpha \otimes e_i\,,
$
and the divergence of a tensor $\zeta = \zeta_i^\alpha e^\alpha \otimes e_i$ as $\diver\,  \zeta = \partial_{x_i} \zeta_i^\alpha  e^\alpha$. When dealing with the integrands of the type $F \colon \Omega \times \er^{N\times n} \to [0, \infty)$, as the one considered in \rif{genF}, second differential of $\partial_{zz} F(\cdot,z)$ is interpreted as $
\partial_{zz} F(\cdot,z) =  \partial_{z_j^\beta z_i^\alpha}F(\cdot, z)(e^\alpha \otimes  e_i) \otimes (e^\beta \otimes e_j) 
$, whenever $z \in \er^{N\times n}$.  For the rest of the paper we shall keep the following notation:
\eqn{Hnot}
$$
E(t):=\sqrt{t^{2}+\mu^{2}}\,,\qquad E_{\varepsilon}(t):=\sqrt{t^{2}+(\mu+\varepsilon)^{2}}
$$
for $ t>0$, $\mu\in [0,1]$,  $\eps >0$.
With $\mathcal B \subset \er^{n}$ being a measurable subset with bounded positive measure $0<|\mathcal B|<\infty$, and with $g \colon \mathcal B \to \er^{k}$, $k\geq 1$, being a measurable map, we denote  
$$
   (g)_{\mathcal B} \equiv \mint_{\mathcal B}  g(x) \, dx  :=  \frac{1}{|\mathcal B|}\int_{\mathcal B}  g(x) \, dx\,.
$$ 
Finally, in the following we denote
\begin{flalign}\label{2***}
2^{*}:=\begin{cases}
\ \frac{2n}{n-2}\quad &\mbox{if} \ \ n>2\\
\ \mbox{any number larger than} \ \ 2\quad &\mbox{if} \ \ n=2\,.
\end{cases}
\end{flalign}
The actual value of $2^{*}$ when $n=2$ will be clear from the context. 
\section{Potentials, functions spaces, iteration lemmas}\label{notsec2}
With $g \in L^{2}(B_r(x_0); \er^k)$ and $B_r(x_0)\subset \er^n$ being any ball, $n\geq 2$,  we consider the following nonlinear potential, that will play a crucial role in this paper:
%\eqn{defi-P} 
$$
{\bf P}_1^g(x_0,r) := \int_0^r \left( \varrho^{2} \mint_{B_\varrho(x_0)} |g|^2 \dx \right)^{1/2} \frac{d \varrho}{\varrho} \,.
$$
This quantity naturally relates to the standard, truncated Riesz potential in the sense that
$$
\int_{B_r(x_0)} \frac{|g(x)|}{|x-x_0|^{n-1}}\, dx \lesssim  {\bf P}_1^g(x_0,2r) \,.
$$
As a matter of fact, $ {\bf P}_1^g$ can be used as an effective replacement of the original Riesz potential when dealing with nonlinear problems. Actually, its mapping properties coincide with those of the classical Riesz potentials on those function spaces that are in a sense smaller than $L^2$. We refer to \cite{KMbull, MH} for more information on such nonlinear potentials and for recent results on Nonlinear Potential Theory in the setting of this paper. The space $\mathfrak X(\cdot)$ in \rif{lox} plays a special role to ensure the local boundedness of $ {\bf P}_1$. Indeed, given concentric balls $B_R \subset B_{R+r} \subset \er^n$, with $r>0$ and $R+r \leq 2$, and $g \in L^{2}(B_{R+r}; \er^k)$, the following inequalities hold:
\eqn{potlore}
$$
\begin{cases}
\left\|{\bf P}_1^{g}(\cdot, r)\right\|_{L^{\infty}(B_{R})}\leq c(n) \| g\|_{L(n,1)(B_{R+r})} \quad n>2\\
\left\|{\bf P}_1^{g}(\cdot, r)\right\|_{L^{\infty}(B_{R})} \leq c(\mathfrak {a})\|g\|_{L^{2}(\log L)^{\mathfrak {a}}(B_{R+r})} \quad n =2\,,
 \end{cases}
 $$
 where, in the last of the two inequalities, it is $c(\mathfrak {a})\to \infty$ when $\mathfrak {a} \searrow 2$. For this we refer to \cite{BM,KMbull}. Note that in the right-hand side of \rif{potlore}$_2$ we find the Luxemburg norm of the Orlicz space $L^{2}(\log L)^{\mathfrak {a}}$ in \rif{llog}, which is defined as 
 $$
\|g\|_{{L^{2}(\log L)^{\mathfrak {a}}(B_{R+r})}}:=  \inf\left\{ \lambda >0 \, \colon  \int_{B_{R+r}}\left( \frac{|g|^2}{\lambda^2}\right)
\log^{\mathfrak {a}}\left(\textnormal{e}+\frac{|g|}{\lambda}\right)
\, dx \leq 1\right\}.
$$
The two following inequalities are well-known and hold for any $\tau >0$:
\eqn{normine}
$$
\begin{cases}
\| g\|_{L(n,1)(B_{R+r})} \leq c(n, \tau)\| g\|_{L^{n+\tau}(B_{R+r})} 
\\
\|g\|_{{L^{2}(\log L)^{\mathfrak {a}}(B_{R+r})}}\leq c(\mathfrak a, \tau)\| g\|_{L^{2+\tau}(B_{R+r})}\,.
\end{cases}
$$
Recalling the standard notation $(v-\kk)_{+} := \max\{v-\kk, 0\}$, we select the following result from  \cite[Lemma 3.1]{BM}, that incorporates the basic elements of De Giorgi's iteration, as used in Nonlinear Potential Theory since the basic work of Kilpel\"ainen \& Mal\'y \cite{kilmal}:
\begin{lemma}\label{noniter}
Let $B_{r_{0}}(x_{0})\subset \mathbb{R}^{n}$, $n\ge 2$, $\delta\in \left(0,1/2\right)$ and $v\in W^{1,2}(B_{r_{0}}(x_{0}))$ be non-negative and $f_{1}, f_{2}\in L^{2}(B_{2r_0}(x_0))$. Assume that there exist positive constants $\tilde{c}, M_{1}, M_{2},M_{3}\ge 1$ and a number $\kk_{0}\ge 0$ such that for all $\kk\ge \kk_{0}$ and every ball $B_{r}(x_{0})\subset B_{r_{0}}(x_{0})$ the inequality
\begin{flalign}\label{varine}
& \mint_{B_{r/2}(x_{0})}\snr{D(v-\kk)_{+}}^{2} \dx \le \frac{\tilde{c}M_{1}^{2}}{r^{2}}\mint_{B_{r}(x_{0})}(v-\kk)_{+}^{2} \dx\nonumber \\
&\qquad\qquad \qquad  + \tilde{c}M_{2}^{2}\mint_{B_{r}(x_{0})}\snr{f_{1}}^{2}\dx+\tilde{c}M_{3}^{2}\mint_{B_{r}(x_{0})}\snr{f_{2}}^{2} \dx
\end{flalign}
holds. If $x_{0}$ is a Lebesgue point for $v$, then
\begin{flalign}\label{var8}
v(x_{0})& \le \kk_{0}+cM_{1}^{1+\max\left\{\delta,\frac{n-2}{2}\right\}}\left(\mint_{B_{r_{0}}(x_{0})}(v-\kk_{0})_{+}^{2} \dx\right)^{1/2}\nonumber\\
&\qquad  \ +cM_{1}^{\max\left\{\delta,\frac{n-2}{2}\right\}}\left[M_{2}\mathbf{P}_{1}^{f_{1}}(x_{0},2r_{0})+M_{3}\mathbf{P}_{1}^{f_{2}}(x_{0},2r_{0})\right]
\end{flalign} holds too, with $c\equiv c(n,\tilde{c},\delta)$.
\end{lemma}
We conclude with another, more classical iteration lemma (see references of \cite{BM}). 
\begin{lemma}\label{l0l}
Let $\mathcal{Z}\colon [\rr,\xi]\to [0,\infty)$ be a bounded function. Let $\varepsilon\in (0,1)$, $a_1,a_2,\mathfrak{b} \ge 0$ be numbers. If
$$
\mathcal{Z}(\tau_1)\le \varepsilon \mathcal{Z}(\tau_2)+ \frac{a_1}{(\tau_2-\tau_1)^{\mathfrak{b}}}+a_2
$$ holds whenever $\xi \le \tau_1<\tau_2\leq \rr$, 
then
$$
\mathcal{Z}(\rr)\le  \frac{ca_1}{(\varrho-\xi)^{\mathfrak{b}}}+ ca_2$$
holds too, with $c\equiv c(\varepsilon,\mathfrak{b})$.
\end{lemma}

\section{Assumptions and general results}\label{resultsec}
In this section we are going to describe a number of conditions on the integrand $F$ in \rif{genF} implying our main results, that is Theorems \ref{t1}-\ref{t4} in Section \ref{general results} below; in turn, these will imply Theorems \ref{sample1}-\ref{sample9} from the Introduction.

\subsection{Basic structural assumptions, and consequences}\label{asssec1} 
We assume that $F$ has radial structure, i.e., there exists $\tilde F: \Omega \times [0, \infty) \to [0, \infty)$
\begin{flalign}\label{0ass}
\begin{cases}
\ F(x,z)=\tilde{F}(x,\snr{z}) \ \mbox{for all $(x,z)\in \Omega\times \mathbb{R}^{N\times n}$}\\
\ t\mapsto \tilde{F}(x,t)\in C^{1}_{\textnormal{loc}}[0,\infty)\cap C^{2}_{\textnormal{loc}}(0,\infty)\quad \mbox{for all $ x\in \Omega$}\\ 
\ x\mapsto \tilde{F}'(x,t)  \in W^{1,n}(\Omega) \ \ \mbox{for  every $t>0$}\,.
\end{cases}
\end{flalign}
We assume in general that $\tilde F, \tilde F'$ are continuous on $\Omega \times (0,\infty)$, while $\tilde F'', \partial_x \tilde F'$ are instead Carath\'eodory regular on the same set. Of course we also assume that $\sup_{\Omega \times K} \partial_{zz}{F}$ is finite for every compact subset $K \subset \er^{N\times n}\setminus\{0_{R^{N\times n}}\}$. Now we describe the minimal and standard assumptions qualifying the functional in \rif{genF} as elliptic. For this, we use three locally bounded functions $g_{i}\colon \Omega\times (0,\infty)\to[0,\infty)$, for $i\in \{1,2,3\}$, in the sense that $\sup_{\Omega \times K} g_{i}$ is finite, for every compact subset $K\subset (0, \infty)$. The first two $g_1,g_2$ are continuous and serve to bound the lowest and the largest eigenvalues of $\partial_{zz}F$, respectively. The third one $g_3$ is Carath\'eodory regular and controls the growth of derivatives with respect to $x$. Specifically, we assume that there exists $T>0$ such that
%%%typoc
\begin{flalign}\label{growth}
\begin{cases}
\ z\mapsto F(x,z) \ \ &\mbox{is convex for all $x\in \Omega$}\\
\ \snr{\partial_{zz}F(x,z)}\le g_{2}(x,\snr{z}) \ \ &\mbox{for all} \ x\in \Omega \ \mbox{on} \ \{\snr{z} \geq  T \}\\
\ g_{1}(x,\snr{z})\snr{\xi}^{2}\le  \partial_{zz}F(x,z)\xi\cdot\xi  \ \ &\mbox{on} \ \{\snr{z} \geq  T \} \ \mbox{and for all $x\in \Omega$}\\
\ |\partial_x\tilde{F}'(x,t)| \le h(x)g_{3}(x,t)  \ \ & \mbox{for a.e.\,$x\in \Omega$ and every $t>0$}\,,
\end{cases}
\end{flalign}
where \rif{growth}$_3$ holds for every $\xi \in \mathbb{R}^{N\times n}$, $h(\cdot)\in L^{n}(\Omega)$ and we assume also that 
\eqn{infima}
$$\inf_{x\in \Omega} g_1(x, T)>0\,.$$ 
Using as in \rif{general1} the notation
$%\begin{flalign}%\label{tildea}
\tilde{a}(x,t):=\tilde{F}'(x,t)/t$, defined for $(x,t)\in \Omega\times (0,\infty)$,
%\end{flalign}
we assume that for fixed numbers $\gamma>1$, $\mu\in [0,1]$, and for every $x \in \Omega$, the minimal $\gamma$-superlinear growth of the lowest eigenvalue of $\partial_{zz}F(x, \cdot)$, that is
\eqn{00}
$$
 t\mapsto \frac{\tilde{a}(x,t)}{(t^{2}+\mu^{2})^{\frac{\gamma-2}{2}}} \  \mbox{and} \  t\mapsto \frac{g_{1}(x,t)}{(t^{2}+\mu^{2})^{\frac{\gamma-2}{2}}}\ \ \mbox{are non-decreasing on} \ (0,\infty).
$$
Note that the global convexity of $F(\cdot)$ in \rif{growth}$_1$ implies, arguing as in \rif{z0}-\rif{23} below, that $\tilde a(\cdot)\geq 0$, that is, $t \mapsto \tilde F(x, t)$ is non-decreasing on $[0, \infty)$, for every $x\in \Omega$. 
\begin{remark}\label{whenrere} When considering partial derivatives of $F$ and $g_1$ in the $x$-variable, as well as the function $g_3$, properties will be given a.e. with respect to $x$ in the Carat\'heodory sense. For instance, \rif{growth}$_4$ means there exists a negligible set $N \subset \Omega $ such that \rif{growth}$_4$ holds for every $t>0$, whenever $x \in \Omega\setminus N $. The same will apply later when considering for instance \rif{7}-\rif{8} and \rif{growthd}. 
\end{remark}
\begin{remark}\label{trire1}
In most of the relevant models it turns out to be $\tilde a (\cdot)\equiv g_1(\cdot)$, and this justifies the double assumption in \rif{00}. Assumptions in \rif{growth} are bound to quantify ellipticity of $\partial_{zz}F(\cdot,z)$ only outside the ball $\{|z|<T\}$ (recall we are assuming $\nu >0$), and this allows to cover functionals loosing their ellipticity properties on a bounded set of $z\in \er^{N\times n}$. This condition not only adds more generality, but helps simplifying the treatment in the case of nonuniformly elliptic problems. We could do the same also with respect to $\partial_x\tilde{F}'$ in \rif{growth}$_4$, but we prefer not, as this would only add technical difficulties, while not covering more examples.
\end{remark}
Let us derive a few consequences of \rif{0ass}-\rif{00}. First, for every $x\in \Omega$ it holds that
\eqn{1bis}
$$
0< s\leq  t\Longrightarrow  \begin{cases} g_{1}(x,s)s\le g_{1}(x,t)t\\
 a(x,s)s\le a(x,t)t\,.\end{cases}
 $$
Indeed, as $t \mapsto (t^{2}+\mu^{2})^{(\gamma-2)/2}t$ is non-decreasing, we have 
\begin{flalign*}
g_{1}(x,s)s  &= \frac{g_{1}(x,s)}{(s^{2}+\mu^{2})^{\frac{\gamma-2}{2}}}(s^{2}+\mu^{2})^{\frac{\gamma-2}{2}}s\\
& \stackrel{\eqref{00}}{\leq} 
\frac{g_{1}(x,t)}{(t^{2}+\mu^{2})^{\frac{\gamma-2}{2}}}(t^{2}+\mu^{2})^{\frac{\gamma-2}{2}}t =g_{1}(x,t)t \,.
\end{flalign*}
The same argument works for $t \mapsto a(\cdot, t)t$. 
Moreover, again from \rif{00}, it readily follows that
\eqn{00bis}
$$
\begin{cases}
 \nu(t^2+\mu^2)^{\frac{\gamma-2}{2}}\leq g_1(x,t)  \  \mbox{for} \ t\geq T\\
 0\stackrel{\rif{infima}}{<} \nu:=\displaystyle\min\left\{\frac{\inf_{x\in \Omega} g_1(x, T)}{(T^2+\mu^2)^{\frac{\gamma-2}{2}}}, 1\right\}\leq 1\,.
 \end{cases}
$$
From the very definition of $\tilde a(\cdot)$, for $(x,t)\in \Omega\times (0,\infty)$ we have 
\eqn{wehaweha}
$$
F(x,z)-F(x,0_{\er^{N\times n}}) =\tilde{F}(x,|z|)-\tilde{F}(x,0) =\int_{0}^{|z|}\tilde{a}(x,s)s \d s
$$ 
%%%ultimo
and therefore $\partial_{xz}F(x,|z|) = \partial_x \tilde a(x,|z|) \otimes z= [\partial_x \tilde F'(x,|z|) /|z|]\otimes z$ holds for $|z|\not =0$.   
It follows that \rif{0ass}$_3$ and \eqref{growth}$_4$ imply that $x\mapsto \partial_{z}F(x,z)  \in W^{1,n}(\Omega;\mathbb{R}^{N\times n})$ for all $z\in \mathbb{R}^{N\times n}$ such that $|z| \not= 0$, and 
\eqn{growthzero}
$$|\partial_{xz}F(x,z)| , |\partial_{x}\tilde{a}(x,\snr{z})||z|\le h(x)g_{3}(x,\snr{z}) \ \mbox{on $\{|z| > 0\}$, a.e. $x\in \Omega$}\,. 
$$
Again from the second-last display it follows that
\begin{flalign}\label{z0}
\partial_{zz}F(x,z)=\partial_{z}[\tilde{a}(x,\snr{z})z]=\tilde{a}(x,\snr{z})\mathds{I}_{N\times n}+\tilde{a}'(x,\snr{z})\snr{z}\frac{z\otimes z}{\snr{z}^{2}}
\end{flalign}
holds for $|z| \not= 0$, 
so that, using \eqref{growth}$_{2,3}$ with $\xi=z$ and $\xi \perp z$ in \rif{z0}, we obtain
\begin{flalign}\label{22}
\begin{cases}
\ \tilde{a}(x,\snr{z})+\tilde{a}'(x,\snr{z})\snr{z}\ge g_{1}(x,\snr{z})\\
\ \tilde{a}(x,\snr{z})\ge g_{1}(x,\snr{z})
\end{cases} \  \mbox{for $(x,z)\in \Omega\times\left\{\snr{z}\geq T\right\}$}
\end{flalign}and 
\begin{flalign}\label{23}
\begin{cases}
\ \tilde{a}(x,\snr{z})+\tilde{a}'(x,\snr{z})\snr{z}\le g_{2}(x,\snr{z})\\
\ \tilde{a}(x,\snr{z})\le g_{2}(x,\snr{z})
\end{cases}\  \mbox{for $(x,z)\in \Omega\times\left\{\snr{z}\geq T\right\}$}\,.
\end{flalign}
In particular, it follows that 
\eqn{banaleg1g2}
$$
\begin{cases}
\ 0< g_1(\cdot, t) \leq g_2(\cdot, t) & \mbox{for $t\geq T$}\\
\   \tilde a(\cdot, t) \leq c (t^2+\mu^2)^{(\gamma-2)/2} &\mbox{for $0< t \leq T$}.
\end{cases}
$$ 
The first inequality in \rif{banaleg1g2} is a consequence of \rif{00bis}. In \rif{banaleg1g2}$_2$, as a constant $c$ we can take 
$c = \|\tilde a(\cdot, T)\|_{L^{\infty}(\Omega)}(T^2+\mu^2)^{(2-\gamma)/2}\leq \|g_2(\cdot, T)\|_{L^{\infty}(\Omega)}(T^2+\mu^2)^{(2-\gamma)/2}$. This follows from \rif{00} and again from \rif{23}. %Moreover, as for \rif{1bis}, we have that $s \mapsto \tilde a (x, s)s$ is non-decreasing, and it follows that $\|\tilde F' \|_{L^{\infty}(\Omega \times K)}$ is finite, for every compact subset $K\subset [0, \infty)$.
\subsection{The energy functions}
These are two functions bound to quantify the minimal energy controlled by the functional in \eqref{genF}; they will play a crucial role in the rest of the paper. For every $(x,t)\in \Omega\times (0,\infty)$, we set
\begin{flalign}\label{GG}
\begin{cases}
\displaystyle G(x,t):=\int_{T}^{\max\{t,T\}}g_{1}(x,s)s \d s \\
 \ \bar{G}(x,t):=G(x,t)+(T^2+1)^{\gamma/2}\,.
 \end{cases}
\end{flalign}
\subsection{Quantifying nonuniform ellipticity} \label{gi} Here we quantify ellipticity and growth of $F$ in the nonuniformly elliptic case. For this, we consider numbers $d, \sigma, \hat \sigma \geq 0$ such that
\eqn{sigmad}
$$
h(\cdot)\in L^{d}(\Omega) \,, \qquad d >n
$$
and 
\eqn{sigma}
$$
  \sigma + \hat \sigma <
\left\{
\begin{array}{ccc}
\min\left\{\frac{1}{n}-\frac{1}{d},\frac{4}{\vartheta(n-2)}\left(1-\frac{1}{\gamma}\right)\right\}  &\mbox{if}& n>2\\  [8 pt]
 \min\left\{\frac{1}{2}-\frac{1}{d},\frac{2}{\vartheta}\left(1-\frac{1}{\gamma}\right)\right\}
 &\mbox{if}& n=2\,.
 \end{array}
 \right.
$$
Here it is $\vartheta\equiv \vartheta(\gamma)=1$ when $\gamma \geq 2$, and $\vartheta=2$ otherwise.
We assume that $x\mapsto g_{1}(x,t)\in W^{1,d}(\Omega)$ for all $t\geq T$ and that $\partial_{x}g_1(\cdot)$ is  Carath\'eodory regular on $\Omega \times [T, \infty)$. Then, for a fixed constant $\cb\geq 1$, we consider the assumptions
\eqn{7}
$$
\snr{\partial_{x}g_{1}(x,t)}\le  h(x)[\bar{G}(x,t)]^{\hat{\sigma}}g_{1}(x,t) \,,
$$
\eqn{8}
$$
\left\{
\begin{array}{c}
\ g_{3}(x,t)t\le \cb[\bar{G}(x,t)]^{1+\sigma}\\ [9 pt]
\displaystyle \ \frac{[g_{3}(x,t)]^{2}}{g_{1}(x,t)}\le \cb[\bar{G}(x,t)]^{1+2\sigma}\,,
\end{array} 
\right.
$$
\eqn{9}
$$
\frac{g_{2}(x,t)}{g_{1}(x,t)}\le \cb[\bar{G}(x,t)]^{\sigma}
$$
for every $t \geq T$, every $x \in \Omega$ in the case of \rif{9}, and for a.e.\,$x\in \Omega$ in the case of \rif{7}-\rif{8}.  
Comments on the meaning of \rif{7}-\rif{9} can be found in Section \ref{diffsec} below. 
\subsection{The uniformly elliptic case} \label{gigi3} Here we instead describe conditions relevant to the uniformly elliptic case, and therefore to models as in \rif{funz-uni}. Such conditions are a slightly more general version of those considered in \cite{CM0, CM1, liebe1}, that are classical. We retain the structure conditions \rif{0ass}-\rif{00} from Section \ref{asssec1} but, instead of using \rif{sigmad}-\rif{9} from Section \ref{gi}, for a fixed constant $\cu\geq 1$ this time we consider
\begin{flalign}\label{uni0}
\begin{cases}
\ t>0 \Longrightarrow g_{i}(x,t)\equiv g_{i}(t), \ \  i\in \{1,2,3\}\\
\ T\leq  t\Longrightarrow   g_{2}(t)\le \cu g_{1}(t)\\
\ T\leq  t\Longrightarrow   g_{1}(t)t\le g_{3}(t)\le \cu g_{1}(t)t\,.
\end{cases}
\end{flalign}
\subsection{General results}\label{general results}
%We are now ready to state our main results in full generality. 
In the following, we abbreviate the assumptions considered as
\eqn{ipotesi}
$$
\begin{cases}
\setm:=\left\{\rif{0ass}-\rif{00}, \rif{sigmad}-\rif{9} 
\right\}\\
\setuni:= \left\{\rif{0ass}-\rif{00}, \rif{uni0}\right\}
\end{cases}
$$
and these are going to be used in the nonuniformly elliptic setting, and in the uniformly elliptic one, respectively.
Accordingly, we also gather the parameters influencing the constants in the a priori estimates as
\eqn{dati}
$$
\begin{cases}
\data:=\left(n,N,\nu,\gamma,T, \cb,d,\sigma,\hat{\sigma}, \mathfrak{a}\right) \\
\datauni:=\left(n,N,\nu, \gamma,T,\cu, g_1(1),\mathfrak{a}\right).
\end{cases}
$$
The presence of $\mathfrak{a}$ only occurs in the two-dimensional case $n=2$. 
\begin{theorem}\label{t1}
Let $u\in W^{1,1}_{\loc}(\Omega;\RN)$ be a minimizer of $\mathcal F$ in \eqref{genF}, under assumptions $\setm$ in \trif{ipotesi}, with $f \in \mathfrak X(\Omega)$ as defined in \trif{lox}. Then $Du\in L^{\infty}_{\textnormal{loc}}(\Omega;\R^{N\times n})$. Moreover,  there exists a positive radius $R_{*}\equiv R_{*}(\data,f(\cdot))\leq 1$ such that if $\B \Subset \Omega$ is a ball with $\rad(\B)\leq R_*$, then
\begin{flalign}
\notag &\nr{G (\cdot,\snr{Du})}_{L^{\infty}(s\B)} \\
&\qquad   \le \frac{c}{(1-s)^{\beta}[\rad(\B)]^{\beta}}
\left[\|F(\cdot,Du)\|_{L^1(\B)}^{ \theta}+\nr{f}_{\mathfrak X(\B)}^{ \theta}+1\right] \label{mainest1}
\end{flalign}
holds for every $s \in (0,1)$, where $c\equiv c(\data,\nr{h}_{L^{d}(\B)})\geq 1$, $ \beta,\theta\equiv \beta,\theta(n,\gamma, \linebreak d,\sigma,\hat{\sigma})>0$.
\end{theorem}
\begin{remark}\label{megliobound0} When either $\gamma \geq 2$ or $f\equiv 0$, in Theorem \ref{t1} we can replace the upper bound on $\sigma +\hat \sigma $ in \rif{sigma} by the less restrictive $\sigma + \hat \sigma < 1/n-1/d$. See Remark \ref{megliobound} below.
\end{remark}
The main Lipschitz regularity result in the uniformly elliptic case is in the next
\begin{theorem}\label{t2}Let $u\in W^{1,1}_{\loc}(\Omega;\RN)$ be a minimizer of $\mathcal F$ in \eqref{genF}, under assumptions $\setuni$ in \trif{ipotesi}, and with $f, h \in \mathfrak X(\Omega)$ as defined in \trif{lox}. Then $Du\in L^{\infty}_{\textnormal{loc}}(\Omega;\R^{N\times n})$. Moreover, there exists a positive radius  $R_{*}\equiv R_{*}(\datauni,h(\cdot))\leq 1$ such that if $\B \Subset \Omega$ is a ball with $\rad(\B)\leq R_*$, then 
\begin{flalign}
&
\|F(\cdot, Du) \|_{L^\infty(s\B)}\notag \\
& \qquad  \le
 \frac{c}{(1-s)^{n}[\rad(\B)]^{n}}\|F(\cdot, Du) \|_{L^1(\B)}+ c\|f\|_{\mathfrak X(\B)}^{\gamma/(\gamma-1)}+c \label{finafinapre}
\end{flalign}
holds  for every $s \in (0,1)$, where $c\equiv c(\datauni)$.
\end{theorem}
In deriving Theorem \ref{t1} we need to prove higher integrability bounds for the gradient, that are worth being singled out in the following: 
\begin{theorem}\label{t3} Let $u\in W^{1,1}_{\loc}(\Omega;\RN)$ be a minimizer of $\mathcal F$ in \eqref{genF}, under assumptions $\setm$ in \trif{ipotesi} with $n>2$, and with $ f \in L^n(\Omega)$. Then $Du\in L^{p}_{\textnormal{loc}}(\Omega;\er^{N\times n})$ for every $p\in [1,\infty)$. Moreover, for every $p\in [1,\infty)$ there exists a positive radius $R_{*}\equiv R_{*}(\data,f(\cdot),p)\leq 1$ such that, if $\B \Subset \Omega$ is a ball with $\rad(\B)\leq R_*$, then
\begin{flalign}
&\|G(\cdot,\snr{Du})\|_{L^p(s\B)} \notag \\ & \qquad \le \frac{c}{(1-s)^{\beta_{p}}[\rad(\B)]^{\beta_p}}\left[\|F(\cdot,Du)\|_{L^1(\B)} ^{\theta_p}+\nr{f}_{L^{n}(\B)}^{\theta_p}+ 1\right]\label{e13}
\end{flalign}
holds for every $s \in (0,1)$, where $c\equiv c(\data,\nr{h}_{L^{d}(\B)},p)$, $\theta_{p}, \beta_{p}\equiv \theta_{p}, \beta_{p}(n,\linebreak \gamma, d,\sigma,\hat{\sigma},p)$.
\end{theorem}
\begin{remark}\label{lpremark}
In order to prove Theorem \ref{t3} the full strength of \rif{sigma} is not actually needed. Assuming 
$  \sigma + \hat \sigma < 1/n-1/d$ suffices.
\end{remark}
\begin{remark}\label{remT} In Theorems \ref{t1}-\ref{t3}, as well as in the other estimates in this paper, the constants depending on $\data$ and $\datauni$  blow-up when $T\to \infty$ (complete loss of ellipticity). On the contrary, they remain bounded when $T \to 0$ (full recovery of ellipticity), as long as the quantity $\nu\equiv \nu(T)$ in \rif{00bis} stays bounded away from zero. This is for instance the case in Theorems \ref{sample1}-\ref{sample6}; see Section \ref{applications} below. As a matter of fact, there is an additional dependence of the constants on the specific operator considered. Indeed, the dependence on $T$ typically shows up via quantities controlled by $\| \tilde a(\cdot, T)\|_{L^{\infty}}(T^2+\mu^2)+T^\gamma$; this can be in turn controlled as $\lesssim \| \tilde a(\cdot, 1)\|_{L^{\infty}}(T^\gamma+\mu^\gamma)$ for $T\leq 1$ via \rif{00}. All in all, the dependence of the constants on the specific operator (i.e., $\partial_x F$) appears only through $\nu$ defined in \rif{00bis}, and $\| \tilde a(\cdot, 1)\|_{L^{\infty}}$ when $T$ approaches zero. See also the appearance of $g_1(1)$ in  $\datauni$ in the uniformly elliptic case.
\end{remark}
Finally, we give another general result concerning an intermediate form of nonuniform ellipticity. This for instance applies to the case of functionals with special structure as the double phase one in \rif{pqfunctional}. More comments on this will follow in Section \ref{diffsec} below. 
\begin{theorem}\label{t4}
Let $u\in W^{1,1}_{\loc}(\Omega;\RN)$ be a minimizer of $\mathcal F$ in \eqref{genF}, under assumptions \textnormal{$\setm$} in \trif{ipotesi} with $\sigma =  \hat \sigma$ in \eqref{7}-\eqref{8} and $ f$ as in \trif{lox}. Moreover, replace \trif{sigma} and \trif{9} by 
\eqn{sigma-int}
$$
 \hat \sigma <
\left\{
\begin{array}{ccc}
1/n-1/d&\mbox{if}& n>2\\  [6 pt]
\min\{ 1/2-1/d, \gamma-1\}
 &\mbox{if}& n=2
 \end{array}
 \right.
 \quad \mbox{and} \quad \frac{g_{2}(x,t)}{g_{1}(x,t)}\le \cb\,,
$$
respectively. Then $Du\in L^{\infty}_{\textnormal{loc}}(\Omega;\R^{N\times n})$ and \trif{mainest1} holds. \end{theorem}
\begin{remark}
Replacing the integrand $F$ in \rif{genF} with the new one $F_0(x,z):= F(x,z)-\tilde{F}(x,0)$, does not change the set of minimizers of $\mathcal F$, and gets another functional still satisfying the conditions of Theorems \ref{t1}-\ref{t4}. Moreover, as $\tilde{F}(x,0)\geq 0$, once the estimates in \rif{mainest1}-\rif{e13} are proved in the case $F_0(x,z)$ is considered, then they also hold in the original case. Therefore for the rest of the paper we can always assume that $\tilde{F}(x,0)=0$ and by \rif{wehaweha}
%\eqn{assumemore}
%$$\tilde{F}(x,0)=0  \qquad \mbox{holds for every $x\in \Omega$}\,.$$
it follows that
\begin{flalign}\label{tildef}
\tilde{F}(x,t)=\int_{0}^{t}\tilde{a}(x,s)s \d s\quad \mbox{for} \ (x,t)\in \Omega\times (0,\infty)\,.
\end{flalign}
Such a replacement makes sense provided $\tilde{F}(x,0)\in L^1_{\loc}(\Omega)$, but this is always the case in our situation. Specifically, this is true as soon as the functional $\mathcal F$ in \eqref{genF} admits a minimizer $u$, which is therefore such that $\tilde F(\cdot, |Du|)\in L^1_{\loc}(\Omega)$. This implies that $0\leq \tilde F (x, 0)\leq 
\tilde F(x, |Du(x)|)$ (recall that $t \mapsto \tilde F(x, t)$ is non-decreasing), so that $\tilde{F}(\cdot,0)\in L^1_{\loc}(\Omega)$. 
\end{remark}
\subsection{Different notions of nonuniform ellipticity}\label{diffsec} What do we call nonuniform ellipticity of the integrand  $F$ here? To provide a measure, in the autonomous case $F(x,z)\equiv F(z)$, it is rather standard to use the ellipticity ratio  
\eqn{ratio1}
$$
\mathcal R_{F} (z):=\frac{\mbox{highest eigenvalue of}\ \partial_{zz}  F(z)}{\mbox{lowest eigenvalue of}\  \partial_{zz} F(z)} \,.
$$
The occurrence of nonuniform ellipticity then refers to the fact that $\mathcal R_{F} (z) \to \infty$ as $|z|\to \infty$; accordingly, the uniformly elliptic case occurs when $\mathcal R_{F} (\cdot)\approx 1$. In the nonautonomous case this leads to define the same pointwise quantity $\mathcal R_{F} (x,z)$, by taking into account the $x$-dependence in the right-hand side of \rif{ratio1}. In this paper we have considered assumptions aimed at bounding $\mathcal R_{F} (x,z)$ pointwise, i.e.,
\eqn{ratio11}
$$
\mathcal R_F (x,z) \lesssim \frac{g_2(x,|z|)}{g_1(x,|z|)} \stackrel{\rif{9}}{\lesssim} \left(\int^{|z|}_Tg_1(x,s)s \, ds\right)^{\sigma}+1\,.
$$
On the other hand, the ratio $\mathcal R_F (x,z)$ misses to properly encode full information in the nonautonomous one. This leads to consider the nonlocal quantity 
\eqn{ratio2}
$$\mathfrak {R}_{F}(z,B):=\frac{\sup_{x\in B}\, \mbox{highest eigenvalue of} \ \partial_{zz}F(x,z)}{\inf_{x\in B}\, \mbox{lowest eigenvalue of} \ \partial_{zz}F(x,z)}\,,$$
with $B\subset \Omega$ being any ball. The new ratio $\mathfrak {R}_{F}(z,B)$ naturally occurs in integral estimates and, in fact, turns out to be the right quantity to describe intermediate notions of nonuniform ellipticity in the nonautonomous case. 
Obviously, from the definitions it follows that 
$\sup_{x\in B}\, \mathcal R_{F} (x,z) \leq \mathfrak {R}_{F}(z,B)$, with equality in the autonomous case. 

\begin{remark}[Double face of the double phase functional \eqref{pqfunctional}]\label{doubleremark}To further motivate \rif{ratio2}, recall that for general uniformly elliptic problems one typically recovers classical theories as Schauder's  \cite{liebe1} and Calder\'on-Zygmund's \cite{BLO}. This is in general not the case for the double phase integral in \rif{pqfunctional}, due to the counterexamples in \cite{sharp, FMM}. Nevertheless the integrand $H(\cdot)$ turns out to be uniformly elliptic when using the quantity in \rif{ratio1} as test, in the sense that $\mathcal R_{H} (\cdot) \approx 1$. On the other hand, when considering a ball $B$ such that $B\cap \{a(x)=0\}\not =0$, we have that  $\mathfrak {R}_{H}(z,B)\approx |B|^{1/n-1/d}|z|^{q-p}+1$, therefore $\mathfrak {R}_{H}(z,B) \to \infty$ when $|z|\to \infty$. In such sense, the use of quantity in \rif{ratio2} resolves the above ambiguity. In this case, the assumptions \rif{lox22} provide a way to correct the growth of $\mathfrak {R}_{H}(z,B)$ with respect to $|z|$ with the smallness of $|B|$, as pointed out in the Introduction.
 \end{remark}
In this paper we consider different assumptions, playing with the parameters $\sigma, \hat \sigma$ in \rif{7}-\rif{9}, in order to tune different degrees of nonuniform ellipticity, involving both $\mathcal R_{F} (\cdot)$ and $\mathfrak {R}_{F}(\cdot)$. In Theorems \ref{t1}-\ref{t2}, we prescribe a direct pointwise bound on $\mathcal R_{F} (\cdot)$ as in \rif{ratio11}, and then control the behaviour with respect to $x$ of the derivatives of $F$ via \rif{7}-\rif{8}. This has the overall effect of providing an indirect control on $\mathfrak {R}_{F}(\cdot)$ too. In Theorem \ref{t4}, we bound $\mathcal R_F (\cdot)\lesssim 1$, that is, uniform ellipticity is assumed in the pointwise sense, still allowing for nonuniform ellipticity in the nonlocal sense of \rif{ratio2}. This leads for instance to get better bounds and opens the way to Theorem \ref{sample3}, dealing with the intermediate case of double phase functionals in the sense of Remark \ref{doubleremark}.

\section{Approximation of integrands}\label{apro} Here we implement a truncation scheme aimed at approximating the original integrand $F$ in \rif{genF} with a family $\{F_{\varepsilon}\}$ of integrands with standard polynomial growth, and converging to $F$ locally uniformly. The new integrands preserve structure properties as \eqref{0ass}-\eqref{growth}, with corresponding control functions $g_{i, \varepsilon},$ still satisfying relations as those of the original ones. For this we will exploit a few arguments used in \cite{BM} as a starting point. In this section we shall permanently assume that \eqref{0ass}-\eqref{00} are in force, so that their consequences \eqref{1bis}-\eqref{banaleg1g2} can be used as well. Additional assumptions as \rif{7}-\rif{9} shall also be considered; in case, this will be explicitly mentioned. In the following we use a parameter $\varepsilon$ that will always be such that $0<\varepsilon<\min\{1,T\}/4$ and use that $\tilde F(x, 0)\equiv 0$.  

\subsection{General setup}\label{genset} Given $T$ and $\mu$ in \rif{growth} and \rif{00}, respecively, we introduce the numbers
\begin{flalign}\label{z1}
\mu_{\varepsilon}:=\mu+\varepsilon\,, \qquad  T_{\varepsilon}:=T+1/\varepsilon\,.
\end{flalign}
With $\tilde{a}\colon \Omega \times (0,\infty)\mapsto [0, \infty)$ that has been defined in \rif{general1} as 
\begin{flalign}\label{tildea}
\tilde{a}(x,t):=\frac{\tilde{F}'(x,t)}{t}
\end{flalign}
for $t>0$, which is continuous on $\Omega \times (0, \infty)$, we start introducing the functions $\tilde{a}_{\varepsilon}\colon \Omega\times [0,\infty)\to [0,\infty)$ as
\begin{flalign}\label{z2}
\tilde{a}_{\varepsilon}(x,t):=\begin{cases} 
\ \frac{\tilde{a}(x,\varepsilon)}{(\varepsilon^{2}+\mu^{2}_{\varepsilon})^{\frac{\gamma-2}{2}}}(t^{2}+\mu^{2}_{\varepsilon})^{\frac{\gamma-2}{2}} \ \ &\mbox{if}  \ t\in [0,\varepsilon)\\
\ \tilde{a}(x,t)\ \ &\mbox{if}  \ t\in [\varepsilon,T_{\varepsilon})\\
\ \frac{\tilde{a}(x,T_{\varepsilon})}{(T_{\varepsilon}^{2}+\mu^{2}_{\varepsilon})^{\frac{\gamma-2}{2}}}(t^{2}+\mu^{2}_{\varepsilon})^{\frac{\gamma-2}{2}} \ \ &\mbox{if} \ t\in [T_{\varepsilon},\infty)\,,
\end{cases}
\end{flalign}
for every $x\in \Omega$, where $\gamma$ also appears in \rif{00}; accordingly, we define
\eqn{z5}
$$
\begin{cases}
\displaystyle F_{\varepsilon}(x,z):=\tilde{F}_{\varepsilon}(x,\snr{z}), \  \mbox{for} \  \tilde{F}_{\varepsilon}(x,t):=\int_{0}^{t}\tilde{a}_{\varepsilon}(x,s)s \d s+\varepsilon L_{\gamma,\varepsilon}(t)\\
\displaystyle 
L_{\gamma,\varepsilon}(t):=\frac{1}{\gamma}\left[(t^{2}+\mu_{\varepsilon}^{2})^{\gamma/2}-\mu_{\varepsilon}^{\gamma}\right]=\int_{0}^{t}(s^{2}+\mu_{\varepsilon}^{2})^{\frac{\gamma-2}{2}}s \d s\,,
\end{cases}
$$
and, finally
\begin{flalign}\label{arad}
\notag &\bar{a}_{\eps}(x,t):= \tilde{a}_{\eps}(x,t)+\eps(t^{2}+\mu_{\varepsilon}^{2})^{\frac{\gamma-2}{2}}\\ 
& \qquad  \Longrightarrow  a_{\eps}(x,z):=\bar{a}_{\eps}(x,\snr{z})z=\partial_z F_{\eps}(x,z) \,. 
\end{flalign} 
Note that $\tilde a_{\eps}$ is continuous on $\Omega \times [0,\infty)$ and $\partial_x \tilde a_{\eps}$ is Carath\'eodory regular; $(x, t) \mapsto \tilde a'_\eps(x, t)$ is instead Borel regular.  
In view of \eqref{0ass} and \rif{tildef}, it follows
\begin{flalign}\label{z3}
\begin{cases}
\ t\mapsto \tilde{F}_{\varepsilon}(x,t)\in C^{1}_{\textnormal{loc}}[0,\infty)\cap C^{2}_{\textnormal{loc}}\left([0,\infty)\setminus\{\varepsilon,T_{\varepsilon}\}\right)\ \ \mbox{for all} \ x\in \Omega\\
\ x\mapsto \tilde{F}'_{\varepsilon}(x,t)\in W^{1,n}(\Omega) \ \ \mbox{for all} \ \ t\in [0,\infty)\\
\ t\mapsto \tilde{a}_{\varepsilon}(x,t)\in \textnormal{Lip}_{\textnormal{loc}}[0,\infty)\cap C^{1}_{\textnormal{loc}}\left([0,\infty)\setminus\{\varepsilon,T_{\varepsilon}\}\right)\ \ \mbox{for all} \ x\in \Omega\\
\ t\mapsto \tilde{F}_{\varepsilon}(x,t)\ \ \mbox{is strictly convex and non-decreasing for all $x\in \Omega$}\\
\ F_{\eps} \ \mbox{is continuous on $\Omega \times \er^{N\times n}$ and $\tilde F_\eps(x, 0)=0$}\\
\ F_{\eps}\to F\quad  \mbox{uniformly on $ \Omega \times \mathcal K$ as $\eps \to 0$, $\forall \, \mbox{compact} \, \mathcal K \subset \er^{N\times n}$ }\\
\ \mbox{$\partial_x \tilde a_{\eps}$ is Carath\'eodory-regular on $\Omega \times [0, \infty)$}\,.
\end{cases}
\end{flalign}
For \rif{z3}$_4$, see \rif{einc} below. 
For the proof of \rif{z3}$_5$, recall that $\tilde F'$ is continuous on $\Omega \times (0, \infty)$. 
For details about \rif{z3}$_6$, see also Lemma \ref{ultimolemma} below. For \rif{z3}$_7$ use that  $\partial_x \tilde F'$ is Carath\'eodory.  

Definitions \rif{z1}-\rif{z2} lead to introduce new control functions $g_{i,\varepsilon}\colon \Omega\times [0,\infty)\to (0,\infty)$, $i\in \{1,2,3\}$ as:
\begin{flalign}
&g_{1,\varepsilon}(x,t):=\mathfrak{g}_{1}\begin{cases}\label{g1e}
\ \frac{g_{1}(x,\varepsilon)}{(\varepsilon^{2}+\mu_{\varepsilon}^{2})^{\frac{\gamma-2}{2}}}(t^{2}+\mu_{\varepsilon}^{2})^{\frac{\gamma-2}{2}}\quad &\mbox{if} \ t\in [0,\varepsilon)\\  
\ g_{1}(x,t)\quad &\mbox{if} \ t\in [\varepsilon,T_{\varepsilon})\\
\ \frac{g_{1}(x,T_{\varepsilon})}{(T_{\varepsilon}^{2}+\mu_{\varepsilon}^{2})^{\frac{\gamma-2}{2}}}(t^{2}+\mu_{\varepsilon}^{2})^{\frac{\gamma-2}{2}}\quad &\mbox{if} \ t\in [T_{\varepsilon},\infty)\,,
\end{cases}\\
&g_{2,\varepsilon}(x,t):=\go_{2}\begin{cases}\label{g2e}
\ \left[\frac{g_{2}(x,\varepsilon)}{(\varepsilon^{2}+\mu_{\varepsilon}^{2})^{\frac{\gamma-2}{2}}}+\varepsilon\right](t^{2}+\mu_{\varepsilon}^{2})^{\frac{\gamma-2}{2}}\quad &\mbox{if} \ t\in [0,\varepsilon)\\
\ g_{2}(x,t)+\varepsilon(t^{2}+\mu_{\varepsilon}^{2})^{\frac{\gamma-2}{2}}\quad &\mbox{if} \ t\in [\varepsilon,T_{\varepsilon})\\
\ \left[ \frac{g_{2}(x,T_{\varepsilon})}{(T_{\varepsilon}^{2}+\mu_{\varepsilon}^{2})^{\frac{\gamma-2}{2}}}+\varepsilon\right](t^{2}+\mu_{\varepsilon}^{2})^{\frac{\gamma-2}{2}}\quad &\mbox{if} \ t\in [T_{\varepsilon},\infty)\,,
\end{cases}\\
&g_{3,\varepsilon}(x,t):=\begin{cases}\label{g3e}
\ \frac{g_{3}(x,\varepsilon)}{(\varepsilon^{2}+\mu_{\varepsilon}^{2})^{\frac{\gamma-2}{2}}\eps}(t^{2}+\mu_{\varepsilon}^{2})^{\frac{\gamma-2}{2}}t\quad &\mbox{if} \ t\in [0,\varepsilon)\\
\ g_{3}(x,t)\quad &\mbox{if} \ t\in [\varepsilon,T_{\varepsilon})\\
\ \frac{g_{3}(x,T_{\varepsilon})}{(T_{\varepsilon}^{2}+\mu_{\varepsilon}^{2})^{\frac{\gamma-2}{2}}T_\eps}(t^{2}+\mu_{\varepsilon}^{2})^{\frac{\gamma-2}{2}}t\quad &\mbox{if} \ t\in [T_{\varepsilon},\infty)\,,
\end{cases}
\end{flalign}
%2 va bene al posto di 20
where the constants $\go_{1}$ and $\go_{2}$ are defined by
\begin{flalign}\label{g123}
\go_{1}:=\min\{1,\gamma-1\}\le 1 \le \go_{2}:=20(Nn+\gamma)\,.%, \ \  \go_{3}:=\sqrt{2}\,.
\end{flalign}
Observe that $g_{1,\eps}, g_{2, \eps}$ are still continuous and $g_{3,\eps}$ is still Carath\'eodory regular. We next introduce also the truncated counterparts of the functions defined in \eqref{GG}, i.e.,
\eqn{gte}
$$
\begin{cases}
\displaystyle G_{\varepsilon}(x,t):=\int_{T}^{\max\{t,T\}}g_{1,\varepsilon}(x,s)s \d s \\
 \bar{G}_{\varepsilon}(x,t):=G_{\varepsilon}(x,t)+(T^2+1)^{\gamma/2}\,.
\end{cases}
$$
In the following, recalling \rif{z1}, we shall repeatedly use
\eqn{repeat}
$$
1 \leq  \frac{s^2+\mu_\eps^2}{s^2+\mu^2}  \leq 3, \  \mbox{provided $s \geq \eps$}
$$
and that this is a decreasing of $s$.
\subsection{Four technical lemmas}
\begin{lemma}\label{gle} Under assumptions \eqref{0ass}-\eqref{00}. 
\begin{itemize}
\item  
There exist constants $\{\Lambda_{\varepsilon}\}$ and $\{L_{\varepsilon}\}$ such that the following properties hold:
\begin{flalign}\label{growthd}
\begin{cases}
\ \snr{\partial_{zz}F_{\varepsilon}(x,z)}\le g_{2,\varepsilon}(x,\snr{z}) \ &\mbox{if} \  (x,z)\in \Omega\times\left\{\snr{z}\geq T\,, \snr{z}\not =T_{\varepsilon}\right\}\\
\ g_{1,\varepsilon}(x,\snr{z})\snr{\xi}^{2}\le  \partial_{zz}F_{\varepsilon}(x,z)\xi\cdot\xi  \ &\mbox{if} \  (x,z)\in \Omega\times\left\{\snr{z}\geq T\,, \snr{z}\not =T_{\varepsilon}\right\}\\
\ \snr{\partial_{xz}F_{\varepsilon}(x,z)}\le h(x)g_{3,\varepsilon}(x,\snr{z}) \ &\mbox{for a.e.\,$x\in \Omega$ and every $z\in \mathbb{R}^{N\times n}$}\\
\  |\partial_{x}\bar{a}_{\eps}(x,\snr{z})||z|\le h(x)g_{3,\varepsilon}(x,\snr{z}) \ &\mbox{for a.e.\,$x\in \Omega$ and every $z\in \mathbb{R}^{N\times n}$}
\end{cases}
\end{flalign}
and
\begin{flalign}\label{zgrowthd}
\begin{cases}
\ \snr{\partial_{zz}F_{\varepsilon}(x,z)}\le \Lambda_{\varepsilon}(\snr{z}^{2}+\mu_{\varepsilon}^{2})^{\frac{\gamma-2}{2}}\ &\mbox{if} \  (x,z)\in \Omega\times \left\{\snr{z}\not=\varepsilon,T_{\varepsilon}\right\}\\
\ \varepsilon\go_{1}(\snr{z}^{2}+\mu_{\varepsilon}^{2})^{\frac{\gamma-2}{2}}\snr{\xi}^{2}\le \partial_{zz}F_{\varepsilon}(x,z)\xi\cdot \xi\  &\mbox{if} \  (x,z)\in \Omega\times \left\{\snr{z}\not =\varepsilon,T_{\varepsilon}\right\}\\
\ \snr{\partial_{xz}F_{\varepsilon}(x,z)}\le L_{\varepsilon}h(x)(\snr{z}^{2}+\mu_{\varepsilon}^{2})^{\frac{\gamma-2}{2}}|z|\ &\mbox{for a.e.\,$x\in \Omega,\ \forall \, z\in \mathbb{R}^{N\times n}$}
\end{cases}
\end{flalign}
for all $\xi \in \mathbb{R}^{N\times n}$. As a consequence, it follows that
\begin{flalign}\label{a1}
 g_{1,\varepsilon}(x,\snr{z})\leq  \bar{a}_{\eps}(x,\snr{z})\le g_{2,\varepsilon}(x,\snr{z})
\end{flalign}
holds for all $(x,z)\in \Omega\times \{\snr{z}\geq T\}$, and
\begin{flalign}\label{a2}
\begin{cases}
\ \bar{a}_{\eps}(x,\snr{z})+\bar{a}_{\eps}'(x,\snr{z})\snr{z}\ge g_{1,\varepsilon}(x,\snr{z})\\
\ \bar{a}_{\eps}(x,\snr{z})+\bar{a}_{\eps}'(x,\snr{z})\snr{z}\le g_{2,\varepsilon}(x,\snr{z})
\end{cases}
\end{flalign}
hold for all $(x,z)\in \Omega\times \{\snr{z}\geq T, \ \snr{z}\not = T_{\varepsilon}\}$. 
    \item It holds that 
    \eqn{einc}
$$
%x\in \Omega\,, \   
0 < s \leq t \  \Longrightarrow \begin{cases}
 g_{1,\eps}(x,s)s \leq
g_{1,\eps}(x,t)t\\\
 \tilde a_{\eps}(x,s)s \leq
\tilde a_{\eps}(x,t)t\,.
\end{cases}
$$
In particular, the function $t \mapsto G_{\varepsilon}(\cdot, t)$ in \trif{gte} is convex. 
\item For every $(x, z)\in \Omega\times \{\snr{z}\geq T\}$, it holds that:
\begin{flalign}\label{z6}
\begin{cases}
\ c(\nu,\gamma)(|z|^2+\mu_{\varepsilon}^2)^{(\gamma-2)/2}\leq g_{1,\varepsilon}(x,|z|)\\
\  G(x,\snr{z})\leq F(x,z)\\
G_{\varepsilon}(x,\snr{z})\leq F_{\varepsilon}(x,z)\,.
\end{cases}
\end{flalign}
For every $(x,t) \in \Omega \times [T, \infty)$, there holds
    \eqn{z10a}
    $$
    \left\{
    \begin{array}{c}
  \displaystyle  (t^{2}+\mu_{\varepsilon}^{2})^{\gamma/2}\le c(\nu, \gamma) \bar{G}_{\varepsilon}(x,t)\\ 
    %  \displaystyle  1\le cg_{1,\varepsilon}(x,t)[\bar{G}_{\varepsilon}(x,t)]^{1/\gamma} \quad \mbox{for}\  c=(\go g_{1}^{-}(T))^{-1}2^{1+\frac{1}{\gamma}}\\ \\
    \displaystyle   {G}_{\varepsilon}(x,t)\leq c(\nu, \gamma) [\bar{G}_{\varepsilon}(x,t)]^{2/\gamma}g_{1,\varepsilon}(x,t)\,.
    \end{array}
    \right.
    $$
Again for every $(x, z)\in \Omega\times \{\snr{z}\geq T\}$, it holds    
\begin{flalign}\label{z7}
\begin{cases}
\ c(\nu, \gamma) F(x,z)\ge (\snr{z}^{2}+\mu^{2})^{\gamma/2}-(T^{2}+\mu^{2})^{\gamma/2}\\
\  c(\nu, \gamma)F_{\varepsilon}(x,z)\ge (\snr{z}^{2}+\mu^{2})^{\gamma/2}-(T^{2}+\mu^{2})^{\gamma/2}\,.
\end{cases}
\end{flalign}
\item For every $x \in \Omega$ and for a fixed constant $c\equiv c (\nu, \gamma)$, we have 
\eqn{z6bis}
$$
\begin{cases}
\tilde a_{\eps} (x,t)\leq c \tilde a (x,t) \quad  \mbox{for $t \geq \eps$} \\
(t^2+\mu_{\eps}^2)^{(\gamma-2)/2}\leq  c\tilde a (x,t) \quad \mbox{for $t\geq T$}\,.
\end{cases}
$$
\item  Finally, for another constant $c$ depending on  $\nu$ and $\|\tilde a(\cdot, T)\|_{L^{\infty}}(T^2+\mu^2)+T^\gamma$, we have that 
\eqn{dominaF}
$$
F_{\eps}(x,z) \leq cF(x,z)+ c \quad \mbox{holds for all }  (x,z) \in \Omega \times \er^{N\times n}
$$ 
and that
\eqn{dimenticata}
$$
\frac{\eps}{\gamma} (t^2+\mu_\eps^2)^{\gamma/2}- \frac{\eps\mu_\eps^\gamma}{\gamma} \leq \tilde F_\eps(x,t)\leq c_{\eps}(t^2+\mu_{\eps}^2)^{\gamma/2}
$$
holds for all $(x,t)\in \Omega\times [0,\infty)$\,.
\end{itemize}
\end{lemma}
\begin{proof}
By the very definitions in \eqref{z2}-\eqref{z5}, we note that $\partial_{zz}F_{\varepsilon}(x,z)$ exists for all $(x,z)\in \Omega\times \left\{\snr{z}\not=\varepsilon,T_{\varepsilon}\right\}$ with
\begin{flalign*}
\partial_{zz}F_{\varepsilon}(x,z)=\begin{cases}
\ (\snr{z}^{2}+\mu_{\varepsilon}^{2})^{\frac{\gamma-2}{2}}\left[\frac{\tilde{a}(x,\varepsilon)}{(\varepsilon^{2}+\mu_{\varepsilon}^{2})^{\frac{\gamma-2}{2}}}+\varepsilon\right]\mathcal{C}_{\varepsilon}(z) \\ \qquad   \quad \mbox{if} \ (x,z)\in \Omega\times \{\snr{z}<\varepsilon\}\\
\ \partial_{zz}F(x,z)+\varepsilon (\snr{z}^{2}+\mu_{\varepsilon}^{2})^{\frac{\gamma-2}{2}} \mathcal{C}_{\varepsilon}(z)\\ 
\qquad \quad \mbox{if} \ (x,z)\ \in \Omega\times\{\varepsilon < \snr{z}<T_{\varepsilon}\}\\
\ (\snr{z}^{2}+\mu_{\varepsilon}^{2})^{\frac{\gamma-2}{2}}\left[\frac{\tilde{a}(x,T_{\varepsilon})}{(T_{\varepsilon}^{2}+\mu_{\varepsilon}^{2})^{\frac{\gamma-2}{2}}}+\varepsilon\right]\mathcal{C}_{\varepsilon}(z) \\
\qquad \quad \mbox{if} \ (x,z)\in \Omega\times\{T_{\varepsilon} < \snr{z}\}\,,
\end{cases}
\end{flalign*}
where
\begin{flalign*}
\mathcal{C}_{\varepsilon}(z):=\mathbb{I}_{N\times n}+(\gamma-2)\frac{z\otimes 
z}{\snr{z}^{2}+\mu_{\varepsilon}^{2}}\quad \mbox{for}  \  z\in \mathbb{R}^{N\times n}\,.
\end{flalign*}
Moreover, recalling \rif{z2}, by the definition of weak derivatives we have 
\begin{flalign*}
\partial_{xz}F_{\varepsilon}(x,z)=\begin{cases} 
\ \frac{\partial_x\tilde{a}(x,\varepsilon)}{(\varepsilon^{2}+\mu^{2}_{\varepsilon})^{\frac{\gamma-2}{2}}}(|z|^{2}+\mu^{2}_{\varepsilon})^{\frac{\gamma-2}{2}} z &\mbox{if} \ (x,z)\in \Omega\times \{\snr{z}<\varepsilon\}\\
\ \partial_x\tilde{a}(x,|z|)z &\mbox{if} \ (x,z)\ \in \Omega\times\{\varepsilon\le\snr{z}<T_{\varepsilon}\}\\
\ \frac{\partial_x\tilde{a}(x,T_{\varepsilon})}{(T_{\varepsilon}^{2}+\mu^{2}_{\varepsilon})^{\frac{\gamma-2}{2}}}(|z|^{2}+\mu^{2}_{\varepsilon})^{\frac{\gamma-2}{2}}z&\mbox{if} \ (x,z)\in \Omega\times\{ T_{\varepsilon}\leq \snr{z}\}\,.
\end{cases}
\end{flalign*}
Then \eqref{growthd}-\eqref{zgrowthd} directly follow by \eqref{growth}, \rif{growthzero}, \eqref{z2}, \eqref{g1e}-\eqref{g3e}, and with the choice\begin{flalign*}
&\Lambda_{\varepsilon}:=\go_{2}\left[1+\sup_{x\in \Omega}\frac{2g_{2}(x,T_{\varepsilon})}{(T_{\varepsilon}^2+\mu_{\varepsilon}^{2})^{\frac{\gamma-2}{2}}}+\sup_{x\in \Omega,\snr{z}\in [\varepsilon,T_{\varepsilon}]}\frac{\snr{\partial_{zz}F(x,z)}}{(\snr{z}^{2}+\mu_{\varepsilon}^{2})^{\frac{\gamma-2}{2}}}\right],\\
&L_{\varepsilon}:=\sup_{t\in [\eps,T_{\varepsilon}]}\frac{\|g_{3}(\cdot,t)\|_{L^\infty(\Omega)}}{(t^{2}+\mu_{\varepsilon}^{2})^{\frac{\gamma-2}{2}}t}\,.
\end{flalign*}
Note that we are also using \rif{23}$_2$ and \rif{00} in order to get upper bounds for $|\partial_{zz}F_{\varepsilon}(x,z)|$, in \rif{growthd}$_1$ and \rif{zgrowthd}$_1$, respectively; also \rif{repeat} is used to estimate
$$
\frac{\tilde{a}(x,\eps)}{(\varepsilon^2+\mu_{\varepsilon}^{2})^{\frac{\gamma-2}{2}}}\leq 
\frac{2\tilde{a}(x,T_{\varepsilon})}{(T_{\varepsilon}^2+\mu_{\varepsilon}^{2})^{\frac{\gamma-2}{2}}}\leq
\frac{2g_{2}(x,T_{\varepsilon})}{(T_{\varepsilon}^2+\mu_{\varepsilon}^{2})^{\frac{\gamma-2}{2}}}\,.
$$
We also used \rif{22}$_2$ to get \rif{growthd}$_2$. 

As for \eqref{a1}-\eqref{a2}, when $|z| \not = T_{\eps}$, these follow from definition \rif{arad} and \eqref{growthd}, arguing exactly as for \eqref{22}-\eqref{23}. The case $|z| = T_{\eps}$ of \eqref{a1} follows by continuity.

The property in \rif{einc} readily follows from \rif{1bis} and the definitions in \eqref{g1e}-\eqref{g2e}, also using the fact that $t \mapsto (t^{2}+\mu_{\varepsilon}^{2})^{(\gamma-2)/2}t$ is increasing.

We now come to \rif{z6}-\rif{z7}. For $\eqref{z6}_{2}$-$\eqref{z7}_{1}$ we note that 
\begin{eqnarray}
F(x,z)&\geq & \int_{T}^{\snr{z}}\tilde{a}(x,s)s \d s \stackrel{\eqref{22}_{2}}{\ge} \int_{T}^{\snr{z}}g_{1}(x,s)s \d s \stackrel{\eqref{GG}}{=}G(x,\snr{z})\nonumber \\
& \stackrel{\eqref{00bis}}{\geq}& \nu\int_{T}^{\snr{z}}(s^{2}+\mu^{2})^{\frac{\gamma-2}{2}}s \d s\stackrel{\eqref{Hnot}}{=}\frac{\nu}{\gamma}\left\{[E(\snr{z})]^{\gamma}-[E(T)]^{\gamma}\right\}.
\label{plano0}
\end{eqnarray}
\indent Let us now take care of \eqref{z10a}$_1$. For $t\in [T,T_{\varepsilon})$, using \eqref{00bis} and \eqref{repeat}, we see that (recall that $\eps \leq T/4$)
\begin{flalign}
\notag \bar{G}_{\varepsilon}(x,t)& \ge \frac{\go_1\nu}{\gamma}\left[(t^{2}+\mu^{2})^{\gamma/2}-(T^{2}+\mu^{2})^{\gamma/2}\right]+(T^2+1)^{\gamma/2}\\
&\ge \frac{(t^{2}+\mu_{\varepsilon}^{2})^{\gamma/2}}{c(\nu, \gamma)}\,,\label{plano1}
\end{flalign}
while, when $t\ge T_{\varepsilon}$ by analogous means, we have
\begin{flalign}
\notag& \bar{G}_{\varepsilon}(x,t)=\go_{1}\int_{T_{\varepsilon}}^{t}\frac{g_{1}(x,T_{\varepsilon})}{(T_{\varepsilon}^{2}+\mu_{\varepsilon}^{2})^{\frac{\gamma-2}{2}}}(s^{2}+\mu_{\varepsilon}^{2})^{\frac{\gamma-2}{2}}s \d s \\
& \notag \qquad \qquad + \go_{1}\int_{T}^{T_{\varepsilon}}g_{1}(x,s)s \d s+(T^2+1)^{\gamma/2}\nonumber \\
&\qquad \ \ \ \ge \frac{1}{c(\nu, \gamma)}\int_{T}^{t}(s^{2}+\mu^{2})^{\frac{\gamma-2}{2}}s \d s+(T^2+1)^{\gamma/2}\ge \frac{(t^{2}+\mu_{\varepsilon}^{2})^{\gamma/2}}{c(\nu, \gamma)}\,.\label{plano2}
\end{flalign}
Implicit in the arguments from \rif{plano1}-\rif{plano2} is also the proof of $\eqref{z6}_{1}$.  
Inequalites $\eqref{z6}_{3}$ and $\eqref{z7}_{2}$ follow using the lower bound in \rif{a1} and then  \rif{z6}$_1$, as in \rif{plano0}. As for $\eqref{z10a}_{2}$, note that 
\begin{eqnarray}
\notag G_{\varepsilon}(x,t) =  \int_{T}^t g_{1, \eps}(x,s) s \, ds & \stackleq{einc} &
g_{1, \eps}(x,t) t \int_{T}^t  \, ds\\
& = &  g_{1,\eps}(x,t)t (t-T)\leq g_{1, \eps}(x,t) t^2 \label{wewewe}
\end{eqnarray}
holds whenever $t \geq T$. Now $\eqref{z10a}_{2}$ follows using \rif{wewewe} with $\eqref{z10a}_{1}$. 

The proof of \rif{z6bis}$_1$ follows straightaway from the definition of $\tilde a_{\eps}$ in \rif{z2}, the first property in \rif{00} and \rif{repeat}; the relation in \rif{z6bis}$_2$ instead follows from \rif{00bis}, \rif{22}$_2$ and again \rif{repeat}. 

For the proof of \rif{dominaF} we use \rif{00bis} and \rif{repeat}, to get, in the case $t\leq T$
\begin{flalign}
\notag \tilde F_{\eps}(x,t)  & \leq c \left[\frac{\tilde a(x,T)}{(T^2+\mu_{\eps}^2)^{\frac{\gamma-2}{2}}}+\eps\right]\int_0^t (s^2+\mu_\eps^2)^{\frac{\gamma-2}{2}} s \, ds\\
&  \leq  c\|\tilde a(\cdot, T)\|_{L^{\infty}}(T^2+\mu^2)+ c (T^\gamma +1)\,. \label{slipffff}
\end{flalign}
On the other hand, when $t >T$, using \rif{z6bis} gives that $\tilde F_{\eps}(x,t) - \tilde F_{\eps}(x,T) \leq  c \tilde F (x, t)$ and this, together with the content of the last display, gives \rif{z6bis} again. Finally, the proof of \rif{dimenticata} follows straightaway from the definitions in \rif{z2}-\rif{z5}. \end{proof}
\begin{lemma}\label{zlem0}
Under assumptions \eqref{0ass}-\eqref{00}
\begin{itemize}
    \item If \eqref{9} is also in force for some $\sigma\geq 0$, then for every $(x,t) \in \Omega \times [T, \infty)$ it holds that 
    \begin{flalign}\label{z9}
  \frac{g_{2,\varepsilon}(x,t)}{g_{1,\varepsilon}(x,t)}\le c[\bar{G}_{\varepsilon}(x,t)]^{\sigma}\,, \quad \mbox{where $c\equiv c(n,N,\nu,\gamma, \cb)$}\,.
    \end{flalign}
   In particular, if $g_2/g_2\leq \cb$ holds as in \trif{sigma-int}, then it also holds that
       \begin{flalign}\label{z9hold}
  \frac{g_{2,\varepsilon}(x,t)}{g_{1,\varepsilon}(x,t)}\le  c(n,N,\nu,\gamma, \cb)\,.
    \end{flalign}
    \item If \eqref{7} is also in force, then, with $\go_1$ as in \trif{g123}, for a.e.\,$x\in \Omega$ and every $z\in \mathbb{R}^{N\times n}$ with $|z|\geq T$, it holds that
    \begin{flalign}\label{z10}
    \snr{\partial_{x}g_{1,\varepsilon}(x,t)}t\le \frac{h(x)}{\mathfrak{g}_1^{\hat \sigma}(1+\hat \sigma)}\partial_{t}[\bar{G}_{\varepsilon}(x,t)]^{1+\hat{\sigma}}\,.
        \end{flalign}
    %\,, \ \ \mbox{where $c\equiv c(\nu,\gamma)$}\,.
    \item For all $(x,t) \in \Omega \times [T, \infty)$, the following holds:
\eqn{cresceg}
$$
\eps_1 < \eps_2 <\min\{1,T\}/4\Longrightarrow g_{1,\varepsilon_2}(x,t) \leq c(\gamma)g_{1,\varepsilon_1}(x,t)
$$
so that, this time for all $(x,t) \in \Omega \times [0, \infty)$, it follows
\eqn{cresceg2}
$$G_{\varepsilon_{2}}(x,t) \le c(\gamma)G_{\varepsilon_{1}}(x,t)\quad \mbox{and}\quad \lim_{\eps \to 0}\, G_{\varepsilon}(x,t)=\go_{1} G(x,t)\,.$$
The last convergence occurs uniformly on compact subsets of $[0, \infty)$. 
\end{itemize}
\end{lemma}
\begin{proof}
For \rif{z9}, note that  \rif{00bis}, \rif{banaleg1g2}$_1$ and \rif{repeat} imply 
\eqn{ccrrss}
$$
T \leq t \leq T_\eps \quad \Longrightarrow \quad \frac{1}{\mathfrak g_1}\frac{g_{2}(x,t)}{g_{1}(x,t)}\leq  \frac{g_{2, \eps}(x,t)}{g_{1, \eps}(x,t)}\leq  c(\nu, \gamma)\, \frac{g_{2}(x,t)}{g_{1}(x,t)}
$$
and that the definitions in \eqref{g1e}-\eqref{g2e} imply
\eqn{ccrrss1} 
$$
 t >T_\eps  \Longrightarrow \frac{g_{2,\eps}(x,t)}{g_{1,\varepsilon}(x,t)}=\frac{g_{2,\eps}(x,T_\eps)}{g_{1,\eps}(x,T_\eps)}\,.
$$
Then, for
$T\le t\le T_{\varepsilon}$ we have 
\begin{flalign}
& \frac{g_{2,\eps}(x,t)}{g_{1,\varepsilon}(x,t)}\stackrel{\eqref{ccrrss}}{\le}c\frac{g_{2}(x,t)}{g_{1}(x,t)}\stackrel{\eqref{9}}{\le}c[\bar{G}(x,t)]^{\sigma}\stackrel{\rif{g1e}}{\leq }c[\bar{G}_{\varepsilon}(x,t)]^{\sigma}\label{banal}
\end{flalign}
while, for $ T_{\varepsilon} < t$, we have 
\begin{flalign*}
\frac{g_{2,\eps}(x,t)}{g_{1,\varepsilon}(x,t)}\stackrel{\rif{ccrrss1} }{=}\frac{g_{2,\eps}(x,T_\eps)}{g_{1,\eps}(x,T_\eps)}\stackrel{\eqref{banal}}{\le}[\bar{G}_{\varepsilon}(x,T_\eps)]^{\sigma}\leq c[\bar{G}_{\varepsilon}(x,t)]^{\sigma}\,,
\end{flalign*}
and \rif{z9} is completely proved. 
The proof of \rif{z9hold} follows as in the last two displays by formally taking $\sigma =0$. The proof of \eqref{z10} is a straightforward consequence of the definition in \eqref{g1e} and \eqref{7} when $t \leq T_\eps$. In the case $t >T_\eps$, we instead have, also using \rif{7}
\begin{eqnarray*}
\snr{\partial_{x}g_{1,\varepsilon}(x,t)}t & \stackrel{\rif{g1e}}{=} & \frac{\mathfrak{g}_{1}|\partial_x g_{1}(x,T_{\varepsilon})|}{(T_{\varepsilon}^{2}+\mu_{\varepsilon}^{2})^{\frac{\gamma-2}{2}}}(t^{2}+\mu_{\varepsilon}^{2})^{\frac{\gamma-2}{2}}t\\
&
\leq &
\frac{\mathfrak{g}_{1}h(x)[\bar{G}(x,T_\eps)]^{\hat{\sigma}}g_{1}(x,T_\eps)}{(T_{\varepsilon}^{2}+\mu_{\varepsilon}^{2})^{\frac{\gamma-2}{2}}}(t^{2}+\mu_{\varepsilon}^{2})^{\frac{\gamma-2}{2}}t\\
&\leq &\frac{h(x)}{\mathfrak{g}_1^{\hat \sigma}}[\bar{G}_{\eps}(x,t)]^{\hat{\sigma}}g_{1, \eps}(x,t)t \\ &= & \frac{h(x)}{\mathfrak{g}_1^{\hat \sigma}(1+\hat \sigma)}\partial_{t}[\bar{G}_{\varepsilon}(x,t)]^{1+\hat{\sigma}}
\end{eqnarray*}
that is, \rif{z10}. Finally, the properties in \eqref{cresceg}-\eqref{cresceg2}, follow directly from the definitions in \eqref{g1e} and \eqref{gte}, using \rif{00} and \rif{repeat}.
\end{proof}
\begin{lemma}
Under assumptions \eqref{0ass}-\eqref{00} and \trif{8}, the inequalities 
\eqn{z11}
$$
\left\{
\begin{array}{c}
\displaystyle g_{3,\varepsilon}(x,t) t\le c[\bar{G}_{\varepsilon}(x,t)]^{1+\sigma}\\ \\
\displaystyle
 \frac{[g_{3,\varepsilon}(x,t)]^{2}}{g_{1,\varepsilon}(x,t)}\le c[\bar{G}_{\varepsilon}(x,t)]^{1+2\sigma}
\end{array}
\right.
$$
hold for all $t\geq T$ and for a.e.\,$x\in \Omega$, where $c\equiv c(n,N,\nu,\gamma,\cb)$.
\end{lemma}
\begin{proof}
When $t\in [T,T_{\varepsilon})$, the proof of \eqref{z11} follows directly by the definitions \eqref{g1e}, \eqref{g3e}, \eqref{gte} and from assumption \eqref{8}. Therefore we restrict ourselves to the case $t \geq T_\eps$. For this, we set $t_{\varepsilon}:=t/T_{\varepsilon}$, and \eqn{followb}
$$
\mathcal{Q}_{\varepsilon}(t):=\frac{\bar{G}_{\varepsilon}(x,T_{\varepsilon})}{\bar{G}_{\varepsilon}(x,t)}\leq 1
$$
and bound via $\eqref{8}_{1}$ and \rif{repeat} as follows:
\begin{flalign}
\notag g_{3,\varepsilon}(x,t)t &\leq    cg_{3}(x,T_{\varepsilon})T_{\varepsilon}\left(\frac{t^{2}+\mu_{\varepsilon}^{2}}{T_{\varepsilon}^{2}+\mu_{\varepsilon}^{2}}\right)^{\frac{\gamma-2}{2}}t_{\eps}^{2}\nonumber\\ &\le  c[\bar{G}_{\varepsilon}(x,T_\eps)]^{1+\sigma}t_{\eps}^{\gamma}\leq   c(n,\gamma,\cb)[\mathcal{Q}_{\varepsilon}(t)]^{1+\sigma}t_{\eps}^{\gamma}[\bar{G}_{\varepsilon}(x,t)]^{1+\sigma}\,.\label{followb2}
\end{flalign}
In order to bound $\mathcal{Q}_{\varepsilon}(t)$ we start observing that if $t_\eps \leq 1000$, then \rif{z11}$_1$ follows using \rif{followb}. When $t_\eps > 1000$, we instead estimate
\begin{eqnarray*}
\nonumber \mathcal{Q}_{\varepsilon}(t)&\stackleq{gte}& \frac{\bar{G}(x,T_{\varepsilon})}
{\displaystyle \int_{T_\eps}^{t}g_{1,\varepsilon}(x,s)s \d s+(T^2+1)^{\gamma/2}}\nonumber \\
\nonumber&\stackleq{g1e}& \frac{\bar{G}(x,T_{\varepsilon})}{\frac{\go_{1}g_{1}(x,T_{\varepsilon})}{\gamma(T_{\varepsilon}^{2}+\mu_{\varepsilon}^{2})^{\frac{\gamma-2}{2}}}T_{\varepsilon}^{\gamma}\left[\left(t_{\varepsilon}^{2}+(\mu_{\varepsilon}/T_{\varepsilon})^{2}\right)^{\gamma/2}-\left(1+(\mu_{\varepsilon}/T_{\varepsilon})^{2}\right)^{\gamma/2}\right]}\nonumber \\
\nonumber&\stackrel{t_{\varepsilon}> 1000}{\leq} &\frac{c( \gamma)\bar{G}(x,T_{\varepsilon}) (T_{\varepsilon}^{2}+\mu_{\varepsilon}^{2})^{\frac{\gamma-2}{2}}}{g_{1}(x,T_{\varepsilon})T_{\varepsilon}^{\gamma}t_{\varepsilon}^{\gamma}}\nonumber \\
\nonumber&\stackleq{wewewe}&c(\gamma)\left[\frac{g_{1}(x,T_{\varepsilon})T_{\varepsilon}^{2}(T_{\varepsilon}^{2}+\mu_{\varepsilon}^{2})^{\frac{\gamma-2}{2}}}{g_{1}(x,T_{\varepsilon})T_{\varepsilon}^{\gamma}t_{\varepsilon}^{\gamma}}+\frac{(T^{\gamma}+1)(T_{\varepsilon}^{2}+\mu_{\varepsilon}^{2})^{\frac{\gamma-2}{2}}}{g_{1}(x,T_{\varepsilon})T_{\varepsilon}^{\gamma}t_{\varepsilon}^{\gamma}}\right]\nonumber \\
&\stackrel{\eqref{00bis}}{\leq}& \frac{c(\gamma)}{t_{\varepsilon}^{\gamma}}\left[1+\frac{T^\gamma+1}{\nu T_{\varepsilon}^{\gamma}}\right]
\,.
\end{eqnarray*}
It follows that 
\eqn{qqqq}
$$
 \mathcal{Q}_{\varepsilon}(t)\leq  \frac{c(\nu, \gamma )}{t_{\varepsilon}^{\gamma}}
$$
Inserting this last estimate in \rif{followb2}, we get 
\eqn{we get we get }
$$
g_{3,\varepsilon}(x,t)t \leq c t_{\eps}^{-\sigma\gamma} [\bar{G}_{\varepsilon}(x,t)]^{1+\sigma} \leq c [\bar{G}_{\varepsilon}(x,t)]^{1+\sigma},
$$
where $c\equiv c(n,\nu,\gamma,\cb)$ and the proof of \rif{z11}$_1$ follows in the case $t_\eps > 1000$ too. As for \rif{z11}$_2$, similarly to \rif{z11}$_1$, when $t\in [T,T_{\varepsilon})$ the proof follows from \eqref{g1e}, \eqref{g3e} and assumption $\eqref{8}_{2}$, while for $t\ge T_{\varepsilon}$, using $\eqref{8}_{2}$ we have 
\begin{flalign*}
\frac{[g_{3,\varepsilon}(x,t)]^{2}}{g_{1,\varepsilon}(x,t)} &\le c\frac{[g_{3}(x,T_\eps)]^{2}}{g_{1}(x,T_\eps)} \left(\frac{t^{2}+\mu_{\varepsilon}^{2}}{T_{\varepsilon}^{2}+\mu_{\varepsilon}^{2}}\right)^{\frac{\gamma-2}{2}}t_{\eps}^{2}\\
&
\leq c(n,N,\gamma)[\mathcal{Q}_{\varepsilon}(t)]^{1+2\sigma}t_{\eps}^{\gamma}[\bar{G}_{\varepsilon}(x,t)]^{1+2\sigma}
\end{flalign*}
and \rif{z11}$_2$ follows using \rif{qqqq} in the last estimate and again arguing as for \rif{we get we get }. \end{proof}
\begin{lemma} %%checkagain
\label{ultimolemma} Under assumptions \eqref{0ass}-\eqref{00} and \eqref{9} for some $\sigma \geq 0$, 
consider a ball $B\Subset \Omega$ with $\rad(B)\le 1$, numbers $0<\varepsilon_{1}<\varepsilon_{2} \leq \min\{1,T\}/4$ and let $w \in W^{1,\gamma}(B; \er^{N})$. If $\bar G(\cdot,\snr{Dw})\in L^p(B)$ for some $p>1+\sigma$, 
then
\eqn{stimadiff}
$$
 \int_{B}\left|F_{\varepsilon_{1}}(x,Dw)-F_{\varepsilon_{2}}(x,Dw)\right| \dx  \le \oo(\eps_2)
\int_{B}[\bar G_{\varepsilon_{1}}(x,\snr{Dw})]^{p} \dx+ \oo(\eps_2)\,,
$$
holds, where $ \oo(\eps_2)$ denotes a quantity such that $ \oo(\eps_2)\to 0$ as $\eps_2\to 0$, and
\eqn{stimadiff2}
$$
 \int_{B}\left|F(x,Dw)-F_{\varepsilon_{2}}(x,Dw)\right| \dx  \le \oo(\eps_2)
\int_{B}[\bar G(x,\snr{Dw})]^{p} \dx+ \oo(\eps_2)\,. 
$$
\end{lemma}
\begin{proof} Note that \rif{cresceg2} implies $G_{\varepsilon_1}(x,t) \le c(\gamma)G (x,t)$, and therefore we also have that 
$\bar G_{\eps_1}(\cdot,\snr{Dw})\in L^p(B)$. Next, we denote
\eqn{auxgrowth0}
$$\bar F_\eps (x,t)= \tilde  F_\eps (x,t) -\eps L_{\gamma,\varepsilon}(t) := \int_0^t \tilde a_\eps (x, s)s\, ds$$ and in the following we always take $x\in B$. 
%As for \eqref{22}-\eqref{23}, thanks to \eqref{growthd} we get
%\eqn{auxgrowth}
%$$
%\tilde{a}_\eps(x,\snr{z})\le g_{2, \eps}(x,\snr{z})\,, \quad \mbox{for every $\eps \leq \min\{1,T\}/4$} \,,
%$$
%provided $|z|\geq T$. 
In the case it is $|z|\leq \eps_2$, by \rif{00} and \rif{repeat} we easily have
\begin{align*}
&\nonumber \left|\bar F_{\eps_1}(x,|z|)-\bar F_{\eps_2}(x, |z|) \right|  \leq \left|\bar F_{\eps_1}(x,\eps_2)\right|+\left|\bar F_{\eps_2}(x, \eps_2)\right| \\
\nonumber  &\   \leq c  \| \tilde a(\cdot, 1) \|_{L^{\infty}(B)} \int_0^{\eps_2} (s^2+\mu_{\eps_1}^2)^{\frac{\gamma-2}{2}}s \, ds \\
& \qquad \quad + c \| \tilde a(\cdot, 1) \|_{L^{\infty}(B)} \int_0^{\eps_2} 
(s^2+\mu_{\eps_2}^2)^{\frac{\gamma-2}{2}}s \, ds \leq \oo(\eps_2) %\qquad  \mbox{provided $|z|\leq \eps_2$} 
 \,.
%\label{dec1}
\end{align*}
When $\eps_2 \leq |z| \leq T_{\eps_2}$, recalling that $\tilde{a}_{\varepsilon_1}(x,|z|)\equiv \tilde{a}_{\varepsilon_2}(x,|z|)$, by also using the information in the last display we find
\eqn{dec2}
$$
\left|\bar F_{\eps_1}(x,|z|)-\bar F_{\eps_2}(x, |z|) \right| \leq \left|\bar F_{\eps_1}(x,\eps_2)\right|+\left|\bar F_{\eps_2}(x, \eps_2)\right|\leq \oo(\eps_2)  \,.
$$
Finally, note that \rif{00} and \rif{repeat} imply 
\eqn{triine}
$$ \tilde a_{\eps_2} (x, s)\leq c (\gamma) \tilde a_{\eps_1} (x, s)
\leq c (\gamma) \bar a_{\eps_1} (x, s) \quad \mbox{for $s> T_{\eps_2}$}$$
 therefore, when $|z|> T_{\eps_2}$, we get
%dettagli 
\begin{eqnarray*}
\nonumber  \left|\bar F_{\eps_1}(x,|z|)-\bar F_{\eps_2}(x, |z|) \right| 
& \stackleq{triine} & \left|\bar F_{\eps_1}(x,T_{\eps_2})-\bar F_{\eps_2}(x,T_{\eps_2}) \right|\\
&& \qquad +c\int_{T_{\eps_2}}^{|z|} \tilde a_{\eps_1} (x, s)s \, ds \\& \stackleq{dec2}  &c\int_{T_{\eps_2}}^{|z|} \bar a_{\eps_1} (x, s)s \, ds+ \oo(\eps_2)
\\
& \stackrel{\eqref{a1}}{\leq} &c\int_{T_{\eps_2}}^{|z|}   g_{2,\eps_1} (x, s)s \, ds+ \oo(\eps_2)\\
& \stackleq{z9} &c \int_{T_{\eps_2}}^{|z|}   [\bar{G}_{ \eps_1}(x,s) ]^{\sigma}g_{1,\eps_1} (x, s)s \, ds+ \oo(\eps_2)\\ 
& \leq &c [\bar{G}_{ \eps_1}(x,|z|) ]^{\sigma} \int_{T_{\eps_2}}^{|z|}  g_{1,\eps_1} (x, s)s \, ds+ \oo(\eps_2)\\& \leq & c [\bar{G}_{ \eps_1}(x,|z|) ]^{1+\sigma}+ \oo(\eps_2)\,,
%\qquad  \mbox{provided $
%|z|> T_{\eps_2}$}%\label{dec4}
\end{eqnarray*}
and we have also used that $s \mapsto \bar{G}_{ \eps_1}(x,s) $ is non-decreasing. 
Using the content of the last four displays, and using also H\"older inequality, we get
\begin{flalign*}
 & \int_{B}\left|\bar  F_{\varepsilon_{1}}(x,|Dw|)-\bar  F_{\varepsilon_{2}}(x,|Dw|)\right| \dx \\
 &\quad \le \oo(\eps_2)|B|
 + c \int_{B\cap \{|Dw|> T_{\eps_2}\}}[\bar G_{\varepsilon_{1}}(x,\snr{Dw})]^{1+\sigma} \dx\\
 & \quad \leq 
  \oo(\eps_2)|B|+ c |B\cap \{|Dw|> T_{\eps_2}\}|^{\frac{p-1-\sigma}{p}} \left(\int_{B}[\bar G_{\varepsilon_{1}}(x,\snr{Dw})]^{p} \dx\right)^{\frac{1+\sigma}{p}}
\end{flalign*}
for $c\equiv c(\nu,\gamma,p)$. On the other hand, observe that 
\begin{eqnarray*}
 |B\cap \{|Dw|> T_{\eps_2}\}| & \leq  &T_{\eps_2}^{-\gamma p}\int_{B\cap \{|Dw|> T_{\eps_2}\}} |Dw|^{\gamma p}\, dx \\
 & \stackrel{\eqref{z10a}_1}{\leq} & c \eps_2^{\gamma p} \int_{B}[ \bar G_{\varepsilon_{1}}(x,\snr{Dw})]^p\, dx\,.
\end{eqnarray*}
Last two inequalities yield 
\begin{flalign*}
 \int_{B}\left| \bar F_{\varepsilon_{1}}(x,|Dw|)-\bar F_{\varepsilon_{2}}(x,|Dw|)\right| \dx   &\le|B| \oo(\eps_2)
 \\
 & \quad + c\eps_2^{\gamma(p-1-\sigma)}\int_{B}[\bar G_{\varepsilon_{1}}(x,\snr{Dw})]^{p} \dx,
\end{flalign*}
where $c\equiv c(n,N,\nu,\gamma,\ca, \cb)$. 
By using again \rif{z10a}$_1$, we get
\begin{flalign*}
 \int_{B}\left|  \varepsilon_1 L_{\gamma,\varepsilon_1}(Dw)-\varepsilon_2 L_{\gamma,\varepsilon_2}(Dw)\right| \dx  &\leq c \eps_2 \int_B (|Dw|^2+1)^{\gamma/2}\, dx \\
 &\leq  c \eps_2 \int_{B}[\bar G_{\varepsilon_{1}}(x,\snr{Dw})]^{p} \dx\,.
\end{flalign*}
Combining the content of the last two displays and recalling \rif{z5} and \rif{auxgrowth0} we arrive at \rif{stimadiff}. As for  \rif{stimadiff2}, this follows from \rif{z3}$_6$ and \rif{stimadiff} letting $\eps_1 \to 0$. Indeed, note that Fatou's lemma works for the left-hand one; as for the right-hand side, we use again that $G_{\varepsilon}(x,t) \lesssim G (x,t)$ and the second information in \rif{cresceg2}, and finally Lebesgue dominated convergence. 
\end{proof}

\section{A priori estimates}\label{apri}
In this section we develop basic a priori estimates. These are obtained for solutions to certain elliptic systems associated to the integrands defined in Section \ref{apro}, \rif{z5}. Unless differently specified, we shall permanently assume that $\setm$ in \trif{ipotesi} is in force; all properties \rif{0ass}-\rif{9} will be therefore available as well.
With $0<\varepsilon<\min\{1,T\}/4$ and $\B\Subset \Omega$ being a ball such that $\rad(\B)\leq 1$, we consider a weak solution $u\in W^{1,\gamma}(\B;\RN)$ to the system
\eqn{a0}\
$$
\begin{cases}
-\diver \, a_{\eps}(x,Du)=f \ \  \mbox{in} \ \B
\\  f\in L^\infty(\B;\RN),  \  |f|\leq |\mathfrak {f}|, \ \mbox{where} \ \mathfrak{f} \in L^n(\B;\RN)\,,
\end{cases}
$$
with $a_{\eps}\colon \Omega\times \mathbb{R}^{N\times n}\to \mathbb{R}^{N\times n}$ being defined as in \rif{arad}. This setting will be kept for the rest of Section \ref{apri}. In particular, all the balls considered in the following will have radius $\leq 1$. Eventually, we shall consider additional restrictions of the type $\rad(\B)\leq R_*$, where $R_*$ will be  a (small) threshold radius to be determined as a function of the fixed parameters of the problem, but independently of the solution $u$ considered. From \rif{a0}, and taking \rif{zgrowthd} and \rif{dimenticata} into account, it follows that $u$ is a minimizer of the functional
\begin{flalign}\label{avp}
w\mapsto \int_{\B}\left[F_\eps(x,Dw)-f\cdot w\right] \dx= \int_{\B}\left[\tilde{F}_\eps(x,\snr{Dw})-f\cdot w\right] \dx\,.
\end{flalign}
We recall that $\bar{a}_{\eps}(\cdot)$, defined in \rif{arad}, is such that $x\mapsto \bar{a}_{\eps}(x,t)\in W^{1,d}(\Omega)$ for all $t\geq0$, $t\mapsto \bar{a}_{\eps}(x,t)\in W_{\textnormal{loc}}^{1,\infty}[0,\infty)\cap C^{1}_{\textnormal{loc}}([0,\infty)\setminus \{\eps, T_{\eps}\})$ i.e., it is locally $C^{1}$-regular outside $\{\eps, T_{\eps}\}$ and it is such that $\bar{a}_{\eps}'(x,0)=0$ for all $x\in \Omega$. This implies that $x\mapsto a_{\eps}(x,z)\in W^{1,d}(\Omega;\mathbb{R}^{N\times n})$ for all $z\in \mathbb{R}^{N\times n}$ and $z\mapsto a_{\eps}(x,z)\in W^{1,\infty}_{\textnormal{loc}}(\mathbb{R}^{N\times n};\er^{N\times n})$ for all $x\in \Omega$.
Finally, the functions $ t\mapsto \bar{a}_{\eps}(\cdot,t)$ and $g_{1,\eps}(\cdot,t)$ are non-decreasing when $\gamma\geq 2$; this is indeed an easy consequence of assumption \rif{00}$_1$ (see display after \rif{1bis}). A direct consequence of \eqref{a1}-\eqref{a2} is
\begin{flalign}\label{a3}
\snr{\bar{a}_{\eps}'(x,\snr{z})}\snr{z}\le g_{2,\varepsilon}(x,\snr{z})\quad \mbox{for all} \ \ (x,z)\in \Omega\times \{\snr{z}\geq T, \ \snr{z}\not =T_{\varepsilon} \}\,. 
\end{flalign}
Indeed, if $\bar{a}_{\eps}'(x,\snr{z})\geq 0$, then \rif{a3} trivially follows from \rif{a2}$_2$; otherwise \rif{a2}$_1$, and then \rif{a1}, give $|\bar{a}_{\eps}'(x,\snr{z})|\snr{z}=-\bar{a}_{\eps}'(x,\snr{z})\snr{z}\leq -g_{1,\varepsilon}(x,\snr{z})+\bar{a}_{\eps}(x,\snr{z})\leq \bar{a}_{\eps}(x,\snr{z}) \leq g_{2,\varepsilon}(x,\snr{z})$.
Similarly to \eqref{a1}, by \rif{z2} and \rif{zgrowthd}, we have that
\eqn{a4}
$$
\begin{cases}
\ \eps\go_1(\snr{z}^{2}+\mu_{\varepsilon}^{2})^{(\gamma-2)/2}\leq \bar{a}_{\eps}(x,\snr{z})\le \Lambda_{\eps} (\snr{z}^{2}+\mu_{\varepsilon}^{2})^{(\gamma-2)/2} \\
\  |\bar{a}_{\eps}'(x,\snr{z})| \leq c_{\eps} (\snr{z}^{2}+\mu_{\varepsilon}^{2})^{(\gamma-3)/2}
 \end{cases}
$$
holds this time for all $(x,z)\in \Omega\times\er^{N\times n}$ (must be $|z|\not = \eps, T_{\eps}$ for \rif{a4}$_2$). From \rif{growthd}-\rif{zgrowthd} we can apply Proposition \ref{tapp} from Section \ref{appendix}, and this yields 
\begin{flalign}\label{a6}
Du \in L^{\infty}_{\textnormal{loc}}(\B;\mathbb{R}^{N\times n}) \quad \mbox{and} \quad  u\in W^{2,2}_{\textnormal{loc}}(\B;\mathbb{R}^{N})\,.
\end{flalign}
In turn, this implies that 
\eqn{a6bis}
$$
 a_{\eps}(\cdot,Du)=\bar a_{\eps}(\cdot,|Du|)Du\in W^{1,2}_{\textnormal{loc}}(\B;\mathbb{R}^{N\times n})\,.
$$
This follows for instance from the results in \cite[Theorem 1.5]{dele}, together with the expression (chain rule) of $Da_{\eps}(\cdot,Du)$. It is sufficient to check that $\bar a_{\eps}(\cdot,|Du|)\in W^{1,2}_{\textnormal{loc}}(\B)$ and that the corresponding chain rule applies to $D_s(\bar  a_{\eps}(\cdot,|Du|))$, for every $s \in \{1,\ldots, n \}$. After extending $t \mapsto \bar a_{\eps}(x, t)$ to $\er$ by even reflection, we apply \cite[Theorem 1.5]{dele} to the vector field $B\colon \B \times \er \mapsto  \er^n$ defined by $(B(\cdot))_j\equiv ( \bar a_{\eps}(\cdot)\delta_{sj})$. For this, note that for every $t\in \er$, $x \mapsto \tilde a_{\eps}(x, t) \in W^{1,n}(\B)$ (this is \cite[Theorem 1.5, (i)]{dele}), and that $\partial_x \bar a_{\eps}(\cdot)=\partial_x \tilde a_{\eps}(\cdot)$ is still Carath\'eodory regular;  this follows from the definitions \rif{tildea}-\rif{z2} and by the fact that $\partial_x \tilde F'$ is Carath\'eodory regular by assumptions (this is \cite[Theorem 1.5, (ii)]{dele}). The crucial point to apply the results from \cite{dele} is that, as described in \rif{z3}, the set of non-differentiable points of the (extended) partial function $t \mapsto \bar a_{\eps}(x, t)$ is contained in $ \{-T_{\eps}, -\eps,\eps, T_{\eps}\}$ for every $x \in \B$, and it is therefore a null set which is independent of $x$ (this is \cite[Theorem 1.5, (iii)]{dele}). Finally, \cite[Theorem 1.5, (iv)]{dele} is verified thanks to \rif{growthd}$_1$ and \rif{a4}$_2$. Therefore, by \rif{a6}, \cite[Theorem 1.5]{dele} applies and \rif{a6bis} follows with the corresponding chain rule (see \rif{diso} and the related discussion a few lines below). Note that we have used \rif{a6}, that, together with \rif{zgrowthd}$_4$, also implies $\partial_{x_{s}}a_{\eps}(\cdot,Du)\in L^n_{\loc}(\Omega;\er^{N\times n})$.

 Let us write, with abuse of notation (keep in mind $\partial_{z}a_{\eps}(x,z)\equiv \partial_{zz}F_{\eps}(x,z)$ is not defined when $|z|= \eps, T_\eps$)
\begin{flalign}\label{diso2}
(\partial_{z}a_{\eps}(x,z))^{\alpha\beta}_{ij}=\bar{a}_{\eps}(x,\snr{z})\delta_{ij}\delta_{\alpha\beta}+ \mathds{1}_{ \mathcal{D}}(|z|)\bar{a}_{\eps}'(x,\snr{z})\snr{z}\frac{z^{\alpha}_{i}z^{\beta}_{j}}{\snr{z}^{2}}
\end{flalign}
for $z \in \er^{N\times n}\setminus\{0\}$, and here we are denoting by $ \mathds{1}_{\mathcal{D}}(\cdot)$ the indicator function of the  set $\mathcal{D}:=\er\setminus \{\eps, T_{\eps}\}$. We explain \rif{diso2} as follows. 
Recalling that $|Du| \in W^{1,2}_{\loc}(\Omega)$ by \rif{a6}-\rif{a6bis}, we have that 
\begin{eqnarray} 
\notag D_s[a_{\eps}(\cdot, Du)] &= &\partial_{x_{s}}a_{\eps}(x,Du) +\partial_{z}a_{\eps}(x,Du)DD_{s}u\\
&=&\partial_{x_s}\bar a_{\eps}(x,|Du|)Du + \bar a_{\eps}(x,|Du|)\mathds{I}_{N\times n} DD_su \notag \\
&& + \mathds{1}_{ \mathcal{D}}(|Du|)  \bar a_{\eps}'(x,|Du|)D_s|Du|  \label{diso}
\end{eqnarray}
holds a.e.\,in $\B$, for every $s \in \{1, \ldots, n\}$. Here we are using Einstein's convention on repeated indexes and $\mathds{I}_{N\times n}\equiv (\delta_{ij}\delta_{\alpha\beta})$.   
Exactly as in the autonomous case, the presence of $\mathds{1}_{ \mathcal{D}}(|Du|)$ in \rif{diso2}-\rif{diso} then accounts for the fact that the term $ \bar a_{\eps}'(x,|Du|)D_s|Du|$ is interpreted as zero at those points where $D|Du|=0$, and in particular, for a.e.\,$x$ such that $|Du(x)| \in \{\eps, T_{\eps}\}$, i.e., where $\bar a_\eps'(x,|Du(x)|)$ alone does not make sense; see \cite[Theorem 1.5]{dele}. Note that, in particular, from \rif{diso2}, \rif{growthd}$_{1,2}$ and \rif{a1}, it follows that
\eqn{inaggiunta}
$$
g_{1, \eps}(x, |z|)|\xi|^2 \leq \partial_{z}a_{\eps}(x,z)\xi \cdot \xi\,, \quad | \partial_{z}a_{\eps}(x,z)|\leq g_{2, \eps}(x, |z|)\,,
$$
whenever $|z|\geq T$ and $\xi \in \er^{N\times n}$. 

Similar arguments apply to $G_{ \eps}(\cdot, |Du|)$, this time using directly \cite[Corollary 1.7]{dele}. Indeed, change/extend $g_{1,\eps}(\cdot)$ on $\Omega \times (\infty, T)$ as $g_{1,\eps}(x, t):= g_{1,\eps}(x, T)$ for $t \leq T$ (this is continuous as $g_{1, \eps}$ is continuous), and, in \cite{dele}, take $b\colon \B \times \er\to \er^n$ defined as $(b)_j := (g_{1,\eps}(x, t)t \delta_{ij})_j$ for every $i \in \{1, \ldots, n\}$, with $w(\cdot)\equiv T$, $u(\cdot)\equiv \max\{|Du|, T\}\in W^{1,2}_{\loc}(\B)\cap L^{\infty}_{\loc}(\B)$ (here $u(\cdot)$ is the one from \cite{dele}, and not the solution considered in \rif{a0}).  Finally, note that $\partial_{x_i} g_{1, \eps} \in L^n_{\loc}(\B \times \er)$ by \rif{z10}. Then \cite[Corollary 1.7]{dele} implies 
\eqn{differenziaG}
$$G_{ \eps}(\cdot,|Du|)\in W^{1,2}_{\textnormal{loc}}(\B)$$ and 
\begin{flalign*}
D_{i}G_{\varepsilon}(x,\snr{Du}) &= \int_{T}^{\max\{ \snr{Du}, T\}}\partial_{x_{i}}g_{1,\varepsilon}(x,t)t \d t \\
& \qquad + g_{1,\varepsilon}(x,\snr{Du})D_i\max\{|Du|, T\}
\end{flalign*}
holds a.e. in $\B$. In particular, on $\{\snr{Du}>T\}$, we have
\begin{flalign}\label{a8}
&D_{i}G_{\varepsilon}(x,\snr{Du})=g_{1,\varepsilon}(x,\snr{Du})\snr{Du}D_{i}\snr{Du}+\int_{T}^{\snr{Du}}\partial_{x_{i}}g_{1,\varepsilon}(x,t)t \d t\nonumber \\
&\quad = g_{1,\varepsilon}(x,\snr{Du})\sum_{\alpha=1}^{N}\sum_{s=1}^{n}D_{i}D_{s}u^{\alpha}D_{s}u^{\alpha}+\int_{T}^{\snr{Du}}\partial_{x_{i}}g_{1,\varepsilon}(x,t)t \d t
\end{flalign}%%%checkagain diffe
for every $i \in \{1,\ldots, n\}$, where we have also used that
\begin{flalign}\label{a9}
D_{i}\snr{Du}=\frac1{\snr{Du}}\sum_{\alpha=1}^{N}\sum_{s=1}^{n}D_{i}D_{s}u^{\alpha}D_{s}u^{\alpha}\,.
\end{flalign}
Note also that $D_{i}G_{\varepsilon}(x,\snr{Du})\equiv 0$ a.e. on the complement of $\{\snr{Du}>T\}$. 
\begin{lemma}\label{help}
For $\lambda\equiv (\lambda_i) \in \er^n$ and $z=(z^\alpha_i)\in\er^{N\times n}$, $1\leq i\leq n$ and $1\leq \alpha\leq N$, with $|z|\geq T$
\begin{flalign}
& \bar{a}_{\eps}(x,\snr{z})\lambda\cdot \lambda+ \mathds{1}_{ \mathcal{D}}(|z|)\bar{a}_{\eps}'(x,\snr{z})\snr{z}\sum_{\alpha=1}^N\frac{\snr{\lambda\cdot z^\alpha}^{2}}{\snr{z}^{2}}\nonumber \\
& \qquad \ge  \min\left\{\bar{a}_{\eps}(x,\snr{z}),\bar{a}_{\eps}(x,\snr{z})+ \mathds{1}_{ \mathcal{D}}(|z|)\bar{a}_{\eps}'(x,\snr{z})\snr{z}\right\}\snr{\lambda}^{2}\notag \\
&
  \qquad  \geq  g_{1,\varepsilon}(x,\snr{z})\snr{\lambda}^{2} \label{a7}
\end{flalign}
holds for every $x \in \Omega$. 
\end{lemma}
\begin{proof} Indeed, \rif{a7} is trivial by \rif{a1} 
%%%typoc 
if $\mathds{1}_{ \mathcal{D}}(|z|)\bar{a}_{\eps}'(x,\snr{z})\geq 0$ (note that \rif{a1} holds when $|z|=T_{\eps}$ too). Otherwise, we can estimate simply $\bar{a}_{\eps}'(x,\snr{z})\snr{\lambda \cdot z^\alpha}^2 \geq \bar{a}_{\eps}'(x,\snr{z})|z^\alpha|^2|\lambda|^2$ for every $\alpha$ and then use $\eqref{a2}_{1}$. 
\end{proof}
%Finally, whenever $B\subset \B$ is a ball, we let
%\eqn{zuppi}
%$$
%{\bf S}_{\eps}(B):=
%sup_{x\in B}\frac{g_{2,\varepsilon}(x,\snr{Du(x)})}{g_{1,\varepsilon}(x, \snr{Du(x)})}\quad \mbox{and}\quad {\bf G}_{\eps}(B):=\sup_{x\in B}G_{\varepsilon}(x,\snr{Du(x)})\,.%=\sup_{x\in B_{\varrho}}\frac{g_{2,\varepsilon}\left(x, M\right)}{g_{1,\varepsilon}\left(x,M\right)}
%$$
\subsection{Caccioppoli inequality for powers, when $n>2$}\label{cs}  
Up to Section \ref{sr} included, we concentrate on the case $n>2$, and for the following we set
\eqn{ilchi}
$$
 m:=\frac{d}{d-2}>1 \quad \mbox{and} \quad 
1\stackrel{n<d}{<} \chi:=\frac{2^{*}}{2m} = \frac{n}{n-2}\frac{d-2}{d}\,.
$$
The main result here is 
\begin{proposition}\label{lp}
Let $u\in W^{1,\gamma}(\B;\RN)$ be a solution to $\eqref{a0}$, under assumptions $\setm$ in \trif{ipotesi}  for $n>2$, and replace \trif{sigma}$_1$ by
\eqn{sigmasss}
$$
\begin{cases}
\sigma+\hat \sigma < \frac {s_*}n - \frac {s_*}d\\
1\leq s_* < \min\left\{2m (1+\sigma+\hat \sigma), \frac{2n}{n-2}\right\}\,.
\end{cases}
$$
%Then, $G_{\varepsilon}(\cdot,\snr{Du})\in W^{1,p}_{\textnormal{loc}}(B)$ for all $p\in [1,\infty)$. In particular
Then, for every $p\in [1,\infty)$, there exists a positive radius $R_{*}\equiv R_{*}(\data,\mathfrak f(\cdot),p)\linebreak \leq 1$ such that if $\rad(\B)\le R_{*}$ and $B_{\varsigma} \Subset  B_{\rr}$ are concentric balls contained in $\B$, then 
\begin{flalign}\label{a44}
\|G_{\varepsilon}(\cdot,\snr{Du})\|_{L^p(B_\varsigma)} \leq 
 \frac{c}{(\rr-\varsigma)^{\beta_{p}}}\left[\|F_{\eps}(\cdot,Du)\|_{L^{s_*}(B_{\rr})}^{\theta_{p}}+1\right]
\end{flalign} holds with $c\equiv c(\data,\nr{h}_{L^{d}(\B)})\geq 1$, $\beta_{p}, \theta_{p}\equiv \beta_{p}, \theta_{p}(n,d,\sigma,\hat{\sigma},p)>0$. 
\end{proposition} %%checkagain  dipendenza da norma di h
Needless to say, \rif{sigmasss} is implied by \rif{sigma}$_1$ for $s_*=1$. The proof of Proposition \ref{lp} will take this and the subsequent Sections \ref{mosersec}-\ref{sr}; in the following, all the balls considered but $\B$, will be concentric with $B_{\varsigma} \Subset  B_{\rr}$ from the statement. We recall it is $\rad(\B)\le 1$; the size of $R_*$ will be chosen in due course of the proof. Observe that all the foregoing computations, except those involving $f$, still work in the case $n=2$; this case will be treated in Section \ref{casetwo} below. To start with the proof of Proposition \ref{lp}, by \eqref{a6}-\eqref{diso2} we pass to the differentiated form of system \eqref{a0}, that is
\begin{flalign}\label{a2v}
\sum_{s=1}^{n}\int_{\B}\left[\partial_{z}a_{\eps}(x,Du)DD_{s}u\cdot D\varphi+\partial_{x_{s}}a_{\eps}(x,Du)\cdot D\varphi+f\cdot D_{s}\varphi\right] \dx=0 \,, %%%checkagain
\end{flalign} which holds for all $\varphi\in W^{1,2}_0(\B;\RN)$ with compact support in $\B$, by density. Indeed, note that the terms multiplying $D\varphi$ belong to $L^2_{\loc}(\B)$ and $L^n_{\loc}(\B)$; this follows from \rif{growthd}$_3$ and \rif{a6}. We now consider concentric balls 
$B_{\varsigma}\subset B_{\tau_{1}} \Subset B_{\tau_{2}} \subset B_{\rr}$; in particular it is $\varsigma \leq \tau_1 < \tau_2\leq \varrho\leq 1$. 
In \eqref{a2v} we take $\varphi\equiv \varphi_{s}:=\eta^{2}[G_{\varepsilon}(x,\snr{Du})]^{\kk+1}D_{s}u$ for $s \in \{1, \ldots, n\}$, where $\kk\ge 0$ and $\eta\in C^{1}_{{\rm c}}(B_{\rr})$ satisfies
$\mathds{1}_{B_{\tau_{1}}}\le \eta\le \mathds{1}_{B_{\tau_{2}}}$ and $\snr{D\eta}\lesssim 1/(\tau_{2}-\tau_{1})$. Note that \rif{a6} and \rif{differenziaG} imply that $\varphi_s\in W^{1,2}_0(\B;\RN)$, has support contained in $B_{\tau_{2}}$ and, as such, it is admissible in \rif{a2v}. 
It follows
\begin{flalign}
\notag D\varphi_{s}=&\, \eta^{2}[G_{\varepsilon}(x,\snr{Du})]^{\kk+1}DD_{s}u\\& +(\kk+1)\eta^{2}[G_{\varepsilon}(x,\snr{Du})]^{\kk}D_{s}u\otimes DG_{\varepsilon}(x,\snr{Du})\nonumber \\
&+2\eta [G_{\varepsilon}(x,\snr{Du})]^{\kk+1}D_{s}u\otimes D\eta\,. \label{latesta}
\end{flalign}
By \rif{a6} and \rif{differenziaG} it follows that $\varphi\in W^{1,2}_0(\B;\RN)$ and it is therefore admissible in \rif{a2v}. 
By the last identity we have 
\begin{flalign}
& \notag \sum_{s=1}^{n}\int_{\B}\partial_{z}a_{\eps}(x,Du)DD_{s}u\cdot D\varphi_{s} \dx\\
&\quad  =\sum_{s=1}^{n}\int_{\B}\eta^{2}[G_{\varepsilon}(x,\snr{Du})]^{\kk+1}\partial_{z}a_{\eps}(x,Du)DD_{s}u\cdot DD_{s}u \dx\nonumber \\ & \qquad 
\notag +(\kk+1)\sum_{s=1}^{n}\int_{\B}\eta^{2}[G_{\varepsilon}(x,\snr{Du})]^{\kk}\partial_{z}a_{\eps}(x,Du)DD_{s}u\\
& \hspace{52mm} \cdot \left( D_{s}u\otimes DG_{\varepsilon}(x,\snr{Du})\right) \dx\nonumber \\ &\qquad 
\notag
+2\sum_{s=1}^{n}\int_{\B}\eta [G_{\varepsilon}(x,\snr{Du})]^{\kk+1}\partial_{z}a_{\eps}(x,Du)DD_{s}u\cdot (D_{s}u\otimes D\eta) \dx\\ &\quad    =:\mbox{(I)}_{z}+\mbox{(II)}_{z}+\mbox{(III)}_{z} \,. \label{equainizio}
\end{flalign}
Note that in the display above and in the following ones until \rif{a20}, as $G_{\varepsilon}(x,t)\equiv 0$ for $t\leq T$, all the integrals above actually extend only over $B\cap\{\snr{Du}>T\}$, therefore we can always use \rif{a8} when computing the derivatives of $G_{ \eps}$.   
To proceed, we have
\begin{flalign*}
\mbox{(I)}_{z}\stackrel{\rif{inaggiunta}}{\ge} &\int_{\B}\eta^{2}[G_{\varepsilon}(x,\snr{Du})]^{\kk+1}g_{1,\varepsilon}(x,\snr{Du})\snr{D^{2}u}^{2} \dx=:\mathcal{S}_1\,.
\end{flalign*}
We temporarily shorten the notation as follows (recall \rif{banaleg1g2}$_1$):
\begin{flalign*}
&\mathcal{H}_{\kk}(Du):=(\kk+1)[G_{\varepsilon}(x,\snr{Du})]^{\kk}\frac{\bar{a}_{\eps}(x,\snr{Du})}{g_{1,\varepsilon}(x,\snr{Du})}\,,\\
&\mathcal{H}_{\kk}'(Du):=(\kk+1)\mathds{1}_{\mathcal{D}}(|Du|)[G_{\varepsilon}(x,\snr{Du})]^{\kk}\frac{\bar{a}_{\eps}'(x,\snr{Du})\snr{Du}}{g_{1,\varepsilon}(x,\snr{Du})}\,.
\end{flalign*}
As all the integrals extend over $\{|Du|>T\}$ by \rif{00bis}, it is $g_{1,\varepsilon}(x,\snr{Du})>0$. 
Recalling \eqref{diso2}-\eqref{a9}, we then re-write
\begin{flalign}
\notag \mbox{(II)}_{z}& =\int_{\B}\eta^{2}\mathcal{H}_{\kk}(Du)DG_{\varepsilon}(x,\snr{Du})\cdot DG_{\varepsilon}(x,\snr{Du}) \dx\nonumber \\
\notag &\hspace{6mm}+\sum_{\alpha=1}^{N}\int_{\B}\eta^{2}\mathcal{H}_{\kk}'(Du)\frac{\left(DG_{\eps}(x,\snr{Du})\cdot Du^{\alpha}\right)^{2}}{\snr{Du}^2} \dx\nonumber \\
\notag &\hspace{6mm}-\int_{\B}\eta^{2}\mathcal{H}_{\kk}(Du)\int_{T}^{\snr{Du}}\partial_{x}g_{1,\varepsilon}(x,t)t \d t\cdot DG_{\varepsilon}(x,\snr{Du}) \dx\nonumber \\
\notag &\hspace{6mm}-\sum_{\alpha=1}^{N}\int_{\B}\eta^{2}\mathcal{H}_{\kk}'(Du)\left(\int_{T}^{\snr{Du}}\partial_{x}g_{1,\varepsilon}(x,t)t \d t\cdot Du^{\alpha}\right)\\ & \hspace{43mm} \times \frac{(DG_{\varepsilon}(x,\snr{Du})\cdot Du^{\alpha})}{\snr{Du}^2} \dx\nonumber \\
& =: \mbox{(II)}_{z,1}+\mbox{(II)}_{z,2}+\mbox{(II)}_{z,3}+\mbox{(II)}_{z,4}\,. \label{stimaII}
\end{flalign}
We now observe that 
\eqn{stimaIIbis}
$$
\mbox{(II)}_{z,1}+\mbox{(II)}_{z,2}  \ge(\kk+1)\int_{\B}\eta^{2}[G_{\varepsilon}(x,\snr{Du})]^{\kk}\snr{DG_{\varepsilon}(x,\snr{Du})}^{2} \dx=: \mathcal{S}_2\,.
$$
Indeed, this follows from \rif{a7} with $\lambda \equiv DG_{\eps}$ and $z\equiv Du$. Then we have, by using \rif{a1}, \rif{z9}, \rif{z10} and \rif{a3}
\begin{flalign*}
& \ \  |\mbox{(II)}_{z,3}|+|\mbox{(II)}_{z,4}|   \leq c(\kk+1)\int_{\B}\eta^{2}h(x)[G_{\varepsilon}(x,\snr{Du})]^{\kk}\frac{g_{2,\varepsilon}(x,\snr{Du})}{g_{1,\varepsilon}(x,\snr{Du})}\\ & \hspace{35mm} \times \left(\int_{T}^{\snr{Du}}\partial_{t}[\bar{G}_{\varepsilon}(x,t)]^{1+\hat{\sigma}}  \d t\right) \snr{DG_{\varepsilon}(x,\snr{Du})} \dx\\
%%%typoc
&\quad   \leq c(\kk+1)\int_{\B}\eta^{2}h(x)[G_{\varepsilon}(x,\snr{Du})]^{\kk}[\bar{G}_{\varepsilon}(x,\snr{Du})]^{1+\sigma+\hat{\sigma}}\snr{DG_{\varepsilon}(x,\snr{Du})} \dx\\
&\quad \leq  \bar \eps \mathcal{S}_2 + \frac{c(\kk+1)}{\bar \varepsilon} \int_{\B}\eta^{2}[h(x)]^2[G_{\eps}(x,\snr{Du})]^{\kk}[\bar{G}_{\eps}(x,\snr{Du})]^{2(1+\sigma+\hat{\sigma})} \dx\\
&\quad \leq  \bar \eps \mathcal{S}_2 + \frac{c(\kk+1)\nr{h}_{L^{d}(\B)}^{2}}{\bar \varepsilon}\\ &\hspace{2cm} \times \left(\int_{\B}\eta^{2m}[G_{\eps}(x,\snr{Du})]^{m\kk}[\bar{G}_{\eps}(x,\snr{Du})]^{2m(1+\sigma+\hat{\sigma})} \dx\right)^{1/m}\,,
\end{flalign*}
where $c\equiv c(n,N,\nu,\gamma,\cb)$, and arbitrary $\bar \eps \in (0,1)$.  Note that in last two lines we have used Young and H\"older's inequalities. 
Similarly to \rif{stimaII}, we also have
\begin{flalign*}
\notag \mbox{(III)}_{z}&=2\int_{\B}\frac{\eta G_{\varepsilon}(x,\snr{Du})\mathcal{H}_{\kk}(Du)}{\kk+1}DG_{\varepsilon}(x,\snr{Du})\cdot D\eta \dx\nonumber\\
\notag &\ \  +2\sum_{\alpha=1}^{N}\int_{\B}\frac{\eta G_{\varepsilon}(x,\snr{Du})\mathcal{H}_{\kk}'(Du)}{(\kk+1)|Du|^2}(DG_{\varepsilon}(x,\snr{Du})\cdot Du^{\alpha})(Du^{\alpha}\cdot D\eta) \dx\nonumber \\
\notag &\ \ -2\int_{\B}\frac{\eta G_{\varepsilon}(x,\snr{Du})\mathcal{H}_{\kk}(Du)}{\kk+1}\left(\int_{T}^{\snr{Du}}\partial_{x}g_{1,\varepsilon}(x,t)t \d t\cdot D\eta \right)\dx\nonumber\\
\notag &\ \  -2 \sum_{\alpha=1}^{N}\int_{\B}\frac{\eta G_{\varepsilon}(x,\snr{Du})\mathcal{H}_{\kk}'(Du)}{(\kk+1)|Du|^2}
\\
& \ \  \qquad \quad  \times \left(\int_{T}^{\snr{Du}}\partial_{x}g_{1,\varepsilon}(x,t)t \d t\cdot Du^{\alpha}\right)(Du^{\alpha}\cdot D\eta) \dx\nonumber \\
& =:\mbox{(III)}_{z,1}+\mbox{(III)}_{z,2}+\mbox{(III)}_{z,3}+\mbox{(III)}_{z,4}\,.%\label{asfor} %no
\end{flalign*}
%%%typoc
Using again \rif{a1}, \rif{z9}, \rif{a3} and Young inequality, we have, for every $\bar \eps \in (0,1)$
\begin{flalign*}
&|\mbox{(III)}_{z,1}|+|\mbox{(III)}_{z,2}| \\ 
& \ \leq c\int_{\B}\eta[G_{\varepsilon}(x,\snr{Du})]^{\kk+1}\frac{g_{2,\varepsilon}(x,\snr{Du})}{g_{1,\varepsilon}(x,\snr{Du})}\snr{DG_{\varepsilon}(x,\snr{Du})}|D\eta| \dx\\
&\  \leq c\int_{\B}\eta[G_{\varepsilon}(x,\snr{Du})]^{\kk+1}[\bar{G}_{\varepsilon}(x,\snr{Du})]^{\sigma}\snr{DG_{\varepsilon}(x,\snr{Du})}|D\eta| \dx\\
& \  \leq  \bar  \eps \mathcal{S}_2  +\frac{c}{\bar \varepsilon(\kk+1)}\int_{\B}\snr{D\eta}^{2}[G_{\eps}(x,\snr{Du})]^{\kk}[\bar{G}_{\varepsilon}(x,\snr{Du})]^{2(1+\sigma)} \dx\\
&\  \leq  \bar \eps \mathcal{S}_2  +\frac{c}{\bar \varepsilon(\kk+1)}\left(\int_{\B}\snr{D\eta}^{2m}[G_{\eps}(x,\snr{Du})]^{m\kk}[\bar{G}_{\varepsilon}(x,\snr{Du})]^{2m(1+\sigma)} \dx\right)^{1/m}
\end{flalign*}
(recall that $\bar{G}_{\varepsilon}\geq 1$) and, using now also \rif{z10}, we get
\begin{flalign}
& |\mbox{(III)}_{z,3}|+|\mbox{(III)}_{z,4}|  \notag \\ \notag
& \    \leq  c\int_{\B}\eta h(x)[G_{\varepsilon}(x,\snr{Du})]^{\kk+1}%%%typoc
 [\bar{G}_{\varepsilon}(x,|Du|)]^{1+\sigma+\hat{\sigma}}\snr{D\eta} \dx\\
&  \  \le  c\int_{\B}\eta^{2}[G_{\varepsilon}(x,\snr{Du})]^{\kk+2} \dx \nonumber
\\
\notag  &\quad \  + c\int_{\B}\snr{h(x)}^2 \snr{D\eta}^{2}[G_{\varepsilon}(x,\snr{Du})]^{\kk}[\bar{G}_{\varepsilon}(x,\snr{Du})]^{2(1+\sigma+\hat{\sigma})} \dx\\
 & \ \le c\int_{\B}\eta^{2}[G_{\varepsilon}(x,\snr{Du})]^{\kk}[\bar{G}_{\varepsilon}(x,\snr{Du})]^{2(1+\sigma+\hat \sigma )} \dx \nonumber \\
&\quad \  +c\nr{h}^{2}_{L^{d}(\B)}\left(\int_{\B}\snr{D\eta}^{2m}[G_{\varepsilon}(x,\snr{Du})]^{m\kk}[\bar{G}_{\varepsilon}(x,\snr{Du})]^{2m(1+\sigma+\hat \sigma)} \dx\right)^{1/m}\nonumber\\
&\   \leq c\left(\nr{h}^{2}_{L^{d}(\B)}+1\right)  \label{comewin}\\
& \  \quad \times \left(\int_{\B}(\eta^{2m}+\snr{D\eta}^{2m})[G_{\varepsilon}(x,\snr{Du})]^{m\kk}[\bar{G}_{\varepsilon}(x,\snr{Du})]^{2m(1+\sigma+\hat \sigma)} \dx\right)^{1/m}\notag 
\end{flalign}
%%%typoc
with $c\equiv c(n,N,\nu,\gamma, \cb)$. Now we look at the second group of terms stemming from \eqref{a2v}
\begin{flalign*}
& \sum_{s=1}^{n}\int_{\B}\partial_{x_{s}}a_{\eps}(x,Du)\cdot D\varphi_{s} \dx \\ &\stackrel{\rif{latesta}}{=} \sum_{s=1}^{n}\int_{\B}\eta^{2}[G_{\eps}(x,\snr{Du})]^{\kk+1}\partial_{x_{s}}a_{\eps}(x,Du)\cdot DD_{s}u \dx\nonumber \\
& \ +(\kk+1)\sum_{s=1}^{n}\int_{\B}\eta^{2}[G_{\varepsilon}(x,\snr{Du})]^{\kk}\partial_{x_{s}}a_{\eps}(x,Du)\cdot (D_{s}u\otimes DG_{\varepsilon}(x,\snr{Du})) \dx\nonumber \\
&\ +2\sum_{s=1}^{n}\int_{\B}\eta [G_{\varepsilon}(x,\snr{Du})]^{\kk+1}\partial_{x_{s}}a_{\eps}(x,Du)\cdot (D_{s}u\otimes D\eta) \dx\\
& \quad =:\mbox{(I)}_{x}+\mbox{(II)}_{x}+\mbox{(III)}_{x}\,.
\end{flalign*}
From $\eqref{growthd}_{3}$,$\eqref{z11}_{2}$, H\"older and Young inequalities, we obtain, for arbitrary $\bar \eps \in (0,1)$
\begin{flalign*}
& \snr{\mbox{(I)}_{x}}\leq \, c\int_{\B}\eta^{2}h(x)[G_{\varepsilon}(x,\snr{Du})]^{\kk+1}g_{3,\varepsilon}(x,\snr{Du})\snr{D^{2}u} \dx\nonumber \\
&\  \leq  \bar \varepsilon \mathcal{S}_1 +\frac{c}{\bar \varepsilon}\int_{\B}\eta^{2}[h(x)]^{2}[G_{\varepsilon}(x,\snr{Du})]^{\kk+1}\frac{[g_{3,\varepsilon}(x,\snr{Du})]^{2}}{g_{1,\varepsilon}(x,\snr{Du})} \dx\nonumber \\
& \ \le  \bar \varepsilon \mathcal{S}_1+\frac{c}{\bar \varepsilon}\nr{h}^{2}_{L^{d}(\B)}\left(\int_{\B}\eta^{2m}[G_{\varepsilon}(x,\snr{Du})]^{m\kk}[\bar{G}_{\varepsilon}(x,\snr{Du})]^{2m(1+\sigma)} \dx\right)^{1/m},
\end{flalign*}
where $c\equiv c(n,N,\nu,\gamma, \cb)$. Using this time $\eqref{z11}_{1}$ we get
\begin{flalign*}
\snr{\mbox{(II)}_{x}}\le&\, c(\kk+1)\int_{\B}\eta^{2}h(x)[G_{\varepsilon}(x,\snr{Du})]^{\kk}g_{3,\varepsilon}(x,\snr{Du})\snr{Du}\snr{DG_{\varepsilon}(x,\snr{Du})} \dx\nonumber \\
\le &\, \bar \varepsilon \mathcal{S}_2 +\frac{c(\kk+1)}{\bar \varepsilon}\int_{\B}\eta^{2}[h(x)]^{2}
[G_{\varepsilon}(x,\snr{Du})]^{\kk}[g_{3,\varepsilon}(x,\snr{Du})]^{2}\snr{Du}^{2} \dx \nonumber \\
\le &\,\bar \varepsilon\mathcal{S}_2 +\frac{c(\kk+1)}{\bar \varepsilon}\nr{h}^{2}_{L^{d}(\B)}\\
& \hspace{1cm} \times \left(\int_{\B}\eta^{2m}[G_{\varepsilon}(x,\snr{Du})]^{m\kk}[\bar{G}_{\varepsilon}(x,\snr{Du})]^{2m(1+\sigma)} \dx\right)^{1/m}
\end{flalign*}
with $c\equiv c(n,N,\nu,\gamma,\cb)$. Again, $\eqref{z11}_{1}$, Young and H\"older inequalities, as in \rif{comewin}, give
\begin{flalign}
\snr{\mbox{(III)}_{x}}\le&\,c\int_{\B}\eta h(x)[G_{\varepsilon}(x,\snr{Du})]^{\kk+1}g_{3,\varepsilon}(x,\snr{Du})\snr{Du}\snr{D\eta} \dx\nonumber \\
\le &\,c\int_{\B}\eta^{2}[G_{\varepsilon}(x,\snr{Du})]^{\kk+2} \dx \nonumber \\
&+c\int_{\B}\snr{D\eta}^{2}[h(x)]^{2}[G_{\varepsilon}(x,\snr{Du})]^{\kk}[g_{3,\varepsilon}(x,\snr{Du})]^{2}\snr{Du}^{2}\dx \nonumber \\
\le &\, c\left(\nr{h}^{2}_{L^{d}(\B)}+1\right)\label{finoa}\\
& \times \left(\int_{\B}(\eta^{2m}+\snr{D\eta}^{2m})[G_{\varepsilon}(x,\snr{Du})]^{m\kk}[\bar{G}_{\varepsilon}(x,\snr{Du})]^{2m(1+\sigma)} \dx\right)^{1/m} \notag
\end{flalign}
where $c\equiv c(n,N,\nu,\gamma,\cb)$. Finally, we examine the contributions to \rif{a2v} coming from the terms featuring $f$:
\begin{flalign*}
& \sum_{s=1}^{n}\int_{\B}f\cdot D_{s}\varphi_{s} \dx\stackrel{\rif{latesta}}{=}\sum_{s=1}^{n}\int_{\B}\eta^{2}[G_{\varepsilon}(x,\snr{Du})]^{\kk+1}(f\cdot D_{s}D_{s}u) \dx\nonumber \\
&\qquad +(\kk+1)\sum_{s=1}^{n}\int_{\B}\eta^{2}[G_{\varepsilon}(x,\snr{Du})]^{\kk}D_{s}G_{\varepsilon}(x,\snr{Du})(f\cdot D_{s}u) \dx\nonumber \\
&\qquad +2\sum_{s=1}^{n}\int_{\B}\eta [G_{\varepsilon}(x,\snr{Du})]^{\kk+1}D_{s}\eta(f\cdot D_{s}u) \dx \\ &=: \, \mbox{(I)}_{f}+\mbox{(II)}_{f}+\mbox{(III)}_{f}\,.
\end{flalign*}
Here recall that $n>2$. Using $\eqref{z10a}_{2}$, H\"older and Young inequalities we get
%%%typoc
\begin{flalign*}
\snr{\mbox{(I)}_{f}}\le&\, c\int_{\B}\eta^{2}[G_{\varepsilon}(x,\snr{Du})]^{\kk+1}\snr{f}\snr{D^{2}u} \dx \nonumber \\
\le &\,\bar \varepsilon\mathcal{S}_1+\frac{c}{\bar \varepsilon}\int_{\B}\eta^{2}[G_{\varepsilon}(x,\snr{Du})]^{\kk} \snr{f}^{2}[g_{1,\varepsilon}(x,\snr{Du})]^{-1}G_{\varepsilon}(x,\snr{Du}) \dx\nonumber \\
\le &\,\bar \varepsilon \mathcal{S}_1+\frac{c}{\bar \varepsilon}\int_{\B}\eta^{2}\snr{f}^{2}[G_{\varepsilon}(x,\snr{Du})]^{\kk}[\bar{G}_{\varepsilon}(x,\snr{Du})]^{2/\gamma} \dx\nonumber \\
\le &\,\bar \varepsilon\mathcal{S}_1+\frac{c}{\bar \varepsilon}\nr{f}_{L^{n}(\B)}^{2}\left(\int_{\B}\eta^{2^{*}}[G_{\varepsilon}(x,\snr{Du})]^{2^{*}\kk/2}[\bar{G}_{\varepsilon}(x,\snr{Du})]^{2^{*}} \dx\right)^{2/2^{*}}\,,
\end{flalign*}
with $c\equiv c(n,N,\nu, \gamma)$ and $2^{*}$ is as in \rif{2***}. Moreover, by $\eqref{z10a}_{1}$, we get
\begin{flalign*}
\snr{\mbox{(II)}_{f}}\le &\,c(\kk+1)\int_{\B}\eta^{2}[G_{\varepsilon}(x,\snr{Du})]^{\kk}\snr{f}\snr{DG_{\varepsilon}(x,\snr{Du})}\snr{Du} \dx\nonumber \\
\le &\,\bar \varepsilon\mathcal{S}_2+\frac{c(\kk+1)}{\bar \varepsilon}\int_{\B}\eta^{2}\snr{f}^{2}[G_{\varepsilon}(x,\snr{Du})]^{\kk}[\bar{G}_{\varepsilon}(x,\snr{Du})]^{2/\gamma} \dx\nonumber \\
\le &\,\bar \varepsilon\mathcal{S}_2+\frac{c(\kk+1)\nr{f}^{2}_{L^{n}(\B)}}{\bar \varepsilon}\\
&   \qquad \quad  \times \left(\int_{\B}\eta^{2^{*}}[G_{\varepsilon}(x,\snr{Du})]^{2^{*}\kk/2}[\bar{G}_{\varepsilon}(x,\snr{Du})]^{2^{*}} \dx\right)^{2/2^{*}}\,,
\end{flalign*}
for $c\equiv c(n,N,\nu, \gamma)$. Finally, again from $\eqref{z10a}_{1}$ we deduce that
\begin{flalign}
\nonumber & \snr{\mbox{(III)}_{f}}\le c\int_{\B}\eta [G_{\varepsilon}(x,\snr{Du})]^{\kk+1}\snr{f}\snr{D\eta}\snr{Du} \dx\nonumber \\
\nonumber & \quad \le c\int_{\B}\snr{D\eta}^{2}[G_{\varepsilon}(x,\snr{Du})]^{\kk+2} \dx \nonumber \\
&\qquad \qquad +c\int_{\B}\eta^{2}\snr{f}^{2}[G_{\varepsilon}(x,\snr{Du})]^{\kk}[\bar{G}_{\varepsilon}(x,\snr{Du})]^{2/\gamma} \dx\nonumber \\
 &\quad \le c\left(\int_{\B}\snr{D\eta}^{2m}[G_{\varepsilon}(x,\snr{Du})]^{m\kk}[\bar{G}_{\varepsilon}(x,\snr{Du})]^{2m} \dx\right)^{1/m} \nonumber \\
&\qquad \  +c\nr{f}_{L^{n}(\B)}^2\left(\int_{\B}\eta^{2^{*}}[G_{\varepsilon}(x,\snr{Du})]^{2^{*}\kk/2}[\bar{G}_{\varepsilon}(x,\snr{Du})]^{2^{*}} \dx\right)^{2/2^{*}}\label{equafine}
\end{flalign}
where $c\equiv c(n,N,\nu, \gamma)$. In the previous three displays, we also used that, by $\eqref{gte}_{2}$ it is $\bar{G}_{\varepsilon}(\cdot)\ge 1$ for all $(x,t)\in \Omega\times [0,\infty)$, thus, for instance $[\bar{G}_{\varepsilon}(\cdot)]^{1/\gamma}\le \bar{G}_{\varepsilon}(\cdot)$. 

Merging the content of displays from \rif{equainizio} to \rif{equafine}, with \rif{a2v}, choosing $\bar \varepsilon>0$ small enough (in order to reabsorb terms in the usual way), we get, after a few standard manipulations, and again using that $\bar{G}_{\varepsilon}(\cdot)\geq 1$, 
\begin{flalign}
& \mathcal{S}_1 + \mathcal{S}_2   \leq c(\kk+1)\left(\nr{h}_{L^{d}(\B)}^{2}+1\right)\label{a20} \\
& \quad \times \left(\int_{\B}(\eta^{2m}+\snr{D\eta}^{2m})[G_{\varepsilon}(x,\snr{Du})]^{m\kk}[\bar{G}_{\varepsilon}(x,\snr{Du})]^{2m(1+\sigma+\hat{\sigma})} \dx\right)^{1/m}\nonumber \\
&\qquad +c(\kk+1)\nr{f}_{L^{n}(\B)}^2\left(\int_{\B}\eta^{2^{*}}[G_{\varepsilon}(x,\snr{Du})]^{2^{*}\kk/2}[\bar{G}_{\varepsilon}(x,\snr{Du})]^{2^{*}} \dx\right)^{2/2^{*}}\notag\,,
\end{flalign}with $c\equiv c(n,N,\nu,\gamma,\cb)$. Now we note that, by $\eqref{gte}_{2}$ it follows
\begin{flalign}\label{bg}
G_{\varepsilon}(x,t)\le \bar{G}_{\varepsilon}(x,t)\le   (T^2+1)^{\gamma/2} \left[G_{\varepsilon}(x,t)+1\right]\,,
\end{flalign}for all $(x,t)\in \Omega\times [0,\infty)$, therefore, recalling \rif{stimaIIbis}, estimate \eqref{a20} can be rearranged as
\begin{flalign*}
 &\int_{\B}\eta^{2}[G_{\varepsilon}(x,\snr{Du})]^{\kk}\snr{DG_{\varepsilon}(x,\snr{Du})}^{2} \dx \notag \\
 &  \notag \le c\left(\nr{h}_{L^{d}(\B)}^{2}+1\right)\\
 & \qquad \times\left(\int_{\B}(\eta^{2m}+\snr{D\eta}^{2m}) \left[[G_{\varepsilon}(x,\snr{Du})]^{m(\kk+2+2\sigma+2\hat{\sigma})}+1\right] \dx\right)^{1/m} \\
&\ \ \ \  +c\nr{f}_{L^{n}(\B)}^2\left(\int_{\B}\eta^{2^{*}}\left[[G_{\varepsilon}(x,\snr{Du})]^{\frac{2^{*}}{2}(\kk+2)}+1 \right] \dx\right)^{2/2^{*}}\,,
\end{flalign*}
%%%typoc
where $c\equiv c(n,N,\nu, \gamma, T, \cb)$ and every number $\kk \geq 0$. Note that this is the first time a dependence of the constants on $T$ appears, via \rif{bg}; no blow-up occurs when $T\to 0$, compare with Remark \ref{remT}. Sobolev-Poincar\'e inequality and the previous display give
\begin{eqnarray}
&&\notag  \left(\int_{\B}\eta^{2^{*}}\left[[G_{\varepsilon}(x,\snr{Du})]^{\frac{2^{*}}{2}(\kk+2)}+1\right] \dx\right)^{2/2^{*}}\\
\notag &&  \hspace{-2mm}\le\left(\int_{\B}\eta^{2^{*}}\left[[G_{\varepsilon}(x,\snr{Du})]^{\frac\kk 2+1}+1\right]^{2^*}\dx\right)^{2/2^{*}}\nonumber \\
&&  \hspace{-2mm}\le c\int_{\B}\snr{ D(\eta([G_{\varepsilon}(x,\snr{Du})]^{\frac\kk 2+1}+1)) }^{2} \dx\nonumber \\
%&&  \notag \hspace{-2mm}\le c\left(\int_{\B}\snr{D\eta}^{2m}\left[[G_{\varepsilon}(x,\snr{Du})]^{m(\kk+2)}+1\right] \dx\right)^{1/m}\\
%&& \quad  +c(\kk+1)^2  \int_{\B}\eta^{2}[G_{\varepsilon}(x,\snr{Du})]^{\kk}\snr{DG_{\varepsilon}(x,\snr{Du})}^{2} \dx\nonumber \\
&&  \hspace{-2mm}\le c(\kk+1)^{2}\left(\nr{h}_{L^{d}(\B)}^{2}+1\right)\label{a22} \\
&& \quad \times \left(\int_{\B}(\eta^{2m}+\snr{D\eta}^{2m})\left[[G_{\varepsilon}(x,\snr{Du})]^{m(\kk+2+2\sigma+2\hat{\sigma})}+1\right] \dx\right)^{1/m}\nonumber \\
&&\qquad \ \ +c_{\textnormal{ab}}(\kk+1)^{2}\nr{f}_{L^{n}(\B)}^2\left(\int_{\B}\eta^{2^{*}}\left[[G_{\varepsilon}(x,\snr{Du})]^{\frac{2^{*}}{2}(\kk+2)}+1 \right] \dx\right)^{2/2^{*}}\notag,
\end{eqnarray}with $c, c_{\textnormal{ab}} \equiv c,c_{\textnormal{ab}}(n,N,\nu, \gamma, T, \cb)$. We now determine the radius $R_{*}$ such that $\rad(\B)\leq R_{*}$ as in the statement of Proposition \ref{lp}. Specifically, we fix $\bar{\kk}\geq  0$ and, using the absolute continuity of the integral, determine $R_{*}\equiv R_{*}(n,N,\nu, \gamma, \linebreak T, \cb,\mathfrak f(\cdot),\bar \kk)\in (0,1)$, such that
\begin{flalign}\label{a23}
 c_{\textnormal{ab}}(\bar \kk+1)^{2}\nr{f}_{L^{n}(\B)}^2\stackleq{a0} c_{\textnormal{ab}}(\bar \kk+1)^{2}\nr{\mathfrak f}_{L^{n}(\B)}^2\le 1/2\,.
\end{flalign}
With \rif{a23} being now in force, for all $\kk \leq \bar \kk$ estimate \eqref{a22} becomes
\begin{eqnarray}
&& \left(\int_{\B}\eta^{2^*}\left[[G_{\varepsilon}(x,\snr{Du})]^{\frac{2^{*}}{2}(\kk+2)}+1\right] \dx\right)^{2/2^{*}}\nonumber \\
&& \quad \  \le c(\kk+1)^{2}\left(\nr{h}_{L^{d}(\B)}^{2}+1\right)\label{a24} \\
&& \quad\quad  \times \left(\int_{\B}(\eta^{2m}+\snr{D\eta}^{2m})\left[[G_{\varepsilon}(x,\snr{Du})]^{m(\kk+2+2\sigma+2\hat{\sigma})}+1\right] \dx\right)^{1/m} \notag
\end{eqnarray}for $c\equiv c(n,N,\nu, \gamma, T, \cb)$. 
\subsection{Moser's iteration in finite steps}\label{mosersec}
With $1\leq i \in \N$, we inductively define the exponents
\eqn{recallss0}
$$
\kk_{1}:=0,\quad \kk_{i+1}:=\chi(\kk_{i}+2)-2(1+\sigma+\hat{\sigma}),\quad s_{i}:=m(\kk_{i}+2+2\sigma+2\hat{\sigma})\,.
$$
Note that $s_*/n - s_*/d < \chi-1$ if and only if $s_* < 2n/(n-2)$, and therefore the second condition in \rif{sigmasss} implies, again together with the first one, that
\eqn{mainbound2} 
$$\sigma+\hat{\sigma}<\frac {s_*}n - \frac {s_*}d < \chi -1\,.$$
This, in turn, implies that $\{\kk_i\}$ and $\{s_i\}$ are increasing sequences; moreover, it holds that
$$
\begin{cases}
\displaystyle \kk_{i+1} =2\chi^{i}-2(\sigma+\hat{\sigma})\sum_{j=0}^{i-1}\chi^{j}-2 \quad\mbox{ for $i \geq 1$}
\\ \displaystyle s_{i+1} =2m
\chi^{i}-2m(\sigma+\hat{\sigma})\sum_{j=1}^{i-1}\chi^{j} \quad  \mbox{for $i \geq 2$ \ \ and \ \ $
s_2= 2m\chi=2^*$}.
\end{cases}
$$
Note that this implies 
\eqn{recallss}
$$
\begin{cases}
\displaystyle s_{i+1}= \frac{2^{*}}{2}(\kk_{i}+2)= \frac{n(\kk_{i}+2)}{n-2} \Longrightarrow k_{i+1}\leq s_{i+1} \leq 2m\chi^{i}\\
\displaystyle k_{i+1}\leq 2 \chi^i \,.
\end{cases}
$$
Again \eqref{mainbound2} implies, for $i \geq 1$
\begin{flalign}\label{aa30}
\notag & s_{i+1}=2m\chi^{i}\left[1-\frac{(\sigma+\hat{\sigma})(1-\chi^{1-i})}{\chi-1}\right]>2m\chi^{i}\left(1-\frac{\sigma+\hat{\sigma}}{\chi-1}\right)\\
& \qquad  \Longrightarrow \lim_{i\to \infty}s_{i+1}=\infty\,,
%>\lim_{i\to \infty}\left[2m\chi^{i}\left(1-\frac{1}{2^{*}}\right)\right]= \infty\,.
\end{flalign}
so that from the first relation in \rif{recallss} it also follows that $\kk_i\to \infty$. 
For $0<\varsigma\le \tau_{1}<\tau_{2}\le \rr $, we consider a sequence $\{B_{\rr_{i}}\}$ of shrinking, concentric balls, 
$B_{\tau_1} \Subset B_{\rr_{i}+1} \Subset B_{\rr_{i}} \Subset B_{\tau_{2}}$, where $\rr_{i}:=\tau_{1}+(\tau_{2}-\tau_{1})2^{-i+1}$. Note that $\{\rr_{i}\}$ is a decreasing sequence with $\rr_{1}=\tau_{2}$ and $\rr_{i}\to \tau_{1}$, therefore it is $\bigcap_{i\in \N}B_{\rr_{i}}=\overline{B_{\tau_{1}}}$ %%%typoc
and $B_{\rr_{1}}=B_{\tau_{2}}$. Accordingly, we fix corresponding cut-off functions $\{\eta_{i}\}\subset C^{1}_{{\rm c}}(B)$ with $\mathds{1}_{B_{\rr_{i+1}}}\le\eta_{i}\le \mathds{1}_{B_{\rr_{i}}}$ and $\snr{D\eta_{i}}
\lesssim  1/(\rr_{i}-\rr_{i+1}) \approx 2^{i}/(\tau_{2}-\tau_{1})$. 
%\begin{flalign*}
%\mathds{1}_{B_{\rr_{i+1}}}\le\eta_{i}\le \mathds{1}_{B_{\rr_{i}}}\quad \mbox{and}\quad \snr{D\eta_{i}}\lesssim \frac{1}{\rr_{i}-\rr_{i+1}}\sim \frac{2^{i}}{\tau_{2}-\tau_{1}}\,.
%\end{flalign*}
Choosing $\eta\equiv \eta_{i}$ in \eqref{a24}, elementary manipulations also based on \rif{recallss0} and \rif{recallss} give that
\begin{flalign}\label{aa31}
&\left(\int_{B_{\rr_{i+1}}}\left[[G_{\varepsilon}(x,\snr{Du})]^{s_{i+1}}+1\right] \dx\right)^{1/\chi}\nonumber \\
&\qquad\le \frac{\left(c\nr{h}_{L^{d}(\B)}+c\right)^{2m}2^{2mi}(\kk_{i}+1)^{2m}}{(\tau_{2}-\tau_{1})^{2m}}\int_{B_{\rr_{i}}}\left[[G_{\varepsilon}(x,\snr{Du})]^{s_{i}}+1 \right] \dx
\end{flalign} holds whenever $\kk_i \leq \bar \kk$ and $i\geq 1$, with $c\equiv c(n,N,\nu, \gamma, T, \cb)$. Finally, we set $c_{h}:=(c\nr{h}_{L^{d}(\B)}+c)^{2^*}$ and
\begin{flalign*}
\mathcal{G}_{i}:=\left(\int_{B_{\rr_{i}}}\left[[G_{\varepsilon}(x,\snr{Du})]^{s_{i}} +1\right] \dx\right)^{1/s_{i}}\,,
\end{flalign*}
so that \rif{aa31} (recall that $2^* = 2m\chi$ by \rif{ilchi}) becomes
$$\mathcal{G}_{i+1}\leq \left[\frac{c_h2^{2^*i}(\kk_{i}+1)^{2^*}}{(\tau_{2}-\tau_{1})^{2^*}}\right]^{\frac{1}{s_{i+1}}} \mathcal{G}_{i}^{\frac{\chi s_i}{s_{i+1}}}\,. $$
%$$\mathcal{G}_{i+1}\leq \left[\frac{c_h2^{2mi}(\kk_{i}+1)^{2m}}{(\tau_{2}-\tau_{1})^{2m}}\right]^{\frac{\chi}{s_{i+1}}} \mathcal{G}_{i}^{\frac{\chi s_i}{s_{i+1}}}\,. $$
Iterating the above inequality yields that
\begin{flalign}
\mathcal{G}_{i+1}\le \prod_{j=0}^{i-1}\left[\frac{c_{h}2^{2^*(i-j)}(\kk_{i-j}+1)^{2^*}}{(\tau_{2}-\tau_{1})^{2^*}}\right]^{\frac{\chi^{j}}{s_{i+1}}}\mathcal{G}_{1}^{\frac{\chi^{i}s_{1}}{s_{i+1}}} \quad \mbox{holds provided $\kk_{i} \leq \bar \kk$}\,.\label{iwonder}
\end{flalign}
Now, from \eqref{aa30} we deduce that
\eqn{recalls2}
$$
\frac{\chi^{i+1}}{s_{i+1}}\leq  \frac{\chi(\chi-1)}{2m(\chi-1-\sigma-\hat{\sigma})}\,.
$$
The function $t \mapsto t/ \chi^{t}$ is decreasing on $[1/\log \chi , \infty)$ and using this fact one sees that
\eqn{recalls3}
$$
\sum_{j=1}^{\infty}\frac{j}{\chi^{j}}\lesssim \frac{1}{(\log \chi)^2}\leq \frac{c}{(\chi-1)^2} \,.
$$
We then write
\begin{flalign*}
&%\limsup_{i\to \infty}
\hspace{-5mm}
\prod_{j=0}^{i-1}\left[\frac{c_{h}2^{2^*(i-j)}(\kk_{i-j}+1)^{2^*}}{(\tau_{2}-\tau_{1})^{2^*}}\right]^{\frac{\chi^{j}}{s_{i+1}}}\nonumber \\
& =\exp\left\{\log \left(\frac{c_h}{(\tau_{2}-\tau_{1})^{2^*}}\right)\frac{1}{s_{i+1}} \sum_{j=0}^{i-1} \chi^j+
\frac{2^*\log 2}{s_{i+1}}\sum_{j=0}^{i-1}  (i-j)\chi^j \right.\\
& \left. \hspace{2cm}+ 
\frac{2^*}{s_{i+1}}\sum_{j=0}^{i-1}  \chi^j \log  (\kk_{i-j}+1)\right\},
\end{flalign*}
and note that, for every integer $i \geq 1$, we have 
$$
\frac{1}{s_{i+1}} \sum_{j=0}^{i-1} \chi^j \stackleq{recalls2} c\,, \quad 
\frac{1}{s_{i+1}}\sum_{j=0}^{i-1}  (i-j)\chi^j \leq c \sum_{j=0}^{i-1}  \frac{i-j}{\chi^{i-j}} \stackleq{recalls3} c 
$$
and 
$$
\frac{2^*\log 2}{s_{i+1}}\sum_{j=0}^{i-1}  \chi^j \log  (\kk_{i-j}+1)\stackrel{\rif{recallss}, \rif{recalls2}}{\leq }  c\sum_{j=0}^{i-1}  \frac{i-j}{\chi^{i-j}} \stackleq{recalls3} c\,,
$$
where $c \equiv c(n,d, \sigma, \tilde \sigma)$ in all cases. Using the content of the last three displays yields
\eqn{inverno}
$$ \prod_{j=0}^{i-1}\left[\frac{c_{h}2^{2^*(i-j)}(\kk_{i-j}+1)^{2^*}}{(\tau_{2}-\tau_{1})^{2^*}}\right]^{\frac{\chi^{j}}{s_{i+1}}}\leq 
 \frac{c}{(\tau_{2}-\tau_{1})^{\beta}}<\infty\,,$$
where $c\equiv c(\data,\nr{h}_{L^{d}(\B)})\geq 1$ and $\beta\equiv \beta(n,d,\sigma,\hat{\sigma})\geq 1$. Note that such constants blow-up when $\chi \to 1+ \sigma +\hat{\sigma}$; in particular, this happens when $d \to n$. Using \rif{inverno} in \rif{iwonder}, and keeping \rif{recalls2} in mind, yields
\begin{flalign}\label{a30}
\nr{G_{\varepsilon}(\cdot,\snr{Du})}_{L^{s_{i+1}}(B_{\tau_{1}})}&\le \mathcal{G}_{i+1}\nonumber \le \frac{c}{(\tau_{2}-\tau_{1})^{\beta}}
\mathcal{G}_{1}^{\frac{\chi^{i}s_{1}}{s_{i+1}}}
\\ & \le \frac{c}{(\tau_{2}-\tau_{1})^{\beta}}
\left[\nr{G_{\varepsilon}(\cdot,\snr{Du})}_{L^{s_1}(B_{\tau_{2}})}^{\frac{\chi^{i}s_{1}}{s_{i+1}}}+1\right]\,,
\end{flalign}with $c\equiv c(\data,\nr{h}_{L^{d}(\B)})$ and $\beta\equiv \beta(n,d,\sigma,\hat{\sigma})$, for every $i \in \en$ such that $\kk_i \leq \bar \kk$. 
\subsection{Sobolev regularity}\label{sr} By \rif{sigmasss} we have $1\leq s_*< s_1 = 2m(1+\sigma +\hat \sigma)$ so that, for every integer index $i \geq 1$, we consider the interpolation inequality
\eqn{interpo}
$$
\nr{G_{\varepsilon}(\cdot,\snr{Du})}_{L^{s_{1}}(B_{\tau_{2}})}\le \nr{G_{\varepsilon}(\cdot,\snr{Du})}_{L^{s_{i+1}}(B_{\tau_{2}})}^{\lambda_{i+1}}\nr{G_{\varepsilon}(\cdot,\snr{Du})}_{L^{s_*}(B_{\tau_{2}})}^{1-\lambda_{i+1}}\,,
$$
with $\lambda_{i+1}$ being defined by 
\eqn{il-la}
$$
\frac{1}{s_{1}}=\frac{1-\lambda_{i+1}}{s_*}+\frac{\lambda_{i+1}}{s_{i+1}}\Rightarrow \lambda_{i+1}=\frac{s_{i+1}(s_{1}-s_*)}{s_{1}(s_{i+1}-s_*)}\,.
$$
Let us show that, thanks to \rif{sigmasss}, 
there exist $\bar \vartheta\equiv \bar \vartheta(n,d, \sigma, \hat{\sigma})<1$ and $i_1\in \N$ such that
\begin{flalign}\label{claim}
 i>i_1 \Longrightarrow \frac{\lambda_{i+1}\chi^{i}s_{1}}{s_{i+1}}\leq \bar\vartheta <1\,.
\end{flalign}
Indeed, for this it is sufficient to observe that
\begin{flalign*}%\label{claim2}
\lim_{i\to \infty} \, \frac{\lambda_{i+1}\chi^{i}s_{1}}{s_{i+1}}& \stackrel{\rif{il-la}}{=} \lim_{i\to \infty} \,\frac{\chi^i(s_{1}-s_*)}{s_{i+1}-s_*}\\
& = l_* :=
 \frac{\chi-1}{2m}\frac{2m(1+\sigma +\hat{\sigma})- s_*}{\chi -1- \sigma -\hat{\sigma}}%:= \frac{1+\sigma +\hat{\sigma}- \frac{s_*}{2m}}{1- \frac{\sigma +\hat{\sigma}}{\chi-1}}
\stackrel{\rif{sigmasss}}{<}1\,.
\end{flalign*}
Note that the last inequality is actually equivalent to \rif{sigmasss} and that \rif{sigmasss} is in fact implied by \rif{sigma} for $s_*=1$. Now, we consider the number $p>1$ for the statement of Proposition \ref{lp}, and determine another index $k\equiv k(p) > i_1$ such that $s_{k+1}\geq p$; accordingly, we consider the number $\kk_{k}$ related to $s_{k}$ via \eqref{recallss}. We now choose the number $\bar\kk\equiv \bar \kk (p)$ in \eqref{a23} as $\bar \kk := \kappa_k$, and accordingly we determine $R_*\equiv R_*(p)$ via \eqref{a23}. It follows that \eqref{a30} holds in the case $i\equiv k$ and therefore we can plug \eqref{interpo} in it, thereby obtaining
\begin{flalign}\label{a40}
& \nr{G_{\varepsilon}(\cdot,\snr{Du})}_{L^{s_{k+1}}(B_{\tau_{1}})}\nonumber \\
& \qquad \notag 
\le \frac{c}{(\tau_{2}-\tau_{1})^{\beta}}\nr{G_{\varepsilon}(\cdot,\snr{Du})}_{L^{s_{k+1}}(B_{\tau_{2}})}^{\frac{\lambda_{k+1}\chi^{k}s_{1}}{s_{k+1}}}\nr{G_{\varepsilon}(\cdot,\snr{Du})}_{L^{s_*}(B_{\tau_{2}})}^{\frac{(1-\lambda_{k+1})\chi^{k}s_{1}}{s_{k+1}}}\\
& \qquad \quad  +\frac{c}{(\tau_{2}-\tau_{1})^{\beta}}\,.
\end{flalign}
On the other hand, as $k > i_1$, then \rif{claim} holds with $i\equiv k$; we can therefore apply
Young inequality %with conjugate exponents $\left(\frac{s_{k+1}}{\lambda_{k+1}\chi^{k}s_{1}},\frac{s_{k+1}}{s_{k+1}-\lambda_{k+1}\chi^{k}s_{1}}\right)$ 
in \eqref{a40}; this yields
\begin{flalign*}
\nr{G_{\varepsilon}(\cdot,\snr{Du})}_{L^{s_{k+1}}(B_{\tau_{1}})}  &\le  \frac{1}{2}\nr{G_{\varepsilon}(\cdot,\snr{Du})}_{L^{s_{k+1}}(B_{\tau_{2}})}\nonumber\\
&  \qquad +\frac{c}{(\tau_{2}-\tau_{1})^{\beta_{k}}}\left[\nr{G_{\varepsilon}(\cdot,\snr{Du})}_{L^{s_*}(B_{\rr})}^{\theta_{k}}+1\right],
\end{flalign*}
where 
\eqn{nuoviesponenti}
$$\beta_{k}:=\frac{s_{k+1}\beta}{s_{k+1}-\lambda_{k+1}\chi^{k}s_{1}}\quad \mbox{and}\quad \theta_{k}:=\frac{(1-\lambda_{k+1})\chi^{k+1}s_{1}}{s_{k+1}-\lambda_{k+1}\chi^{k}s_{1}}\,.$$
Lemma \ref{l0l} with the choice $\mathcal{Z}(t)\equiv \nr{G_{\varepsilon}(\cdot,\snr{Du})}_{L^{s_{k+1}}(B_{t})}$, finally yields
\begin{flalign*}%\label{a41}
\nr{G_{\varepsilon}(\cdot,\snr{Du})}_{L^{s_{k+1}}(B_{\varsigma})}\le \frac{c}{(\rr-\varsigma)^{\beta_{k}}}\left[\nr{G_{\varepsilon}(\cdot,\snr{Du})}_{L^{s_*}(B_{\rr})}^{\theta_{k}}+1\right]\,,
\end{flalign*}
with $c\equiv c(\data,\nr{h}_{L^{d}(\B)})$. This last inequality holds provided the bound in \rif{sigmasss} holds. All in all, recalling \eqref{z6}$_{2}$, \rif{bg} and that $p \leq s_{k+1}$, completes the proof of Proposition \ref{lp} with $\beta_p:=\beta_{k(p)}, \theta_p:=\theta_{k(p)}$. We remark that in \rif{a44} the exponents $\theta_{p},\beta_{p}$ can be replaced by exponents $\tilde \beta, \tilde \theta\equiv \tilde \beta, \tilde \theta(n,d,\sigma,\hat{\sigma})$ that are independent of $p$. For this, observe that
\eqn{convy}
 $$
 \begin{cases}
 \beta_k, \beta_p \stackrel{\rif{nuoviesponenti}}{\to} \frac{\beta}{1-l_*}=:\tilde \beta\,,\quad  \lambda_k \stackrel{\rif{il-la}}{\to} 1- \frac{s_*}{s_1}\\
 \theta_k, \theta_p \stackrel{\rif{nuoviesponenti}}{\to} \frac{\chi l_*s_*}{(s_1-s_*)(1-l_*)}:=\tilde \theta\,,
 \end{cases}
 $$
 when $k, p\to \infty$. 
The only dependence on $p$ in \rif{a44} comes through the threshold radius $R_{*}$; it is $R_{*}(p)\to 0$ as $p\to \infty$ unless $f \equiv 0$. 

\begin{remark}\label{t4uno}
 In the proof of Proposition \ref{lp}, assume that $g_{2,\eps}/g_{2,\eps} \leq c_{\rm b}$ as in \rif{z9hold}, and that $\sigma =\hat \sigma$ (compare with the assumptions of Theorem \ref{t4}). It is then easy to see that $\sigma +\hat \sigma$ can be replaced by $\hat \sigma$, everywhere, starting from \rif{a24}. This finally reflects in the same replacement in \rif{sigmasss}; in particular, we can assume $\hat \sigma < s_*/n-s_*/d$.
\end{remark}

\subsection{A Lipschitz bound in the homogeneous case $f \equiv 0$.} The above reasoning, eventually culminating in Proposition \ref{lp}, immediately leads to a priori Lipschitz estimates when $f\equiv 0$. The result, when combined with the approximation of Section \ref{pt1} below, extends those in \cite{M2, M3, MP} to the case of nonautonomous functionals with superlinear growth as in \rif{genF}. For this, the key observation is that it is not any longer necessary to consider balls with small radii $R_*$ as in \rif{a23}, as the last term in \rif{a22} does not appear. Therefore we can take everywhere, and in particular in \rif{a44}, $R_*=1$, independently of the value of $p$. It follows we can let $p\to \infty$ in \rif{a44}, and recalling \rif{convy} we conclude with
\begin{proposition}\label{lpinf}
Let $u\in W^{1,\gamma}(\B;\RN)$ be a solution to $\eqref{a0}$, under assumptions $\setm$ in \trif{ipotesi} for $n\geq 2$ and with $f\equiv 0$. Moreover, replace \trif{sigma} by the weaker $\sigma +\hat\sigma < 1/n-1/d$. 
If $B_{\varsigma}\Subset  B_{\rr} $ are concentric balls contained in $\B$, then 
$$%\label{a41bis}
\nr{G_{\varepsilon}(\cdot,\snr{Du})}_{L^{\infty}(B_{\varsigma} )} \le
 \frac{c}{(\rr-\varsigma)^{\tilde \beta}}\left[\|F_{\varepsilon}(\cdot,Du)\|_{L^1(B_{\rr})}^{\tilde \theta}+1\right] \,,
$$
holds with $c\equiv c(\data,\nr{h}_{L^{d}(\B)})\geq 1$, $\tilde\beta, \tilde \theta\equiv \tilde \beta,\tilde \theta(n,d,\sigma,\hat{\sigma})>0$.
\end{proposition}
Notice that here we are using Proposition \ref{lp} with the choice $s_*=1$. Note also that Proposition \ref{lp} refers to the case $n>2$. The remaining two dimensional case can be obtained via minor modifications to the proof of Proposition \ref{lp}, by choosing $2^*/2$ large enough (see \rif{2***}) in order to get $\chi >1$ in \rif{ilchi}. Anyway, the two dimensional case $n=2$ will be treated in Section \ref{casetwo} directly for the general case $f \not=0$. In that situation the proof cannot be readapted from the one of Proposition \ref{lp} as for Proposition \ref{lpinf}.

\subsection{Caccioppoli inequality on level sets}\label{cls} This is in the following:
\begin{lemma}[Caccioppoli inequality]\label{caclem}
Let $u\in W^{1,\gamma}(\B;\RN)$ be a solution to $\eqref{a0}$, under assumptions $\setm$ in \trif{ipotesi} for $n\geq 2$, and let $B_{r}(x_0)\Subset \B$ be a ball. 
%and $M\geq 0$ such that
%\eqn{duecond}
%$$
%\|Du\|_{L^{\infty}(B_{r}(x_0))} \leq M \quad \mbox{and} \quad T %\leq M\,.
%$$
Then
\begin{flalign}\label{a45}
&\int_{B_{r/2}(x_{0})}\snr{D(G_{\varepsilon}(x,\snr{Du})-\kk)_{+}}^{2} \  \dx\nonumber \\
&\qquad  \le \frac{c}{r^{2}}\nr{\bar{G}_{\varepsilon}(\cdot,\snr{Du})}_{L^{\infty}(B_{r}(x_{0}))}^{\vartheta\sigma}\int_{B_{r}(x_{0})}(G_{\varepsilon}(x,\snr{Du})-\kk)_{+}^{2} \  \dx\nonumber \\
&\quad \qquad  +c\int_{B_{r}(x_{0})}[h(x)]^{2}[\bar{G}_{\varepsilon}(x,\snr{Du})]^{2(1+\sigma+\hat{\sigma})} \dx\nonumber 
\\&\quad \qquad +c\nr{Du}_{L^{\infty}(B_{r}(x_{0}))}^{2}\int_{B_{r}(x_{0})}\snr{f}^{2} \  \dx
\end{flalign} holds whenever $\kk \geq 0$, with $c\equiv c(\data)$ and $\vartheta\equiv \vartheta(\gamma)$ is as in \eqref{sigma}. 
\end{lemma}
\begin{proof} 
For $s \in \{1, \ldots, n\}$, we take $\varphi_{s}:=\eta^{2}(G_{\varepsilon}(x,\snr{Du})-\kk)_{+}D_{s}u$ in \eqref{a2v}, where $\eta\in C^{1}_{{\rm c}}(B_{r}(x_{0}))$ satisfies $\mathds{1}_{B_{r/2}(x_{0})}\le \eta\le \mathds{1}_{B_{r}(x_{0})}$ and $\snr{D\eta}\lesssim 1/r$. Admissibility of $\varphi_s$ follows by \rif{a6} and \rif{differenziaG}. Note that all the integrals stemming from \eqref{a2v} extend over $\B^{\kk}:=\B\cap\{G_{\varepsilon}(\cdot,\snr{Du})>\kk\}$; in particular, by the very definition of $G_{\varepsilon}$, we can always restrict to the case it is $|Du|>T$ and $g_{1, \eps}(\cdot, |Du|)>0$ (recall \rif{00bis}). We start expanding the terms resulting from \rif{a2v}
\begin{flalign}
\notag \sum_{s=1}^{n}&\int_{\B^{\kk}}\partial_{z}a_{\eps}(x,Du)DD_{s}u\cdot D\varphi_{s} \dx \\ & =\sum_{s=1}^{n}\int_{\B^{\kk}}\eta^{2}(G_{\varepsilon}(x,\snr{Du})-\kk)_{+}\partial_{z}a_{\eps}(x,Du)DD_{s}u\cdot DD_{s}u \dx\nonumber \\
\notag &\quad +\sum_{s=1}^{n}\int_{\B^{\kk}}\eta^{2}\partial_{z}a_{\eps}(x,Du)DD_{s}u\cdot \left[D_{s}u \otimes  D(G_{\varepsilon}(x,\snr{Du})-\kk)_{+}\right]\dx\nonumber \\
\notag &\quad +2\sum_{s=1}^{n}\int_{\B^{\kk}}\eta(G_{\varepsilon}(x,\snr{Du})-\kk)_{+}\partial_{z}a_{\eps}(x,\snr{Du})DD_{s}u\cdot D_{s}u\otimes D\eta \dx\\ &  =:\mbox{(IV)}_{z}+\mbox{(V)}_{z}+\mbox{(VI)}_{z}\,.\label{posi1}
\end{flalign}
Moreover, it is
\begin{flalign}
& \notag \sum_{s=1}^{n}\int_{\B^{\kk}}\partial_{x_{s}}a_{\eps}(x,Du)\cdot D\varphi_{s} \dx\\
&=\sum_{s=1}^{n}\int_{\B^{\kk}}\eta^{2}(G_{\varepsilon}(x,\snr{Du})-\kk)_{+}\partial_{x_{s}}a_{\eps}(x,Du)\cdot DD_{s}u \dx\nonumber \\
\notag
&+\sum_{s=1}^{n}\int_{\B^{\kk}}\eta^{2}\partial_{x_{s}}a_{\eps}(x,Du)\cdot \left(D_{s}u \otimes D(G_{\varepsilon}(x,\snr{Du})-\kk)_{+}\right) \dx \nonumber \\
\notag &+2\sum_{s=1}^{n}\int_{\B^{\kk}}\eta (G_{\varepsilon}(x,\snr{Du})-\kk)_{+}\partial_{x_{s}}a_{\eps}(x,Du)\cdot \left(D_{s}u\otimes D\eta\right) \dx\notag \\ &   =:\mbox{(IV)}_{x}+\mbox{(V)}_{x}+\mbox{(VI)}_{x}\,. \label{posi22222}  
\end{flalign}
By \rif{inaggiunta} we have
\eqn{posi3}
$$
\mbox{(IV)}_{z}\ge \int_{\B^{\kk}}\eta^{2}(G_{\varepsilon}(x,\snr{Du})-\kk)_{+}g_{1,\varepsilon}(x,\snr{Du})\snr{D^{2}u}^{2} \dx=:\mathcal{S}_3
$$
and for later use we also define
\eqn{posi2}
$$
 \mathcal{S}_4:= \int_{\B^{\kk}}\eta^{2}\snr{D(G_{\varepsilon}(x,\snr{Du})-k)_{+}}^{2} \dx\,.
$$
We then consider two different cases. 

{\em Case 1: $1< \gamma <2$ in \trif{00}, and therefore it is $\vartheta=2$}. 
We proceed estimating the terms $\mbox{(V)}_{z}$ and $\mbox{(VI)}_{z}$. The estimate for the term (V)$_{z}$ is similar to the one for (II)$_{z}$ in \rif{stimaII}, see in particular those for the terms $\mbox{(II)}_{z,3}$ and $\mbox{(II)}_{z,4}$; indeed, again using \rif{a1}, \rif{z9}, \rif{z10}, \rif{a3} and \rif{a7}, we have
\begin{flalign*}
\notag \mbox{(V)}_{z} & \geq \mathcal{S}_4  -c\int_{\B^{\kk}}\eta^{2}h(x)[\bar{G}_{\varepsilon}(x,s)]^{1+\sigma+\hat{\sigma}}\snr{D(G_{\varepsilon}(x,\snr{Du})-k)_{+}}\dx\\
&\ge \frac 12 \mathcal{S}_4 -c\int_{B_{r}(x_0)}[h(x)]^{2}[\bar{G}_{\varepsilon}(x,\snr{Du})]^{2(1+\sigma+\hat{\sigma})} \dx\,,
%\label{stimaIItris} %no
\end{flalign*}
with $c\equiv c(n,N,\nu,\gamma, \cb)$. Using \rif{diso2}-\rif{a9}, we then have 
%%%typoc puntino
\begin{flalign*}
\mbox{(VI)}_{z}&=2\int_{\B^{\kk}}\eta (G_{\varepsilon}(x,\snr{Du})-\kk)_{+}\frac{\bar{a}_{\eps}(x,\snr{Du})}{g_{1,\varepsilon}(x,\snr{Du})}DG_{\varepsilon}(x,\snr{Du})\cdot D\eta \dx\nonumber\\
&\qquad +2\sum_{\alpha=1}^{N}\int_{\B^{\kk}}\eta \mathds{1}_{ \mathcal{D}}(|Du|)(G_{\varepsilon}(x,\snr{Du})-\kk)_{+}\frac{\bar{a}_{\eps}'(x,\snr{Du})\snr{Du}}{g_{1,\varepsilon}(x,\snr{Du})\snr{Du}^2}\\ & \hspace{35mm} \times (DG_{\varepsilon}(x,\snr{Du})\cdot Du^{\alpha})(Du^{\alpha}\cdot D\eta) \dx\nonumber \\
&\qquad -2\int_{\B^{\kk}}\eta (G_{\varepsilon}(x,\snr{Du})-\kk)_{+}\frac{\bar{a}_{\eps}(x,\snr{Du})}{g_{1,\varepsilon}(x,\snr{Du})}\\
& \hspace{35mm} \times \left(\int_{T}^{\snr{Du}}\partial_{x}g_{1,\varepsilon}(x,t)t \d t\cdot D\eta \right) \dx\nonumber\\
&\qquad -2\sum_{\alpha=1}^{N}\int_{\B^{\kk}}\eta \mathds{1}_{ \mathcal{D}}(|Du|)(G_{\varepsilon}(x,\snr{Du})-\kk)_{+}\frac{\bar{a}_{\eps}'(x,\snr{Du})\snr{Du}}{g_{1,\varepsilon}(x,\snr{Du})\snr{Du}^2}\\  & \hspace{25mm} \times \left(\int_{T}^{\snr{Du}}\partial_{x}g_{1,\varepsilon}(x,t)t \d t\cdot Du^{\alpha}\right)(Du^{\alpha}\cdot D\eta) \dx\nonumber\,.
\end{flalign*}
%%%typoc formula in piÃ¹
Using \rif{z9}, \rif{z10}, \rif{a3}-\rif{a4}, and Young inequality, we have
%%%typoc sup di troppo
\begin{flalign}
\nonumber |\mbox{(VI)}_{z}| &\leq \frac 14\mathcal{S}_4 + \frac{c}{r^2} \nr{\bar{G}_{\varepsilon}(\cdot,\snr{Du})}_{L^{\infty}(B_{r}(x_{0}))}^{2\sigma}
\int_{B_{r}(x_0)}(G_{\varepsilon}(x,\snr{Du})-\kk)_{+}^2 \, dx\\
& \quad + c\int_{B_{r}(x_0)}[h(x)]^{2}[\bar{G}_{\varepsilon}(x,\snr{Du})]^{2(1+\hat{\sigma})} \, dx\,, \label{stimaIIquater}
\end{flalign}
where $c\equiv c(n,N,\nu,\gamma, \cb)$. 
%Note that, as all the integrals above are restricted to the set $\{|Du|>T\}$, we are here allowed to use \rif{einc} in conjunction to \rif{duecond}. 
Gathering the content of displays from \rif{posi1} to \rif{stimaIIquater}, and using them in \rif{a2v}, some further elementary estimations we have 
%%%typoc2
\begin{flalign}
\mathcal S_3 +\mathcal S_4  & \leq  c|\mbox{(IV)}_{x}|+c|\mbox{(V)}_{x}|+c|\mbox{(VI)}_{x} |\nonumber \\
 \notag &  \quad + \frac{c}{r^2}\nr{\bar{G}_{\varepsilon}(\cdot,\snr{Du})}_{L^{\infty}(B_{r}(x_{0}))}^{\vartheta \sigma}
\int_{B_{r}(x_0)}(G_{\varepsilon}(x,\snr{Du})-\kk)_{+}^2 \, dx\\
&\quad + c\int_{B_{r}(x_0)}[h(x)]^{2}[\bar{G}_{\varepsilon}(x,\snr{Du})]^{2(1+\sigma+\hat{\sigma})} \, dx \notag \\
&  \quad + c\sum_{s=1}^n\int_{\B^{\kk}} \snr{f\cdot D_s\varphi_s}\  \dx  \label{incon}
\end{flalign}
where $\vartheta=2$ and $c\equiv c(n,N,\nu,\gamma, \cb)$.
%%%typoc2

{\em Case 2: $\gamma \geq 2$ in \trif{00}, and therefore it is $\vartheta=1$}.  In this case we use that the function $t \mapsto \bar a_{\eps} (\cdot, t)$ is non-decreasing, so that $\bar{a}_{\eps}'(\cdot)$ is non-negative (when it exists). We note that 
\begin{flalign}
\notag &  \mbox{(V)}_{z}+\mbox{(VI)}_{z} \notag \\ 
&\notag \quad  \stackrel{\rif{a8}}{=} 
 \sum_{s=1}^n \int_{\B^{\kk}}  \eta^2  g_{1, \eps}(x,\snr{Du}) \snr{Du}   \partial_z a_{\eps} (x,Du)  D D_s u \cdot D_s u\otimes  D \snr{Du}  \dx \\
  \notag &\qquad   +  2\sum_{s=1}^n \int_{\B^{\kk}}  \eta  (G_{\varepsilon}(x,\snr{Du}) - k)_ +\partial_z a_{\eps} (x,Du)  D D_s u \cdot D_{s}u \otimes  D\eta   \dx\\
 & \qquad
+ \sum_{s=1}^{n}\int_{\B^{\kk}}\eta^{2}\partial_{z}a_{\eps}(x,Du)DD_{s}u\cdot D_{s}u \otimes  \int_{T}^{\snr{Du}}\partial_{x}g_{1,\varepsilon}(x,t)t \d t  \dx \notag \\& \quad \ \ =: 
  \mbox{(V)}_{z,1}+\mbox{(VI)}_{z}+  \mbox{(V)}_{z,2} \,.
\label{doppio}
\end{flalign}
%%%typoc2 in turn
In turn, we have 
\begin{eqnarray}
\notag  && \mbox{(V)}_{z,1} \stackrel{\rif{diso2}}{=}\int_{\B^{\kk}}   \eta^2  g_{1, \eps}(x,\snr{Du})\bar a_{\eps} (x, |Du|) |D |Du||^2   \snr{Du}^2  \dx
  \\
 \notag &&\ + \sum_{\alpha =1}^{N}\int_{\B^{\kk}}   \eta^2  g_{1, \eps}(x,\snr{Du})\mathds{1}_{ \mathcal{D}}(|Du|)\bar a_{\eps}' (x, |Du|) |Du|
(D |Du|\cdot Du^{\alpha})^2\, dx\\
&& \stackrel{\rif{a1}, \bar a_{\eps}'(\cdot)\geq 0
}{\geq} \int_{\B^{\kk}}   \eta^2  [g_{1, \eps}(x,\snr{Du}) \snr{Du}]^2  |D |Du||^2 \dx 
\,.\label{VVVV}
\end{eqnarray}
%so that
%\eqn{stti1}
%$$
% \mbox{(V)}_{z,1} \geq   \int_{\B^{\kk}}  \eta^2  [g_{1, \eps}(x,\snr{Du})]^2 \snr{Du}^2 |D \snr{Du}|^2  \dx\,.
%$$
%%%typoc2 $ \mbox{(VI)}_{z}$
For $ \mbox{(VI)}_{z}$, keeping in mind the identity in the last display, we again use \rif{diso2} and $\bar a_{\eps}'(\cdot)\geq 0$ to estimate via Young inequality as follows:
%%%typoc2 levare \frac 14 dove c'Ãš  \mbox{(V)}_{z,1} 
\begin{eqnarray*}
&&\notag  4|\mbox{(VI)}_{z}| \leq 
  8 \int_{\B^{\kk}}  \eta  (G_{\varepsilon}(x,\snr{Du}) - k)_ +\bar{a}_{\eps}(x,\snr{Du}) |Du| |D |Du|\cdot D\eta|\dx  \\ 
    \notag && \qquad \qquad \  +
     8\sum_{\alpha =1}^N \int_{\B^{\kk}}  \eta  (G_{\varepsilon}(x,\snr{Du}) - k)_ +
  \mathds{1}_{ \mathcal{D}}(|Du|)  \bar a_{\eps}' (x, |Du|)\\
  &&\hspace{4cm} \times |(D|Du|\cdot Du^\alpha)(Du^\alpha\cdot D\eta)|\, dx\\
   \notag && \qquad  \leq  \mbox{(V)}_{z,1} + c
 \int_{\B^{\kk}}  \frac{(G_{\varepsilon}(x,\snr{Du}) - k)_ +^2}{ g_{1, \eps}(x,\snr{Du})}   \bar a_{\eps} (x, |Du|) |D\eta|^2 \dx\\
  \notag  &&    \quad \qquad + c\sum_{\alpha =1}^N
 \int_{\B^{\kk}}  \frac{(G_{\varepsilon}(x,\snr{Du}) - k)_ +^2}{ g_{1, \eps}(x,\snr{Du})} \mathds{1}_{ \mathcal{D}}(|Du|) \bar a_{\eps}' (x, |Du|) \frac{(Du^\alpha\cdot D\eta)^2}{|Du|}  \dx\\
 && \qquad \leq  \mbox{(V)}_{z,1} +  \frac{c}{r^2}\nr{\bar{G}_{\varepsilon}(\cdot,\snr{Du})}_{L^{\infty}(B_{r}(x_{0}))}^{\sigma} \int_{B_{r}(x_0)} (G_{\varepsilon}(x,\snr{Du}) - \kk)_+^2  \, dx\,.%\label{letting0}
 \end{eqnarray*}
As for $\mbox{(V)}_{z,2}$, by letting 
\eqn{letting}
$$\mathcal I :=  \int_{T}^{\snr{Du}}\partial_{x}g_{1,\varepsilon}(x,t)t \d t  \Longrightarrow 
|\mathcal I | \stackleq{z10} c h(x)[\bar{G}_{\varepsilon}(x,\snr{Du})]^{1 +\hat \sigma}$$ 
we have, again using that $\bar a_{\eps}'(\cdot)\geq 0$
\begin{eqnarray}
\notag 4 |\mbox{(V)}_{z,2}| &\stackleq{diso2} 
  &  4\int_{\B^{\kk}}  \eta^2 \bar{a}_{\eps}(x,\snr{Du}) |Du| \left|D |Du|\cdot \mathcal I\right|\, dx  \\
  \notag &&   \quad +
     4\sum_{\alpha =1}^N \int_{\B^{\kk}}  \eta^2
  \mathds{1}_{ \mathcal{D}}(|Du|)  \bar a_{\eps}' (x, |Du|)\\ && \hspace{2cm} \times 
|(D|Du|\cdot Du^\alpha)(Du^\alpha\cdot\mathcal I)|\, dx \notag \\
   \notag &  \stackleq{VVVV}&  \mbox{(V)}_{z,1} + c
 \int_{B_{r}(x_0)}\frac{\eta^2 [\bar a_{\eps}(x, |Du|)+\bar a_{\eps}'(x, |Du|)|Du|]}{ g_{1, \eps}(x,\snr{Du})}|\mathcal{I}|^2  \dx\\
  \notag  &  \stackrel{\eqref{a2}_2}{\leq}  &  \mbox{(V)}_{z,1} + c
 \int_{B_{r}(x_0)}\frac{g_{2, \eps}(x,\snr{Du})}{ g_{1, \eps}(x,\snr{Du})} \left|\mathcal I\right|^2 \dx\\
 & \stackrel{\rif{z9}, \rif{letting}}{\leq} &  \mbox{(V)}_{z,1} + c\int_{B_{r}(x_0)}[h(x)]^{2}[\bar{G}_{\varepsilon}(x,\snr{Du})]^{\sigma +2(1+\hat{\sigma})} \dx\label{assembl}.
 \end{eqnarray}
On the other hand, we have
\begin{eqnarray}%\int_{\B^{\kk}}\eta^{2}\snr{D(G_{\varepsilon}(x,\snr{Du})-\kk)_{+}}^{2} \dx
\notag \mathcal{S}_4 
 &\stackleq{a8} &  c\int_{\B^{\kk}}    \eta^2\left(  [g_{1, \eps}(x,\snr{Du}) \snr{Du}]^2  |D |Du||^2 + |\mathcal I|^2\right) \dx \\
\notag & \stackleq{letting} & c\int_{\B^{\kk}}   \eta^2\left( [g_{1, \eps}(x,\snr{Du}) \snr{Du}]^2  |D |Du||^2\right. \\
&& \hspace{2cm}\left. \notag+  |h(x)|^2[\bar{G}_{\varepsilon}(x,\snr{Du})]^{2(1 +\hat \sigma)} \right)\dx \\
&\stackleq{VVVV} & c \mbox{(V)}_{z,1} +c\int_{B_{r}(x_0)}[h(x)]^{2}[\bar{G}_{\varepsilon}(x,\snr{Du})]^{2(1+\hat{\sigma})} \dx\,.\label{SSSS}
\end{eqnarray}
Assembling the content of displays from \rif{doppio} to \rif{SSSS} and using it in \rif{a2v}, we again conclude with \rif{incon}, but this time with $\vartheta =1$. We proceed estimating the $x$-terms coming from \rif{incon} (these have been defined in \rif{posi22222}), using $\eqref{growthd}_{3,4}$, \eqref{z11} as follows:
\begin{flalign*}
\snr{\mbox{(IV)}_{x}}+\snr{\mbox{(V)}_{x}}& \le \bar \eps\mathcal{S}_3  +\bar \eps\mathcal{S}_4+\frac{c}{ \bar\varepsilon}\int_{B_{r}(x_0)}\eta^{2}[h(x)]^2[\bar{G}_{\varepsilon}(x,\snr{Du})]^{2(1+\sigma)} \dx\,,\\
\snr{\mbox{(VI)}_{x}} &\stackrel{\eqref{z11}_1}{\leq} \frac{c}{r^2}\int_{B_{r}(x_0)}(G_{\varepsilon}(x,\snr{Du})-\kk)_{+}^{2}\dx\nonumber\\
& \qquad \ \ +c\int_{B_{r}(x_0)}[h(x)]^{2}[\bar{G}_{\varepsilon}(x,\snr{Du})]^{2(1+\sigma)} \dx\,, 
\end{flalign*}
with $c\equiv c(\data)$ and arbitrary $\bar \eps \in (0,1)$. Finally, the estimate of the terms involving $f$ can be done by using Young inequality and \rif{wewewe} %and \rif{duecond} 
as follows:
%%%typoc2 c/r^2
\begin{eqnarray}
&&\notag  \sum_{s=1}^n\int_{\B^\kk} \snr{f\cdot D_s\varphi_s} \dx  \leq  \bar \eps\mathcal{S}_3  +\bar \eps\mathcal{S}_4 + \frac{c}{\bar \eps}\int_{\B^{\kappa}} |D\eta|^2(G_{\varepsilon}(x,\snr{Du}) - \kk)_+^2 \dx \\
   &&\notag   \qquad   \qquad  \qquad  \qquad  \qquad + \frac{c}{\bar \eps } \int_{\B^{\kappa}}\eta^2 \left[\frac{(G_{\varepsilon}(x,\snr{Du}) - \kk)_+}{g_{1, \eps}(x,\snr{Du})} + \snr{Du}^2 \right]\snr{f}^2   \dx\,\\
   &&  \quad \leq \bar \eps\mathcal{S}_3  +\bar \eps\mathcal{S}_4 + \frac{c}{\bar \eps r^2}\int_{\B^{\kk}} (G_{\varepsilon}(x,\snr{Du}) - \kk)_+^2  \dx\notag \\
   && \qquad \qquad +\frac{c}{\bar{\eps}}\nr{Du}_{L^{\infty}(B_{r_{0}}(x_{0}))}^{2} \int_{\B^{\kk}} \eta^2\snr{f}^2 \dx\,,\label{trattati}
\end{eqnarray}
for $c\equiv c(\data)$ and arbitrary $\bar \varepsilon \in (0,1)$. Collecting the estimates in the last three  displays to \rif{incon}, recalling that 
recalling that $\bar{G}_{\varepsilon}(\cdot)\geq 1$, and selecting $\bar \varepsilon>0$ sufficiently small in order to reabsorb terms, we complete the proof of Lemma \ref{caclem}. 
\end{proof}
\subsection{Nonlinear iterations}\label{nonlineari} In this section we finally derive pointwise gradient bounds. 
This goes via Lemma \ref{stimapp} and Proposition \ref{linf1} below. 
\begin{lemma}\label{stimapp}
Let $u\in W^{1,\gamma}(\B;\RN)$ be a solution to $\eqref{a0}$, under assumptions $\setm$ in \trif{ipotesi}  for $n>2$. If $B_{2r_0}(x_{0})\Subset \B$ is a ball such that $x_0$ is a Lebesgue point of $|Du|$, then 
\begin{eqnarray}\label{a46}
&&\nonumber G_{\varepsilon}(x_{0},\snr{Du(x_{0})})\\
&& \quad \le \kk+c\nr{\bar{G}_{\varepsilon}(\cdot,\snr{Du})}_{L^{\infty}(B_{r}(x_{0}))}^{n\vartheta\sigma/4}\left(\mint_{B_{r_{0}}(x_{0})}(G_{\varepsilon}(x,\snr{Du})-\kk)_{+}^{2} \dx\right)^{1/2}\nonumber \\
&&\quad\ \ \ \   +c\nr{\bar{G}_{\varepsilon}(\cdot,\snr{Du})}_{L^{\infty}(B_{r}(x_{0}))}^{(n-2)\vartheta\sigma/4}\notag \\
&& \quad \qquad \ \times \left[\mathbf{P}_{1}^{\mathfrak{h}}(x_{0},2r_{0})+\nr{Du}_{L^{\infty}(B_{r_{0}}(x_{0}))}\mathbf{P}_{1}^{f}(x_{0},2r_{0})\right]\,,
\end{eqnarray}holds for all $\kk\ge 0$, with $c\equiv c(\datai)$, where $\vartheta\equiv \vartheta(\gamma)$ is as in \eqref{sigma},  
and
%%%typoc
\begin{flalign}\label{fk0}
\mathfrak{h}(x):=h(x)[\bar G_{\varepsilon}(x,\snr{Du})]^{1+\sigma+\hat{\sigma}}\,.
\end{flalign}
\end{lemma}
\begin{proof} Note that we can assume that $|Du(x_0)|>T$, otherwise \rif{a46} is trivial by the very definition of $G_{\varepsilon}$; this obviously implies $ \|Du\|_{L^\infty(B_r(x_0))}>T$. Let us first note that $x_0$ is also a Lebesgue point of $x \mapsto G_{\varepsilon}(x,\snr{Du(x)})$ and it is
\eqn{cc0}
$$
\lim_{r \to 0} \mint_{B_r(x_0)} G_{\varepsilon}(x,\snr{Du(x)}) \, dx = G_{\varepsilon}(x_0,\snr{Du(x_0)})\,,
$$
i.e., the right-hand side denotes the precise representative of $G_{\varepsilon}(\cdot,\snr{Du(\cdot)})$ at the point $x_0$. Indeed, note that 
\begin{eqnarray}
\nonumber & &\limsup_{r \to 0} \mint_{B_r(x_0)} |G_{\varepsilon}(x,\snr{Du(x)})-G_{\varepsilon}(x_0,\snr{Du(x_0)})| \, dx \\ && \qquad \leq 
 \limsup_{r \to 0} \mint_{B_r(x_0)} |G_{\varepsilon}(x,\snr{Du(x)})-G_{\varepsilon}(x_0,\snr{Du(x)})| \, dx \nonumber \\ 
&& \qquad \qquad + \limsup_{r \to 0} \mint_{B_r(x_0)} |G_{\varepsilon}(x_0,\snr{Du(x)})-G_{\varepsilon}(x_0,\snr{Du(x_0)})| \, dx\nonumber \\
&&\qquad \qquad \qquad  =: \limsup_{r\to 0}\mathcal C_{1}(r) +\limsup_{r\to 0} \mathcal C_2(r)\,. \label{cc1}
\end{eqnarray}
As $x_0$ is a Lebesgue point for $Du$, $t \mapsto G_{\varepsilon}(x_0, t)$ is locally Lipschitz-regular, and $Du$ is locally bounded, we have 
\eqn{cc2}
$$    \lim_{r \to 0} \, \mathcal C_2(r) \lesssim \lim_{r \to 0} \mint_{B_r(x_0)} |\snr{Du(x)}-\snr{Du(x_0)}| \, dx=0\,.
$$
As for the term $\mathcal{C}_{1}(\cdot) $, we have, also using Fubini's
\begin{flalign*}
& \nonumber \mathcal C_{1}(r)  \leq  \mint_{B_r(x_0)} \int_T^{\max\{|Du(x)|, T\}}|g_{1, \eps}(x,s)-g_{1, \eps}(x_0,s)|s \, ds \, dx\\
\nonumber&\quad \leq  \int_T^{\|Du\|_{L^\infty(B_r(x_0))}}  \mint_{B_r(x_0)} |g_{1, \eps}(x,s)-g_{1, \eps}(x_0,s)| \, dx \, s \, ds\\
\nonumber & \quad \leq  \|Du\|_{L^\infty(B_r(x_0))}^2 \sup_{s\in [T, \|Du\|_{L^\infty(B_r(x_0))}]} \mint_{B_r(x_0)} |g_{1, \eps}(x,s)-g_{1, \eps}(x_0,s)| \, dx.
\end{flalign*}
Recall that $g_1(\cdot)$ is assumed to be continuous on $\Omega\times (0, \infty)$. By the definition in \rif{g1e}, this implies that also $g_{1, \eps}$ is continuous and therefore it is uniformly continuous on $\overline{B_{r_0}(x_0)}\times [T, \|Du\|_{L^\infty(B_r(x_0))}]$. This and the content of the last display it is sufficient to infer that $\mathcal C_{1}(r) \to 0$ as $r\to 0$. This fact, together with \rif{cc1}-\rif{cc2}, yields \rif{cc0}. 
Thanks to \rif{a45} we can verify \eqref{varine} with the choices $v(\cdot)\equiv G_{\varepsilon}(\cdot,\snr{Du(\cdot)})$, 
$f_{1}\equiv \mathfrak{h}$, $f_{2}\equiv f$, $M_{1}\equiv\nr{\bar{G}_{\varepsilon}(\cdot,\snr{Du})}_{L^{\infty}(B_{r}(x_{0}))}^{\vartheta\sigma/2}$, $M_{2}\equiv 1$ and $M_{3}\equiv \nr{Du}_{L^{\infty}(B_{r_{0}}(x_{0}))}$. Applying Lemma \ref{noniter}, inequality \rif{var8} yields \eqref{a46}. 
\end{proof}
%For the proof of the next result, we need some more notation. With $ B \subset \B$ being a ball, we recall  \rif{assinf1} and, as in \eqref{gte}, we define 
%\begin{flalign}\label{igte}
%G_{\varepsilon,B_{\rr}}(t):=\int_{T}^{\max\{t,T\}}g_{1,\varepsilon}(x_{B_\rr},s)s \d s%\quad  \bar{G}_{\varepsilon,B_{\rr} }(t):=G_{\varepsilon,B_{\rr}}(t)+\go
%\end{flalign}
%for all $t\in [0,\infty)$; then \eqref{z8} and the above definition imply 
%\begin{flalign}\label{bgi}
%\begin{cases}
%\  G_{\varepsilon,r}(t)+1 \le \bar{G}_{\varepsilon,B_{\rr}}(t) \le \go\left[G_{\varepsilon,B_{\rr}}(t)+1\right]\quad &\mbox{for all} \ \ t\in [0,\infty)\\
%\displaystyle 
%\ [E_{\varepsilon}(t)]^{\gamma}-[E_{\varepsilon}(T)]^{\gamma}\le c(\nu, \gamma)G_{\varepsilon,B_\rr}(t)\quad &\mbox{for all} \ \ t\in [T,\infty)\\
%\ \inf_{x\in B_{\rr}}\,  G_{\eps}(x,t) \leq G_{\varepsilon,B_{\rr}}(t)\,.
%\end{cases}
%\end{flalign}
\begin{proposition}\label{linf1}
Let $u\in W^{1,\gamma}(\B;\RN)$ be a solution to $\eqref{a0}$, under assumptions $\setm$ in \trif{ipotesi} for $n>2$. There exists a positive radius $R_{*}\equiv R_{*}(\data,\mathfrak f(\cdot))\leq 1$ such that if $\rad(\B)\le R_{*}$ and $B_{\varsigma} \Subset  B_{\rr}$ are concentric balls contained in $\B$, then
\begin{flalign}\label{a52}
& [E_{\varepsilon}(\nr{Du}_{L^{\infty}(B_{\varsigma})})]^{\gamma}+\nr{G_{\varepsilon}(\cdot,\snr{Du})}_{L^{\infty}(B_{\varsigma})}\nonumber \\
 &\qquad \leq \frac{c}{(\rr-\varsigma)^{ \beta }}
\left[\|F_{\varepsilon}(\cdot,Du)\|_{L^1(B_{\rr})}^{\theta}+\nr{f}_{L(n,1)(B_{\rr})}^{\theta}+1\right]
\end{flalign}
holds with  
$c\equiv c( \data,\nr{h}_{L^{d}(\B)})\geq 1$ and $ \beta ,\theta \equiv \beta ,\theta (n,d,\gamma,\sigma,\hat{\sigma})>  0$.
\end{proposition}
\begin{proof} We take numbers $\tau, p\equiv \tau, p(n,d, \sigma,\hat{\sigma})$ such that
\begin{flalign}\label{ppp}
0< \tau < d-n \quad \mbox{and} \quad p:=\frac{(1+\sigma+\hat{\sigma})(n+\tau)d}{d-n-\tau}\,,
\end{flalign}
where $d>n$ is the exponent from \rif{sigma}, and fix $R_{*}\equiv R_{*}( \data,\mathfrak{f}(\cdot))>0$ as the radius from Proposition \ref{lp}, so that \eqref{a44} holds such $p$. With $B_{\varsigma} \Subset B_{\rr}$ being the balls considered in the statement, with no loss of generality we can assume that $\nr{Du}_{L^{\infty}(B_{\varsigma})} \geq T$
%\eqn{assTT}
%$$\nr{Du}_{L^{\infty}(B_{\varsigma})} \geq T\,,$$ 
otherwise  the assertion in \rif{a52} is trivial by choosing $c$ large enough. Let $\varrho_1:= \varsigma + {(\varrho-\varsigma)/2}$ and  consider concentric balls $B_{\varsigma}\Subset B_{\tau_{1}}\Subset B_{\tau_{2}}\Subset B_{\rr_1}\Subset B_{\rr}$, a point $x_{0}\in B_{\tau_{1}}$ which is a Lebesgue point for $|Du|$, and $r_{0}:=(\tau_{2}-\tau_{1})/8$, so that  $B_{2r_{0}}(x_{0})\Subset B_{\tau_{2}}$. Needless to say, a.e. point in $B_{\tau_1}$ qualifies. By \rif{bg}
we find
%%%typoc2 costante 
\begin{flalign}\label{sos0}
&\nr{\bar{G}_{\varepsilon}(\cdot,\snr{Du})}_{L^{\infty}(B_{r_{0}}(x_{0}))}\le (T^2+1)^{\gamma/2}\left[\nr{G_{\varepsilon}(\cdot,\snr{Du})}_{L^{\infty}(B_{\tau_2})}+1\right].
\end{flalign}
This and \rif{z10a}$_1$ imply
\begin{flalign}\label{sos0.1}
\nr{Du}_{L^{\infty}(B_{\tau_2})} \le c\nr{G_{\varepsilon}(\cdot,\snr{Du})}_{L^{\infty}(B_{\tau_2})}^{1/\gamma}+c\,, 
\end{flalign}
for $c\equiv c(\nu,\gamma,T)$. We then apply \rif{a46} 
on $B_{r_0}(x_0)$, 
and also using \eqref{sos0}-\eqref{sos0.1} and H\"older inequality (by \rif{ppp} it is $p\geq 2$), we obtain
%%%typoc shortening
\begin{flalign}\label{a48}
\notag & G_{\varepsilon}(x_{0},\snr{Du(x_{0})}) \\
& \le cr_{0}^{-n/p}\left[\nr{G_{\varepsilon}(\cdot,\snr{Du})}_{L^{\infty}(B_{\tau_2})}^{n\vartheta\sigma/4}+1\right]\|G_{\varepsilon}(\cdot,\snr{Du})\|_{L^p(B_{r_{0}}(x_0))}\nonumber \\
&\qquad \quad +c\left[\nr{G_{\varepsilon}(\cdot,\snr{Du})}_{L^{\infty}(B_{\tau_2})}^{(n-2)\vartheta\sigma/4}+1\right]\notag \\
& \qquad \qquad \quad \times \left[\mathbf{P}_{1}^{\mathfrak{h}}(x_{0},2r_0)+\nr{Du}_{L^{\infty}(B_{\tau_2})}\mathbf{P}_{1}^{f}(x_{0},2r_0)\right],
\end{flalign}
with $c\equiv c(\data)$. With \eqref{a44}  (where we take $\xi \equiv r_0$ and $\varrho \equiv 2r_0$) we further bound
$$
\|G_{\varepsilon}(\cdot,\snr{Du})\|_{L^p(B_{r_{0}}(x_0))}\le  cr_0^{-\beta_{p}}\left[\|F_{\varepsilon}(\cdot,Du)\|_{L^1(B_{\rr})}^{\theta_{p}}+1\right],
$$
for $c\equiv c( \data,\nr{h}_{L^{d}(\B)})$. Using also \eqref{sos0.1} in \eqref{a48}, recalling that $x_{0}\in B_{\tau_{1}}$ is a arbitrary Lebesgue point for $|Du|$, we obtain
\begin{flalign}\label{a49}
&\nr{G_{\varepsilon}(\cdot,\snr{Du})}_{L^{\infty}(B_{\tau_{1}})}\notag \\ &\le \frac{c}{(\tau_{2}-\tau_{1})^{n/p+\beta_{p}}}\nr{G_{\varepsilon}(\cdot,\snr{Du})}_{L^{\infty}(B_{\tau_{2}})}^{n\vartheta\sigma/4}\left[\nr{F_{\varepsilon}(\cdot,Du)}_{L^{1}(B_{\rr})}^{\theta_{p}}+1\right]\nonumber \\
&\quad +c\nr{G_{\varepsilon}(\cdot,\snr{Du})}_{L^{\infty}(B_{\tau_{2}})}^{(n-2)\vartheta\sigma/4}\| \mathbf{P}_{1}^{\mathfrak{h}}(\cdot,(\tau_{2}-\tau_{1})/4)\|_{L^{\infty}(B_{\tau_{1}})}\nonumber \\
&\quad +c\nr{G_{\varepsilon}(\cdot,\snr{Du})}_{L^{\infty}(B_{\tau_{2}})}^{(n-2)\vartheta\sigma/4+1/\gamma}\| \mathbf{P}_{1}^{f}(\cdot,(\tau_{2}-\tau_{1})/4)  \|_{L^{\infty}(B_{\tau_{1}})}\nonumber \\
&\quad +\frac{c}{(\tau_{2}-\tau_{1})^{n/p+\beta_{p}}}\left[\nr{F_{\varepsilon}(\cdot,Du)}_{L^{1}(B_{\rr})}^{\theta_{p}}+1\right]\nonumber \\
&\quad +c  \| \mathbf{P}_{1}^{\mathfrak{h}}(\cdot,(\tau_{2}-\tau_{1})/4)  \|_{L^{\infty}(B_{\tau_{1}})}+ c \| \mathbf{P}_{1}^{f}(\cdot,(\tau_{2}-\tau_{1})/4)  \|_{L^{\infty}(B_{\tau_{1}})}\,,
\end{flalign}
where $c\equiv c(\datai,\nr{h}_{L^{d}(\B)})$. Now, observe that \eqref{sigma} implies 
\eqn{ledue}
$$
n\sigma\vartheta/4<1  \quad \mbox{and} \quad (n-2)\sigma\vartheta/4<1-1/\gamma\,,$$ 
therefore we can apply Young inequality in \rif{a49}
to end up with
\begin{flalign}\label{a50}
\notag \nr{G_{\varepsilon}(\cdot,\snr{Du})}_{L^{\infty}(B_{\tau_{1}})}& \le \frac{1}{2}\nr{G_{\varepsilon}(\cdot,\snr{Du})}_{L^{\infty}(B_{\tau_{2}})}\\
 & \quad +\frac{c}{(\tau_{2}-\tau_{1})^{\beta_{*}}}\left[\nr{F_{\varepsilon}(\cdot,Du)}_{L^{1}(B_{\rr})}^{\theta_{*}}+1\right]\nonumber \\
&\quad +c \| \mathbf{P}_{1}^{\mathfrak{h}}(\cdot,(\tau_{2}-\tau_{1})/4) \|_{L^{\infty}(B_{\tau_{1}})}^{\theta_{*}}\notag \\ 
&\quad +c\| \mathbf{P}_{1}^{f}(\cdot,(\tau_{2}-\tau_{1})/4) \|_{L^{\infty}(B_{\tau_{1}})}^{\theta_{*}}+c\,,
\end{flalign}
where $c\equiv c(\datai,\nr{h}_{L^{d}(\B)})$ where it is $ \beta_{*} , \theta_{*} \equiv \beta_{*} , \theta_{*} (n,d,\gamma,\sigma,\hat{\sigma})>  0$. 
Inequality \eqref{a50} allows to apply Lemma \ref{l0l} with the obvious choice $\mathcal{Z}(t):= \nr{G_{\varepsilon}(\cdot,\snr{Du})}_{L^{\infty}(B_{t})}$, and this leads to
\begin{flalign}\label{a51}
\nr{G_{\varepsilon}(\cdot,\snr{Du})}_{L^{\infty}(B_{\varsigma})}&\, \le\frac{c}{(\rr-\varsigma)^{\beta_{*}}}\left[\nr{F_{\varepsilon}(\cdot,Du)}_{L^{1}(B_{\rr})}^{\theta_{*}}+1\right]\nonumber \\
&\quad +c\| \mathbf{P}_{1}^{\mathfrak{h}}(\cdot,(\rr-\varsigma)/4) \|_{L^{\infty}(B_{\rr_{1}})}^{\theta_{*}}\notag \\
& \quad +c\| \mathbf{P}_{1}^{f}(\cdot,(\rr-\varsigma)/4) \|_{L^{\infty}(B_{\rr_{1}})}^{\theta_{*}}+c
\end{flalign}
for $c\equiv c(\datai,\nr{h}_{L^{d}(\B)})$. By using \rif{potlore}$_1$ we infer
$$
\| \mathbf{P}_{1}^{f}(\cdot,(\rr-\varsigma)/4)\|_{L^{\infty}(B_{\rr_1})}\le c\nr{f}_{L(n,1)(B_{\rr})}\,.
$$
Moreover, with $\varrho_2:=\varrho_1+(\rr-\varsigma)/4$; also using \rif{ppp} and H\"older inequality yields (recall that $ s_*=1$)
\begin{eqnarray}
\notag \hspace{-.5mm} \|\mathbf{P}_{1}^{\mathfrak{h}}(\cdot,(\rr-\varsigma)/4)\|_{L^{\infty}(B_{\rr_1})}&\stackrel{\eqref{potlore}_1}{\leq} & c\nr{\mathfrak{h}}_{L(n,1)(B_{\rr_2})}\nonumber\\ \nonumber & \stackleq{normine} &c(\tau)\nr{\mathfrak{h}}_{L^{n+\tau}(B_{\rr_2})}\\ & \stackleq{fk0} & c\nr{h}_{L^{d}(B_{\rr})}\nr{\bar{G}_{\varepsilon}(\cdot,\snr{Du})}_{L^{p}(B_{\rr_2})}^{1+\sigma+\hat{\sigma}}\label{servepure} \\
& \stackrel{\rif{a44}}{\leq} & \frac{c\nr{h}_{L^{d}(B_{\rr})}\left[\nr{F_{\varepsilon}(\cdot,Du)}_{L^{1}(B_{\rr})}^{\theta_p(1+\sigma+\hat{\sigma})}+1\right]}{(\rr-\varsigma)^{\beta_{p}(1+\sigma+\hat{\sigma})}}\notag 
\end{eqnarray}
%%%typoc2 valori esponenti
where $c\equiv c(\datai, \nr{h}_{L^{d}(\B)})$. We have applied Proposition \ref{lp} with the choice $s_*=1$, so that \rif{sigmasss} is verified by the assumption \rif{sigma}$_1$. Inserting the above two estimates in \eqref{a51} and recalling also \eqref{gte} and \eqref{z10a}$_1$, we finally end up with \eqref{a52}, where $\beta := \max\{\beta_{*}, \beta_{p}\theta_{*}(1+\sigma+\hat{\sigma})\}$ and $\theta :=\max \{\theta_{*}, \theta_p\theta_{*}(1+\sigma+\hat \sigma)\}.$ 
\end{proof}
\subsection{The case $n=2$}\label{casetwo}
Here we consider the missing two-dimensional case. We start with the following lemma, which is a hybrid counterpart of Proposition \ref{lp}, in the sense that the a priori estimate involved still contains the $L^{\infty}$-norm of $Du$ in the right-hand side. We recall that the number $m$ has been defined in \rif{ilchi}. 
\begin{lemma}\label{lemmadieci} Let $u\in W^{1,\gamma}(\B;\RN)$ be a solution to $\eqref{a0}$ under assumptions $\setm$ in \trif{ipotesi} for $n=2$, where we replace \trif{sigma}$_2$ by
\eqn{uppersss}
$$\sigma + \hat \sigma < \frac {s_*}2-\frac {s_*}d = \frac{s_*}{2m}, \ \  \mbox{for some $s_*$ such that}\ 1\leq s_* <2m(1+\sigma+\hat{\sigma})\,.
$$ If $B_{\varsigma} \Subset  B_{\rr}$ are concentric balls contained in $\B$, then, for every $p\geq 1$, there holds
\begin{flalign}
\notag  \nr{G_{\varepsilon}(\cdot,\snr{Du})}_{L^{p}(B_{\varsigma})} &\le \frac{c}{(\rr-\varsigma)^{\beta_{p}}}\left[\nr{G_{\varepsilon}(\cdot,\snr{Du})}_{L^{s_*}(B_{\rr})}^{\theta_{p}}+1\right]\\
 & \quad +c\nr{Du}_{L^{\infty}(B_{\rr})}\|f\|_{L^2(B_{\varrho})}\label{2.2}
\end{flalign}
with $c\equiv c(\data,\nr{h}_{L^{d}(\B)},p)\geq 1$, $\beta_{p}, \theta_{p}\equiv \beta_{p}, \theta_{p}(d,\sigma,\hat{\sigma},p)>0$.
\end{lemma}
\begin{proof} We can confine ourselves to prove \rif{2.2} for sufficiently large $p$, and we consider
\begin{flalign}\label{tip}
p>\max\left\{2m(1+\sigma+\hat{\sigma}),\frac{2ms_*}{s_*-2m(\sigma+\hat{\sigma})}\right\}=\frac{2ms_*}{s_*-2m(\sigma+\hat{\sigma})}\,.
\end{flalign}
The last equality comes from the second inequality in \rif{uppersss}. 
Note that such a choice is possible thanks to \rif{uppersss}, making the denominator of the last quantity different than zero.  %and, in turn, this is implied by \rif{sigma} for $n=2$. 
In the following lines all the balls will be concentric to $B_{\varrho}$. We look back at the proof of Proposition \ref{lp}, take $\kk=0$ to obtain the test function $\varphi_s=\eta^{2}G_{\varepsilon}(x,\snr{Du})D_{s}u$ for $s \in \{1, \ldots, n\}$, and perform exactly the same calculations made there up to \rif{finoa}. For the terms $\mbox{(I)}_{f}$-$\mbox{(III)}_{f}$ involving the right-hand side $f$, we note that, as $\kk =0$, the test maps $\varphi_s$ used in the proof of Propositions \ref{lp} and Lemma \ref{caclem} do coincide (actually, we take $\kk=0$ in both Propositions \ref{lp} and Lemma \ref{caclem}). Therefore we can use estimate \rif{trattati}, 
where $c\equiv c(\data)$ and $\bar \varepsilon \in (0,1)$; here $\mathcal S_3$ and $\mathcal S_4$ have been defined in \rif{posi3} and \rif{posi2}, respectively. All together, choosing $\bar \varepsilon>0$ small enough and re-absorbing terms in a standard way, we obtain
\begin{flalign}\label{2.0}
\notag \mathcal{S}_3+ \mathcal{S}_4& \le c (\nr{h}_{L^{d}(\B)}^{2}+1)\\
\notag &\qquad  \times \left(\int_{\B}(\eta^{2m}+\snr{D\eta}^{2m})\left[[G_{\varepsilon}(x,\snr{Du})]^{2m(1+\sigma+\hat{\sigma})}+1\right] \dx\right)^{1/m}\nonumber \\
&\quad+c\nr{Du}_{L^{\infty}(B_{\rr})}^{2}\int_{\B}\eta^{2}\snr{f}^{2} \dx\,,
\end{flalign}for $c\equiv c(\data)$. As it is $|D\eta|\lesssim 1/(\tau_2-\tau_1)$, elementary manipulations on \eqref{2.0} yield
\begin{flalign*}
\nr{D(\eta G_{\varepsilon}(\cdot,\snr{Du}))}_{L^{2}(B_{\tau_{2}})}^{2}& \le \frac{c(\nr{h}_{L^{d}(\B)}+1)^{2}}{(\tau_{2}-\tau_{1})^{2}}\nr{G_{\varepsilon}(\cdot,\snr{Du})}_{L^{2m(1+\sigma+\hat{\sigma})}(B_{\tau_{2}})}^{2(1+\sigma+\hat{\sigma})}\nonumber\\ &\quad+\frac{c(\nr{h}_{L^{d}(\B)}+1)^{2}}{(\tau_{2}-\tau_{1})^{2}}
+c\nr{Du}_{L^{\infty}(B_{\rr})}^{2}\int_{B_{\tau_{2}}}\snr{f}^{2} \dx
\end{flalign*}
so that Sobolev embedding gives
\begin{flalign}\label{2.1}
\nr{G_{\varepsilon}(\cdot,\snr{Du})}_{L^{p}(B_{\tau_{1}})}&\le c_{p}\tau_{2}^{2/p}\nr{D(\eta G_{\varepsilon}(\cdot,\snr{Du}))}_{L^{2}(B_{\tau_{2}})}\nonumber \\
&\notag \le \frac{c}{(\tau_{2}-\tau_{1})}\nr{G_{\varepsilon}(\cdot,\snr{Du})}_{L^{2m(1+\sigma+\hat{\sigma})}(B_{\tau_{2}})}^{1+\sigma+\hat{\sigma}}\\
&  \quad +\frac{c}{(\tau_{2}-\tau_{1})} +c\nr{Du}_{L^{\infty}(B_{\rr})}\|f\|_{L^2(B_{\tau_{2}})}\,,
\end{flalign}
with $c\equiv c(\data, \|h\|_{L^d(\B)})$. With $\lambda_{p}\in (0,1)$ being defined through 
\begin{flalign*}
\frac{1}{2m(1+\sigma+\hat{\sigma})}=\frac{1-\lambda_{p}}{s_*}+\frac{\lambda_{p}}{p}\Rightarrow \lambda_{p}=\frac{p[2m(1+\sigma+\hat{\sigma})-s_*]}{2m(p-s_*)(1+\sigma+\hat{\sigma})}\,,
\end{flalign*}
using the interpolation inequality
\begin{flalign*}
\nr{G_{\varepsilon}(\cdot,\snr{Du})}_{L^{2m(1+\sigma+\hat{\sigma})}(B_{\tau_{2}})}\le \nr{G_{\varepsilon}(\cdot,\snr{Du})}_{L^{p}(B_{\tau_{2}})}^{\lambda_{p}} \nr{G_{\varepsilon}(\cdot,\snr{Du})}_{L^{s_*}(B_{\tau_{2}})}^{1-\lambda_{p}}
\end{flalign*}
in \rif{2.1}, we get
\begin{flalign*}
& \nr{G_{\varepsilon}(\cdot,\snr{Du})}_{L^{p}(B_{\tau_{1}})} \\
& \qquad \le\frac{c}{(\tau_{2}-\tau_{1})}\nr{G_{\varepsilon}(\cdot,\snr{Du})}_{L^{p}(B_{\tau_{2}})}^{\lambda_{p}(1+\sigma+\hat{\sigma})}\nr{G_{\varepsilon}(\cdot,\snr{Du})}_{L^{s_*}(B_{\tau_{2}})}^{(1-\lambda_{p})(1+\sigma+\hat{\sigma})}
\nonumber \\
&\qquad\quad   +\frac{c}{(\tau_{2}-\tau_{1})}\nonumber+c\nr{Du}_{L^{\infty}(B_{\rr})}\|f\|_{L^2(B_{\tau_{2}})}\,.
\end{flalign*}
Using \rif{uppersss} and \eqref{tip}, we see that $\lambda_{p}(1+\sigma+\hat{\sigma})<1$, thus Young inequality 
gives
\begin{flalign*}
& \nr{G_{\varepsilon}(\cdot,\snr{Du})}_{L^{p}(B_{\tau_{1}})} \le \frac{1}{2}\nr{G_{\varepsilon}(\cdot,\snr{Du})}_{L^{p}(B_{\tau_{2}})}\nonumber \\
&\qquad +\frac{c}{(\tau_{2}-\tau_{1})^{\beta_{p}}}\left[\nr{G_{\varepsilon}(\cdot,\snr{Du})}_{L^{s_*}(B_{\tau_{2}})}^{\theta_{p}}+1\right]+c\nr{Du}_{L^{\infty}(B_{\rr})}\|f\|_{L^2(B_{\tau_{2}})}
\end{flalign*}
with $c\equiv c(\data,\nr{h}_{L^{d}(\B)})$ and 
\eqn{ilbeta}
$$\beta_{p}:=\frac{2m(p-s_*)}{s_*p-2m[s_*+p(\sigma+\hat{\sigma})]} \ \  \mbox{and}\ \  \theta_{p}:=\frac{s_*p-2ms_*(1+\sigma+\hat{\sigma})}{s_*p-2m[s_*+p(\sigma+\hat{\sigma})]}\,.$$ Lemma \ref{l0l} with the choice $\mathcal{Z}(t)\equiv  \nr{G_{\varepsilon}(\cdot,\snr{Du})}_{L^{p}(B_{t})}$ now gives \rif{2.2}. \end{proof}
We finally come to the a priori gradient bound in the two dimensional case. 
\begin{proposition}\label{duedi}
Let $u\in W^{1,\gamma}(\B;\RN)$ be a solution to $\eqref{a0}$ under assumptions $\setm$ in \trif{ipotesi} for $n=2$. If $B_{\varsigma} \Subset  B_{\rr}$ are concentric balls contained in $\B$, then
\begin{flalign}\label{a52n2}
& [E_{\varepsilon}(\nr{Du}_{L^{\infty}(B_{\varsigma})})]^{\gamma}+\nr{G_{\varepsilon}(\cdot,\snr{Du})}_{L^{\infty}(B_{\varsigma})}\nonumber \\
& \qquad \leq  \frac{c}{(\rr-\varsigma)^{ \beta }}\left[\|F_{\varepsilon}(\cdot,Du)\|_{L^1(B_{\rr})}^{\theta}+\nr{f}_{L^{2}(\log L)^{\mathfrak {a}}(B_{\rr})}^{\theta}+1\right]\,,
\end{flalign}
holds with $c\equiv c(\data,\nr{h}_{L^{d}(\B)})\geq 1$, $\beta,\theta\equiv \beta,\theta(d,\gamma, \sigma,\hat{\sigma})>0$.
\end{proposition}
\begin{proof} We proceed as for the proof of Proposition \ref{linf1}, keeping the notation used there, including \rif{ppp}. By Lemma \ref{caclem} we use Lemma \ref{noniter} with the choice made at the end of the proof of Lemma \ref{stimapp}, and this for $n=2$ gives 
\begin{flalign}
& \notag G_{\varepsilon}(x_{0},\snr{Du(x_{0})}) \\
& \quad \le cr_0^{-2/p}\left[\nr{G_{\varepsilon}(\cdot,\snr{Du})}_{L^{\infty}(B_{\tau_2})}^{(1+\delta)\vartheta\sigma/2}+1\right]\|G_{\varepsilon}(\cdot,\snr{Du})\|_{L^p(B_{r_{0}}(x_0))}\nonumber \\
& \qquad \quad +c\left[\nr{G_{\varepsilon}(\cdot,\snr{Du})}_{L^{\infty}(B_{\tau_2})}^{\delta\vartheta\sigma/2}+1\right]\notag \\
& \qquad\qquad  \quad \times \left[\mathbf{P}_{1}^{\mathfrak{h}}(\cdot,2r_{0})+\nr{Du}_{L^{\infty}(B_{\tau_2})}\mathbf{P}_{1}^{f}(\cdot,2r_{0})\right],\label{arrivaad}
\end{flalign}
for every $\delta \in (0,1/2)$, where $c\equiv c(\data,\nr{h}_{L^{d}(\B)},\delta)$ and $\vartheta$ is as in \rif{sigma}. 
Then \eqref{2.2} gives
%\eqn{2.5}
\begin{flalign*}
&\|G_{\varepsilon}(\cdot,\snr{Du})\|_{L^p(B_{r_{0}}(x_0))} \\
& \quad \le cr_0^{-\beta_{p}}\left[\nr{G_{\varepsilon}(\cdot,\snr{Du})}_{L^{1}(B_{2r_0})}^{\theta_{p}}+1\right] + c\nr{Du}_{L^{\infty}(B_{\tau_2})}\|f\|_{L^2(B_{2r_0})}
\end{flalign*}
where $p$ is as in \eqref{ppp} with $n=2$, $c\equiv c(\data,\nr{h}_{L^{d}(\B)},p)$ and $\beta_{p}, \theta_{p}$ are as in \rif{ilbeta}. Note that here we are using Lemma \ref{lemmadieci} with the choice $s_*=1$, which is allowed as the assumption gives that $\sigma + \hat \sigma < 1/2 - 1/d$, which is \rif{sigma} for $n=2$. Combining the last two estimates, that hold for a.e.\,$x_0\in B_{\tau_1}$, we have
\begin{eqnarray}
&&\notag \nr{G_{\varepsilon}(\cdot,\snr{Du})}_{L^{\infty}(B_{\tau_{1}})}\\
&& \ \ \leq \frac{c}{(\tau_{2}-\tau_{1})^{2/p+\beta_{p}}}\left[\nr{G_{\varepsilon}(\cdot,\snr{Du})}_{L^{\infty}(B_{\tau_2})}^{(1+\delta)\vartheta\sigma/2}+1\right]\left[\nr{G_{\varepsilon}(\cdot,\snr{Du})}_{L^{1}(B_{\rr})}^{\theta_{p}}+1\right]\nonumber \\
&& \ \ \qquad  + \frac{c}{(\tau_{2}-\tau_{1})^{2/p}}\left[\nr{G_{\varepsilon}(\cdot,\snr{Du})}_{L^{\infty}(B_{\tau_2})}^{(1+\delta)\vartheta\sigma/2}+1\right]\nr{Du}_{L^{\infty}(B_{\tau_2})}\|f\|_{L^2(B_{\varrho})} \nonumber \\
&&\ \ \qquad+c\left[\nr{G_{\varepsilon}(\cdot,\snr{Du})}_{L^{\infty}(B_{\tau_2})}^{\delta\vartheta\sigma/2}+1\right]\notag \\
&& \hspace{15mm} \times  \nr{Du}_{L^{\infty}(B_{\tau_2})}\|\mathbf{P}_{1}^{f}(\cdot,(\tau_{2}-\tau_{1})/4)\|_{L^{\infty}(B_{\tau_1})}\notag \\
&&\ \ \qquad +c\left[\nr{G_{\varepsilon}(\cdot,\snr{Du})}_{L^{\infty}(B_{\tau_2})}^{\delta\vartheta\sigma/2}+1\right]\|\mathbf{P}_{1}^{\mathfrak{h}}(\cdot,(\tau_{2}-\tau_{1})/4)\|_{L^{\infty}(B_{\tau_{1}})}\notag\\
&&\ \  =: T_1+T_2+T_3+T_4 \label{2.4}\,,
\end{eqnarray}for $c\equiv c(\data,\nr{h}_{L^{d}(\B)},\delta)$. To estimate the $T$-terms we take $\delta$ such that
\eqn{sceglidelta}
$$
(1+\delta)\vartheta\sigma/2 +1/\gamma < 1\quad \mbox{and} \quad \delta\vartheta\sigma /2+(1+\sigma+\hat{\sigma})/\gamma <1
$$
hold, which is in turn possible by \rif{sigma}$_2$; this fixes $\delta$ as a function of $\sigma, \tilde \sigma, \gamma$. We then have
\begin{eqnarray*}
T_1 &
\stackrel{\eqref{sceglidelta}}{\leq } & \frac 18\, \nr{G_{\varepsilon}(\cdot,\snr{Du})}_{L^{\infty}(B_{\tau_2})}+\frac{c\nr{G_{\varepsilon}(\cdot,\snr{Du})}_{L^{1}(B_{\rr})}^{\theta}+c}{(\tau_{2}-\tau_{1})^{ \beta}}\,,\\
T_2
&\stackrel{\eqref{sos0.1}}{\leq } &\frac{c}{(\tau_{2}-\tau_{1})^{2/p}} [\nr{G_{\varepsilon}(\cdot,\snr{Du})}_{L^{\infty}(B_{\tau_2})}+1]^{(1+\delta)\vartheta\sigma/2 +1/\gamma}\|f\|_{L^2(B_{\varrho})} \nonumber \\
&\stackrel{\eqref{sceglidelta}}{\leq } & \frac 18\, \nr{G_{\varepsilon}(\cdot,\snr{Du})}_{L^{\infty}(B_{\tau_2})}+\frac{c\|f\|_{L^{2}(\log L)^{\mathfrak {a}}(B_{\varrho})}^{\theta}+c}{(\tau_{2}-\tau_{1})^{\beta}}\,,\\
T_3 & \stackrel{ \eqref{potlore}_2, \eqref{sos0.1}}{\leq }  &c[\nr{G_{\varepsilon}(\cdot,\snr{Du})}_{L^{\infty}(B_{\tau_2})}+1]^{\delta\vartheta\sigma/2+1/\gamma}\|f\|_{L^{2}(\log L)^{\mathfrak {a}}(B_{\varrho})}\\
&\stackrel{\eqref{sceglidelta}}{\leq } & \frac 18\, \nr{G_{\varepsilon}(\cdot,\snr{Du})}_{L^{\infty}(B_{\tau_2})}+c\|f\|_{L^{2}(\log L)^{\mathfrak {a}}(B_{\varrho})}^{\theta}+1\,,
\end{eqnarray*}
with $c\equiv c(\data,\nr{h}_{L^{d}(\B)})$. Here, as in the following lines, $\theta, \beta$ denote positive exponents depending on $d,\gamma, \sigma,\hat{\sigma}$; they might change from line to line according to the same convention used to denote a generic constant $c$. To estimate the last term $T_4$ we again use Lemma \ref{lemmadieci} with $s_*=1$. Therefore, we set $\tau_3:= (\tau_1+\tau_2)/2$ so that $\tau_2-\tau_3= (\tau_2-\tau_1)/2$, and, recalling \rif{ppp} for $n=2$, we have
\begin{eqnarray*}
\notag  \|\mathbf{P}_{1}^{\mathfrak{h}}(\cdot,(\tau_{2}-\tau_{1})/4)\|_{L^{\infty}(B_{\tau_{1}})}&\stackrel{\eqref{potlore}_2}{\leq} & c\nr{\mathfrak{h}}_{L^{2}(\log L)^{\mathfrak {a}}(B_{\tau_3})}\\
&\notag \stackleq{normine}  &c(\tau)\nr{\mathfrak{h}}_{L^{2+\tau}(B_{\tau_3})} \\
&\stackleq{fk0} & c\nr{h}_{L^{d}(B_{\tau_3})}\nr{\bar{G}_{\varepsilon}(\cdot,\snr{Du})}_{L^{p}(B_{\tau_3})}^{1+\sigma+\hat{\sigma}}\nonumber \\
& \stackrel{\rif{2.2}}{\leq} & \frac{c}{(\tau_{2}-\tau_{1})^{\beta_{*}}}\left[\nr{G_{\varepsilon}(\cdot,\snr{Du})}_{L^{1}(B_{\tau_2})}^{\theta_{*}}+1\right]\\ && \ \  +c\left[\|Du\|_{L^\infty(B_{\tau_2})}\|f\|_{L^2(B_{\tau_2})}\right]^{1+\sigma+\hat{\sigma}}%\label{2.6}
\end{eqnarray*}
for $c\equiv c(\data,\nr{h}_{L^{d}(\B)})$. Using the above inequality and \eqref{z10a}$_1$, we end up with
\begin{flalign*}
T_4 &\le  \frac{c}{(\tau_{2}-\tau_{1})^{\beta_{*}}}\left[\nr{G_{\varepsilon}(\cdot,\snr{Du})}_{L^{\infty}(B_{\tau_2})}^{\delta\vartheta\sigma /2}+1\right]\left[\nr{G_{\varepsilon}(\cdot,\snr{Du})}_{L^{1}(B_{\rr})}^{\theta_{*}}+1\right]\nonumber \\
&\quad +c\left[\nr{G_{\varepsilon}(\cdot,\snr{Du})}_{L^{\infty}(B_{\tau_2})}^{\delta\vartheta\sigma /2+(1+\sigma+\hat{\sigma})/\gamma}+1\right]
\|f\|_{L^{2}(\log L)^{\mathfrak {a}}(B_{\varrho})}^{1+\sigma+\hat \sigma}
\nonumber \\
& \le  \frac 18  \nr{G_{\varepsilon}(\cdot,\snr{Du})}_{L^{\infty}(B_{\tau_2})}+\frac{c\nr{G_{\varepsilon}(\cdot,\snr{Du})}_{L^{1}(B_{\rr})}^{\theta}+c}{(\tau_{2}-\tau_{1})^{ \beta}}
+c\|f\|_{L^{2}(\log L)^{\mathfrak {a}}(B_{\varrho})}^{\theta}
\end{flalign*}
where $c\equiv c(\data,\nr{h}_{L^{d}(\B)})$, and we have used the second condition in \rif{sceglidelta}. Plugging the estimates for the $T$-terms in \eqref{2.4} yields
\begin{flalign*}
\nr{G_{\varepsilon}(\cdot,\snr{Du})}_{L^{\infty}(B_{\tau_{1}})} & \le \frac{1}{2}\nr{G_{\varepsilon}(\cdot,\snr{Du})}_{L^{\infty}(B_{\tau_{2}})}\\ & \quad +\frac{c\nr{G_{\varepsilon}(\cdot,\snr{Du})}_{L^{1}(B_{\rr})}^{\theta}+c\|f\|_{L^{2}(\log L)^{\mathfrak {a}}(B_{\varrho})}^{\theta}+c}{(\tau_{2}-\tau_{1})^{ \beta}}
\end{flalign*}
for new exponents $\theta, \beta$ as in the statement. 
Lemma \ref{l0l} allows now to conclude with 
\begin{flalign*}
&\nr{G_{\varepsilon}(\cdot,\snr{Du})}_{L^{\infty}(B_{\varsigma})}\le\frac{c\nr{G_{\varepsilon}(\cdot,\snr{Du})}_{L^{1}(B_{\rr})}^{\theta}+c\|f\|_{L^{2}(\log L)^{\mathfrak {a}}(B_{\varrho})}^{\theta}+c}{(\rr-\varsigma)^{ \beta}}\,,
\end{flalign*}
where $c\equiv c(\data,\nr{h}_{L^{d}(\B)})$. 
Finally, \rif{a52n2} follows from this last estimate, \eqref{z6} and \eqref{z10a}$_1$. The proof is complete. 
\end{proof}
\section{Proofs of Theorems \ref{t1} and \ref{t3}}\label{pt1}
We start with the proof of Theorem \ref{t3}, where, in particular, we assume $f \in L^n(\Omega;\er^N)$ and $n>2$. We fix $p$ as in statement and, without loss of generality, we assume that $p> 1+\sigma$ ($\sigma$ being as in \rif{sigma}). Then, for every integer $j\geq 1$, we define $f_j \in L^{\infty}(\Omega;\er^N)$ as $f_{j}(x):=f(x)$ if $\snr{f(x)}\le j$, and $f_{j}(x):=j|f(x)|^{-1}f(x)$ otherwise. It clearly follows that 
\eqn{dominaf}
$$
|f_j|\leq |f| \quad \mbox{for every $j\geq1$}\,,\qquad f_j \to f \ \mbox{in $L^n(\Omega;\er^N)$}\,.
$$ Next, we determine $R_{*}\equiv R_{*}(\data, f(\cdot), p)\leq 1$ according to Proposition \ref{lp}. Pay attention here; with some abuse of notation, the $f$ used here is not the same from Proposition \ref{lp}, but rather corresponds to $\mathfrak f$ from \rif{a0} in the context of Proposition \ref{lp} (and thanks to \rif{dominaf}, $f_j$ corresponds to $f$ in Proposition \ref{lp}). Accordingly, we fix a ball $\B\Subset \Omega$ such that $\rad(\B)\le R_*$. We consider a decreasing sequence of positive numbers $\{\varepsilon_{j}\}$ such that 
$\eps_j \leq \min\{1,T\}/4$ for every $j \in \en$, and, accordingly, we consider the families of functions $\{F_{j}\}\equiv \{F_{\varepsilon_{j}}\}$, $\{G_{j}\}\equiv \{G_{\varepsilon_{j}}\}$ constructed in \eqref{z5}. Note now that any minimizer $u$ of the functional $\mathcal F$ in \rif{genF} belongs to $W^{1,\gamma}_{\loc}(\Omega; \er^N)$ by \rif{z7}$_1$. This allows to define $u_{j}\in u+W^{1,\gamma}_0(\B;\RN)$ as the solution to
\begin{flalign*}
& u_j\mapsto \min_{w\in u+W^{1,\gamma}_0(\B;\RN)} \, \mathcal{F}_j(w; \B)\\
& \qquad := \min_{w\in u+W^{1,\gamma}_0(\B;\RN)} \,  \int_{\B}\left[F_{j}(x,Dw)-f_{j}\cdot w \right]\dx\,.
\end{flalign*} 
Directs Methods of the Calculus of Variations apply here and ensure existence (see for instance \cite[Section 4.4]{BM}). 
As for \cite[(4.55)]{BM}, and recalling \rif{z7}$_2$ and \rif{dominaF}, we find
\begin{flalign}\label{e2oo}
& \notag \|F_{j}(\cdot,Du_{j})\|_{L^1(\B)} + \|Du_{j}\|_{L^\gamma(\B)}^\gamma  \\
& \qquad \leq 
c\left[\|F(\cdot,Du)+1\|_{L^1(\B)} + [\rad(\B)]^{n}\nr{f}_{L^{n}(\B)}^{\gamma/(\gamma-1)}\right]
\end{flalign}
for every $j \geq 1$, where $c\equiv c (n,N, \nu, \gamma, T)$. This implies we can assume that, up to a not relabelled subsequence, $Du_{j}\rightharpoonup D\tilde{u}$ weakly in $L^{\gamma}(\B;\mathbb{R}^{N\times n})$ and 
$u_{j}\to \tilde{u}$ strongly in $L^{\frac{n}{n-1}}(\B;\RN)$, for some $\tilde{u}\in u+W^{1,\gamma}_{0}(\B;\RN)$. Note that Proposition \ref{lp} applies to $u_{j}$ with the choice $s_*=1$, as \rif{sigmasss} is satisfied assuming that $\sigma +\hat\sigma < 1/n-1/d$, which is the case here by \rif{sigma} in $\setm$. %Notice that keeping this fact in mind, and following the rest of the proof, justifies the content of Remark \ref{lpremark}. 
The application of Proposition \ref{lp}, and \rif{a44}, now give that
\begin{flalign}
&
\|G_{j}(\cdot,\snr{Du_j})\|_{L^p(s\B)}\notag \\
&\qquad  \le
 \frac{c}{(1-s)^{\beta_{p}}[\rad(\B)]^{\beta_p}}\left[\|F(\cdot,Du)\|_{L^1(\B)}^{\theta_{p}}+\nr{f}_{L^{n}(\B)}^{\theta_{p}}+ 1\right]
 \label{e0}
\end{flalign}
holds for every $s \in (0,1)$, for new exponents $\beta_{p}, \theta_{p}\equiv \beta_{p}, \theta_{p}(n,\gamma, d,\sigma,\hat{\sigma},p)>0$, where we have also used \rif{e2oo} to bound the right-side coming from \rif{a44}; $c$ depends on $\data$ and $\nr{h}_{L^{d}(\B)}$. We fix $j_{0}\in \N$, and apply \eqref{cresceg2} with $j\geq j_0$ in \eqref{e0}, to get
\begin{flalign}
&\|G_{j_0}(\cdot,\snr{Du_j})\|_{L^p(s\B)} \notag \\
& \qquad \le  \frac{c}{(1-s)^{\beta_{p}}[\rad(\B)]^{\beta_p}}\left[\|F(\cdot,Du)\|_{L^1(\B)}^{\theta_{p}}+\nr{f}_{L^{n}(\B)}^{\theta_{p}}+ 1\right]\,.
\label{e00}
\end{flalign}
Letting $j \to \infty$ in the above inequality, and using weak lower semicontinuity (recall that $G_{j_0}(\cdot)$ is convex and non-decreasing, by \rif{einc}, $\eps\equiv \eps_{j_0}$, therefore $z \mapsto G_{j_0}(\cdot,\snr{z})$ is convex), yields
\begin{flalign}
& \notag
\|G_{j_0}(\cdot,\snr{D\tilde u})\|_{L^p(s\B)}\\ & \qquad \le  \frac{c}{(1-s)^{\beta_{p}}[\rad(\B)]^{\beta_p}}\left[\|F(\cdot,Du)\|_{L^1(\B)}^{\theta_{p}}+\nr{f}_{L^{n}(\B)}^{\theta_{p}}+ 1\right]\,.\label{e2}
\end{flalign}
This holds for every $j_0\in \en$ and therefore, by finally letting 
$j_0\to \infty$ (by Fatou's lemma) and recalling \rif{cresceg2}, we conclude with 
\begin{flalign}
&
\|G(\cdot,\snr{D\tilde u})\|_{L^p(s\B)}\notag \\ & \qquad \le  \frac{c}{(1-s)^{\beta_{p}}[\rad(\B)]^{\beta_p}}\left[\|F(\cdot,Du)\|_{L^1(\B)}^{\theta_{p}}+\nr{f}_{L^{n}(\B)}^{\theta_{p}}+ 1\right]\,,\label{e22}
\end{flalign}
for every $s \in (0,1)$, where $c\equiv c(\data,\nr{h}_{L^{d}(\B)})$ and $\beta_{p}, \theta_{p}\equiv \beta_{p}, \theta_{p}(n,\gamma, d,\sigma,\linebreak \hat{\sigma},p)>0$.  
Next, we trivially write
\begin{flalign}
\mathcal{F}_{j_0}(u_j;s\B)& \le \mathcal{F}_{j}(u_j;s\B)+\|F_{j}(\cdot,Du_{j})-F_{j_{0}}(\cdot,Du_{j})\|_{L^1(s\B)} \notag\\
& \quad +\|(f_{j_{0}}-f_{j})\cdot u_j\|_{L^1(s\B)}\label{eooo}
\end{flalign}
whenever $s \in (0,1)$. 
Properties \eqref{dominaf}-\eqref{e2oo}, H\"older and Sobolev-Poincar\'e inequalities give
\eqn{gigi1}
$$
 \mathcal{F}_{j}(u_j;s\B)\le \mathcal{F}_{j}(u_j;\B)+\|{f_{j}\cdot u_j\|_{L^1(\B \setminus s\B)}} \leq 
 \mathcal{F}_{j}(u;\B)+c\nr{f}_{L^{n}(\B \setminus s\B)},
$$
where $c$ is independent of $s,j$ and we have used minimality of $u_j$. Using \eqref{stimadiff} with $\eps_1\equiv \eps_j$, $\eps_2\equiv \eps_{j_0}$, $w\equiv u_j$, $B \equiv s\B$ (recall it is $p>1+\sigma$) and \eqref{e00}, we have
\eqn{gigi2}
$$
\|F_{j}(\cdot,Du_{j})-F_{j_{0}}(\cdot,Du_{j})\|_{L^1(s\B)} \le c(1-s)^{-\beta_{p}} \oo(j_0)
$$
where $c$ is independent of $s,j,j_0$. Again Sobolev-Poincar\'e inequality \rif{dominaf} and \eqref{e2oo} give
$\|(f_{j_{0}}-f_{j})\cdot u_j\|_{L^1(s\B)} \leq c \|f_{j_{0}}-f_{j}\|_{L^n(s\B)}$, with $c$ independent of $s,j, j_0$. Using this last inequality and \rif{gigi1}-\rif{gigi2} in \rif{eooo}, and finally letting $j\to \infty$, lower semicontinuity yields
\begin{flalign}
\notag \mathcal{F}_{j_0}(\tilde u;s\B) &\le \mathcal{F}(u;\B)+c\nr{f}_{L^{n}(\B \setminus s\B)}\\
& \quad +c\nr{f_{j_{0}}-f}_{L^{n}(\B)} + c(1-s)^{-\beta_{p}} \oo(j_0)\,.\label{letlet}
\end{flalign}
Note that we have used \rif{dominaF} and Lebesgue's dominated convergence theorem to get $\mathcal{F}_{j}(u;\B) \to \mathcal{F}(u;\B)$. In turn, again note that \rif{e22} ensures that $G(\cdot,\snr{D\tilde u}) \in L^p(s\B)$ and therefore allows to apply \rif{stimadiff2}; this yields
$$
\|F(\cdot,D\tilde u)-F_{j_{0}}(\cdot,D\tilde u)\|_{L^1(s\B)} \le c(1-s)^{-\beta_{p}} \oo(j_0),
$$ so that $\mathcal{F}_{j_0}(\tilde u;s\B)\to \mathcal{F}(\tilde u; s\B)$ as $j_0\to \infty$, where we also use \rif{dominaf}. 
In view of this, letting first $j_0\to \infty$ and then $s \to 1$ in \rif{letlet}, yields
%\begin{flalign}\label{e9}
$\mathcal{F} (\tilde u;\B)\leq \mathcal{F} (u;\B)$. %\,.
%\end{flalign}
This and the minimality of $u$ finally give $\mathcal{F}(u;\B)=\mathcal{F}(\tilde{u};B)$, therefore, by standard convexity arguments, see for instance \cite[Section 4.4]{BM}, we end up with
\begin{flalign}\label{e10}
\mbox{either} \ \max\{\snr{Du(x)},\snr{D\tilde{u}(x)}\}\le T \ \ \mbox{or} \  \ Du(x)=D\tilde{u}(x)
\end{flalign}
for a.e.\,$x\in \B$. 
Using this information in \rif{e22} yields \rif{e13} and concludes the proof of Theorem \ref{t3}. Observe that, in order to justify the content of Remark \ref{lpremark} it is sufficient to note that making the a priori estimate \eqref{e0} only requires the bound $\sigma+\hat{\sigma}<1/n-1/d$.

We now come to the proof of Theorem \ref{t1}. We can again use the same approximation employed for Theorem \ref{t3}. Using this time estimate \rif{a52} for the case $n>2$ and estimate \rif{a52n2} when $n=2$, together with \rif{e2oo}, we find
\begin{flalign}\label{a52bis}
& [E_{\varepsilon_j}(\nr{Du_{j}}_{L^{\infty}(s\B)})]^\gamma+\nr{G_{j}(\cdot,\snr{Du_{j}})}_{L^{\infty}(s\B)}\nonumber \\
&\qquad \quad  \leq  \frac{c}{(1-s)^{\beta}[\rad(\B)]^{\beta}}
\left[\|F(\cdot,Du)\|_{L^1(\B)}^{ \theta}+\nr{f}_{\mathfrak X(\B)}^{ \theta}+1\right]
\end{flalign}
for every $s \in (0,1)$, where $c,  \beta , \theta$ have the same dependencies as in \rif{a52}. 
%Note we are using that $\gamma/(\gamma-1) \leq 1/(1-\beta_1)$ as we are assuming $1/\gamma < \beta_1$. 
It follows that for every $s \in (0,1)$ there exists $M_s$ such that $\|Du_j\|_{L^\infty(s\B)}\leq M_s$, for every $j \in \en$. Using a standard diagonalization argument we infer that, up to a not relabelled subsequences, we have that $u_j \rightharpoonup^* \tilde u$ in $W^{1,\infty}_{\loc}(\B;\er^N)$ for some $\tilde{u}\in u+W^{1,\gamma}_{0}(\B;\RN)$ being such that $\|D\tilde u\|_{L^\infty(s\B)}\leq M_s$. Moreover, we can repeat verbatim the argument of Theorem \ref{t3} leading to \rif{e10}. Now, denoting $j_s$ the first integer such that $1/\eps_{j_s}>M_s$, from the very definition of $G_{j}\equiv G_{\eps_j}$ in \rif{gte}, it follows that $\nr{G_{j}(\cdot,\snr{Du_{j}})}_{L^{\infty}(s\B)}=\go_1\nr{G(\cdot,\snr{Du_{j}})}_{L^{\infty}(s\B)}$, 
for every $j\ge j_s$. Therefore, letting $j \to \infty$ in \rif{a52bis}, and using \rif{e10}, yields \rif{mainest1} and the proof of Theorem \ref{t1} is complete too. Notice that here we are using standard lower semicontinuity theorems for supremal functionals with respect to the weak$^*$ convergence in $W^{1,\infty}_{\loc}(\B;\er^N)$ (see for instance \cite{BJW}). 

\begin{remark}\label{megliobound}Let us discuss the case $f \equiv 0$. 
We start by the case $n>2$; the first relation in \rif{ledue} is already implied by $\sigma +\hat \sigma < 1/n-1/d$. Moreover, note that the second condition in \rif{ledue} appears only when $f\not\equiv 0$. This is the only point in the proof of Theorem \ref{t1} where the full bound in \rif{sigma}$_1$ is required; otherwise  
$\sigma +\hat \sigma < 1/n-1/d$ is sufficient. Alternatively, one can apply directly Proposition \ref{lpinf} instead of Proposition \ref{lp} in the proof of Theorem \ref{t1}. When $n=2$ and $f\equiv 0$, \rif{sceglidelta} turns into $
(1+\delta)\vartheta\sigma/2  < 1 $ and $\delta\vartheta\sigma /2<1
$, that are implied by $\sigma +\hat \sigma < 1/2-1/d$ by taking $\delta$ small enough (see estimate of $T_4$ in Proposition \ref{duedi}). Finally, note that in the case $\gamma\geq 2$, the minimum in \rif{sigma} is attained by $1/n-1/d$. All in all, we have justified the content of Remark \ref{megliobound0}. 
\end{remark}

\section{Uniform ellipticity and proof of Theorem \ref{t2}}\label{unisec}
Here we work under assumptions $\setuni$ in \rif{ipotesi}; we keep the full notation introduced in Sections \ref{apro} and \ref{apri}. In particular, we keep on considering a solution $u\in W^{1,\gamma}(\B;\er^N)$ to \rif{a0} in a ball 
$\B\Subset \Omega$ such that $\rad(\B)\leq 1$. With the current choice of $g_1, g_2, g_3$ as in \rif{uni0}, we apply the constructions laid down in Section \ref{apro}, thereby obtaining, in particular, the new functions $\tilde a_{\eps}$, $g_{1,\varepsilon}$, $g_{2,\varepsilon}$, $g_{3,\varepsilon}$ in \rif{z2} and \rif{g1e}-\rif{g3e}. The last three functions are now independent of $x$, as well as $G_{\eps}$ defined in \rif{gte}. Note that \rif{uni0}$_2$ from $\setuni$ implies the validity of \rif{9} with $\cb\equiv \cu$ and $\sigma=0$. Therefore we can use all the properties from Section \ref{apro}, and displayed through Lemmas \ref{gle}-\ref{ultimolemma}, implied by assumptions \eqref{0ass}-\eqref{00} and \rif{9}. In particular, we can use Lemma \ref{ultimolemma} with $\sigma =0$. In addition to such properties, we have
\begin{lemma} Under assumptions $\setuni$ in \trif{ipotesi}
\begin{itemize}
\item The following inequalities hold for every $t\in [T,\infty)$:
\eqn{uni3}
$$
\begin{cases}
\displaystyle \ g_{1,\varepsilon}(t)\le g_{2,\varepsilon}(t) \leq   \tcu g_{1,\varepsilon}(t)\\
  g_{3, \eps}(t)\le \tcu g_{1, \eps}(t)t
  \end{cases}
$$
where $\tcu \equiv \tcu (n,N,\nu, \gamma, \cu)\geq 1$, and 
%\begin{flalign}\label{uni4}
%\frac{1}{2\tcu}g_{1,\eps}(t)t^{2}\le G_{\eps}(t)+\frac{\go_{1}}{2}g_{1}(T)T^{2}\le \frac{\tcu}{2} g_{1,\eps}(t)t^{2}\,.
%\end{flalign}
\begin{flalign}\label{uni4}
\frac{1}{2\tcu}g_{1,\eps}(t)t^{2}\le G_{\eps}(t)+\frac{\go_{1}}{2}g_{1}(T)T^{2}\le \frac{\tcu}{2} g_{1,\eps}(t)t^{2}\,.
\end{flalign}
\item Moreover, for every $x \in \Omega$ and $t \geq 0$, it holds that 
\eqn{doublec}
$$
G_{\eps}(t) \leq \tilde F_{\eps}(x,t) \leq  cG_{\eps}(t)  + cg_1(T)(T^2+\mu_{\varepsilon}^{2})\,, \quad c\equiv c (\datauni)\,. 
$$
\end{itemize}
\end{lemma}
\begin{proof}
Properties \rif{uni3} are immediate and follow from \rif{uni0}, also recalling \rif{00bis}, therefore we concentrate on \rif{uni4}. 
As for the right-hand side of \rif{uni4}, integration by parts yields
\eqn{intpart}
$$
\int_{T}^{t}\left[\tilde{a}_{\eps}(x,s)+\tilde{a}_{\eps}'(x,s)s\right]s \, ds = 
-\int_{T}^{t}\tilde{a}_{\eps}(x,s)s \, ds+\tilde{a}_{\eps}(x,t)t^{2}-\tilde{a}_{\eps}(x,T)T^{2}
$$
and therefore, recalling the definition in \eqref{g1e}, we have
\begin{eqnarray*}
G_{\eps}(t)&\stackrel{\eqref{a2}_2}{\leq}  &\int_{T}^{t}\left[\tilde{a}_{\eps}(x,s)+\tilde{a}_{\eps}'(x,s)s\right]s \, ds\\
&\stackrel{\rif{a1}, \rif{uni3}}{\leq}& -G_{\varepsilon}(t)+ \tcu g_{1,\eps}(t)t^{2}-\go_{1}g_{1}(T)T^{2}
\end{eqnarray*}
that is, the right-hand side of \rif{uni4}.  As for the left-hand side, we similarly have 
\begin{eqnarray*}
G_{\eps}(t)&\stackrel{\rif{a2}_2, \rif{uni3}}{\ge} &\frac{1}{\tcu} \int_{T}^{t}\left[\tilde{a}_{\eps}(x,s)+\tilde{a}_{\eps}'(x,s)s\right]s \, ds \nonumber \\
&\stackrel{\rif{a1}, \rif{intpart}}{\geq} &-\frac{1}{\tilde {\cu}}\int_{T}^{t}g_{2,\eps}(s)s \, ds+\frac{1}{\tilde {\cu}}\left[\tilde{a}_{\eps}(x,t)t^{2}-\tilde{a}_{\eps}(x,T)T^{2}\right]\\
&\stackrel{\eqref{g1e}, \rif{uni3}}{\geq} & -G_{\eps}(t)+\frac{1}{\tilde {\cu}}g_{1,\eps}(t)t^{2}-\go_{1}g_{1}(T)T^{2}\,.
\end{eqnarray*}
We turn to \rif{doublec}. The left-hand side inequality is nothing but \rif{z6}$_3$ (that holds whenever $|z|\geq 0$). For the right-hand side inequality, when $t \leq T$ we have, similarly to \rif{slipffff}
$$
\tilde F_{\eps}(x,t)    \stackrel{ \eqref{uni0}_2, \rif{slipffff}}{\leq} c\left[\frac{g_1(T)}{(T^2+\mu ^2)^{\frac{\gamma-2}{2}}}+\eps\right] (T^2+\mu_\eps^2)^{\gamma/2}\stackleq{00bis}cg_1(T)(T^2+\mu^{2}). 
$$
When $t > T$, we note that \rif{a1} and \rif{uni3} imply $\bar a_{\eps}(x,s) \leq c  g_{1,\eps}(s)$, that implies $\tilde F_{\eps}(x,t) - \tilde F_{\eps}(x,T) \leq  c G_{\eps}(t) $, from which \rif{doublec} follows again using the content of the last display. 
\end{proof}

%Indeed, if $t \leq \eps $\begin{flalign*}
%\tilde F_{\eps}(x,t)  &\leq  \left[\frac{2\tilde a(x,\eps)}{(\eps^2+\mu ^2)^{\frac{\gamma-2}{2}}}+\eps\right]\int_0^t (s^2+\mu_\eps^2)^{\frac{\gamma-2}{2}} s \, ds \leq \left[\frac{2\tilde a(x,T)}{(T^2+\mu ^2)^{\frac{\gamma-2}{2}}}+\eps\right] (T^2+\mu_\eps^2)^{\gamma/2}\\
%&  \leq \left[\frac{2\cu g_1(T)}{(T^2+\mu ^2)^{\frac{\gamma-2}{2}}}+\eps\right] (T^2+\mu_\eps^2)^{\gamma/2} \leq c
%g_1(T)(T^2+\mu_{\varepsilon}^{2})\,,
%\end{flalign*}

\begin{proposition}\label{unile2}
Let $u\in W^{1,\gamma}(\B;\RN)$ be a solution to $\eqref{a0}$, under assumptions $\setuni$ in \trif{ipotesi} for $n\geq 2$. There exists a positive radius $R_{*}\equiv R_{*}(\datauni, h(\cdot))\leq 1$ such that if $\rad(\B)\le R_{*}$ and $B_{\varsigma} \Subset  B_{\rr}$ are concentric balls contained in $\B$, then
\begin{eqnarray}
&&\notag \|F_{\varepsilon}(\cdot, Du) \|_{L^\infty(B_\xi)}+
\|G_{\varepsilon}(\snr{Du}) \|_{L^\infty(B_\xi)} \\
&& \hspace{2cm} \le
\frac{c}{(\varrho-\xi)^n}\|F_\eps(\cdot, Du) \|_{L^1(B_{\rr})}+ c \left[\|f\|_{\mathfrak{X}(B_{\rr})}^{\gamma/(\gamma-1)}+1\right]\label{finafina}
\end{eqnarray}
holds with $c\equiv c(\datauni)$, where $\mathfrak{X}(\cdot)$ has been defined in \eqref{lox}. \end{proposition}
\begin{proof} We start taking a ball $B_{r}(x_0)\Subset \B$ (therefore it is $r < 1$), and prove that
\begin{eqnarray}
%\nonumber &\int_{B_{r/2}(x_{0})}\eta^{2}(G_{\varepsilon}(x,\snr{Du})-\kk)_{+}g_{1,\varepsilon}(x,\snr{Du})\snr{D^{2}u}^{2} \dx\\ & \quad +
&&\int_{B_{r/2}(x_{0})}\snr{D(G_{\varepsilon}(\snr{Du})-\kk)_{+}}^{2} \dx \le \frac{c}{r^{2}}\int_{B_{r}(x_{0})}(G_{\varepsilon}(\snr{Du})-\kk)_{+}^{2} \dx\nonumber \\
&&\notag \hspace{25mm}  +c\left[\nr{G_{\varepsilon}(\snr{Du})}_{L^{\infty}(B_{r}(x_{0}))}+1\right]^2\int_{B_{r}(x_{0})}|h|^2 \dx\\
&&\hspace{30mm}+c\nr{Du}_{L^{\infty}(B_{r}(x_{0}))}^{2}\int_{B_{r}(x_{0})}\snr{f}^{2} \dx\label{a45uni}
\end{eqnarray} holds whenever $\kk \geq 0$, with $c\equiv c(\datauni)$. This is an analogue of \rif{a45} and to get it  we modify the proof of Lemma \ref{caclem}, keeping the notation used there. Proceeding as for the bounds for $\mbox{(IV)}_{z}$-$\mbox{(VI)}_{z}$ in Lemma \ref{caclem}, we have
\begin{flalign}
\notag \mathcal S_3 +\mathcal S_4   & \leq  c|\mbox{(IV)}_{x}|+c|\mbox{(V)}_{x}|+c|\mbox{(VI)}_{x} |+  \frac{c}{r^2}
\int_{B_{r}(x_0)}(G_{\varepsilon}(\snr{Du})-\kk)_{+}^2\dx\\
&\qquad + c\sum_{s=1}^n\int_{\B^{\kk}} \snr{f\cdot D_s\varphi_s} \dx  \,,\label{inconuni}
\end{flalign}
with $c\equiv c(\datauni)$. This estimate can be obtained by adapting those in \rif{posi3}-\rif{stimaIIquater}, and also those for the terms in \rif{stimaII}. One must take into account that now $g_{1,\eps}$ is independent of $x$ (therefore the terms coming from the use of \rif{z10} and featuring $h(\cdot)$ disappear), and the fact that we can formally take $\sigma\equiv 0$ as the ratio $g_{2,\eps}/g_{1,\eps}$ is bounded by a constant by \rif{uni3}$_1$. In turn, the last term in \rif{inconuni} involving the right-hand side $f$ can be treated exactly as in \rif{trattati}. As for the remaining $x$-terms appearing in the first line of \rif{inconuni}, with the help of $\eqref{growthd}_{3}$, \eqref{uni3} and \eqref{uni4}, we estimate
\begin{flalign*}
c|\mbox{(IV)}_{x}|+c|\mbox{(V)}_{x}|+c|\mbox{(VI)}_{x} | &\le \bar \varepsilon\mathcal{S}_3 +
\bar \varepsilon\mathcal{S}_4+ \frac{c}{r^2}
\int_{B_{r}(x_0)}(G_{\varepsilon}(\snr{Du})-\kk)_{+}^2\dx\nonumber\\
&\quad +\frac{c}{\bar \varepsilon}\int_{\B}\eta^{2}|h|^2[G_{\varepsilon}(\snr{Du})+1]^{2}\dx\,,
\end{flalign*}
for $c\equiv c(\datauni)$ and $\bar \varepsilon \in (0,1)$. Merging the content of the above two displays, choosing $\bar \varepsilon\equiv \bar \varepsilon(\datauni)$ small enough, and reabsorbing terms, we end up with \rif{a45uni}, 
where $c\equiv c(\datauni)$. As a consequence, we proceed as for the proofs of Lemma 
\ref{stimapp} and Proposition \ref{linf1}. An application of Lemma \ref{noniter} gives that, if $B_{2r_0}(x_{0})\Subset \B$ is any ball, then 
\begin{flalign}
G_{\varepsilon}(\snr{Du(x_{0})}) &\le cr_0^{-n/2}\left[\nr{G_{\varepsilon}( \snr{Du})}_{L^{\infty}(B_{\tau_2})}+1\right]^{1/2}\|G_{\varepsilon}(\snr{Du}) \|_{L^1(B_{\tau_{2}})}^{1/2}\nonumber \\
&\quad   +c\left[\nr{G_{\varepsilon}( \snr{Du})}_{L^{\infty}(B_{\tau_2})}+1\right] \mathbf{P}_{1}^{h}(x_{0},2r_{0})
\notag \\ &\quad + c \left[\nr{G_{\varepsilon}( \snr{Du})}_{L^{\infty}(B_{\tau_2})}+1\right]^{1/\gamma}\mathbf{P}_{1}^{f}(x_{0},2r_{0}) \label{uni80}
\end{flalign}
holds provided $x_0$ is a Lebesgue point of $|Du|$, where $c\equiv c(\datauni)$, and we have also used \rif{z10a}$_1$. Next, \rif{potlore} gives that 
$$
\|\mathbf{P}_{1}^{h}(\cdot,2r_{0})\|_{L^{\infty}(B_{\tau_{1}})}  
\lesssim \|h\|_{\mathfrak{X}(B_\rr)} \quad \mbox{and} \quad 
\|\mathbf{P}_{1}^{f}(\cdot,2r_{0})\|_{L^{\infty}(B_{\tau_{1}})}
\lesssim  \|f\|_{\mathfrak{X}(B_\rr)}\,.
$$
Using these informations in \rif{uni80} yields
\begin{flalign}\label{uni8}
&\notag \|G_{\varepsilon}(\snr{Du}) \|_{L^\infty(B_{\tau_{1}})}+1\\
& \qquad \le cr_0^{-n/2}\left[\nr{G_{\varepsilon}( \snr{Du})}_{L^{\infty}(B_{\tau_2})}+1\right]^{1/2}\|G_{\varepsilon}(\snr{Du}) \|_{L^1(B_{\tau_{2}})}^{1/2}+1\nonumber \\
&\qquad\quad  \nonumber +c_*\left[\nr{G_{\varepsilon}( \snr{Du})}_{L^{\infty}(B_{\tau_2})}+1\right]  \|h\|_{\mathfrak{X}(B_{\varrho})}
\\ & \qquad  \quad + c\left[\nr{G_{\varepsilon}( \snr{Du})}_{L^{\infty}(B_{\tau_2})}+1\right]^{1/\gamma} \|f\|_{\mathfrak{X}(B_{\varrho})}\,,
\end{flalign}
where $c, c_*\equiv c,c_*(\datauni)$. By absolute continuity, we now determine the radius $R_*\equiv R_*(\datauni, h(\cdot))$ mentioned in the statement in such a way that
\begin{flalign}\label{a23uni}
\rad(\B)\leq R_* \Longrightarrow c_*\nr{ h}_{\mathfrak{X}(B_{\varrho})}\le c_*\nr{ h}_{\mathfrak{X}(\B)}\le 1/6\,.
\end{flalign}
Using this and Young inequality in \rif{uni8}, and yet recalling that $r_{0}:=(\tau_{2}-\tau_{1})/8$, gives
\begin{flalign*}
 \|G_{\varepsilon}(\snr{Du}) \|_{L^\infty(B_{\tau_{1}})}
&\le \frac{1}{2} \|G_{\varepsilon}(\snr{Du}) \|_{L^\infty(B_{\tau_{2}})}\\
&\quad +\frac{c}{(\tau_2-\tau_1)^n}\|G_{\varepsilon}(\snr{Du}) \|_{L^1(B_\rr)}+ c\left[\|f\|_{\mathfrak{X}(B_\rr)}^{\gamma/(\gamma-1)}+1\right]\,.
\end{flalign*}
Inequality \rif{finafina} now follows using Lemma \ref{l0l} with $\mathcal{Z}(t):= \|G_{\varepsilon}(\snr{Du}) \|_{L^\infty(B_{t})}$ and \rif{doublec}. 
\end{proof}
With Proposition \ref{unile2} available, we can now complete the proof of Theorem \ref{t2}. Arguing as in the proof of Theorem \ref{t1} in Section \ref{pt1} (as usual with $F_j \equiv F_{\eps_j}$ and so on), and arrive up to \rif{a52bis}, the analog of which in the present context is
%\eqn{start}
$$
\|F_j(\cdot, Du_j) \|_{L^\infty(s\B)} \le
 \frac{c}{(1-s)^{n}[\rad(\B)]^{n}} \|F(\cdot, Du) \|_{L^1(\B)}+ c\left[\|f\|_{\mathfrak X(\B)}^{\gamma/(\gamma-1)} +1\right]
$$
where $c$ depends on $\datauni$ (again recall the equivalence in \rif{doublec}). The rest of the proof again follows with minor modifications to the proof for Theorem \ref{t1}.

\section{Proof of Theorem \ref{t4}}\label{tsec3}
We revisit the proof of Theorem \ref{t1} starting from the part concerning the a priori estimates of Section \ref{apri} in the case $n>2$. In turn, this uses as a preliminary result Proposition \ref{lp} for $s_*=1$, that now works only assuming $\hat \sigma< 1/n-1/d$, as allowed by \rif{sigma-int}. This relies on the condition $g_2/g_1\leq \cb$ in \rif{sigma-int}, that via \rif{z9} also implies $g_{2, \eps}/g_{1, \eps}\leq c$. Then Remark \ref{t4uno} can be used and we can replace $\sigma +\hat \sigma $ by $\hat \sigma$ in \rif{sigmasss}. The same arguments of Remark \ref{t4uno} also apply in the proof of Lemma \ref{caclem}, and yield the following simplified form of \rif{a45}:
\begin{flalign*}%\label{a46-int-pre}
&\int_{B_{r/2}(x_{0})}\snr{D(G_{\varepsilon}(x,\snr{Du})-\kk)_{+}}^{2} \  \dx \le \frac{c}{r^{2}}\int_{B_{r}(x_{0})}(G_{\varepsilon}(x,\snr{Du})-\kk)_{+}^{2} \  \dx\nonumber \\
&\hspace{44mm} +c\int_{B_{r}(x_{0})}[h(x)]^{2}[\bar{G}_{\varepsilon}(x,\snr{Du})]^{2(1+\hat{\sigma})} \dx\\ & \hspace{44mm} +c\nr{Du}_{L^{\infty}(B_{r}(x_{0}))}^{2}\int_{B_{r}(x_{0})}\snr{f}^{2} \  \dx\,.
\end{flalign*}
This last inequality still holds when $n=2$. Using it as in Lemma \ref{stimapp} yields
\begin{flalign}\label{a46-int}
G_{\varepsilon}(x_{0},\snr{Du(x_{0})})& \le   \kk+c\left(\mint_{B_{r_{0}}(x_{0})}(G_{\varepsilon}(x,\snr{Du})-\kk)_{+}^{2} \dx\right)^{1/2}\nonumber \\
&\ \  + c\left[\mathbf{P}_{1}^{\mathfrak{h}}(x_{0},2r_{0})+\nr{Du}_{L^{\infty}(B_{r_{0}}(x_{0}))}\mathbf{P}_{1}^{f}(x_{0},2r_{0})\right]\,, 
\end{flalign}
for every $\kk \geq 0$, which is in fact the analog of \rif{a46} in the present setting. Here it is $\mathfrak{h}(x):=h(x)[\bar G_{\varepsilon}(x,\snr{Du})]^{1+\hat{\sigma}}$. From this last inequality we again arrive at \rif{a52} exactly as in the proof of 
Proposition \ref{linf1} (needless to say, for different exponents $\theta$ and $\beta$). Here the key observation is that we do not have to verify the two conditions in \rif{ledue}, that in fact are not occurring (actually they occur taking formally $\sigma =0$). It is only the second, and more restrictive one, that requires the stronger bound in \rif{sigma}, that can therefore be relaxed in \rif{sigma-int}. Starting from this fact, the rest of the proof is the same as the one of Theorem \ref{t1} in the case $n>2$, again replacing $\sigma +\hat \sigma$ by $\hat \sigma$. In the case $n=2$ the argument is similar. First, we use Lemma \ref{lemmadieci}, again with $s_*=1$, and, as for the case $n>2$, this works replacing  $\sigma +\hat \sigma$ by $\hat \sigma$ in \rif{uppersss}. Next, we apply this time \rif{a46-int}, thereby getting \rif{arrivaad} (with formally $\sigma =0$) and then \rif{2.4}, with $\nr{G_{\varepsilon}(\cdot,\snr{Du})}_{L^{\infty}(B_{\tau_2})}$ replaced by $1$. From this point on, the proof goes as in the previous case; the only remark is that now the first condition in \rif{sceglidelta} does not appear as a consequence of the new version of \rif{2.4}, and this new bound is exactly the one appearing in \rif{sigma-int} upon considering only $\hat \sigma$.

\section{Proof of Theorems \ref{sample1}-\ref{sample8}}\label{applications}
\subsection{Theorems \ref{sample1} and \ref{sample2}} We just verify that the assumptions of Theorem \ref{sample1} imply those of Theorem \ref{t1} for suitable choices of objects and parameters. For this we take $\gamma =p$, $g_{i}(x,t)\equiv g_{i}(t)$, $i\in \{1,2,3\}$, where $g_{1}(t):=\nu(t^{2}+\mu^{2})^{(p-2)/2}$, $g_{2}(t):=\Lambda (t^{2}+\mu^{2})^{(q-2)/2}+\Lambda (t^{2}+\mu^{2})^{(p-2)/2}$ and $g_{3}(t):=(t^{2}+\mu^{2})^{(q-2)/2}t+(t^{2}+\mu^{2})^{(p-2)/2}t$, 
for all $t\in (0,\infty)$. It follows that $
G(x,t)\equiv G(t)=\nu[(t^{2}+\mu^{2})^{p/2}-(T^{2}+\mu^{2})^{p/2}]/p $ (for $t \geq T$ and it is zero otherwise) and $\bar{G}(x,t)\equiv \bar{G}(t)=G(t)+(T^{2}+1)^{p/2}\geq \nu(t^{2}+\mu^{2})^{p/2}+ (p-\nu)/p$. 
Here, any number $T\in (0,\infty)$ is fine. By \eqref{asp1}-\eqref{asp2} we see that $F$ satisfies \rif{0ass}-\eqref{00}. Note that $\eqref{7}$ holds with $\hat{\sigma}=0$ since $g_{1}$ is $x$-independent. As for \rif{8}-\rif{9}, we verify them choosing $\sigma=q/p-1$, for a suitably large constant $
\cb\equiv \cb(n,N,\nu, p, q, \Lambda)
\geq 1$. With such a choice of $\sigma$, and $\hat \sigma =0$ as above, the condition in \rif{sigma} is implied by the assumed one in \rif{lox2}. All in all, assumptions \rif{sigmad}-\rif{9} are verified too and and we can apply Theorem \ref{t1} to get Theorem \ref{sample1}. Finally, the last assertion concerning the improved bound in \rif{lox2bis} follows from the content of Remark \ref{megliobound0} and this completes the proof of Theorem \ref{sample1}. Theorem \ref{sample2} is a special case of Theorem \ref{sample1}, upon taking $h(\cdot)\approx |D\ccc(\cdot)|$. For this, just observe that the convexity of $z \to F( z)$, and the growth assumptions \rif{asp1}, imply also that $\tilde F'(t) \lesssim (t^2+\mu^2)^{(q-2)/2}t+(t^2+\mu^2)^{(p-2)/2}t$, thereby allowing to verify also \rif{asp2} for the integrand $(x,t)\mapsto \ccc(x)\tilde F(t)$.

\subsection{Theorem \ref{sample3}}\label{sample3sec}
As all our estimates are local in nature, we can assume that $H(\cdot, Du)\in L^1(\Omega)$, where the integrand $H(\cdot)$ has been defined in \rif{pqfunctional}; moreover, by localizing Morrey's embedding theorem, we can also assume that $[a]_{0,1-n/d;\Omega}+\nr{a}_{L^{\infty}(\Omega)}$ is finite. Similarly to \rif{dati}, we denote 
\eqn{datahh}
$$\texttt{data}_{\textnormal{h}}:=(n,\nu,N,p,q,d,[a]_{0,1-n/d;\Omega},\nr{a}_{L^{\infty}},\nr{H(\cdot,Du)}_{L^{1}},\nr{f}_{L^{n}})\,.$$
Here $[a]_{0,1-n/d;\Omega}$ denotes the usual $(1-n/d)$-H\"older seminorm of $a(\cdot)$ on $\Omega$, and all the norms extend over $\Omega$. The proof now proceeds in three steps.
\subsubsection*{Step 1: Quantification of ellipticity} We refer to the framework of Theorem \ref{t4}. We define, for $x \in \Omega$ and $t>0$, 
$g_{1}(x,t):=\min\{p-1,1\}[t^{p-2}+a(x)t^{q-2}]$, $g_{2}(x,t):=2qNn[t^{p-2}+a(x)t^{q-2}]$, $g_{3}(x,t)\equiv g_{3}(t):=t^{q-1}$
that implies
$G(x,t):=\min\{p-1,1\}[(t^{p}/p+a(x)t^{q}/q)-(1/p+a(x)/q)]$ (here we take $T=1$ and $\mu=0$), $\bar{G}(x,t)=G(x,t)+2^{p/2}$ and $g_2/g_1\leq c_{\rm b}$ for $t \geq 1$, where $ c_{\rm b}\equiv  c_{\rm b}(\texttt{data}_{\textnormal{h}})$. We also set $h(\cdot)=\snr{Da(\cdot)}$, $\sigma =\hat \sigma := q/p-1$, $\gamma=p$. We now observe that, replacing $\leq $ by $<$ in the first bound from \rif{lox22}, we can immediately conclude with the local Lipschitz continuity of minima invoking Theorem \ref{t4}, whose assumptions, and in particular \rif{sigma-int}, are satisfied but in the limiting/equality case. In order to get the delicate equality in \rif{lox22}, we have to readapt some points from the proof of Theorem \ref{t4} using some additional results available when the peculiar structure in \rif{pqfunctional} is considered.

\subsubsection*{Step 2: Uniform higher integrability} We modify and extend arguments from \cite{CoM0, CoM, demicz}.  
\begin{lemma}\label{ilprimoh}
Let $u\in W^{1,1}(\Omega;\mathbb{R}^{N})$ be a minimizer of the functional in \eqref{pqfunctional}, under assumptions \eqref{fff} and $0 \leq a(\cdot) \in W^{1,d}(\Omega)$ with \trif{lox2bis}. There exists $s_{**}\equiv s_{**}(\texttt{data}_{\textnormal{h}})>1$ such that $H(\cdot, Du)\in L^{s_{**}}_{\loc}(\Omega)$. 
\end{lemma}
\begin{proof}
Consider concentric balls $B_{\rr} \subset B_t \Subset B_s \subset B_{R}$ contained in $\Omega$, then take a standard cut-off function $\eta\in C^{1}_{{\rm c}}(B_{R})$ so that $\mathds{1}_{B_{t}}\le \eta\le \mathds{1}_{B_{s}}$ and $\snr{D\eta}\lesssim 1/(s-t)$. We take $w:=u-\eta(u-(u)_{B_{R}})$ as competitor to minimality, that yields 
$
\|H(\cdot,Du)\|_{L^1(B_{s})} \le 
\|H(\cdot,Dw)\|_{L^1(B_{s})} +  \|\eta f\cdot (u-(u)_{B_{R}}) \|_{L^1(B_{s})}$. The first term in the right-hand side of this last inequality can be handled as in \cite[Section 9]{CoM0}, while by Sobolev and Young inequalities we have
\begin{flalign}\label{involves}
&\|\eta f\cdot (u-(u)_{B_{R}}) \|_{L^1(B_{s})}\le \|f \|_{L^n(B_{s})}\|\eta(u-(u)_{B_{R}} )\|_{L^{n/(n-1)}(B_{s})}\nonumber \\
& \quad \le cR^{n-n/p}\nr{f}_{L^{n}(\Omega)}\nr{Du}_{L^{p}(B_{s})}+\frac{cR^{n-n/p}}{s-t}\nr{f}_{L^{n}(\Omega)}\|u-(u)_{B_{R}}\|_{L^p(B_{s})}\nonumber \\
&\quad \le \frac14 \|H(\cdot, Du)\|_{L^1(B_{s})}+c\int_{B_{R}}\tilde H\left(x,\frac{|u-(u)_{B_R}|}{s-t}\right)  \, dx +cR^n  \,,
\end{flalign}
where
$c\equiv c(n,p,q,\nr{f}_{L^{n}(B)})$. This last estimate together with the arguments developed in \cite[Section 9]{CoM0}, leads to the reverse H\"older inequality 
$$
\left(\mint_{B_{R/2}} [H(\cdot,Du)]^{1+\delta}\dx \right)^{1/(1+\delta)} \leq c
\mint_{B_{R}} [H(\cdot,Du)]\dx + c\,,
$$  
for $c\equiv c(\texttt{data}_{\textnormal{h}})\geq 1$, $\delta\equiv \delta(\texttt{data}_{\textnormal{h}})>0$, which in turn allows us to conclude via a covering argument.
\end{proof}
Let $B\Subset\Omega$ be a ball with $\rad(B)\le 1$. As is \cite{CoM0, CoM}, we say that the $p$-phase occurs when 
\eqn{pphase}
$$a_{\textnormal{i}}(B):=\inf_{x\in B}a(x)\le 4[\rad(B)]^{1-n/d}[a]_{0,1-n/d;B}$$ holds, otherwise we say that the $(p,q)$-phase occurs. Accordingly, we define the integrand 
$\tilde H^{-}_{B}(|z|):=\snr{z}^{p}/p+a_{\textnormal{i}}(B) \snr{z}^{q}/q.$ 
Finally, let $\tilde H_{\varepsilon}(\cdot)$, $\tilde H^{-}_{B,\varepsilon}(\cdot)$ be the integrands defined by applying the construction in Section \ref{genset}, \rif{z2}-\rif{z5}, to $\tilde H(\cdot)$ and $\tilde H^{-}_{B}(\cdot)$, respectively, with $\gamma=p$ and $g_{1}, g_2, g_{3}$ defined in \emph{Step 1} and $T=1$. Such a construction provides us with approximating integrands that are still of double phase type, i.e.,
\eqn{hhdopo}
$$
\tilde H_{\varepsilon}(x,t) = \int_0^t\tilde{h}_{\varepsilon}(x,s)s\, ds+\varepsilon\int_{0}^{t}(s^{2}+\eps^{2})^{\frac{p-2}{2}}s \, ds\,,
$$
where
\begin{flalign}\label{hhdopo2}
\tilde{h}_{\varepsilon}(x,t):= \begin{cases} 
\ \frac{\eps^{p-2}}{(2\eps^2)^{\frac{p-2}{2}}}(t^{2}+\eps^2)^{\frac{p-2}{2}}
\\ \qquad +a(x) \frac{\eps^{q-2}}{(2\eps^2)^{\frac{p-2}{2}}}(t^{2}+\eps^2)^{\frac{p-2}{2}}\ \ &\mbox{if}  \ t\in [0,\varepsilon)\\
\ t^{p-2}+a(x)t^{q-2}\ \ &\mbox{if}  \ t\in [\varepsilon,T_{\varepsilon})\\
\ \frac{T_{\eps}^{p-2}}{(T_{\eps}^2+\eps^2)^{\frac{p-2}{2}}}(t^{2}+\eps^2)^{\frac{p-2}{2}}\\
\qquad +a(x) \frac{T_{\eps}^{q-2}}{(T_{\eps}^2+\eps^2)^{\frac{p-2}{2}}}(t^{2}+\eps^2)^{\frac{p-2}{2}} \ \ &\mbox{if} \ t\in [T_{\varepsilon},\infty)\,,
\end{cases}
\end{flalign}
with the same representation holding for $\tilde H^{-}_{B,\varepsilon}(\cdot)$, replacing $a(\cdot)$ by $a_{\textnormal{i}}(B)$ in \rif{hhdopo2} above; here it is $T_\eps = 1+1/\eps$. 
Next, we develop an intrinsic Sobolev-Poincar\'e inequality involving $H_{\varepsilon}(\cdot)$; the main point is that the implied constants and exponents are independent of $\eps$. 
\begin{lemma}\label{sopolem}
Assume that $0 \leq a(\cdot) \in W^{1,d}(\Omega)$ with \trif{lox2bis}. Let $B\Subset \Omega$ be a ball with $\rad(B)\le 1$ and $w\in W^{1,p}_{0}(B;\RN)$ and $v\in W^{1,p}(B;\RN)$ be such that  $H_{\varepsilon}(\cdot,Dw),  H_{\varepsilon}(\cdot,Dv)\in L^{1}(B)$. Then, the following Sobolev and Sobolev-Poincar\'e inequalities hold:
\begin{flalign}\label{sopoeps}
\mint_{B}\tilde H_{\varepsilon}\left(x,\frac{|w|}{\rad(B)}\right)  \, dx \le c\left(\mint_{B}[\tilde H_{\varepsilon}(x,|Dw|)+1]^{\tau}  \, dx\right)^{1/\tau}
\end{flalign}
$c\equiv c(n,N,p,q,d,[a]_{0,1-n/d;B},\nr{Dw}_{L^{p}(B)})$ and
\begin{flalign}\label{sopoeps.1}
\mint_{B}\tilde H_{\varepsilon}\left(x,\frac{|v-(v)_{B}|}{\rad(B)}\right)  \, dx \le c\left(\mint_{B}[\tilde H_{\varepsilon}(x,|Dv|)+1]^{\tau}  \, dx\right)^{1/\tau}\,,
\end{flalign}
for $c\equiv c(n,N,p,q,d,[a]_{0,1-n/d;B},\nr{Dv}_{L^{p}(B)})$. In both cases it is $\tau\equiv \tau(n,p,q)\linebreak \in (0,1)$.
\end{lemma}
\begin{proof}
We prove \eqref{sopoeps}, the proof of \eqref{sopoeps.1} being totally similar. We start considering the $p$-phase; this is when \rif{pphase} occurs. In this case, recalling $\eqref{hhdopo}$-$\eqref{hhdopo2}$, we bound
\begin{flalign*}
 &\int_{B}\tilde H_{\varepsilon}\left(x,\frac{|w|}{\rad(B)}\right)  \, dx\\
 & \qquad \le c\varepsilon^{p}[\rad(B)]^{n}+c\int_{B\cap\{\varepsilon\le \snr{w}/\rad(B)<T_{\varepsilon}\}}\left(\int_{\varepsilon}^{\snr{w}/\rad(B)}g_{1}(x,s)s \, ds\right)  \, dx\nonumber \\
&\qquad \qquad +c\int_{B\cap\{\snr{w}/\rad(B)\ge T_{\varepsilon}\}}\left(\int_{\varepsilon}^{\snr{w}/\rad(B)}g_{1,\varepsilon}(x,s)s \, ds\right)  \, dx \\
& \qquad \qquad + c\eps [\rad(B)]^{n}\left(\mint_{B}\snr{Dw}^{p_{*}}  \, dx\right)^{p/p_{*}} \  \big(=:\mbox{(I)}+\mbox{(II)}+\mbox{(III)}++\mbox{(IV)}\big)\,,
\end{flalign*}
where $c\equiv c(n,p,[a]_{0,1-n/d;B})$.  
Now, note that the assumed bound in \trif{lox2bis} implies $p_{*}\le q_{*}<p$, where 
$
b_{*}:=\max \, \{1,nb/(n+b)\}$, for $b\in \{p,q\}$. 
Therefore using again Sobolev-Poincar\'e and H\"older's inequalities, we estimate
\begin{flalign*}
& \mbox{(II)} +\mbox{(III)}\le  c\int_{B}[(|w|/\rad(B))^{p}+a(x)(|w|/\rad(B))^{q} ] \, dx\nonumber \\
&\qquad \qquad  \quad +c\int_{B\cap\{\snr{w}/\rad(B)\ge T_{\varepsilon}\}}[ T_{\varepsilon}^{p} + a(x)T_{\varepsilon}^{q}+ a(x)T_{\varepsilon}^{q-p}(|w|/\rad(B))^{p}]  \, dx\nonumber \\
&\stackleq{pphase}  c[\rad(B)]^{n}\mint_{B}[ (|w|/\rad(B))^{p}+[\rad(B)]^{1-n/d}(|w|/\rad(B))^{q}]  \, dx\\
& \  \leq  c[\rad(B)]^{n}\left(\mint_{B}\snr{Dw}^{p_{*}}  \, dx\right)^{p/p_{*}}\nonumber \\
& \qquad \qquad   + c[\rad(B)]^{n+1-n/d}\left(\mint_{B}\snr{Dw}^{p}  \, dx\right)^{q/p-1}\left(\mint_{B}\snr{Dw}^{p(q_{*}/p)}  \, dx\right)^{p/q_{*}}\nonumber \\
 &\ \stackleq{lox2bis}  c[\rad(B)]^{n}\left(\nr{Dw}_{L^{p}(B)}^{q-p}+1\right)\left(\mint_{B}\snr{Dw}^{p\tilde \tau}  \, dx\right)^{1/\tilde \tau},
\end{flalign*}
with $c\equiv c(n,N,p,q,d,[a]_{0,1-n/d;B})$ and $\tilde \tau:=q_{*}/p<1$. Merging the content of the last two displays  above easily yields
\begin{flalign}\label{degsopo}
\mint_{B}\tilde H_{\varepsilon}\left(x,\frac{|w|}{\rad(B)}\right)  \, dx\le c\left(\mint_{B}[\tilde H_{\varepsilon}(x,|Dw|)+1]^{\tilde \tau}  \, dx\right)^{1/\tilde \tau}
\end{flalign}
for $c\equiv c(n,N,p,q,d,[a]_{0,1-n/d;B}\nr{Dw}_{L^{p}(B)}^{q-p})$. We then pass to the $(p,q)$-phase, that is, when the complementary condition to \rif{pphase} holds; here we follow the arguments from \cite[Section 4]{CoM}. It is then easy to see that $
\tilde H_{\varepsilon}(x,z)\sim \tilde H^{-}_{B,\varepsilon}(z)$ for $(x,z)\in B\times \mathbb{R}^{N\times n}$. Moreover, recalling that $H^{-}_{B,\varepsilon}$ can be written as in \rif{hhdopo}-\rif{hhdopo2} with $a(\cdot)$ replaced by $a_{\textnormal{i}}(B)$, it follows that $\tilde H^{-}_{B,\varepsilon}$ satisfies both $\Delta_2$ and $\nabla_2$ conditions (with constants depending on $n,p,q$, but independent of $\eps$). We are therefore in position to argue as in the proof of \cite[Theorem 7]{DiE}, thereby again arriving at \rif{degsopo}, for a different exponent $\tilde \tau \equiv \tau(n,p,q) \in (0,1)$. At this stage the proof of \rif{sopoeps} follows merging the inequalities found in the two cases, and taking the largest of the two exponents $\tilde \tau$ found for the two phases. 
\end{proof}
Given the inequalities in \rif{sopoeps}-\rif{sopoeps.1}, and the representation in \rif{hhdopo}-\rif{hhdopo2}, we can proceed for instance as in \cite[Lemma 5]{demicz} to get a global gradient integrability result; this also involves estimates as in \rif{involves} to treat the additional $f$-terms appearing here with respect to the case considered in \cite{demicz}. This involves a matching of local and up-to-the-boundary versions of Gehring's lemma (see \cite{KSt}). 
\begin{lemma}\label{bougeh.c}
Let $\B\Subset \Omega$ be a ball with $\rad(\B)\le 1$. Under assumptions \eqref{fff}, and $0 \leq a(\cdot) \in W^{1,d}(\Omega)$ with \trif{lox2bis}, let $u_\eps \in W^{1,p}(\B;\mathbb{R}^{N})$ solve
\begin{flalign}\label{avpeeee}
u_\eps\mapsto \min_{w}  \int_{\B}[\tilde{H}_\eps(x,\snr{Dw})-f\cdot w] \dx \qquad w \in u+W^{1,p}_0(\B;\er^N)\,,
\end{flalign}
where $u\in W^{1,1}(\Omega;\mathbb{R}^{N})$ is a minimizer of the functional in \eqref{pqfunctional}.
Then there exists $s_* \in (1, s_{**})$, depending only on $\texttt{data}_{\textnormal{h}}$, but not on $\eps$, such that 
\begin{flalign}
\notag \|\tilde H_{\varepsilon}(\cdot,|Du_{\varepsilon}|)+1\|_{L^{s_*}(\B)} & \le c\|\tilde H_{\varepsilon}(\cdot,|Du|)+1\|_{L^{s_*}(\B)}\\
&  \leq c\|\tilde H(\cdot,|Du|)+1\|_{L^{s_*}(\B)}\label{s2.5}
\end{flalign}
holds for a constant $c\equiv c ( \texttt{data}_{\textnormal{h}})$, where $s_{**}$ is the exponent coming from Lemma \ref {ilprimoh}. 
\end{lemma}
We finally remark that the constants appearing in \rif{s2.5} should a priori depend also on $\nr{Du_{\varepsilon}}_{L^{p}(B)}$, as Lemma \ref{sopolem} is involved in the derivation of Lemma \ref{bougeh.c} (see again \cite{demicz}), and it is used with the choices $w\equiv u_\eps$ and  $v\equiv u_\eps$. Such a dependence on $\eps$ does not actually occur; indeed, adapting estimate \rif{e2oo} to the present case leads to bounds as $\nr{Du_{\varepsilon}}_{L^{p}(\B)}^{p}\lesssim \|H(\cdot, Du)+1\|_{L^1(\B)}+ [\rad(\B)]^{n}\nr{f}_{L^{n}(\B)}^{p/(p-1)}$, thereby reducing the dependence of the constants on $\texttt{data}_{\textnormal{h}}$ in \rif{datahh}. This fact is indeed used to prove the intermediate last inequality in \rif{s2.5} with an $\eps$-independent constant $c$. The last inequality in \rif{s2.5} is instead a direct consequence of \rif{dominaF}. In turn, the last quantity in \rif{s2.5} is finite by Lemma \ref{ilprimoh}. 

\subsubsection*{Step 3: Completion of the proof of Theorem \ref{sample3}} We proceed almost verbatim as in the proof of Theorem \ref{t4}. We start with the case $n>2$; the only difference here is that we apply Proposition \ref{lp} with $s_*>1$ being equal to the higher integrability exponent found in Lemma \ref{bougeh.c}. With no loss of generality we may assume that $s_*< \min\{2m(1+\hat{\sigma}), 2^*\} =\min\{2mq/p, 2^*\}$, as required in \rif{sigmasss} (applied with $\sigma =0$). Note that this is possible when $  \hat \sigma < s_*/n-s_*/d$ and therefore, in particular, when 
$ \hat \sigma \leq  1/n-1/d$, which is the case considered here. Indeed, with the choice made in Step 1, this last condition translates in $q/p \leq 1 +1/n-1/d$, which is \rif{lox22} for $n>2$. This is the essential point where the higher integrability estimates of Step 2 come into the play, allowing for equality in \rif{lox22}. We then proceed exactly as for Theorem \ref{t4}, as its assumptions are verified but for the equality case in \rif{lox22}, as noticed in Step 1. Proceeding as in the proof of Theorem \ref{t4} we arrive at \rif{a46-int}, and all the foregoing considerations remain the same, but, in order to get a suitable a priori estimate, the term involving $\mathbf{P}_{1}^{\mathfrak{h}}$ must be estimated slightly differently from Proposition \ref{linf1}. More precisely, using \rif{a44} with the current choice of $s_*$ in \rif{servepure}, we finally come to the new uniform bound 
\begin{flalign}
&
\notag \nr{G_{\varepsilon}(\cdot,\snr{Du_\eps})}_{L^{\infty}(B_{\varsigma})}\\
&\qquad \le \frac{c}{(\rr-\varsigma)^{ \beta }}
\left[\|\tilde H_{\varepsilon}(\cdot,|Du_\eps|)\|_{L^{s_*}(B_{\rr})}^{\theta}+\nr{f}_{L(n,1)(B_{\rr})}^{\theta}+1\right]\,, \label{a52-doppia}
\end{flalign}
which replaces \rif{a52} and holds whenever $B_{\varsigma} \Subset  B_{\rr}$ are concentric balls contained in $\B$; the constant $c$ depends on $\texttt{data}_{\textnormal{h}}$. Note that we have applied the argument of Proposition \ref{linf1} directly to $u_\eps$ defined in \rif{avpeeee}. Using \rif{s2.5} in \rif{a52-doppia} yields
\begin{flalign*}
&\|G_{\eps}(\cdot,\snr{Du_\eps})\|_{L^{\infty}(s\B)}\\
& \qquad  \le
\frac{c}{(1-s)^{\beta}[\rad(\B)]^{\beta}}
\left[\|H(\cdot,Du)\|_{L^{s_*}(\B)}^{ \theta}+\nr{f}_{L(n,1)(\B)}^{ \theta}+1\right]
\end{flalign*}
for every $s\in (0,1)$. Again, the right-hand side stays finite by Lemma \ref{ilprimoh}. This last estimate can be used as a replacement of \rif{a52bis} in an approximation argument which is at this stage completely similar to the one used for the proof of Theorem \ref{t1} and this completes the proof in the case $n>2$. It remains to treat the case $n=2$. For this we again turn to the arguments of Section \ref{casetwo}, where we apply Lemma \ref{lemmadieci} with the choice of the number $s_*>1$ again determined in Lemma \ref{bougeh.c}. Such an application is legal as we are assuming that $\hat \sigma = q/p -1 \leq 1/2-1/d$, while again we may assume that $s_* < 2m(1+\hat{\sigma})$. Then we proceed exactly as in the proof of Theorem \ref{t4} case $n=2$, again noting that only the second inequality in \rif{sceglidelta} is needed (formally with $\sigma=0$), and this leads to require that $q/p<p$, which is the second condition in \rif{lox22} for the case $n=2$.

\subsection{Proof of Theorems \ref{sample4}-\ref{sample6} and Theorem \ref{sample8}}

We deduce Theorem \ref{sample6}, from Theorem \ref{t1}; Theorem \ref{sample4} then follows similarly. In exactly the same way, we can then deduce Theorem \ref{sample8} from Theorem \ref{t3}. As usual, we do it by making a suitable choice of the growth functions $g_{1}, g_2, g_3$ and of the parameters $\sigma,\hat{\sigma},\ca,\cb,\gamma, T, \mu$. This requires some preparations; we split the proof in two steps. In the following we denote $\data_k\equiv (n,N,k,\nu_{\mathfrak{m}},L,p_{\mathfrak{m}},p_{m},p_{M})$, for every integer $k \geq 0$. In the following, with abuse of notation, we shall indicate by $\partial_{zz}{\bf e}_{k}(\cdot)$ the Hessian matrix of $z \to {\bf e}_{k}(\cdot, |z|)$, while ${\bf e}_{k}'(x,t)$ keeps on denoting the (partial) derivative of the function $t \to {\bf e}_{k}(\cdot, t)$. By Sobolev-Morrey embedding theorem we can assume that $\{\ccc_k, p_k\}$ are H\"older continuous functions with exponent $\alpha = 1-n/d$. 
\subsubsection*{Step 1: Computations} By induction we have
\begin{flalign}\label{derie.x}
{\bf e}_{k}'(x,t)=\ccc_{k}(x)p_{k}(x)t^{p_{0}(x)-1}\Pi_{k}(x,t){\bf e}_{k}(x,t) \ \ \mbox{for} \ \ k\ge 0\,,
\end{flalign}
for every $x \in \Omega$ and $t>0$, where
$$
\begin{cases}
\Pi_{k}(x,t):=\ccc_{k-1}(x)p_{k-1}(x)[{\bf e}_{k-1}(x,t)]^{p_{k}(x)} \Pi_{k-1}(x,t)\ \ \mbox{for $k\geq 1$}\\
 \Pi_{0}(x,t):=1\,.
 \end{cases}
$$
%\begin{flalign*}
%\Pi_{k}(x,t):=\prod_{j=0}^{k-1}c_{j}(x)p_{j}(x)[{\bf e}_{j}(x,t)]^{p_{j+1}(x)} \ \ \mbox{for $k\geq 1$ \ \ and\ \  } \Pi_{0}(x,t):=1\,.
%\end{flalign*}
Then, observing that 
$$
\Pi_{k}'(x,t)= \Pi_{k}(x,t)t^{p_0(x)-1} \sum_{j=0}^{k-1} p_{j+1}(x)\ccc_{j}(x)p_{j}(x)  \Pi_{j}(x,t) \ \ \mbox{for} \ \ k\ge 1\,,
$$
 identity \rif{derie.x} gives
 $$
 {\bf e}_{0}''(x,t)=t^{p_{0}(x)-1}{\bf e}_{0}'(x,t)\left[\ccc_{0}(x)p_{0}(x)+\frac{p_{0}(x)-1}{t^{p_{0}(x)}}\right]
 $$
 and
\begin{flalign}\label{deri2k}
\notag {\bf e}_{k}''(x,t)&\, =t^{p_{0}(x)-1}{\bf e}_{k}'(x,t)\ccc_{k}(x)p_{k}(x)\Pi_{k}(x,t)\\
& \qquad + t^{p_{0}(x)-1}{\bf e}_{k}'(x,t)\sum_{j=0}^{k-1} p_{j+1}(x)\ccc_{j}(x)p_{j}(x)  \Pi_{j}(x,t)\nonumber \\
&\qquad +t^{p_{0}(x)-1}{\bf e}_{k}'(x,t)\frac{p_{0}(x)-1}{t^{p_{0}(x)}}\qquad \mbox{for} \ \ k\ge 1\,.
\end{flalign}
Finally, it is 
$$
[{[\bf e}_{k}(x,t)]^{p_{k+1}(x)}]' = [{\bf e}_{k}(x,t)]^{p_{k+1}(x)}p_{k+1}(x)\ccc_{k}(x)p_{k}(x)t^{p_0(x)-1}
\Pi_{k}(x,t)\,.
$$
We now come to the $x$-derivatives; note that the following computations are justified as in \rif{esponenciali} we are considering a composition of bounded Sobolev functions $\ccc_{k}, p_{k}$ with smooth functions, and therefore the standard (vectorial) chain rule applies as in the traditional case, i.e., if $\ccc_{k}, p_{k}$ were $C^1$-regular. We use properly defined, auxiliary vector fields $\mathcal D_k, \mathcal L_k \colon \Omega \times (0, \infty) \to \er^n$ for $k\geq 0$, and the notation ${\bf e}_{-1}(x,t)\equiv t$. Then, we have, for a.e.\,$x\in \Omega$ and $t>0$
\begin{flalign}\label{ilDkk}
\begin{cases}
\partial_x {\bf e}_{k}(x,t)={\bf e}_{k}(x,t)[{\bf e}_{k-1}(x,t)]^{p_{k}(x)}\mathcal D_{k}(x,t) \,,\quad  \mbox{for} \ k\geq 0\\ 
\partial_x [{\bf e}_{k-1}(x,t)]^{p_k(x,t)}:=[{\bf e}_{k-1}(x,t)]^{p_k(x,t)} [{\bf e}_{k-2}(x,t)]^{p_{k-1}(x,t)}  
\\ 
\hspace{36mm}  \times \left[ \ccc_{k-1}(x)D p_k(x) 
 + p_k(x) \mathcal D_{k-1}(x,t)\right], \ \mbox{for} \ k \geq 1\,,
\end{cases}
\end{flalign}
where, by induction, we have defined, for $k \geq 0$
\begin{flalign}\label{ilDk}
\begin{cases}
\mathcal D_0(x,t):=D \ccc_0(x)+\ccc_0(x)\log  t\,  D p_0(x)\\
\mathcal D_{k+1}(x,t):= D \ccc_{k+1}(x)+\ccc_{k+1}(x)\ccc_{k}(x)[{\bf e}_{k-1}(x,t)]^{p_{k}(x)}D p_{k+1}(x)\\ 
 \qquad  \qquad  \qquad +\ccc_{k+1}(x)p_{k+1}(x)[{\bf e}_{k-1}(x,t)]^{p_{k}(x)}\mathcal D_{k}(x,t)\,.
\end{cases}
\end{flalign}
Using \rif{ilDkk}-\rif{ilDk}, for $k\ge 1$ we compute, again by induction
\begin{flalign}\label{ilPk}
\begin{cases}
\ \partial_{x}\Pi_{0}(x,t)=0\,, \quad \mathcal{L}_{0}(x,t)=0_{\er^n}\\
\ \partial_{x}\Pi_{k}(x,t)=\Pi_{k}(x,t)\mathcal{L}_{k}(x,t)\\
\ \mathcal{L}_{k}(x,t):=D\log\left(\ccc_{k-1}(x)p_{k-1}(x)\right)\\
\ \quad +[{\bf e}_{k-2}(x,t)]^{p_{k-1}(x)}\left[\ccc_{k-1}(x)D p_{k}(x)+p_{k}(x)\mathcal{D}_{k-1}(x,t)\right]+\mathcal{L}_{k-1}(x,t)\,.
\end{cases}
\end{flalign}
Finally, using \eqref{ilDkk}-\eqref{ilPk} in \eqref{derie.x}, for every $k \geq 0$ we conclude with
\begin{flalign}\label{derie.xx}
&\partial_{x}{\bf e}'_{k}(x,t)={\bf e}'_{k}(x,t)\nonumber \left[D\left(\log(\ccc_{k}(x)p_{k}(x))\right)\right.\\
& \left. \hspace{17mm} + \log t D p_{0}(x)+\mathcal{L}_{k}(x,t)+[{\bf e}_{k-1}(x,t)]^{p_{k}(x)}\mathcal{D}_{k}(x,t)\right].
\end{flalign}
As for the tensor $\partial_{zz}{\bf e}_{k}(x,|z|)$, by a direct computation we have that
\begin{flalign}\label{leduematrici}
\begin{cases}
\partial_{zz} {\bf e}_{k}(x,|z|) \xi \cdot \xi \geq \min \left\{ {\bf e}_{k}''(x,|z|),\frac{ {\bf e}_{k}'(x,|z|)}{ |z|} \right\} |\xi|^2\\
|  \partial_{zz}  {\bf e}_{k}(x,|z|)|^2=[{\bf e}_{k}''(x,|z|)]^2 + \left[\frac{{\bf e}_{k}'(x,|z|) }{|z|}\right]^2(Nn-1)
\end{cases}
\end{flalign}
hold for every choice of $z, \xi \in \er^{N\times n}$ with $z\not=0$ and $x \in \Omega$. 
\subsubsection*{Step 2: Determining $g_{1}$, $g_{2}$ and $g_{3}$.}  For every fixed $k \geq 0$, the constants implied in the symbols $\lesssim$ and $\approx$, will depend on $\data_k$ and we shall use the auxiliary functions
 $$\mathfrak{h}_k(x) := \sum_{i=0}^k \, [|D \ccc_i(x)| + |D p_i(x)|] +1\,, \qquad k\geq 0\,.$$ It follows that $\mathfrak{h}_k \in L^d(\Omega)$ for $d>n$, by assumptions. By \rif{ilDk} it follows $|\mathcal D_0(x,t)|   \lesssim \mathfrak{h}_{0}(x)[|\log t|+1] $, so that induction gives that 
 %\eqn{ddk1}%%ee
$$|\mathcal{D}_{k}(x,t)|\lesssim \mathfrak{h}_{k}(x)\left[(t^{p_{0}(x)}|\log t|+1)\Pi_{k-1}(x,t)+1\right] \quad \mbox{holds for all $k\geq 1$}\,.$$ In turn, this and \rif{ilPk} imply that 
$$|\mathcal{L}_{k}(x,t)|\lesssim \mathfrak{h}_{k}(x)\left[(t^{p_{0}(x)}|\log t|+1)\Pi_{k-1}(x,t)+1\right] \quad \mbox{holds for all $k\geq 1$}\,.$$ 
By \rif{derie.x} it is also ${\bf e}_{k}'(x,t) \lesssim t^{p_0(x)-1} \leq t^{p_{\mathfrak{m}}-1}$ for $t\approx 0$. From this and \rif{derie.xx} it follows there exit constants $m_k\equiv m_k(\data_k)\ge 1$ such that, for a.e.\,$x\in \Omega$
\eqn{<e}
$$
\begin{cases}
\snr{\partial_{x}{\bf e}'_{k}(x,t)}\lesssim m_k\mathfrak{h}_{k}(x) &\  \mbox{for all} \ \ t\in [0,{\rm e}]\\
\snr{\partial_{x}{\bf e}'_{k}(x,t)}\lesssim \mathfrak{h}_{k}(x){\bf e}_{k}'(x,t)t^{p_{0}(x)}\log t\Pi_{k}(x,t) &\ \mbox{for all} \ \ t\in [{\rm e},\infty)\,.
\end{cases}
$$
Now, let $\{c_{k}, d_{k}\}$ be constants larger than $1$, to be eventually chosen large enough, again in dependence on $\data_k$, and $\phi\in C([0,\infty),[0,1])$ be a non-decreasing function such that $\phi(t)=0$ for $t\in [0,{\rm e}/2]$ and $\phi(t)=1$ for $t\in [{\rm e},\infty)$. With $p_{\mathfrak m}$ as in \eqref{assexp}, and $x \in  \Omega$ and $t>0$, we define
\begin{flalign*}%\label{g1g2.ex}
\begin{cases}
\ g_{1}(x,t)\equiv g_{1,k}(x,t):=\phi(t)\min\{p_{\mathfrak m}-1,1\}{\bf e}_{k}'(x,t)/t\\
\ g_{2}(x,t)\equiv g_{2,k}(x,t):=\phi(t)c_{k}t^{p_{0}(x)}\Pi_{k}(x,t){\bf e}_{k}'(x,t)/t\\
\ g_{3}(x,t)\equiv g_{3,k}(x,t):=(1-\phi(t))c_{k}m_k\\
\hspace{38mm}+\phi(t)c_{k}m_k{\bf e}_{k}'(x,t)t^{p_{0}(x)}\log t\Pi_{k}(x,t)\,.
\end{cases}
\end{flalign*}
We are ready to check that $g_{1}, g_{2}, g_{3}$ satisfy the relations prescribed in Section \ref{asssec1} and required to meet the assumptions of Theorem \ref{t1}. We take $\cb\equiv \cb(\data_k)\geq1 $ to be determined in due course of the proof, $T={\rm e}, \nu\approx \min\{p_{\mathfrak m}-1,1\},  \gamma=p_{\mathfrak{m}},  \mu=0$; moreover, any choice of small numbers $\sigma,\hat{\sigma}>0$ will work. In this way it is $\bar G(x,t)=\min\{p_{\mathfrak m}-1,1\}[{\bf e}_{k}(x,t)- {\bf e}_{k}(x,{\rm e})  + ({\rm e}^2+1)^{p_{\mathfrak m}/2}]$ for $t \geq {\rm e}$, so that $\bar G(x,t)\approx {\bf e}_{k}(x,t)$  for $t$ large. By \rif{derie.x}-\rif{deri2k} and \rif{leduematrici}, it follows that \rif{growth}$_{2,3}$ are satisfied provided we take $c_{k}\equiv c_{k}(\data_k)$ large enough. Similarly, \rif{growth}$_{4}$ follows using \rif{<e} and eventually increasing $c_k$. In the same way \eqref{00} holds with the above choice of the parameters by \rif{derie.x}. As for \eqref{7}-\eqref{9}, note that, given $k\in \N$ and $\eps \in (0,1)$, there exists $c \equiv c (\data_k, \eps)$, such that $\Pi_{k}(x,t)t^{p_{0}(x)+1}\log t\le c[{\bf e}_{k}(x,t)]^{\eps}$ for $t \geq {\rm e}$. Using this, \eqref{7} follows, for all $\hat \sigma>0$, from \rif{<e} by taking $h\equiv d_{k}\mathfrak{h}_{k} \in L^d(\Omega)$, provided we take $ d_{k}\equiv  d_{k}(\data_k,\hat \sigma)$ large enough. As for \rif{8}, note that $(g_3^2/g_1)(x,t)\lesssim [{\bf e}_{k}'(x,t)]^{1+\eps}\lesssim [{\bf e}_{k}(x,t)]^{1+2\eps}$ for every $\eps >0$ and $t \geq {\rm e}$; therefore \eqref{8} follows for every $\sigma \in (0,1)$, again taking $\cb \equiv \cb (\data_k, \sigma)$ large enough. In the same way \rif{9} follows, by eventually enlarging $\cb$. This means that the assumptions of Theorem \ref{t1} are satisfied and the proof of Theorem \ref{sample6} is complete.

\subsection{Proof of Theorem \ref{sample7}}
Conditions in \eqref{0ass} are verified by \eqref{uni-ell}; indeed, an easy consequence of \eqref{uni-ell} is that $t \mapsto \tilde{a}(t)/t^{\ia}$ is non-decreasing. It follows that $\tilde a(t)t\leq \tilde a(1)t^{\ia+1}$ for $t\in (0,1]$, and therefore $t \mapsto A(x,t) \in C^1_{\loc}[0,\infty)\cap C^2_{\loc}(0,\infty)$ for every $x \in \Omega$ as it is $\ia >-1$. Moreover 
\begin{flalign*}
\begin{cases}
\ \snr{\partial_{zz}A(x,\snr{z})}\le L\sqrt{Nn}\max\{1,\sa+1\}\tilde{a}(\snr{z})\\
\ \nu\min\{1,\ia+1\}\tilde{a}(\snr{z})\snr{\xi}^{2}\leq \partial_{zz}A(x,\snr{z})\xi\cdot \xi \\
\ \snr{\partial_{xz}A(x,\snr{z})}\le \snr{D\ccc(x)}\tilde{a}(\snr{z})\snr{z}\,,
\end{cases}
\end{flalign*}
hold for every choice of $x\in \Omega$ and $z, \xi\in \er^{N\times n}$, $|z|\not=0$. Notice that here we again indicate
by $\partial_{zz}A(\cdot)$ the Hessian of $z \to A(\cdot, |z|)$. Having in mind to apply Theorem \ref{t2}, this leads to define $g_{1}(t):=\nu\min\{1,\ia+1\}\tilde{a}(t), g_{2}(t):=L Nn\max\{1,\sa+1\}\tilde{a}(t),  g_{3}(t):=\tilde{a}(t)t$ and $h:=\snr{D\ccc}\in \mathfrak{X}(\Omega)$. As $t \mapsto \tilde{a}(t)/t^{\ia}$ is non-decreasing, \eqref{00} is also verified with $\gamma=\ia+2$, $\mu=0$. The definitions given above make sure that conditions $\eqref{uni0}$ are satisfied with 
$\cu\approx \max\{1,\sa+1\}/\min\{1,\ia+1\}$.  
This means that Theorem \ref{t2} applies and \eqref{stimaAA} follows from \rif{finafinapre}. 

\section{Obstacle problems and Theorem \ref{sample9}}\label{obsec}
\subsection{A general result}
We start with the following constrained analog of Theorem \ref{t1}, which will be used to ge the proof of Theorem \ref{sample9}:
\begin{theorem}\label{t1ob}
Let $u\in W^{1,1}_{\loc}(\Omega)$ be a constrained minimizer of the functional $\mathcal F_0$ in \eqref{genF0}, under assumptions $\setm$ in \trif{ipotesi} with $g_3, \partial_{zz}F$ being locally bounded on $ \Omega \times [0, \infty)$ and $\Omega \times \er^{N\times n}$, respectively, and $\gamma\geq 2$. If $\psi\in W^{2,1}_{\loc}(\Omega)$ with $\snr{D^{2}\psi}\in \mathfrak{X}(\Omega)$, then $Du\in L^{\infty}_{\textnormal{loc}}(\Omega;\R^{N\times n})$. 
\end{theorem}
\begin{proof} In $\setm$ we have initially assumed that $g_3$ was locally bounded in $ \Omega \times (0, \infty)$, while here we are assuming it is locally bounded in $ \Omega \times [0, \infty)$, that means that for every $\Omega_0 \Subset \Omega$ and $b \geq 0$ we have that $\sup_{\Omega_0\times [0, b]}\, g_3$ is finite. This is is not really an additional assumption as it is automatically satisfied in all the cases considered for instance in Theorem \ref{sample9}. 
Notice also that $D^2\psi \in \mathfrak{X}(\Omega) \subset L(n,1)$ implies that $D\psi$ is continuous in $\Omega$ \cite{St} and, in particular,   $D\psi \in L^{\infty}(\B;\er^n)$ for any ball $\B\Subset \Omega$; accordingly, we let $T_{\psi} \equiv T_{\psi} (\B):= \|D\psi\|_{L^\infty(\B)}+T+1$. We now fix an arbitrary ball $\B\Subset \Omega$ with $\rad(\B)\le 1$, we consider the family $\{F_{\varepsilon}\}$ constructed in Section \ref{genset}. This time we take $0 < \eps < \min\{ 1/T_{\psi}, T\}/4$ and define $u_{\eps}$ as the solution to the auxiliary problem
\begin{flalign}\label{o.0}
u_\eps \mapsto \min_{(u+W^{1,\gamma}_{0}(\B))\cap \mathcal{K}_{\psi}(\B)} \int_{\B}F_{\varepsilon}(x,Dw)\, dx\,.
\end{flalign}
The existence of $u_\eps$ follows by standard theory, and the variational inequality
\begin{flalign}\label{o.1}
\int_{\B}\partial_{z}F_{\varepsilon}(x,Du_{\varepsilon})\cdot (Dw-Du_{\varepsilon})\, dx \ge 0
\end{flalign}
holds for all $w\in(u+W^{1,\gamma}_{0}(\B))\cap \mathcal{K}_{\psi}(\B)$. 
Thanks to the $\gamma$-polynomial growth conditions in \eqref{zgrowthd}, we are now able to perform the linearization procedure used in \cite[page 237]{fumin} i.e., we can rearrange \eqref{o.1} in the following way:
\begin{flalign}\label{o.2}
\int_{\B}\left[\partial_{z}F_{\varepsilon}(x,Du_{\varepsilon})\cdot D\varphi- \texttt{f}_\eps\cdot \varphi\right]\, dx=0
\end{flalign}
for all $\varphi\in W^{1,\gamma}_{0}(\B)$, where $$\texttt{f}_{\varepsilon}(x):=-\theta_{\varepsilon}(x)\mathds{1}_{\{x\in \B\colon u_{\varepsilon}(x)=\psi(x)\}}\diver \, \partial_{z}F_{\varepsilon}(x,D\psi)\,,$$ for some measurable density $\theta_{\varepsilon}\colon \B\to [0,1]$. Note that the definition of $\texttt{f}_{\varepsilon}$ makes sense in light of the fact that $\psi \in W^{2,n}(\Omega)\cap W^{1,\infty}(\Omega)$ and of the discussion made at the beginning of Section \ref{apri} to prove \rif{diso}. This implies that $\partial_{z}F_{\varepsilon}(\cdot,D\psi) \in W^{1,n}(\B;\er^n)$ and the usual chain rule formula holds as in Section \ref{apri}, again thanks to the results in \cite{dele}; see Remark \ref{linearire} below. We then define the constant
\begin{flalign}
 \go_{\psi}(\B) &:= \|g_{3}\|_{L^\infty(\B\times [0,T_{\psi}))}+\|\tilde a(\cdot, 1)\|_{L^\infty(\B)} \notag\\ 
 &\qquad \quad +\sup_{x\in \B,  |z|\in [0,T_{\psi})} |\partial_{zz}F(x, z)| + T_{\psi}^{\gamma-2}+1\,.
 \label{settig}
\end{flalign}
This quantity is always bounded as $\partial_{zz} F$ is in turn assumed to be locally bounded and $\gamma \geq 2$, here we also use the fact that $g_3$ is locally bounded as again described in the statement of the Theorem. This is essentially the only place where such assumptions come into the play. 
Elementary manipulations based on the first property in \rif{00} and \eqref{growthd}$_3$, give 
%ult
\eqn{o.7pre}
$$\snr{\texttt{f}_{\varepsilon}(x)}\le c(n,N,\nu, \gamma,\cb)\go_{\psi}(\B)[\snr{D^{2}\psi(x)}+h(x)]\,.$$
Here we have also used that $\eps \leq 1/T_\psi$ to exploit the definitions in \rif{z2} and \rif{g3e}. 
In turn, \rif{o.7pre} gives 
\begin{flalign}\label{o.7}
\texttt{f}_{\varepsilon}\in \mathfrak{X}(B)\quad \mbox{with} \ \ \nr{\texttt{f}_{\varepsilon}}_{\mathfrak{X}(B)}\le c\nr{D^{2}\psi}_{\mathfrak{X}(B)}+c \nr{h}_{\mathfrak{X}(B)}\,,
\end{flalign}
for all balls $B\subseteq \B$, where $c\equiv c(n,N,\nu, \gamma,\cb,g_{\psi}(B))$. Since $u_{\varepsilon}$ verifies \eqref{o.2}, the strict convexity of $F_{\varepsilon}(\cdot)$ prescribed by $\eqref{z3}_{4}$ implies that $u_{\varepsilon}$ is the unique solution of Dirichlet problem
\begin{flalign}\label{o.3}
u_\eps \mapsto \min_{u+W^{1,\gamma}_{0}(\B)}\int_{\B}\left[F_{\varepsilon}(x,Dw)-\texttt{f}_{\varepsilon}\cdot w\right]\, dx\,.
\end{flalign}
By \eqref{zgrowthd} and \eqref{o.7}, we see that problem \eqref{o.3} falls in the realm of those covered by Proposition \ref{tapp} in Section \ref{appendix} below, therefore $u_{\varepsilon}\in W^{1,\infty}_{\loc}(B)\cap W^{2,2}_{\loc}(B)$. This is exactly the information in \rif{a6} allowing to justify the all the subsequent calculations in Section \ref{apri} in view of an application to solutions to \rif{o.3}. Moreover, thanks to \rif{o.7pre}, the coercivity estimate 
\rif{e2oo} applied to the case of \rif{o.3},
becomes
\begin{flalign}\label{coeer}
\notag & \|F_{\eps}(\cdot,Du_{\eps})\|_{L^1(\B)} + \|Du_{\eps}\|_{L^\gamma(\B)}^\gamma\\
& \qquad  \leq 
c\|F(\cdot,Du)\|_{L^1(\B)} + c\nr{D^{2}\psi}_{L^{n}(\B)}^{\gamma/(\gamma-1)}+c\nr{h}_{L^{n}(\B)}^{\gamma/(\gamma-1)}+c
\end{flalign}
thanks to \rif{o.7pre}. Using Proposition \ref{linf1} when $n>2$, and Proposition \ref{duedi} when $n=2$ (with $u_\eps\equiv u$ and $\texttt{f}_\eps\equiv f$), we arrive at
\begin{flalign}\label{a52obs}
&[E_{\varepsilon}(\nr{Du_{\eps}}_{L^{\infty}(s\B)})]^\gamma+\nr{G_\eps(\cdot,\snr{Du_\eps})}_{L^{\infty}(s\B)}\nonumber\\ &\hspace{6mm} \le \frac{c}{(1-s)^{\beta}[\rad(\B)]^{\beta}}
\left[\|F(\cdot,Du)\|_{L^1(\B)}^{\theta}+\nr{D^{2}\psi}_{\mathfrak{X}(\B)}^{ \theta}+\nr{h}_{\mathfrak{X}(\B)}^{ \theta}+1\right]\,,
\end{flalign}
where we have used \rif{o.7pre} and \rif{coeer}. Here $c$ depends on $\data,\nr{h}_{L^{d}(\B)}$ and $\go_{\psi}(\B)$, but not on $\eps$. Notice that, accordingly to the content of Section \ref{apri} and in particular recalling \rif{a23}, estimate \rif{a52obs} does not hold for any ball, but it holds provided $\rad(\B)\leq R_*$ where, exactly as in \rif{a23}, the threshold radius z
$R_{*}$ depends now also on $h(\cdot), D\psi(\cdot)$ and $\go_{\psi}(\B)$. As for the dependence on this last quantity, there is no vicious circle here since the set function $\Omega_0 \mapsto \go_{\psi}(\Omega_0)$ defined in \rif{settig} is obviously non-decreasing with respect to general open subsets $\Omega_0 \Subset \Omega$ (specifically, fix $\Omega_0 \Subset \Omega$, determine $\go_{\psi}(\Omega_0)$ as in \rif{settig} and proceed for every ball $\B\Subset \Omega_0$ with radius $\rad(\B)$ whose smallness now depends on $\go_{\psi}(\Omega_0)$).  
Estimate \rif{a52obs} can be now used to replace \rif{a52bis} in the approximation scheme of the proof of Theorem \ref{t1}, with $\eps\equiv \eps_j$. The rest of the proof now follows exactly the proof of Theorem \ref{t1} and leads to the conclusion, along with explicit a priori estimates obtainable by \rif{a52obs} after a suitable passage to the limit. Note that also the argument of \rif{e10} can be repeated verbatim, as the obstacle constraint involved here is still convex. Indeed, note that here we are not passing to the limit in the linearized problems \rif{o.2}, but rather directly in the obstacle problems \rif{o.0}; see for instance \cite[Section 5.5]{dejmaa}  and \cite{fumin} for more details on such approximation arguments.  
\end{proof}

\begin{remark}\label{linearire}The linearization procedure leading to \rif{o.3} has been first introduced in \cite{fuobs}, and it is usually employed assuming that $\psi \in W^{2, \infty}(\B)$ \cite{fumin}. A careful examination of the proof reveals that $\psi \in W^{2, n}(\B)\cap W^{1,\infty}(\B)$ suffices, that in turn implies $\partial_{z}F_{\varepsilon}(\cdot,D\psi) \in W^{1,n}(\B; \er^n)$. For this, a first basic remark is that the weak integration-by-parts formula $\int_{\B} h \varphi_{x_i}\dx =-\int_{\B} h_{x_i} \varphi \dx$ holds for every $i \in \{1, \ldots,n\}$, whenever $h \in W^{1,n}(\B)$ and $\varphi  \in W^{1,\gamma}_0(\B)$, for some $\gamma>1$. In turn, this follows by a standard density argument and Sobolev embedding in the limiting case. 
\end{remark}

\subsection{Proof of Theorem \ref{sample9} and additional results} The proof of Theorem \ref{t1ob} offers a route to get the obstacle version of all the other results presented in the unconstrained case; in particular, Theorem \ref{sample9} follows. The key point is again to employ 
the linearization procedure used in \cite{fuobs, fumin} to pass from a variational inequality as in \rif{o.1} to an equation as in \rif{o.2}, to which the estimates in the unconstrained case immediately apply. The whole procedure then works provided the additional assumptions $\gamma \geq 2$ and $g_3, \partial_{zz}F\in L^{\infty}_{\loc}$ are in force as described in the statement of Theorem \ref{t1ob}. With this path being settled, the reader can now easily obtain the constrained extensions of all the results presented in this paper. A few remarks are in order. First, notice that Theorems \ref{sample1}-\ref{sample2} and Theorems \ref{sample4}-\ref{sample6} in the constrained version, are a direct consequence of Theorem \ref{t1ob} and this can be checked exactly as in the unconstrained version. Next, again as for the case of Theorem \ref{t1ob}, Theorems \ref{t2} and \ref{t4} admit a constrained reformulation. In turn, the former would imply a constrained version of Theorem \ref{sample7}. The latter would instead imply a first constrained version of Theorem \ref{sample3}, where the bounds in \rif{lox22} appear in the $<$-version; see also Step 1 of the proof of Theorem \ref{sample3}. As for the full $\leq$-version in \rif{lox22}, it is then necessary to readapt to the obstacle case the arguments from Step 2 and 3 of the proof of Theorem \ref{sample3} in Section \ref{sample3sec}, along the lines of the proof of Theorem \ref{t1ob}. In this respect, the only worth mentioning difference is that the higher integrability lemmas \ref{ilprimoh} and \ref{sopolem} can be easily obtained in the setting of obstacle problems too, starting from the arguments indicated here. For this see also \cite{chde}.  Finally, note that in the case of functionals with $(p,q)$-growth, including the double phase one in \rif{pqfunctional}, verifying the assumptions $\gamma \geq 2$ and $\partial_{zz}F\in L^{\infty}_{\loc}$ boils down to assume that $p\geq 2$, as indeed done in Theorem \ref{sample9}. The additional (micro)assumption on $g_3$ in the statement of Theorem \ref{t1ob} is instead satisfied in every case. 

\begin{remark}\label{obre} In Theorem \ref{t1ob} we can trade the assumption $\gamma \geq 2$ with $\mu>0$ (non-degenerate case). This eventually leads to constrained versions of Theorems \ref{sample1}-\ref{sample2} assuming that $p>1$ instead of $p\geq 2$, provided it is $\mu>0$.  
\end{remark}

\section{Justification of a priori regularity}\label{appendix}
In this final section we justify the claim in \rif{a6}, which is necessary to carry out the rest of the estimates in Sections \ref{apri} and \ref{unisec}. Keeping \rif{z3} in mind, we therefore consider a functional like \eqref{genF}, with the integrand $F$ being convex in the gradient variable, satisfying the structure condition \eqref{0ass}$_1$, where we assume that $\tilde{F}$ and $\tilde F'$ are continuous functions on $\Omega \times [0,\infty)$ such that $\tilde F(x, 0)\equiv 0$, that $\partial_{zz} F, \partial_{xz} F$ are measurable and 
\begin{flalign}\label{irrreg}
\begin{cases}
\ t\mapsto \tilde{F}(x,t) \in C^{1}_{\textnormal{loc}}[0,\infty) &\mbox{for all} \ \ x\in \Omega\\
\   t\mapsto \tilde{F}'(x,t) =: \tilde a(x,t)t \in \textnormal{Lip}_{\loc}[0,\infty)&\mbox{for all} \ \ x\in \Omega\\
\ x\mapsto \tilde{F}'(x,t)\in  W(1;\mathfrak X)(\Omega) \quad &\mbox{for every} \ \ t\geq 0\,.
\end{cases}
\end{flalign}
The last condition means that $|\partial_x  \tilde{F}'(\cdot,t)|\in \mathfrak{X}(\Omega)$ for every $t \geq 0$, where $\mathfrak X(\Omega)$ has been defined in \eqref{lox}.  In particular, this implies that $x\mapsto \tilde{F}'(x,t)\in   W^{1,n}(\Omega)$.  As in Section \ref{apri}, we also assume that the set of non-differentiability points of $t \mapsto \tilde a(x, t)t$ is independent of $x$ and that
\begin{flalign}\label{ir0}
\begin{cases}
\ \nu_0 (\snr{z}^{2}+\mu^{2})^{\gamma/2}- \Lambda\mu^\gamma\le F(x,z)\le \Lambda (\snr{z}^{2}+\mu^{2})^{\gamma/2}\\
\ \snr{\partial_{zz}F(x,z)}\le \Lambda (\snr{z}^{2}+\mu^{2})^{(\gamma-2)/2}\\
\   \nu_0 (\snr{z}^{2}+\mu^{2})^{(\gamma-2)/2}\snr{\xi}^{2} \leq \partial_{zz}F(x,z)\xi\cdot \xi\\
\ \snr{\partial_{xz}F(x,z)}\le \Lambda h(x)(\snr{z}^{2}+\mu^{2})^{(\gamma-1)/2}\\\
f,h\in \mathfrak X(\Omega)
\end{cases}
\end{flalign}
hold for every $x\in \Omega$, and for every $z\in \er^{N\times n}$, provided $\partial_{zz}F(x,z)$ exists, and a.e.\,$x$ and $z,\xi\in \mathbb{R}^{N\times n}$ in the case of \rif{ir0}$_4$. In $\eqref{ir0}$, it is $\gamma>1$, $0<\nu_0\leq 1 \le \Lambda$, $0 < \mu \leq 1$. This functional is of the type considered in \rif{avp} by \rif{zgrowthd} and \rif{dimenticata}, so that the claim in \rif{a6} is justified by the following:
\begin{proposition}\label{tapp}
Let $u\in W^{1,\gamma}_{\loc}(\Omega;\RN)$ be a minimizer of $\mathcal F$ in \eqref{genF}, under assumptions \trif{irrreg}-\trif{ir0}.
Then $
Du \in L^{\infty}_{\textnormal{loc}}(\Omega;\mathbb{R}^{N\times n})$ and $u\in W^{2,2}_{\textnormal{loc}}(\Omega;\mathbb{R}^{N})$. 
\end{proposition} 
%Notice also that by convexity, see \cite[Lemma 2.1]{ma1}, there holds that $\snr{\partial_{z}F(x,z)}\le c(\Lambda,q)\left[1+\snr{z}^{q-1}\right]$.
\begin{proof}
The proof goes now in three different steps, where we essentially revisit and readapt a few hidden facts in the literature. References \cite{DeMi} and \cite{KM} are particularly relevant here. 
\subsubsection*{Step 1: Regularized integrands.} We revisit the procedure we used in \cite[Theorem 4, Step 1]{DeMi} and start fixing a ball $\B \Subset \Omega$ such that $\rad(\B)\leq 1$. We first extend $\tilde F$ by even reflection making it defined on $\Omega \times \er$, i.e., $\tilde F(x,t):=\tilde F(x,-t)$, and then we consider standard, radially symmetric mollifiers $\phi_1 \in C^{\infty}_{\rm{c}}(B_1), \phi_2 \in C^{\infty}_{\rm{c}}(-1,1), \|\phi_1\|_{L^{1}(\mathbb{R}^{n})} =\|\phi_2\|_{L^{1}(\mathbb{R})} =1, \phi_{1,\delta}(x):=\phi_{1}(x/\delta)/\delta^{n}$, $\phi_{2,\delta}(t):=\phi(t/\delta)/\delta$, $B_{3/4}(0)\subset \textnormal{supp} \,\phi_1, (-3/4, 3/4)\subset \textnormal{supp} \,\phi_2$. With $\delta$ such that  $0 < \delta < \left(0,\dist (\B ,\partial \Omega)/2\right)$, we define
\eqn{molli1}
$$
\begin{cases}
\displaystyle \tilde F_{\delta}(x,t):=\int_{-1}^1\int_{B_{1}}\tilde{F}(x+\delta  y,t+\delta s) \phi_{1}(y)  \phi_{2}(s)\, dy \, ds \\
 h_{\delta}(x):=(h*\phi_{1,\delta})(x)
 \end{cases}
$$
for all $x \in \B$ and $t \in \er$. 
\begin{lemma}
If $F_{\delta}(x,z):=\tilde{F}_{\delta}(x,\laz)$, where $\laz:=\sqrt{|z|^2+\delta^2}$, then
\begin{flalign}\label{l0}
\begin{cases}
\ \frac{1}{\cur }(\snr{z}^{2}+\mu^{2}_{\delta})^{\gamma/2}-\cur \mu_{\delta}^{\gamma}\le F_{\delta}(x,z)\le \cur (\snr{z}^{2}+\mu^{2}_{\delta})^{\gamma/2}\\
\ \snr{\partial_{zz}F_{\delta}(x,z)}\le \cur (\snr{z}^{2}+\mu^{2}_{\delta})^{(\gamma-2)/2}\\
\  \frac{1}{\cur }(\snr{z}^{2}+\mu^{2}_{\delta})^{(\gamma-2)/2}\snr{\xi}^{2} \leq \partial_{zz}F_{\delta}(x,z)\xi\cdot \xi \\
\ \snr{\partial_{xz}F_{\delta}(x,z)}\le \cur h_{\delta}(x)(\snr{z}^{2}+\mu^{2}_{\delta})^{(\gamma-1)/2}
\end{cases}
\end{flalign}
hold for every $x\in \Omega'$, $z,\xi\in \mathbb{R}^{N\times n}$, where $\cur \equiv \cur (n,N,\nu_0,\Lambda,\gamma)\geq 1$ is a positive constant, and, as usual, it is $\mu_{\delta}:=\mu+\delta$. Moreover, we have
%Next, the very definitions of $\tilde F_{\delta}(\cdot)$ and $h_\delta$ yield
\eqn{l2}
$$
\left\{
\begin{array}{c}
F_{\delta}(x,z) \to F(x,z)  \ \ \mbox{uniformly on compact subsets of $\overline{\B}\times \er^n$ as $\delta \to 0$}\\[3 pt]
 \nr{h_{\delta}}_{\mathfrak X (B)}\le c \nr{h}_{\mathfrak X(B+\delta B_1(0))}
 \end{array}\right.
$$
whenever $B\subset \B$ is a ball such that $B+\delta B_1(0) \Subset \Omega$, where $c$ is independent of $\delta$ and $h(\cdot)$. 
\end{lemma}
\begin{proof}
The arguments we are going to use here build on those employed in \cite[Section 4.5]{DeMi}. The upper bound in \rif{l0}$_1$ follows directly from the definition \rif{molli1}, while the lower bound follows verbatim from \cite[Section 4.5, last display]{DeMi}, upon   replacing $\gamma-2$ there with $\gamma$ here, and taking the case $n=1$ there. We now go for \rif{l0}$_3$. Denoting $\tilde a_{\delta}(x,t) :=\tilde F_{\delta}'(x,t)/t$ for $t\not =0$, we note that \rif{l0}$_3$ follows from
\eqn{fififi}
$$
\begin{cases}
(t^{2}+\mu^{2}_{\delta})^{(\gamma-2)/2}\lesssim \tilde a_{\delta}(x,t)\\
 (t^{2}+\mu^{2}_{\delta})^{(\gamma-2)/2}\lesssim \tilde a_{\delta}(x,t)+\tilde a_{\delta}'(x,t)t= \tilde F_{\delta}''(x,t)
 \end{cases}
 $$
for $t\geq \delta$. 
This in turn follows observing that 
\eqn{lookatme}
 $$
 \partial_{zz} F_\delta(x,z) = a_{\delta}(x,\laz)\mathds{I}_{N\times n} + a_{\delta}'(x,\laz)\laz \frac{z\otimes z}{\laz^2}\,,
 $$
holds for every $z \in \er^{N \times n}$, and recalling the arguments for \rif{a1}-\rif{a2} and Lemma \ref{help}; also observe that $\laz^2+\mu_\delta^2 \approx 
|z|^2+\mu_\delta^2$. Here, as in the rest of the proof of the Lemma, all the implied constants in the symbol $\lesssim$ depend only on $n,N,\nu_0, \Lambda$ and $\gamma$, but remain otherwise independent of $\delta$.   
To prove \rif{fififi}$_1$, note that $\tilde F_{\delta}(x,t)$ is still such that $\tilde F_{\delta}(x,t)=\tilde F_{\delta}(x,-t)$ for every $t\in \er$ so that $\tilde F_{\delta}'(x,0)=0$. Moreover, from \rif{ir0}$_3$, arguing as in \rif{22}, it follows that $(t^{2}+\mu^{2})^{(\gamma-2)/2} \approx \tilde F''(x,t)$, for $t\not =0$, provided $F''(x,t)$ exists. From this and the definition in \rif{molli1}, following again the same argument in \cite[Section 4.5, last two displays]{DeMi}, we gain $(t^{2}+\mu^{2}_{\delta})^{(\gamma-2)/2} \lesssim \tilde F''_\delta(x,t)$, which is \rif{fififi}$_2$. In turn, integrating this last inequality and using $\tilde F_{\delta}'(x,0)=0$, yields $(t^{2}+\mu^{2}_{\delta})^{(\gamma-2)/2} t/\max\{1, \gamma-1\} \leq 
 \int_0^t (s^{2}+\mu^{2}_{\delta})^{(\gamma-2)/2}\, ds \lesssim \tilde F'_{\delta}(x,t)$ for $t>0$, which is in fact \rif{fififi}$_1$ and \trif{l0}$_3$ is verified. We now derive \rif{l0}$_2$. By \rif{lookatme}, it is sufficient to prove that $\tilde a_{\delta}(x,t) + |\tilde a_{\delta}'(x,t)t| \lesssim (t^{2}+\mu^{2}_{\delta})^{(\gamma-2)/2}$. For this we can again use the argument in \cite[Proof of (4.52)$_3$]{DeMi} to start showing that $|\tilde a_{\delta}(x,t)+\tilde a_{\delta}'(x,t)t|=|\tilde F_{\delta}''(x,t)| \lesssim  (t^{2}+\mu^{2}_{\delta})^{(\gamma-2)/2}$ for $t \geq \delta$. Therefore it remains to prove that $\tilde a_{\delta}(x,t) t =\tilde F_{\delta}'(x,t) \lesssim  (t^{2}+\mu^{2}_{\delta})^{(\gamma-2)/2}t$. Again by $\tilde F_{\delta}'(x,0)=0$, it follows that  
$\tilde F_{\delta}'(x,t) \lesssim  \int_0^t (s^{2}+\mu^{2}_{\delta})^{(\gamma-2)/2}\, ds \leq\max\{1,1/(\gamma-1)\} (t^{2}+\mu^{2}_{\delta})^{(\gamma-2)/2}t$, and this completes the proof of \rif{l0}$_2$. 
For the proof of \rif{l0}$_4$, note that this implies that $|\partial_{x}\tilde F'(x,t)|\leq  \Lambda h(x)(t^{2}+\mu^{2})^{(\gamma-1)/2}$, from which \rif{l0}$_4$ follows from the definitions in \rif{molli1} as for \rif{l0}$_1$. Finally, \rif{l2}$_1$ is again an immediate consequence of the definitions in \rif{molli1} since $\tilde F$ is continuous;  \rif{l2}$_2$ follows from the definition of Lorentz and $L^{2}(\log L)^{\mathfrak {a}}$-norms and the boundedness of maximal operators in the corresponding spaces. 
\end{proof}
We further define $f_{\delta} \in L^{\infty}(\Omega;\er^N)$ as $f_{\delta}(x):=f(x)$ if $\snr{f(x)}\le 1/\delta$, and $f_{\delta}(x):=\delta^{-1}|f(x)|^{-1}f(x)$ otherwise. Finally, we let $u_{\delta}\in u+W^{1,\gamma}_{0}(\B;\mathbb{R}^{N\times n})$, defined as the unique solution of the Dirichlet problem
%\eqn{lpdj}
$$
u_{\delta}\mapsto \min_{w \in u+W^{1,\gamma}_{0}(\B;\RN)} \int_{\B}\left[F_{\delta}(x,Dw)-f_{\delta} \cdot w\right] \dx\,.
$$
Up to now, we have required that $\delta$ is small enough to have  $\delta<\dist (\B ,\partial \Omega)/2$. In the next step we shall choose additional smallness conditions on $\delta$. 

\subsubsection*{Step 2:  $Du \in L^{\infty}_{\textnormal{loc}}(\Omega;\mathbb{R}^{N\times n})$}
Thanks to \rif{l0}, standard regularity theory yields
\eqn{l3}
$$
\begin{cases}Du_{\delta}\in L^{\infty}_{\loc}(\B;\mathbb{R}^{N\times n}),\quad u_{\delta}\in W^{2,2}_{\loc}(\B;\RN)\\
 \partial_{z}F_{\delta}(x,Du_{\delta})\in W^{1,2}_{\loc}(\B;\mathbb{R}^{N\times n})\,. 
 \end{cases}
$$
See, for instance, \cite{tolk}. 
We can therefore proceed exactly as in the proof of Proposition \ref{unile2} arguing on the Euler-Lagrange $-\diver\, \partial_zF_{\delta}(x, Du_{\delta})=f_{\delta}$. This yields the existence of $\delta_0 \equiv \delta_0 (n,N,\nu_0, \gamma, \Lambda, h(\cdot))\in (0,1)$ and $R_{*}\equiv R_{*}(n,N,\nu_0, \gamma, \linebreak \Lambda, h(\cdot))\leq 1$ such that the estimate
\begin{flalign}
& \notag \|Du_\delta \|_{L^\infty(s\B)} \\
& \qquad  \le
\frac{c}{[(1-s)\rad(\B)]^{n/\gamma}}\left[\|F_{\delta}(\cdot, Du_{\delta}) \|_{L^1(\B)}^{1/\gamma}+1\right]+ c\|f_{\delta}\|_{\mathfrak X(\B)}^{1/(\gamma-1)}
\label{bababa}
\end{flalign}
holds whenever $s \in (0,1)$, provided $\delta \leq \min \{\delta_0, \dist (\B ,\partial \Omega)/2\}$ and $\rad(\B)\leq R_{*}$, where $c\equiv c (n,N, \nu_0, \Lambda, \gamma)\geq 1$ is independent of $\delta$. 
Indeed, the setting of Proposition \ref{unile2} applies with the obvious choices $g_{1,\eps}(t)\equiv (t^{2}+\mu^{2}_{\delta})^{(\gamma-2)/2}/\cur $,  $g_{2, \eps}(t)\equiv  \cur (t^{2}+\mu^{2}_{\delta})^{(\gamma-2)/2}$, $g_{3, \eps}(t)\equiv \cur (t^{2}+\mu^{2}_{\delta})^{(\gamma-2)/2}t$, $h(\cdot)\equiv h_{\delta}(\cdot)$ and $T \equiv \mu_{\delta}$; note that 
$(t^{2}+\mu^{2}_{\delta})^{(\gamma-1)/2}\leq \sqrt{2} (t^{2}+\mu^{2}_{\delta})^{(\gamma-2)/2}t$ for $t \geq \mu_{\delta}$. 
Moreover, the only qualitative properties of the solution $u_{\delta}$ needed to argue as in Proposition \ref{unile2} are those in \rif{a6}, that are exactly those in \rif{l3}. Therefore the whole set of estimates developed there applies here verbatim once we observe that \rif{growthd}$_3$ in only needed here when $t \geq T$, which is the case thanks to \rif{l0}$_4$. Notice that, proceeding as in Proposition \ref{unile2}, and recalling \rif{a23uni} (here applied with $h \equiv h_{\delta}$), the radius $R_{*}$ here should exhibit a dependence on $h_\delta$, and therefore ultimately on $\delta$. However, $R_*$ can be made independent of $\delta$, thanks to \rif{l2}$_2$ by further taking $\delta$ small enough, and without creating vicious circles. Specifically, we arrive at \rif{uni8} with the above choice of $g_{1, \eps}, g_{2, \eps}, g_{3, \eps}$ and \rif{a23uni} turns out to be
$c_*\nr{ h_{\delta}}_{\mathfrak{X}(\B)}\le 1/6$, where $c_*$ is independent of $\delta$ but only depends of $\datauni$. With the current choice of the parameters this means that $c_*$ depends only on $n,N,\nu_0, \Lambda, \gamma, \mathfrak{a}$. We use \rif{l2}$_2$ to reduce the last condition to 
$cc_*\nr{h}_{\mathfrak X(\B+\delta B_1(0))}\le 1/6$, where $c$ is the constant appearing in \rif{l2}$_2$ and it is independent of $\delta$. Therefore, by absolute continuity we find $\delta_0, R_* \equiv \delta_0, R_*(n,N,\nu_0, \gamma, \Lambda, h(\cdot))$ as described above, such that the last inequality is satisfied. This allows to set inequality \rif{bababa} free from any dependence on $\delta$. 
Next, using \rif{ir0}$_1$ and \rif{l0}$_1$, and finally the minimality of $u_\delta$ in \rif{bababa}, we gain
\begin{flalign}
 &\|Du_\delta \|_{L^\infty(s\B)} \notag  \\
 & \qquad \le
\frac{c}{[(1-s)\rad(\B)]^{n/\gamma}}\left[\|F(\cdot, Du) \|_{L^1(\B)}^{1/\gamma}+1\right]+ c\|f\|_{\mathfrak X(\B)}^{1/(\gamma-1)}  \label{bababa2}
\end{flalign}
with $c$ being independent of $\delta$. 
From this, a standard convergence argument based on \rif{l2}$_1$ (see again the proof of Theorem \ref{t1}) extracting a subsequence $\{u_{\delta}\equiv u_{\delta_j}\}$ such that $u_\delta\rightharpoonup^{*} u$ weakly in $W^{1,\infty}(s\B;\mathbb{R}^{N})$, leads to 
$$
\|Du \|_{L^\infty(s\B)} \le
\frac{c}{[(1-s)\rad(\B)]^{n/\gamma}}\left[\|F(\cdot, Du) \|_{L^1(\B)}^{1/\gamma}+1\right]+ c\|f\|_{\mathfrak X(\B)}^{1/(\gamma-1)}\,.
$$
As this holds whenever $\B \Subset \Omega$, then $Du$ is locally bounded. 

\subsubsection*{Step 3: $u \in W^{2,2}_{\textnormal{loc}}(\Omega;\mathbb{R}^{N})$} For this we shall revisit some arguments from \cite[Theorems 4.5-4.6]{KM}. We test the weak formulation of the Euler-Lagrange system $-\diver \, \partial_{z}F_{\delta}(x,Du_{\delta})=f_{\delta} $ by $D_s \varphi$, for $s \in \{1, \dots, n\}$ and $\varphi \in C^{\infty}_{0}(\B)$; integration by parts yields
\eqn{usingusing}
$$
 \int_{\B}  D_s [\partial F_{\delta}(x,Du_{\delta})] \cdot D\varphi \dx = - \int_{\B} f_{\delta}\cdot  D_s\varphi \dx\,.
$$
We then take $\eta\in C_0^\infty(\B/2;[0,1])$ with $\eta \equiv 1$ in $\B/4$, $\|D\eta\|_{L^\infty(\B)}\lesssim 1/ [\rad(\B)]$, and we define $\varphi \equiv \varphi_s := \eta^2D_s u_{\delta}$ so that $\varphi  \in W^{1,2}_0(\B; \er^N)$ and has compact support in $\B$. Using $\varphi$ as test function in \rif{usingusing}, summing over $s\in \{1,\cdots n\}$ and taking~\eqref{l0} into account, yields
\begin{flalign*}
&\int_{\B}  (|Du_{\delta}|^2+\mu_{\delta}^2)^{\frac{\gamma-2}{2}} |D^2u_{\delta}|^2\eta^2 \dx\\
&\quad  \leq c \int_{\B} \eta   (|Du_{\delta}|^2+\mu_{\delta}^2)^{\frac{\gamma-2}{2}} |D^2u_{\delta}| |D u_{\delta}| |D\eta| \dx
 \\
 &  \qquad  \quad  + c \int_{\B} h_{\delta} (|Du_{\delta}|^2+\mu_{\delta}^2)^{\frac{\gamma-1}{2}}[\eta^2|D^2u_\delta|+ 
 \eta |D\eta||Du_\delta|]\dx\\ & \qquad \qquad +\sum_{s=1}^n \int_{B} |f_{\delta}||D\varphi_s| \dx\,.
\end{flalign*}
Estimating the third integral in a standard way (see for instance in \cite[pag. 395]{KM}) we get
\begin{flalign}
&\nonumber \int_{\B}  (|Du_{\delta}|^2+\mu_{\delta}^2)^{\frac{\gamma-2}{2}} |D^2u_{\delta}|^2\eta^2 \dx
\\
& \qquad  \leq c[\rad(\B)]^{-2} \int_{\B} \eta [1+h_{\delta}^2] (|Du_{\delta}|^2+\mu_{\delta}^2)^{\gamma/2}\dx  \label{bababa3} 
 \\& \hspace{1cm}+c[\rad(\B)]^{-1}\|f\|_{L^2(\B)}\|Du_\delta\|_{L^2(\B/2)}+ c\|f\|_{L^2(\B)}\|\eta^2 D^2u_\delta\|_{L^2(\B)}\,.\notag
\end{flalign}
The involved constant $c$ only depends on $n,N,\nu, \Lambda$ and $\gamma$ and is otherwise independent of $\delta \in (0,1)$.
We now set 
$
M:= \sup_{\delta}\,  \|Du_\delta\|_{L^\infty(\B/2)}+1, 
$
which is a finite quantity by \rif{bababa2}. We start considering the case $\gamma \geq 2$, where we have 
\begin{flalign*}
\mu^{\gamma-2}\|\eta D^2u_\delta\|_{L^2(\B)}^2
& \leq  c\|1+h_{\delta}\|_{L^2(\B)}^2(M^2+\mu_{\delta}^2)^{\gamma/2} + 
c\|f\|_{L^2(\B)}M \\& \qquad + c\|f\|_{L^2(\B)}\|\eta^{2} D^2u_\delta\|_{L^2(\B)}
\end{flalign*}
where $c \equiv c(n,N,\nu_0, \Lambda, \gamma, \rad(\B)) $
and therefore, via Young inequality, we get
\begin{flalign*}
\| D^2u_\delta\|_{L^2(\B/4)}^2
 & \leq  c\mu^{2-\gamma}\|1+h_{\delta}\|_{L^2(\B)}^2[M^2+1]^{\gamma/2} \\
 & \qquad    +c\mu^{2-\gamma}\|f\|_{L^2(\B)}M+ c\mu^{2(2-\gamma)}\|f\|_{L^2(\B)}^2
\end{flalign*}
which is a uniform (with respect to $\delta$) local bound for $\{D^2u_\delta\}$:
\eqn{boundy}
$$
\|D^2u_\delta\|_{L^2(\B/4)}\leq c (n,N,\nu_0, \gamma, \Lambda, \|h\|_{L^2(\B)},  \|f\|_{L^2(\B)}, \rad(\B), M, \mu)\,.
$$
In the case $1< \gamma <2$, we can argue exactly as after \rif{bababa3}, but replacing $\mu$ by $M$, thereby getting again \rif{boundy}. Starting from \rif{boundy}, using the same approximation argument for the proof of Theorem \ref{t1} and in Step 2 here, we can let $\delta \to 0$ (via a subsequence) in \rif{boundy} finally getting a local upper bound for $D^2u$ in $L^2$. The assertion then follows via the usual covering argument.   
\end{proof}

\subsection*{Acknowledgements.}  This work is supported by the Engineering and Physical Sciences Research Council (EPSRC): CDT Grant Ref. EP/L015811/1, and by the University of Parma via the project ``Regularity, Nonlinear Potential Theory and related topics".

\address{
Dipartimento di Matematica "Giuseppe Peano"\\ Universit\`a di Torino\\
Via Carlo Alberto 10, I-10123, Torino, Italy\\
\email{cristiana.defilippis@unito.it}
\and
Dipartimento SMFI\\
Universit\`a di Parma\\
Parco Area delle Scienze 53/a, I-43124, Parma, Italy \\
email:{giuseppe.mingione@unipr.it}}


\begin{thebibliography}{}

\bibitem {Baroni} 
P.~Baroni, Riesz potential estimates for a general class of quasilinear equations, 
\emph{Calc. Var. \& PDE} \textbf{{\bf 53}} (2015), 803--846.

\bibitem{BCM} P.~Baroni, M.~Colombo G.~Mingione, Regularity for general functionals with double phase,  \emph{Calc. Var. \& PDE} \textbf{{\bf 57}}:62 (2018), 48 pages. 

\bibitem{BJW}  E. N.~Barron, R. R.~Jensen  \& C. Y.~Wang, Lower semicontinuity of $L^\infty$ functionals, {\em Ann. Inst. H. Poincar\'e, Anal. Non Lin\'eaire} \textbf{{\bf 18}} (2001), 495--517.

\bibitem{BM} L. Beck \& G. Mingione, Lipschitz bounds and nonuniform ellipticity, \emph{Comm. Pure Appl. Math.}
\textbf{{\bf 73}} (2020), 944--1034. 

\bibitem{BS} P. Bella \& M. Sch\"affner, Local boundedness and Harnack inequality for solutions of linear nonuniformly elliptic equations, \emph{Comm. Pure Appl. Math.}  \textbf{{\bf 74}} (2021), 453--477.


\bibitem{Bi} M. Bildhauer, \emph{Convex variational problems. Linear, nearly linear and anisotropic growth conditions},  Lecture Notes in Mathematics, 1818. Springer-Verlag, Berlin, 2003. x+217 pp.

\bibitem{BF1} M. Bildhauer \& M. Fuchs, $C^{1, \alpha}$-solutions to nonautonomous anisotropic variational problems, \emph{Calc. Var. \& PDE} \textbf{{\bf 24}} (2005), 309--340. 

\bibitem{bcrippa} P. Bouchut \& G. Crippa, Lagrangian flows for vector fields with gradient given by
a singular integral, \emph{J. Hyperbolic Diff. Equ.} \textbf{{\bf11}} (2018), 813--854.

\bibitem{BB2} P. Bousquet \& L. Brasco, Lipschitz regularity for orthotropic functionals with nonstandard growth conditions, \emph{Rev. Mat. Iberoamericana}, \textbf{{\bf36}} (2020), 1989--2032.

\bibitem {BOR} 
S.-S.~Byun, J. Ok \& S. Ryu, Global gradient estimates for elliptic equations of $p(x)$-Laplacian type with BMO nonlinearity, \emph{J. Reine Angew. Math.} \textbf{{\bf715}} (2016), 1--38.


\bibitem {BO} 
S.-S.~Byun \& J.~Oh, Global gradient estimates for nonuniformly elliptic equations, \emph{Calc. Var. \& PDE} \textbf{{\bf56}}, 36 Pages, 2017.


\bibitem {BY} 
S.-S.~Byun \& Y.~Youn, Riesz potential estimates for a class of double phase problems,  \emph{J.~Diff. Equ.} \textbf{{\bf264}} (2018), 1263--1316. 

\bibitem {BLO} S.-S.~Byun, S. Liang \& J. Ok, Irregular double obstacle problems with Orlicz growth,  \emph{J. Geom. Anal.} \textbf{{\bf30}}  (2020), 1965--1984. 




\bibitem {Cellinastaicu} 
A.~Cellina \& V.~Staicu, On the higher differentiability of solutions to a class of variational problems of fast growth, \emph{Calc. Var. \& PDE} \textbf{{\bf57}}:66 (2018). 

\bibitem {chde} I. Chlebicka \& C. De Filippis, Removable sets in non-uniformly elliptic problems,  \emph{Ann. Mat. Pura Appl. (IV)} \textbf{{199}}  (2020), 619--649. 



\bibitem {choe}  H. J. Choe, A regularity theory for a general class of quasilinear elliptic partial differential equations and obstacle problems, \emph{Arch. Ration. Mech. Anal.} \textbf{{\bf 114}} (1991), 383--394.

\bibitem {CL}  H. J. Choe \& J.L. Lewis, On the obstacle problem for quasilinear elliptic equations of $p$-Laplace type, {\em SIAM J. Math. Anal.} \textbf{{\bf 22}} (1991), 623--638. 




\bibitem {CiGA} A.~Cianchi, Maximizing the $L^\infty$-norm of the gradient of solutions to the Poisson equation, \emph{J. Geom. Anal.} \textbf{{2}} (1992), 499--515.

\bibitem{CM0}
A.~Cianchi \& V. G.~Maz'ya, Global Lipschitz regularity for a class of quasilinear equations,  
\emph{Comm. PDE}  \textbf{{\bf 36}} (2011), 100--133.

\bibitem{CM1}
A.~Cianchi \& V. G.~Maz'ya, Global boundedness of the gradient for a
class of nonlinear elliptic systems, \emph{Arch. Ration. Mech. Anal.} \textbf{{\bf 212}} 
(2014), 129--177.

\bibitem{CM2}
A.~Cianchi \& V. G.~Maz'ya, Gradient regularity via rearrangements for $p$-Laplacian type elliptic boundary value problems, \emph{J. Eur. Math. Soc.} \textbf{{\bf 16}} (2014),  571--595. 


\bibitem{CM3}
A.~Cianchi \& V. G.~Maz'ya, Global gradient estimates in elliptic problems under minimal data and domain regularity,  \emph{Comm. Pure Appl. Anal.} \textbf{{\bf14}} (2015), 285--311.

\bibitem{CM4}
A.~Cianchi \& V. G.~Maz'ya, Second-order two-sided estimates in nonlinear elliptic problems, \emph{Arch. Ration. Mech. Anal.} \textbf{{\bf229}} (2018), 569--599. 

\bibitem{CoM0}
M.~Colombo \& G.~Mingione, Regularity for double phase variational problems, \emph{Arch. Ration. Mech. Anal.} \textbf{{\bf215}} (2015), 443--496. 

\bibitem{CoM}
M.~Colombo \& G.~Mingione, Bounded minimisers of double phase variational integrals, \emph{Arch. Ration. Mech. Anal.} \textbf{{\bf218}} (2015), 219--273.


\bibitem{crippadelellis}
G.~Crippa \& C. De Lellis, Estimates and regularity results for the DiPerna-Lions flow, \emph{J. Reine Angew. Math.} 
\textbf{{\bf616}} (2008), 15--46.


\bibitem{dele}V. De Cicco \& G. Leoni, A chain rule in $L^1(\textnormal{div};\Omega)$ and its applications to lower semicontinuity, \emph{Calc. Var. \& PDE} \textbf{{\bf 19}} (2004), 23--51.

\bibitem{dejmaa} C. De Filippis, Regularity results for a class of nonautonomous obstacle problem with $(p,q)$-growth, \emph{J. Math. Anal. Appl.}, \textbf{{\bf 501}} (2021) 123450. 


\bibitem{demicz} C. De Filippis \& G. Mingione, A borderline case of Calder\'on-Zygmund estimates for nonuniformly elliptic problems, \emph{St.  Petersburg Math. J.} \textbf{{\bf31}} (2020), 455--477. 

\bibitem{DeMi1} C. De Filippis \& G. Mingione, Manifold constrained nonuniformly elliptic problems, \emph{J. Geom. Anal.} \textbf{{\bf30}}  (2020), 1661--1723. 

\bibitem{DeMi} C. De Filippis \& G. Mingione, On the regularity of minima of nonautonomous functionals, \emph{J. Geom. Anal.}  \textbf{{\bf30}} (2020), 1584--1626. 



\bibitem{DM} T. Di Marco \& P. Marcellini, A-priori gradient bounds for elliptic systems under either slow or fast growth conditions, \emph{Calc. Var. \& PDE} \textbf{{\bf59}}:120 (2020). 


\bibitem{DiE} L. Diening \& F. Ettwein, Fractional estimates for non-differentiable elliptic systems with general growth, \emph{Forum Math.} \textbf{{20}} (2008), 523--556. 



 \bibitem {diening} L.  Diening, P. Harjulehto, P. H\"ast\"o \& M. Ruzicka, {\em Lebesgue and Sobolev spaces with variable exponents}, 2017. Springer, Heidelberg, 2011.


\bibitem{DoK} G. Dolzmann \& J. Kristensen, Higher integrability of minimizing Young measures, \emph{Calc. Var. \& PDE} \textbf{{22}} (2005), 283--301. 

\bibitem{DE} D. M.~Duc \& J.~Eells, Regularity of exponentially harmonic functions, 
 \emph{Internat. J. Math.} \textbf{{2}} (1991), 395--408. 
 
 \bibitem{EMM} M. Eleuteri, P. Marcellini, \& E. Mascolo, Lipschitz estimates for systems with ellipticity conditions at infinity, {\em Ann. Mat. Pura Appl. (IV)} \textbf{{195}}  (2016), 1575--1603. 


\bibitem{sharp} L. Esposito, F. Leonetti \& G. Mingione, Sharp regularity for functionals with $(p,q)$ growth, \emph{J. Diff. Equ.} \textbf{{204}}  (2004), 5--55. 


\bibitem{evans} L. C. Evans, Some new PDE methods for weak KAM theory, {\em Calc. Var. \& PDE} \textbf{{\bf17}} (2003), 159--177. 

\bibitem{FGS} M. Focardi, F. Geraci \& E. Spadaro, Quasi-monotonicity formulas for classical obstacle problems with Sobolev coefficients and applications, \emph{J. Optim. Theory Appl.} \textbf{{\bf184}} (2020), 125--138. 

\bibitem{FMM} I. Fonseca, J. Mal\`y \& G. Mingione, 
Scalar minimizers with fractal singular sets, \emph{Arch. Ration. Mech. Anal.} \textbf{{\bf172}} (2004), 295--307.



\bibitem{fuobs} M. Fuchs, H\"older continuity of the gradient for degenerate variational inequalities, \emph{Nonlinear Anal.} \textbf{{\bf15}} (1990), 85--100. 

\bibitem{fumin} M. Fuchs \& G. Mingione, 
Full $C^{1,\alpha}$-regularity for free and constrained local
minimizers of elliptic variational integrals with nearly linear
growth, {\em manuscripta math.} \textbf{{\bf102}}  (2000), 227--250.

\bibitem{ivanov} A. V. Ivanov, Quasilinear degenerate and nonuniformly elliptic and parabolic equations of second order, {\em Proc. Steklov Inst. Math.} 1984, no. 1 (160), xi+287 pp.

\bibitem{KL} A. Karppinen \& M. Lee, H\"older continuity of the minimizer of an obstacle problem with generalized Orlicz growth, \emph{Inter. Math. Res. Notices}, \url{https://doi.org/10.1093/imrn/rnab1502020}. 

\bibitem{kilmal} T.~Kilpel{\"a}inen and J.~Mal{\'y}, The Wiener test and potential
estimates for quasilinear elliptic equations, \emph{Acta Math.} \textbf{{\bf 172}}
(1994), 137--161.


\bibitem{KS} D. Kinderlehrer \& G. Stampacchia, {\em An introduction to variational inequalities and their applications.} New York-San Francisco-London: Academic Press, 1980

\bibitem {KM} J.~Kristensen \& G.~Mingione, Boundary regularity in variational problems, \emph{Arch.~Ration.~Mech.~Anal.}  \textbf{{\bf198}} (2010), 369--455. 

\bibitem {KSt} J.~Kristensen \& B. Stroffolini, The Gehring lemma: dimension free estimates,  \emph{Nonlinear Anal.} \textbf{{\bf177}} (2018), 601--610.

\bibitem{KMbull} T.~Kuusi \& G.~Mingione, Guide to nonlinear potential estimates, \emph{Bull. Math. Sci.} \textbf{{\bf4}} (2014), 1--82. 

\bibitem{KMstein} T.~Kuusi \& G.~Mingione, A nonlinear Stein theorem, 
\emph{Calc. Var. \& PDE} \textbf{{\bf51}} (2014), 45--86.

\bibitem{KMjems} T.~Kuusi \& G.~Mingione,
Vectorial nonlinear potential theory, \emph{J.~Europ.~Math.~Soc.} \textbf{{\bf20}} (2018), 929--1004. 

\bibitem{HH} P. Harjulehto \& P. H\"ast\"o, \emph{Orlicz spaces and generalized Orlicz spaces},  Lecture Notes in Mathematics, 2236. Springer, 2019.

\bibitem{HO} P. H\"ast\"o \& J. Ok, Maximal regularity for local minimizers of nonautonomous functionals, \emph{J. Europ. Math. Soc.}, to appear. 

\bibitem{HS} J. Hirsch \& M. Sch\"affner, Growth conditions and regularity, an optimal local boundedness result, {\em Comm. Cont. Math} \textbf{{\bf23}} (2021), 2050029. 

\bibitem{mazyaex} T. Jin, V. G. Maz'ya \& J. Van Schaftingen, Pathological solutions to elliptic problems in divergence form with continuous coefficients.
{\em C. R. Math. Acad. Sci. Paris} \textbf{{\bf 347}} (2009), 773--778.

\bibitem {LU} O. A.~Ladyzhenskaya \& N. N.~Ural'tseva, Local estimates for gradients of solutions of nonuniformly elliptic
and parabolic equations, \emph{Comm.~Pure Appl.~Math.} \textbf{{23}} (1970), 677--703. 

\bibitem {liebe1} G. M. Lieberman, The natural generalization of the
natural conditions of Ladyzhenskaya and Ural'tseva
 for elliptic equations,  \emph{Comm. PDE}  \textbf{{16}} (1991),  311--361.
 
\bibitem {liebe2} G. M. Lieberman, 
Regularity of solutions to some degenerate double obstacle problems, 
 \emph{Indiana Univ. Math. J.} \textbf{{\bf 40}} (1991), 1009--1028.





  
\bibitem{M1} P.~Marcellini, Regularity and
existence of solutions of elliptic equations with $p,q$-growth
conditions, \emph{J.~Diff.~Equ.} \textbf{{\bf90}} (1991), 1--30.



\bibitem{M2} P.~Marcellini, Everywhere regularity for a class of elliptic systems without
growth conditions, \emph{Ann.~Scuola Norm.~Sup.~Pisa Cl.~Sci.~(IV)} 
\textbf{{\bf23}} (1996), 1--25.






\bibitem{M3} P.~Marcellini, Regularity for some scalar variational problems under general
growth conditions,  \emph{J. Optim. Theory Appl.} \textbf{{\bf90}} (1996), 
161--181.





\bibitem{M4}
P.~Marcellini, Growth conditions and regularity for weak solutions to nonlinear elliptic pdes,  \emph{J. Math. Anal. Appl.}, \textbf{{501}} (2021), 124408. 



\bibitem{MP}
P.~Marcellini \& G.~Papi, Nonlinear elliptic systems with general growth, \emph{J.~Diff.~Equ.}
\textbf{{\bf221}} (2006), 412--443. 



\bibitem {MH} V. G.~Maz'ya \& M.~Havin,  A nonlinear potential theory,  \emph{Russ.~Math.~Surveys} \textbf{{27}} (1972), 71--148.

\bibitem{dark}
G. Mingione \& V. R\v adulescu, Recent developments in problems with nonstandard growth and nonuniform ellipticity, \emph{J. Math. Anal. Appl.} \textbf{{\bf 501}} (2021), article no. 125197.

    
\bibitem{RR}
V.~Radulescu \& D.~Repovs, \emph{Partial Differential Equations with Variable Exponents: Variational Methods and Qualitative Analysis}. Chapman \& Hall/CRC Monographs and Research Notes in Mathematics. 2015. 

\bibitem{sarason} D. Sarason, Functions of vanishing mean oscillation, \emph{Trans. Amer. Math. Soc.} \textbf{{\bf 207}} (1975), 391--405. 


\bibitem{sil} S. Sil, 
Nonlinear Stein theorem for differential forms, 
\emph{Calc. Var. \& PDE} \textbf{{\bf 58}}:154 (2019), 32 pp.



\bibitem{Simon} L.~Simon, Interior gradient bounds for nonuniformly elliptic equations, \emph{Indiana Univ. Math. J.} \textbf{{\bf25}} (1976), 821--855.



\bibitem {St} E. M.~Stein, 
Editor's note: the differentiability of functions in $\er^n$, 
\emph{Ann.~of Math. (II)}  \textbf{{\bf113}} (1981), 383--385. 




\bibitem{SY1} V. \v{S}ver\'{a}k \& X.~Yan, Non-Lipschitz
minimizers of smooth uniformly convex variational integrals,   
\emph{Proc.~Natl.~Acad.~Sci.~USA} \textbf{{\bf99/24}} (2002), 15269--15276.

\bibitem{tolk} P. Tolksdorf, Everywhere-regularity for some quasilinear systems with a lack of ellipticity, {\em Ann. Mat. Pura Appl. (IV)} \textbf{{\bf134}} (1983), 241--266.


\bibitem{Uh}
K.~Uhlenbeck, Regularity for a class of nonlinear elliptic systems, 
\emph{Acta Math.} \textbf{{\bf138}} (1977), 219--240.


\bibitem{Ur} N. N. Ural'tseva, Degenerate quasilinear elliptic systems, \emph{Zap.~Na.
Sem.~Leningrad.~Otdel.~Mat.~Inst.~Steklov.~(LOMI)} \textbf{{7}} (1968), 184--222.




\bibitem{UU} N. N.~Ural'tseva \& A. B.~Urdaletova, The boundedness of
the gradients of generalized solutions of degenerate quasilinear
nonuniformly elliptic equations, \emph{Vestnik Leningrad Univ. Math.}
\textbf{{\bf19}} (1983) (russian) english. tran.: \textbf{{\bf16}} (1984), 263--270.

\bibitem{Z1} V. V.~Zhikov,  Averaging of functionals of the calculus of variations and
elasticity theory,  \emph{Izv. Akad. Nauk SSSR Ser. Mat.}  \textbf{{\bf50}} (1986),
 675--710.
 
\bibitem{Z2} V. V.~Zhikov, On some variational problems, \emph{Russian J. Math. Phys.}  \textbf{{\bf5}} (1997), 105--116. 
\end{thebibliography}
\end{document}